\documentclass[11pt]{article}
\usepackage{amsfonts,amsmath,amssymb,amsthm}
\usepackage{xcolor}
\usepackage[skins,breakable]{tcolorbox}
\usepackage{graphicx}
\usepackage[numbers,sort&compress]{natbib}

\definecolor{extensionteal}{HTML}{99E0D9}
\definecolor{extensiongrey}{HTML}{F2F2F2}
\usepackage[a4paper,margin=26mm]{geometry}
\usepackage{mathtools,microtype}
\usepackage{enumerate}
\definecolor{darktopaz}{HTML}{0B5F7A}
\definecolor{blackpurple}{HTML}{1D0242}
\definecolor{blackred}{HTML}{750D2C}
\definecolor{blackblue}{HTML}{01033B}
\definecolor{blackgreen}{HTML}{072612}
\definecolor{quasipseudodarkgreen}{HTML}{022B19}
\definecolor{midgreen}{HTML}{0B5C17}
\definecolor{midred}{HTML}{570432}
\definecolor{midpurple}{HTML}{6B0941}

\usepackage{algorithm,algpseudocode}

\usepackage[
    unicode,
    breaklinks,
    colorlinks=true,
    linkcolor=darktopaz,
    citecolor=darktopaz,
    urlcolor=darktopaz
]{hyperref}

\usepackage{enumitem}

\usepackage{booktabs}
\usepackage{tabularx}
\usepackage{array}

\usepackage{float}

\usepackage{verbatim}

\tcbset{
    greytheorem/.style={
        enhanced,
        breakable,
        colback=black!5,       % Pale-grey background.
        colframe=black!75,     % Dark-grey border.
        boxrule=0.5mm,         % Border thickness.
        arc=1mm,              % Gently rounded corners.
        boxsep=1mm,
        left=4mm,
        right=4mm,
        top=2mm,
        bottom=2mm,
        before skip=12pt,
        after skip=12pt
    }
}
\tcbset{
    lightgreytheorem/.style={
        enhanced,
        breakable,
        colback=black!2,       % Very light grey background.
        colframe=black!20,     % Light grey border.
        boxrule=0.5mm,
        arc=1mm,
        boxsep=1mm,
        left=4mm,
        right=4mm,
        top=2mm,
        bottom=2mm,
        before skip=12pt,
        after skip=12pt
    }
}

\tcbset{
    superlightgreytheorem/.style={
        enhanced,
        breakable,
        colback=black!0.6,     % Extremely light grey background.
        colframe=black!8,      % Very subtle grey border.
        boxrule=0.35mm,
        arc=1mm,
        boxsep=1mm,
        left=4mm,
        right=4mm,
        top=2mm,
        bottom=2mm,
        before skip=12pt,
        after skip=12pt
    }
}
\newtcolorbox{extensionbox}{
    colback=extensionteal!12!extensiongrey,
    colframe=black!70,
    boxrule=0.8pt,
    arc=2mm,
    boxsep=0pt,
    left=4mm,
    right=4mm,
    top=2mm,
    bottom=2mm,
    before skip=10pt,
    after skip=10pt,
    before upper={
        \setlength{\abovedisplayskip}{0pt}
        \setlength{\belowdisplayskip}{0pt}
        \setlength{\abovedisplayshortskip}{0pt}
        \setlength{\belowdisplayshortskip}{0pt}
    }
}

\allowdisplaybreaks[2]
\numberwithin{equation}{section}
\newtheorem{setting}{Setting}[section]
\newtheorem{dataformat}{Data Format}[section]
\newtheorem{theorem}{Theorem}[section]
\newtheorem{lemma}[theorem]{Lemma}
\newtheorem{proposition}[theorem]{Proposition}
\newtheorem{highlightedproposition}[theorem]{Proposition}
\newtheorem{corollary}[theorem]{Corollary}

\theoremstyle{definition}
\newtheorem{definition}[theorem]{Definition}

\theoremstyle{remark}
\newtheorem{remark}[theorem]{Remark}
\tcolorboxenvironment{theorem}{greytheorem}
\tcolorboxenvironment{corollary}{superlightgreytheorem}
\tcolorboxenvironment{highlightedproposition}{lightgreytheorem}

\newcommand{\R}{\mathbb{R}}
\newcommand{\N}{\mathbb N}
\newcommand{\Nzero}{\mathbb N_0}
\newcommand{\dd}{\mathrm d}
\newcommand{\Lip}{\operatorname{Lip}}
\newcommand{\Var}{\operatorname{Var}}
\newcommand{\KL}{\operatorname{KL}}
\newcommand{\Cov}{\operatorname{Cov}}
\newcommand{\conv}{\operatorname{conv}}
\newcommand{\op}{\mathrm{op}}
\newcommand{\eqdef}{\ensuremath{\stackrel{\mbox{\upshape\tiny def.}}{=}}}

\newcommand{\norm}[1]{\left\lVert#1\right\rVert}

\usepackage{amsmath,amssymb}
\usetikzlibrary{arrows.meta,positioning,calc}

\title{%
    A Computationally-Efficient Closed-Form $C^{s,1}$-Extension Formula
    \\[0.6em]
    {\large
    Using
        $\mathcal{O}(N)$ storage and compilation cost
        and $\mathcal{O}(\log N)$ query time
    \\
    Which is $\mathbb{R}_{\operatorname{exp}}$-Definable and a.e.\ Real-Analytic
    \par
    }%
}
\author{%
    Anastasis Kratsios\thanks{Corresponding author.}
    \,\thanks{Department of Mathematics and Statistics; 
        McMaster University, the Vector Institute, and The Fields Institute's Centre for Mathematical AI.
        \href{mailto:kratsioa@mcmaster.ca}
        {\nolinkurl{kratsioa@mcmaster.ca}}.%
    }
    \and
    Philipp Zimmermann\thanks{Institute of Mathematics;
        \'{E}cole Polytechnique F\'{e}d\'{e}rale de Lausanne (EPFL).
        \href{mailto:philipp.zimmermann@epfl.ch}
        {\nolinkurl{philipp.zimmermann@epfl.ch}}.%
    }
}
\date{}
\date{ }

\begin{document}
\maketitle
\vspace{-2.5em}

\begin{abstract}
We identify explicit \textit{closed-form} and variational \textit{formulae} for reconstructing a $C^{s,1}$-function $f:\mathbb{R}^d\to\mathbb{R}$ from its exact values and derivatives through order $s$ observed at $N$ distinct points in $[0,1]^d$, with bounds on the reconstruction's global $C^{s,1}$-seminorm and Lipschitz constant on $[0,1]^d$ that are \textit{independent of the sample size $(N)$} (the latter is the ``sharp Whitney'' condition).  For $s\ge2$, our formula is real-analytic away from the datapoints and is definable in the o-minimal structure $\mathbb{R}_{\operatorname{exp}}$.
Our \textit{nonlinear} extension formula  coincides with the $C^{0,1}$-formula of~\cite{McShane1934} and the $C^{1,1}$-formula of~\cite{AzagraLeGruyerMudarra2018,le2015sup}, and is new in the $s\ge 2$ regime.

In the $s\ge 2$ regime, our formula is computable by a parallelized circuit (using only elementary unary and binary real-operations) with $\mathcal{O}(N)$ gates (storage space) which evaluates (forward pass) in $\mathcal{O}(\log(N))$ time-complexity; initialized (one-time work) with time-complexity $\mathcal{O}(N)$.  Thus, although our formula is nonlinear it improves on the initialization cost (one-time work)  of~\cite{FeffermanKlartag2009II} by a \textit{logarithmic factor}; which matching its run-time costs.
\end{abstract}


\section{Introduction}
\label{s:Intro}

In modern computational disciplines such as machine learning, high dimensional statistics, and scientific computing; one is typically faced with the task on reconstructing an unobserved function $f:\mathbb{R}^d\to \mathbb{R}$ from finitely many ($N$) samples $\mathbb{X}^N\eqdef \{X_n\}_{n=1}^N\subset \mathbb{R}^d$; that is, identifying a function which \textit{coincides with the sample data}.  
While inferring a function which coincides with the sample measurements is rather straightforward using finitely parameterized function families; e.g.\ splines or neural networks~\cite{vardi2022on,hong2026closed}; doing so while \textit{match} the order of regularity of $f$ is exceedingly complicated.  
For instance, if $f$ is $1$-Lipschitz then the McShane extension formula~\cite{McShane1934} provides such a $1$-Lipschitz reconstruction, and there are also neural networks which exist which provide a $\mathcal{O}_d(1)$-Lipschitz reconstruction~\cite{hong2024bridging} or approximate reconstruction~\cite{murari2026approximation}.

The situation complicates significantly $f$ is of $C^{s,1}$-regularity and we wish to match the size of its order $s$-derivatives, while also \textit{matching} the observed partial derivatives of $f$ at the sample points $\mathbb{X}^N$.  When $s\ge 1$, we have access to the data 
\begin{equation}
\label{eq:the_data}
 \mathcal D^N\eqdef\{(X_n,P_n)\}_{n=1}^N,
 \qquad
 P_n(x)=J_{X_n}^sf(x)
 \eqdef\sum_{k=0}^s\frac1{k!}D^kf(X_n)[(x-X_n)^k].
\end{equation}
Here $D^kf(X_n)[h^k]=D^kf(X_n)[h,\ldots,h]$ for $k\ge1$, and the $k=0$ summand is $f(X_n)$.
While~\cite{AzagraLeGruyerMudarra2018} recently identified a formula for the $s=1$ case%
\footnote{The formula in~\cite{AzagraLeGruyerMudarra2018} becomes finitely parameterized in our setting as $\mathbb{X}^N$ is finite.}, no known neural network approximator of $C^{s,1}$-functions is known to interpolate derivative data while enjoying first partial derivatives on the order of $\mathcal{O}_{d,s}(1)$; indeed, this is because, they glue together Taylor polynomials of the function using bump functions~\cite{lu2021deep} or a partitions of unity~\cite{de2021approximation} both of which already create an explosion in the neural network's first-order derivatives as the number of samples $N$ diverge.

\paragraph{Our Closed-Form Formula (the $s\ge 2$ case):}
We now present our formula in the case where $s\ge 2$.  
This is because, our construction exhibits a \textit{phase transition}, wherein our extension formula is both: \textit{analytic} away from the data and \underline{closed form} formula in this regime is simply
\begin{extensionbox}
%%% LONGER FORM
% \begin{equation}
% \label{eq:ExtenstionFormula_Expression}
% \resizebox{0.925\hsize}{!}{$%
%     F_{\operatorname{ext}}(x)
% \eqdef
%     % \begin{cases}
%     (1-s)MR(x)^{s+1}
%     \,
%     \log
%     \biggl(
%         \sum_{i=1}^N
%         \,
%         % \theta_i(x)
%         \frac{
%             \mu_i|x-X_i|^{-q}
%             %%%
%             \,
%             \exp\big(
%                 % \frac{
%                     \big(
%                         \tfrac{P_i(x)}{MR(x)^{s+1}}
%                         +
%                         A_s
%                         \big(
%                         \tfrac{|x-X_i|}{R(x)}
%                         \big)^{2s}
%                     \big)
%                 % }{1-s}
%                 /(1-s)
%             % }
%             \big)
%         %%%
%         }{
%             \sum_{k=1}^N\,
%                 \mu_k|x-X_k|^{-q}
%         }
%     \biggr)
%     %     & \mbox{ : if  } x\not\in \{X_i\}_{i=1}^N
%     % \\
%     % \sum_{i=1}^N\,
%     %     P_i(X_i)\, I_{x=X_i}
%     %     & \mbox{ : if  } x\in \{X_i\}_{i=1}^N
%     % \end{cases}
% $}
% \end{equation}
%%%% SHORTER FORM
\begin{equation}
\label{eq:ExtenstionFormula_Expression}
\resizebox{0.925\hsize}{!}{$%
    F_{\operatorname{ext}}(x)
\eqdef
    % \begin{cases}
    (1-s)MR(x)^{s+1}
    \,
    \log
    \Big(
        % \sum_{i=1}^N
        \,
        \theta(x)^{\top}
        % \frac{
            % \mu_i|x-X_i|^{-q}
            %%%
            \Big(
            \exp\big(
                % \frac{
                    \big(
                        \tfrac{P_i(x)}{MR(x)^{s+1}}
                        +
                        A_s
                        \big(
                        \tfrac{|x-X_i|}{R(x)}
                        \big)^{2s}
                    \big)
                % }{1-s}
                /(1-s)
            % }
            \big)
            \Big)_{i=1}^N
        %%%
        % }{
        %     \sum_{k=1}^N\,
        %         \mu_k|x-X_k|^{-q}
        % }
    \Big)
    %     & \mbox{ : if  } x\not\in \{X_i\}_{i=1}^N
    % \\
    % \sum_{i=1}^N\,
    %     P_i(X_i)\, I_{x=X_i}
    %     & \mbox{ : if  } x\in \{X_i\}_{i=1}^N
    % \end{cases}
$}
\end{equation}
\end{extensionbox}
\noindent
and $F_{\operatorname{ext}}(X_i)\eqdef P_i(X_i)$ for every $i\in[N]$. Here $D\ge1$ and $\mu\in\Delta_N^\circ$ are hyperparameters, and
\begin{equation}
% \label{eq:ExtenstionFormula_HelperFunctions___introversion}
    \theta(x)
    % & 
    \eqdef 
    \Big(
        \tfrac{\mu_n|x-X_n|^{-q}}
        {\sum_{k=1}^{N}\mu_k|x-X_k|^{-q}}
    \Big)_{n=1}^N
% $
%%%
 % and 
\quad\mbox{ and }\quad
%%%
% $
R(x)^{2s}
\eqdef
\sum_{i=1}^N\theta_i(x)|x-X_i|^{2s}
% $ 
\end{equation}
and the the following explicit parameters induced by the regularity of the problem $
% \begin{equation}
% \label{eq:ExtenstionFormula_Parameters}
    \beta\eqdef s-1$,  
    % \qquad
    $p\eqdef 2s$, 
    % \qquad
    $A_s\eqdef \frac{1}{(s+1)!}+\frac{s-1}{2}$, 
    % \qquad
    $q\eqdef 2s+\left\lceil\log_2D\right\rceil+2$.
% \end{equation}
% \noindent
In contrast, in the ``low regularity'' regime $s=0,1$ the reconstruction is available only in explicit albeit \textit{variational form}; namely, it is explicit up to a minimum over a convex set.  This is also true of the $C^{1,1}$-formula of~\cite{AzagraMudarra2017}, with which our construction coincides when $s=1$, and of an equivalent variational expression for the McShane extension formula~\cite{McShane1934} when $s=0$.  
For $s\ge2$, Proposition~\ref{prop:entropy-only-reconstruction} identifies the function in~\eqref{eq:ExtenstionFormula_Expression} with the entropy-only at $F_{\operatorname{ext}}=F_0$ of~\eqref{eq:entropy-variance-extension__primal}. 
% When $D=2^{46d}$ and $\mu\in\Delta_{N:D}$, Proposition~\ref{prop:variance-parameter-family} gives it the same interpolation and $\dot C^{s,1}$ bounds as the unifying variational formulation.

\subsection{Our Contribution}
\label{s:Intro__ss:Ourcontrib}
We identify an \textit{closed-form} and \textit{finitely parameterized} (nonlinear) formula $F:\mathbb{R}^d\to \mathbb{R}$, cf.~\eqref{eq:entropy-variance-extension__primal} given below, for reconstructing a real-valued function $C^{s,1}$-function the data in~\eqref{eq:the_data}.  
Our main result, namely Theorem~\ref{thrm:main}, guarantees that $F$ has the following properties:
\subsubsection{Analytic Properties}
\begin{enumerate}
    \item[(i)] \textbf{Reconstruction:} $D^rF(X_n)=D^rP_n(X_n)$ for every $n\in[N]$ and $r=0,\ldots,s$,
    \item[(ii)] \textbf{Whitney-Extension Property:} $[F]_{s,1}\le K_{d,s}M$, where $K_{d,s}$ is independent of the number, locations, and separation of the sites.
    
    \item[(iii)] 
    \textbf{Analyticity Away from the Data:} If $s\ge2$, then $F$ is real-analytic on $\mathbb{R}^d\setminus\{X_n\}_{n=1}^N$.
    \item[(iv)] \textbf{Combinatorial Tameness:} $F$ is definable in the o-minimal structure $\mathbb{R}_{\operatorname{exp}},
    $
\end{enumerate}
where we write $\mathbb R_{\operatorname{exp}}\eqdef(\mathbb R,<,+,\cdot,0,1,\exp)$, which is o-minimal by~\cite[Theorem~5.2, p.~140]{Wilkie2009OMinimal}.

\subsubsection{Computational Properties}
We quantify circuit assembly using the exact-real word-RAM model specified in Section~\ref{s:Background_Complexity__ss:ER_WR}. For evaluation, we bound the size and depth of bounded-fan-in circuits over the gate set~\eqref{eq:Gates}, recording total gate work and parallel gate-evaluation time separately. 

\begin{table}[H]
\centering
\renewcommand{\arraystretch}{1.28}
\begin{tabularx}{\textwidth}{
    @{}
    >{\raggedright\arraybackslash}X
    >{\centering\arraybackslash}p{0.3\textwidth}
    >{\centering\arraybackslash}p{0.27\textwidth}
    @{}
}
\toprule
\textbf{Complexity measure}
&
\textbf{Closed-Form Formula}
&
\textbf{Algorithm of~\cite{FeffermanKlartag2009II}}
\\
\midrule

Initialization cost
% \newline
{\small (One-time work)}
&${\color{blackgreen}O_{d,s}(N)}$
&${\color{black}O_{d,s}(N\log N)}$
\\
\midrule
Parallel Query Time (Forward pass)
&$\color{blackgreen} O_{d,s}(\log N)$
& $\mathcal{O}_{d,s}(\log N)$
\\[1mm]
Storage Cost {\small (Space Complexity)}
&
$\color{black} \mathcal{O}_{d,s}(N)$
&
$\color{black} O_{d,s}(N)$
\\
\bottomrule
\end{tabularx}
\caption{Computational costs for fixed $d,s$ under the respective computational models. Our bounds are stated in Theorems~\ref{thrm:complexity} and~\ref{thrm:complexity__Assembly}; assembly takes the weights and integer $q$ as supplied inputs. The Fefferman--Klartag bounds are from~\cite[Theorem~6(C), p.~188]{FeffermanKlartag2009II}.}
\label{tab:complexity-comparison}
\end{table}

For fixed $d,s$ with $s\ge2$, Theorems~\ref{thrm:complexity} and~\ref{thrm:complexity__Assembly} give a circuit computing $F_{\operatorname{ext}}=F_0$ with $\mathcal O_{d,s}(N)$ total gate work and storage, $\mathcal O_{d,s}(1+\log N)$ parallel depth, and $\mathcal O_{d,s}(N)$ assembly work from supplied sites, polynomial coefficients, $M$, weights, and integer $q$. The assembly bound excludes constructing weights satisfying the doubling condition.

For comparison, the algorithm of~\cite[Theorem~6(C), p.~188]{FeffermanKlartag2009II} uses $\mathcal O_{d,s}(N\log N)$ preprocessing work, $\mathcal O_{d,s}(N)$ storage, and $\mathcal O_{d,s}(\log N)$ sequential query time. Under the respective computational models, our circuit therefore attains the same logarithmic upper bound on query latency with sufficient parallel resources, while using linear total work per query. Its linear assembly bound concerns supplied parameters and does not establish an end-to-end preprocessing improvement from the original derivative data.

\paragraph{Role of Parallelism and Modern (Manycore) GPU-Based Computing:}
Parallelism plays a critical role in the computational efficiency of our formula, as reported in Table~\ref{tab:complexity-comparison}.  This schema aligns with modern manycore GPU programming, cf.~\cite{NickollsBuckGarlandSkadron2008}, computing which is centred around massive parallelism.
If one followed a na\"{i}ve sequential implementation of our formula then it would exhibit a \textit{sequential} query time of $\mathcal{O}(N)$; however, it would be un-intuitive to implement it sequentially (at-least when $N$ is large).

\subsection*{Organization of the Paper}

Section~\ref{s:Notation} fixes the notation used throughout.  
Section~\ref{s:TheFormula} introduces the variational family and, for $s\ge2$, identifies its entropy-only endpoint with our closed-form extension formula.  
Section~\ref{s:MainResults} states the interpolation, regularity, definability, and off-site analyticity guarantees.  
Section~\ref{s:Coincidences} shows that the construction recovers the McShane formula when $s=0$ and the Azagra--Le Gruyer--Mudarra formula when $s=1$.  
The proofs of the main results are collected in Section~\ref{s:Proofs}, while Appendix~\ref{a:NonAnalticityChecks} contains the auxiliary non-analyticity comparisons used in the introduction.

\subsection{Related Work}
\label{s:Related Work}
We briefly discuss how our result fits into the broader Whitney extension theory literature.  

\paragraph{Linear Extension Operators}
For higher-order prescribed jets, a classical approach is the linear Whitney construction~\cite{Whitney1934,Fefferman2006}, which blends Taylor polynomials using a partition of unity:
\begin{equation}
\label{eq:classicalWhitney}
        F_{\operatorname{Whitney}}(x)
    =
        \begin{cases}
        \sum_{Q\in \mathcal{W}_{\mathbb{X}^N}}\,
            \theta_Q(x)
            \,
            P_{x(Q)}(x)
        %%%
            & \mbox{ : if } x \not\in \mathbb{X}^N
        \\
        \sum_{n=1}^N\, P_n(x)\, I_{x=x_n} 
        %%%
            & \mbox{ : if } x\in \mathbb{X}^N
        \end{cases}
\end{equation}
where $\mathcal{W}_{\mathbb{X}^N}$ denotes the Whitney decomposition of $[0,1]^d$ into \textit{infinitely} many ``Whitney cubes''%
\footnote{See e.g.~\cite[Section 2.1.1]{fefferman2020fitting} for a short and accessible overview.}%
, $(\theta_Q)_{Q\in \mathcal{W}_{\mathbb{X}^N}}$ is the associated partition of unity where each smooth-enough bump function $\theta_Q$ is supported on its respective Whitney cube $Q$, and $P_{x(Q)}$ denotes a closest polynomial to any given Whitney cube $Q$.
In the special case, perhaps most relevant to data-driven domains,~\cite{FeffermanKlartag2009I,FeffermanKlartag2009II} showed that~\eqref{eq:classicalWhitney} can be computed relatively efficiently.  
Related finite geometric constructions include the well-separated pair decomposition of~\cite[Theorem~4.3, p.~73]{callahan1995decomposition}.
Indeed, in~\cite{FeffermanKlartag2009II} it is shown that for our finite-data setting,~\eqref{eq:classicalWhitney} admits a ``sparse'' representation
\begin{equation}
\label{eq:classicalWhitney__sparse}
        F_{\operatorname{Whitney}}(x)
    =
        \sum_{n=1}^N\,
            f(X_n)
            \,
            \phi_{n,\mathbb{X}^N}(x)
\end{equation}
which determined using $\mathcal{O}_{d,s}(N)$ parameters and such that each $\phi_{n,\mathbb{X}^N}$ can be computed using $\mathcal{O}_{d,s}(\log(N))$ queries on an idealized ``Von Neumann'' computer.  Nevertheless, the representation of each $\phi_{n,\mathbb{X}^N}$ in~\eqref{eq:classicalWhitney__sparse} remains relatively complicated and is \underline{not} available in a \textit{closed-form}.

\paragraph{Non-Analyticity Away from the Data $\{X_n\}_{n=1}^N$}
The classical (sharp) Whitney formula in~\eqref{eq:classicalWhitney} does not by itself guarantee membership in $C^\omega(\mathbb{R}^d\setminus\{X_n\}_{n=1}^N)$, i.e.\ real analyticity in a neighbourhood of every point outside $\{X_n\}_{n=1}^N$.
For the finite-data construction in~\cite{FeffermanKlartag2009II}, the global gluing formula in~\cite[Section~34, equation~(25), p.~180]{FeffermanKlartag2009II} is a \textit{finite} sum of compactly supported terms, and hence the resulting extension is compactly supported. If $d\ge2$, any nonzero extension obtained in this way therefore cannot belong to $C^\omega(\mathbb{R}^d\setminus\{X_n\}_{n=1}^N)$.
A similar line of reasoning\footnote{See Proposition~\ref{prop:parusinski_rainer_nonanalyticity} in the appendix.} implies that the finite-set extension of~\cite{ParusinskiRainer2024} also cannot belong to $C^\omega(\mathbb{R}^d\setminus\{X_n\}_{n=1}^N)$ when $d\ge2$ and the prescribed jets are nonzero. Likewise, for $d\ge2$, the finite-depth linear extension operator recalled in~\cite[equations~(0.2)-(0.5), pp.~430-431]{fefferman2009fitting} produces, for at least one choice of interpolation data, an extension that is not real analytic away from the sites whenever $N$ exceeds its depth bound.\footnote{See Proposition~\ref{prop:NO_Analytic_FIII} in the appendix.}

\paragraph{Combinatorial Tameness}
Membership in $C^{s,1}(\mathbb{R}^d)\cap C^\omega(\mathbb{R}^d\setminus\{X_n\}_{n=1}^N)$ does not imply definability\footnote{For instance, $\sin$ is analytic but not definable in any o-minimal expansion of $(\mathbb{R},+,<)$.} in an o-minimal expansion of the ordered real field $(\mathbb{R},<,+,\cdot,0,1)$, which precludes pathological behaviour such as infinitely many isolated zeros in one variable; a property useful in contemporary machine learning~\cite{kinn2026fast,kratsios2026every}. While~\cite[Theorem~1.3]{ParusinskiRainer2024} exhibits definable $C^{s,1}$-Whitney extensions with uniform norm bounds within each fixed definable bounded family of jets, for $s\ge2$ these bounds are not asserted with constants depending only on $d,s$, uniformly in $N$; their dependence on the number of sample points $N$ is therefore left unspecified.

\paragraph{Known Formulae}
The known $C^{0,1}$-~\cite{McShane1934} and $C^{1,1}$-\cite{AzagraLeGruyerMudarra2018} extension formulae depend nonlinearly on their data, as does our entropy-variance construction. For $s\ge1$, the auxiliary rational interpolant in~\eqref{eq:explicit-whitney-witness} is linear in the polynomials $P_i$ when the sites and geometry weights are fixed. 
It supplies the explicit calibration constant in Lemma~\ref{lem:finite-whitney-seminorm__bound}.
Theorem~\ref{thrm:main} concerns the variational family and its quantitative regularity bounds, with the exact low-order specializations established in Section~\ref{s:Coincidences}.
It is also worth noting that the best available constant on the $C^{s,1}$-norm, cf.~\cite{chang2017whitney}.
For each fixed $m\in\mathbb N_+$, \cite[Theorem~1.3, p.~624]{chang2017whitney} gives a linear $C^m$ Whitney extension operator with norm $\mathcal O_m(d^{5m/2})$, using the norms in Definitions~1.1--1.2 of that paper.
Another algorithm for the $s=0,1$ cases is also identified in~\cite{HerbertVossHirnMcCollum2017}.

\paragraph{Organization of the Paper.}
Section~\ref{s:Notation} introduces the notation and terminology, while Section~\ref{s:TheFormula} defines the variational family and identifies its closed-form specialization. Section~\ref{s:MainResults} presents the main analytic and computational results, and Section~\ref{s:Coincidences} establishes coincidence with the known low-regularity formulae. 
All proofs are collected in Sections~\ref{s:Proofs} and~\ref{s:Proofs_Complexity}. 
Appendix~\ref{s:Background_Complexity} reviews the formal models of computation used herein. Appendix~\ref{a:NonAnalticityChecks} proves the auxiliary non-analyticity claims made in the introduction.

\section{Notation}
\label{s:Notation}
We now collect the notation used in our manuscript; both in the main body and in its proofs.  
The non-negative integers are denoted by $\mathbb{N}$, $\mathbb{N}_+$ denotes the positive integers and for every $n\in \mathbb{N}_+$ we write $[n]\eqdef \{m\in \mathbb{N}_+:\,m\le n\}$.   For any $t\in \mathbb{R}$, we write $t_+ \eqdef \max\{t,0\}$, we denote its integer ceiling by $\lceil t\rceil$, and if $t>0$ then we use $\log(t)$ to denote its natural logarithm.  
\paragraph{Norms}
For any $d,k \in \mathbb{N}_+$ and $x\in \mathbb{R}^d$ we write $|x|$ for its Euclidean norm, we denote the \textit{operator norm} of any $k$-linear map $T:(\mathbb{R}^d)^k\to\mathbb{R}$, by
% \begin{equation}
% \label{eq:multilinear}
$
        \|T\|_{\mathrm{op}}
    \eqdef
        \sup_{\substack{h_1,\ldots,h_k\in\mathbb{R}^d\\ |h_1|,\ldots,|h_k|\le1}} 
            \, 
                \big|T[h_1,\ldots,h_k]\big|
$.
% \end{equation}
For $r\in\mathbb{N}$, the space $C^r(\mathbb{R}^d)$ consists of real-valued functions with continuous derivatives through order $r$. 
For $f\in C^r(\mathbb{R}^d)$, we  write $\nabla f$ for its gradient when $r\ge1$.
For $s\in\mathbb N$ and $g\in C^s(\mathbb R^d)$, put $D^0g=g$. For $s\ge1$, $D^sg(x)$ is a symmetric $s$-linear form on $\mathbb R^d$, equipped with the operator norm $\|\cdot\|_{\mathrm{op}}$ above. At order zero, this norm means absolute value. Define
\begin{equation}
\label{eq:Lip_Seminorm}
 [g]_{s,1}\eqdef
 \sup_{\substack{x,y\in\mathbb R^d\\x\ne y}}
 \frac{\|D^sg(x)-D^sg(y)\|_{\mathrm{op}}}{|x-y|}.
\end{equation}
Thus $[g]_{0,1}=\operatorname{Lip}(g)$. The homogeneous class $\dot C^{s,1}(\mathbb R^d)$ consists of $g\in C^s(\mathbb R^d)$ with $[g]_{s,1}<\infty$; it imposes no global boundedness requirement on $g$ or its derivatives.
The \textit{homogeneous class}\footnote{E.g. $x^2\in \dot{C}^{s,1}(\mathbb{R})$ for $s=1,2$ but not for $s=0$.} $\dot{C}^{s,1}(\mathbb{R}^d)$ consists of all $g\in \dot{C}^{s}(\mathbb{R}^d)$ for which the seminorm in~\eqref{eq:Lip_Seminorm} is finite;  i.e.\ the finiteness of $[g]_r$ for any integer $0<r<s$ is not required.  
Finally, for the restrictions of $g\in C^s(\mathbb R^d)$ to $[0,1]^d$, use the full norm
\begin{equation}
\label{eq:full-Cs1-norm-cube}
 \|g\|_{C^{s,1}([0,1]^d)}\eqdef
 \sum_{r=0}^s\sup_{x\in[0,1]^d}\|D^rg(x)\|_{\op}
 +\Lip_{[0,1]^d}(D^sg).
\end{equation}
Now, for $d,s\in\mathbb N_+$, $\alpha=(\alpha_1,\ldots,\alpha_d)\in\mathbb N_0^d$, and $z\in\mathbb R^d$, write
$
|\alpha|\eqdef\sum_{j=1}^d\alpha_j$,  $z^\alpha\eqdef\prod_{j=1}^dz_j^{\alpha_j}$, and $\Lambda_{d,s}\eqdef\{\alpha\in\mathbb N_0^d:|\alpha|\le s\}$ and with $t^0=1$ for every $t\in\mathbb R$.

\paragraph{Probabilistic Notation}
For any $N\in\mathbb{N}_+$, we denote the (probability) $n$-simplex by $
        \Delta_N
         \eqdef 
         \{u \in[0,1]^N:
        \sum_{n=1}^N\,u_n=1\}
    $ and its relative interior to be $\Delta_N^{\circ} \eqdef  \{u\in \Delta_N:\,u_n>0\,\forall n\in [N]\}$.
These, respectively, parameterize the set of probability measures $\mathcal{P}(\{z_n\}_{n=1}^N)$ on an $N$-point set $\{z_n\}_{n=1}^N$; and the subset of the probability measures which place mass on all points; the correspondence being $\Delta_N\ni u\leftrightarrow \sum_{n=1}^N\, u_n\,\delta_{z_n}\in \mathcal{P}(\{z_n\}_{n=1}^N)$.  
Now, for any $\lambda\in\Delta_N$ and any weight (matrix) $w=(w_n)_{n=1}^N\in (\mathbb{R}^d)^N$, we define the $\lambda$-\textit{weighted variance} of $w$ by\footnote{
Note that, when $d=1$ then,~\eqref{eq:weighted_variance} reduces to the familiar variance of a real-valued random variable $X$ with $\mathbb{P}(X=w_n)=\lambda_n$; i.e.\ 
$\operatorname{Var}_{\lambda}(w) = 
\mathbb{E}_{\mathbb{P}}\big[
    \big(X-\mathbb{E}_{\mathbb{P}}[X]\big)^2
\big]$ and $\mathbb{E}_{\mathbb{P}}[X]=\sum_{n=1}^N\, w_n\,\lambda_n$.
}
\begin{equation}
\label{eq:weighted_variance}
    \operatorname{Var}_{\lambda}(w)
 \eqdef 
    \sum_{n=1}^{N}
    \,
    \lambda_n|w_n|^2
        -\Big\|
            \sum_{n=1}^{N}\lambda_n\,w_n
        \Big\|^2_2
.
\end{equation}
For any $\bar{\lambda}\in \Delta_N^{\circ}$the ``relative entropy'' of $\lambda$ with respect to $\bar{\lambda}$ is
$
    \operatorname{KL}(\lambda\mid\bar{\lambda})
 \eqdef 
    \sum_{n=1}^{N}\lambda_n\,
        \log\tfrac{\lambda_n}{\bar{\lambda}_n}
$ 
where we \textit{notationally} convene that $0\log(0)\eqdef 0$.

For any finite set $\{X_n\}_{n=1}^N\subseteq \mathbb{R}^d$ and any $D\ge 1$, we designate the space $\Delta_{N:D}\subseteq \Delta_N^{\circ}$ consisting of those ``weights'' $\mu\in \Delta_N^{\circ}$ and a constant $D\ge 1$ satisfying the doubling-type condition on~\cite[page (1.2)]{han2021geometric}: for every $i\in[N]$ and each $r>0$
\begin{equation}
\label{eq:doubling}
        \sum_{\substack{j\in[N]\\|X_j-X_i|<2r}}\mu_j
    \le
        D\sum_{\substack{j\in[N]\\|X_j-X_i|<r}}\mu_j
.
\end{equation}
Note that, for any such $\mu\in \Delta_N^{\circ}$ such as a $D$ always exists: since, we may always that left-hand side of~\eqref{eq:doubling} is at most $1$, while its right-hand side is at least $D\mu_i$; thus,  $D=(\min_{n\in[N]}\mu_n)^{-1}$ suffices; moreover $D$ is finite since $\mu\in \Delta_N^{\circ}$ not only $\Delta_N$%%
\footnote{The minimal $D$ an upper-bound for the \textit{upper-dimension} of $\sum_{n=1}^N\,\mu_n\delta_{X_n}$; cf.~\citep[Equation (1.3)]{han2021geometric}.}.

\section{The Variational Formulation}
\label{s:TheFormula}
We now define a $1$-parameter family of \textit{variational} formulae, which we will show coincides with our extension formula; in the high-regularity regime where $s\ge 2$.  However, we will prove our Whitney-type extension results for this broader family, and then specialize our computational claimed for the closed-form formula $F_{\operatorname{ext}}$ later.
We maintain the setting:
\begin{setting}[Training Data]
\label{TheSetting}
Fix $d,N\in \mathbb{N}_+$, $s\in \mathbb{N}$, $M>0$, distinct sites $X_1,\dots,X_N\in \mathbb{R}^d$, an (unknown) $f\in \dot{C}^{s,1}(\mathbb{R}^d)$, any polynomials $P_1,\dots,P_N\in \mathbb{R}[x_1,\dots,x_d]$ of degree at-most $s$ such that: for each $n\in [N]$
\begin{equation}
\label{eq:training_data}
P_n = J_{X_n}^s f
\qquad
\mbox{ and } \operatorname{Lip}(D^s\,f)\le M
.
\end{equation}
\end{setting}
\noindent
We refer to $\mathbb{D}^N \eqdef  \{(X_n,J_{X_n}^s f)\}_{n=1}^N$ as the training data %
 and $\mathbb{X}^N \eqdef  \{X_n\}_{n=1}^N$ as the measurement sites.  
The \textit{fill radius} of our data is
$
    h\eqdef\sup_{x\in(0,1)^d}
        \min_{n\in[N]}\|x-X_n\|_2
$.
Fix $D\ge 1$ and $\mu\in \Delta_{N:D}$.  
Once $s$, and $D$ are fixed; then so are the constants
\begin{equation}
\label{eq:fancy_constants}
\begin{aligned}
    \beta&  \eqdef (s-1)_+,
    \qquad
    p \eqdef s+1+\beta,
    \qquad
    A_s \eqdef \frac{1}{(s+1)!}+\frac{\beta}{2}
    ,
    \qquad
    q  \eqdef p+\lceil\log_2D\rceil+2
.
\end{aligned}
\end{equation}
Define the helper functions, for $x\in\mathbb{R}^d\setminus\mathbb{X}^N$, by
\begin{equation}
\label{eq:entropy-variance-extension__primal__THEVARIABLES}
\resizebox{0.9\hsize}{!}{$%
\begin{aligned}
    \theta(x)
    % & 
    \eqdef 
    \left(
        \frac{\mu_n|x-X_n|^{-q}}
        {\sum_{k=1}^{N}\mu_k|x-X_k|^{-q}}
    \right)_{n=1}^N
,\quad
    R(x)^p
     \eqdef 
        \sum_{i=1}^{N}\theta_i(x)|x-X_i|^p
% \\
,\quad
    v(x)
    % &
    \eqdef 
    \left(
        \nabla P_i(x)
        +\frac{M}{s!}|x-X_i|^{s-1}(x-X_i)
    \right)_{i=1}^N
\end{aligned}
$}
\end{equation}
We define the admissible variance-parameter set
% \begin{equation}
% \label{eq:entropy-variance-admissible-parameters}
$\mathcal I_s
\eqdef \{1\}$ for $s=0,1$ and $\mathcal I_s=[0,1]$ when $s\ge 2$.
% \begin{cases}
% \{1\},&s\in\{0,1\},\\
% [0,1],&s\ge2.
% \end{cases}
% \end{equation}
Analogously to the classical (linear) Whitney formula, cf.~\eqref{eq:classicalWhitney}, for every fixed $\vartheta\in\mathcal I_s$, define $F_\vartheta:\mathbb R^d\to\mathbb R$ at the measurement sites by
\begin{equation}
\label{eq:entropy-variance-extension__primal___SITES}
F_\vartheta(X_i)\eqdef P_i(X_i),
\qquad i\in[N],
\end{equation}
and, for $x\in\mathbb R^d\setminus\mathbb X^N$, by
\begin{extensionbox}
\begin{equation}
\label{eq:entropy-variance-extension__primal}
\resizebox{0.925\hsize}{!}{$%
\begin{aligned}
F_\vartheta(x)
\eqdef
\min_{\lambda\in\Delta_N}
\Bigg\{
&\sum_{i=1}^N\lambda_i
\left[
P_i(x)+\frac{A_sM|x-X_i|^p}{R(x)^\beta}
\right]-\frac{
    \vartheta s
    \,
    \operatorname{Var}_\lambda(v(x))
}{4MR(x)^{s-1}}
+\beta MR(x)^{s+1}
\operatorname{KL}\bigl(\lambda\mid\theta(x)\bigr)
\Bigg\}.
\end{aligned}
$}
\end{equation}
\end{extensionbox}
We set
\begin{equation}
\label{eq:entropy-variance-extension__specialization}
F\eqdef F_1.
\end{equation}
The main result tying our \textit{closed-form formula} $F_{\operatorname{ext}}$ to this family of variational formulae is the following identity in the ``high regularity'' regime where $s\ge 2$.
Thus, we will present all our analytic arguments for the variational formula in~\eqref{eq:entropy-variance-extension__primal} which is valid for all regimes
$s\in \mathbb{N}$ and then specialize our computational arguments for the closed-form version in~\eqref{eq:ExtenstionFormula_Expression}.
%%%
\begin{highlightedproposition}[Closed Form of the Entropy-Only Variational Problem]
\label{prop:entropy-only-reconstruction}
Let $s\ge 2$.  
Under Setting~\ref{TheSetting}, let $D\ge1$ and $\mu\in\Delta_N^\circ$, and $q$ be as in~\eqref{eq:fancy_constants}, $\theta$ and $R$ be as in~\eqref{eq:entropy-variance-extension__primal__THEVARIABLES}.  
\hfill\\
Then, $F_0$ (cf.~\eqref{eq:entropy-variance-extension__primal}) with $\vartheta=0$ and $
F_{\operatorname{ext}}
$ (cf.~\eqref{eq:ExtenstionFormula_Expression}) satisfy: for every $x\in\mathbb R^d$
\[
F_{\operatorname{exp}}(x)=F_0(x).
\]
\end{highlightedproposition}
%%%
\begin{proof}
See Section~\ref{s:Proofs__ss:EntropyOnly}.
\end{proof}

\section{Main Results}
\label{s:MainResults}
We now present our paper's main results; we begin with the analytic and model theoretic result, and then progress to the theoretic computer science (TCS) results.
\subsection{Analytic Results}
\label{s:MainResults__ss:Analtic}
we emphasize that the value $D=2^{46d}$ is only sufficient.
\begin{theorem}[Variational Extension Formula]
\label{thrm:main}
Under Setting~\ref{TheSetting}, let $D=2^{46d}$ and choose $\mu\in\Delta_{N:D}$%
\footnote{Note that such a $\mu$ must exist by Lemma~\ref{lem:three-doubling-weights}.}.  %  
Then the calibrated reconstruction $F=F_1$ defined by~\eqref{eq:entropy-variance-extension__primal___SITES}--\eqref{eq:entropy-variance-extension__specialization} is definable in $\mathbb{R}_{\operatorname{exp}}$ and satisfies:
\begin{enumerate}
    \item[(i)] \textbf{Interpolation:} For every $n\in [N]$ and $j\in\{0,\dots,s\}$ we have
\[
    D^jF(X_n)
=
    D^jP_n(X_n)
.
\]
\item[(ii)] \textbf{Regularity:} $F\in\dot{C}^{s,1}(\mathbb{R}^d)$ and there is a constant\footnote{Depending only on $d,s$.} $K_{d,s}\ge1$ satisfying
\[
    [F]_{s,1}\le K_{d,s}M,
\]
Moreover, $K_{d,s}
    \in
    \mathcal{O}
    (
        \mathrm{e}^{cd}(d+s)^{cs}
    )$ for some absolute $c>0$ and 
$K_{d,s}=1$ when $s<2$.
%%%
\item[(iv)] \textbf{Extrapolation:} If, moreover, $s\ge 2$ then, $F$ is real-analytic on $\mathbb{R}^d\setminus\{X_n\}_{n=1}^N$.
\end{enumerate}
\end{theorem}
\begin{proof}[{Proof of Theorem~\ref{thrm:main}}]
Part (i) is proved in Section~\ref{s:Proofs__ss:Interpolation}, part (ii) in Section~\ref{s:Proofs__ss:Cs1_Smoothness}, 
and part (iv) in Lemma~\ref{lem:offsite-analyticity}.  The proof of the definability of $F$ in $\mathbb{R}_{\operatorname{exp}}$ is given in Lemma~\ref{lem:reconstruction-definability} in Section~\ref{s:Proofs__ss:Definability}.
\end{proof}

Now, the interpolation property (i) and the homogeneous seminorm bound (ii) in Theorem~\ref{thrm:main} over the entire hypercube $[0,1]^d$ yields full control of the $C^{s,1}([0,1]^d)$-norm.

\begin{corollary}[Full Norm Control on the Cube]
\label{cor:full-Cs1-control}
In the setting of Theorem~\ref{thrm:main}, suppose that $\mathbb X^N\subseteq[0,1]^d$. Then, for every $s\in\mathbb N$,
\begin{equation}
\label{eq:full-Cs1-control}
 \|F\|_{C^{s,1}([0,1]^d)}
 \le \mathrm e^{\sqrt d}\left(
 \min_{i\in[N]}\sum_{r=0}^s\|D^rP_i(X_i)\|_{\op}
 +K_{d,s}M\right).
\end{equation}
In particular, the bound has no dependence on the sample size or separation beyond the displayed measured jet values.
\end{corollary}

\begin{remark}[Control of the global $C^{s,1}(\mathbb{R}^d)$ Norm]
\label{rem:full_control}
We note that, for control of the norm on all of $\mathbb{R}^d$ then one simply need to consider the product $\bar{F}\eqdef \chi F$ where $\chi$ is a plateau function which is $1$ on $[0,1]^d$ and vanishes outside a neighbourhood of $[0,1]^d$.  However, doing so implies that $\hat{F}$ is real-analytic on $(0,1)^d\setminus\{X_n\}_{n=1}^N$.
\end{remark}

\subsection{Complexity Theoretic Results}
\label{s:MainResults__ss:TCS}

For fixed data, a circuit evaluates $F_{\operatorname{ext}}$ at a query point. A uniform compiler constructs that circuit from the supplied data. We specify separately the gate model for evaluation and the exact-real word-RAM model for compilation.
\subsection{Computational Complexity of \texorpdfstring{$F_{\operatorname{ext}}$}{The Extension Formula}}
\label{s:TCS_Guarantees_ss:complexity_Fext}

We establish $O_{d,s}(1+\log N)$ parallel gate-evaluation time and $\Theta_{d,s}(N)$ storage for the circuit realizing~\eqref{eq:ExtenstionFormula_Expression}. A sequential evaluation of every gate uses $\Theta_{d,s}(N)$ gate operations.
The circuit viewpoint records not only the input--output function but also the elementary operations, their compositional structure, and the resulting size and depth of the computation.   Background on the circuit complexity formalism used in this section in reviewed in Appendix~\ref{s:Background_Complexity__ss:Circuits}.

We use unit-cost exact real-valued gates.
The admissible elementary operations, or \emph{gates}, belong to the language
\begin{equation}
\label{eq:Gates}
\mathbb G
\eqdef
\{\mathsf c_z:z\in\mathbb R\}
\cup
\{+_2,\times_2,\div_2,\exp,\log\}
\cup
\{\mathsf{pow}_r:r>0\}
\cup
\{I_{\cdot=b}:\, b\in \mathbb{R}\}
,
\end{equation}
% where:
\begin{enumerate}
\item[(i)] \textbf{Binary - Arithmetic operations:} the binary addition and multiplication gates are
$
+_2,\times_2:\mathbb R^2\to\mathbb R$ 
and the binary division gate is
$
\div_2:\mathbb R\times(\mathbb R\setminus\{0\})\to\mathbb R
$ and $\div_2(u,v)\eqdef\frac uv$.  
\item[(ii)] \textbf{Unary - Power and radical operations:} for every $r>0$, the unary gate
$
\mathsf{pow}_r:[0,\infty)\to[0,\infty)
$ mapping any $t\in \mathbb{R}$ to $\mathsf{pow}_r(t)\eqdef t^r$.
% ; using the unique non-negative real branch.
\item[(iii)] \textbf{Unary - Exponential operations:} $
\exp:\mathbb R\to(0,\infty)$ and $
\log:(0,\infty)\to\mathbb R$.  
\item[(iv)] \textbf{Unary -  Constants:} for every $z\in\mathbb R$ let $
\mathsf c_z:\mathbb R\ni x \to z$,
\item[(v)] \textbf{Unary - Equality Satisfaction:} For every $b\in \mathbb{R}$, the map $I_{\cdot=b}:\mathbb{R}\mapsto \{0,1\}$ returning $1$ if and only if $\cdot=b$ and $0$ otherwise.
\end{enumerate}

Our closed-form formula, i.e.\ in the $s\ge 2$ setting, has the the following circuit complexity:

\begin{theorem}[Circuit Complexity of the Closed-Form Formula]
\label{thrm:complexity}
Fix $d,N,s\in\mathbb N_+$ with $s\ge2$, $M>0$, pairwise distinct $X_1,\ldots,X_N\in\mathbb R^d$, polynomials $P_i\in\mathbb R[x_1,\ldots,x_d]$, $i\in[N]$, of degree at most $s$, $D\ge1$, and $\mu\in\Delta_N^\circ$. Set $q\eqdef2s+\lceil\log_2D\rceil+2$ and $U\eqdef\mathbb R^d\setminus\{X_1,\ldots,X_N\}$. There exists a bounded-fan-in $\mathbb G$-circuit $\widehat{\mathcal C}_{\operatorname{ext}}$, with $d$ input vertices and lifting channel $I_d$, which computes $F_{\operatorname{ext}}$; i.e.\
\[
    f_{\hat{\mathcal{C}}_{\operatorname{ext}}}(x) = F_{\operatorname{ext}}(x)
\qquad
\mbox{ for all } x\in \mathbb{R}^d
\]
satisfying:
\begin{enumerate}[label=(\roman*), ref=(\roman*)]
    \item \label{eq:closed-form-circuit-size}
    \textbf{Storage Cost:} Size of $\mathcal O_s(Nd^s)$,
    \item \label{eq:closed-form-circuit-depth}
    \textbf{Parallel Query Time:} $\mathcal O_s(\log(d+N))$,
    \item \label{eq:closed-form-circuit-width}
    \textbf{Parallelism Used:} Width of $\mathcal O_s(Nd^s)$.
\end{enumerate}
\end{theorem}
\begin{proof}
See Section~\ref{s:Proofs_Complexity}.
\end{proof}

We next construct the circuit from its supplied site coordinates, polynomial coefficients, positive weights, and integer $q$.

\subsection{The Complexity of Assembling \texorpdfstring{$F_{\operatorname{ext}}$}{The Extension Formula}}
\label{s:TCS_Guarantees__ss:Assembly}

We use the following exact-real word-RAM model for compilation. The bound below concerns compilation from the supplied coefficients, weights, and integer $q$; it excludes finding admissible weights and computing $q$ from a real-valued $D$.

\begin{dataformat}
\label{dataformat}
Fix $d,s\in\mathbb N_+$. The word inputs are $N\in\mathbb N_+$ and an integer $q\ge2s+2$. The exact-real inputs are $M>0$ and $(X_i,\mu_i,(c_{i,\alpha})_{\alpha\in\Lambda_{d,s}})_{i\in[N]}$, where the sites $X_i\in\mathbb R^d$ are pairwise distinct and $\mu=(\mu_i)_{i\in[N]}\in\Delta_N^\circ$. Each polynomial $P_i\in\mathbb R[x_1,\ldots,x_d]$, of degree at most $s$, is represented by its centered coefficients:
\[
P_i(x)\eqdef\sum_{\alpha\in\Lambda_{d,s}}c_{i,\alpha}(x-X_i)^\alpha,\qquad i\in[N].
\]
\end{dataformat}

\begin{theorem}[Circuit Assembly Complexity]
\label{thrm:complexity__Assembly}
Fix $d,s\in\mathbb N_+$ with $s\ge2$. There exists a uniform exact-real word-RAM algorithm $\texttt{Compile-Circuit}$ (Algorithm~\ref{alg:compile-circuit})\footnote{The program depends only on $d,s$.} such that, for every input satisfying Data Format~\ref{dataformat} 
on a word length 
% $w\in\mathbb N_+$ satisfying $2^w>q+C_{\rm mem}(d,s)(N+1)$
of $\mathcal O_{d,s}\!\bigl(\log(q+N)\bigr)$%
, the algorithm returns a $\mathbb G$-circuit $\widehat{\mathcal C}_{\operatorname{ext}}$\footnote{Encoded by the decorated adjacency representation $\operatorname{Adj}(\widehat{\mathcal C}_{\operatorname{ext}})$ specified in~\eqref{eq:compiler-record}.} satisfying Theorem~\ref{thrm:complexity} with $D=2^{q-2s-2}$. Moreover,
\[
\begin{aligned}
\operatorname{Work}(\texttt{Compile-Circuit})&=\mathcal O_s(Nd^s),\\
\operatorname{OutputStorage}(\texttt{Compile-Circuit})&=\mathcal O_s(Nd^s).
\end{aligned}
\]
\end{theorem}

We close our discussion by comparing our construction to known extension formulae.

\section{Coincidence with the Known \texorpdfstring{$s=0,1$}{Low-Regularity} Formulae}
\label{s:Coincidences}
We wrap-up our results by relating our formula to the other two known nonlinear closed-form and finitely parameterized Whitney-extension formulae.  These handle the cases $s=0,1$.
We have the following coincidence result:
\begin{proposition}[Coincidence $s=0$: McShane Formula]
\label{prop:McShane}
Let $d,N\in \mathbb{N}_+$, and $s=0$; then, for every $x\in\mathbb{R}^d$
\[
    F(x)
=
    \min_{n\in [N]}\,
        \left\{f(X_n) + M\|x-X_n\|_2\right\}
.
\]
\end{proposition}
Each term $f(X_n)+M\|x-X_n\|_2$ is an upper bound on the value at $x$ imposed by the measurement at $X_n$ and the Lipschitz constraint. Taking the smallest of these bounds gives the pointwise largest $M$-Lipschitz function agreeing with all the measured values.
For $s=1$ we recover the Azagra-Le Gruyer-Mudarra formulae, cf.~\cite{AzagraMudarra2017}:
\begin{proposition}[{Coincidence $s=1$: Azagra-Le Gruyer-Mudarra Formula~\cite[Theorem~3.4]{AzagraLeGruyerMudarra2018}}]
\label{prop:Carlos}
Let $d,N\in \mathbb{N}_+$, and $s=1$; then, for every $x\in\mathbb{R}^d$
\[
\begin{aligned}
    F(x)
& = 
    \operatorname{conv}\!\left(
        g
    \right)(x)
    -\frac{M}{2}|x|^2,
\\
g(y)
&  \eqdef 
        \min_{1\le i\le N}
        \left\{P_i(y)+\frac{M}{2}|y-X_i|^2\right\}
        +\frac{M}{2}|y|^2
.
\end{aligned}
\]
Here $\operatorname{conv}$ denotes the lower-semicontinuous convex envelope.
\end{proposition}

Here $\operatorname{conv}(g)$ is the largest convex function lying below $g$. Adding the common quadratic before taking this envelope, and subtracting it afterwards, combines the local models into a function that matches both the measured values and gradients while satisfying $\operatorname{Lip}(\nabla F)\le M$.

\section{Future Work}
\label{s:Conclusion}
In future work, it would be very nice to find sharpened version of our formula which achieves the minimal possible constant.  We also hope others are inspired by it, and apply it to algorithmic foundations of machine learning.

\section{\texttt{LEAN} Certificate}
\label{s:LEAN}
Theorems~\ref{thrm:main} (except the definability in $\mathbb{R}_{\operatorname{exp}}$ due to the unavailability of that model in \texttt{LEAN} currently), Theorem~\ref{thrm:complexity}, and all supporting results were certifiably-verified in \texttt{LEAN} (see ArXiV upload for the \texttt{LEAN} code and the \texttt{LEAN} certificate).

\section{Acknowledgement and Funding}
\label{s:Funding}
\noindent
A.\ Kratsios gratefully acknowledges financial support from an NSERC Discovery Grant No.\ RGPIN-2023-04482 and No.\ DGECR-2023-00230.  They acknowledge that resources used in preparing this research were provided, in part, by the Province of Ontario, the Government of Canada through CIFAR, and companies sponsoring the Vector Institute\footnote{\href{https://vectorinstitute.ai/partnerships/current-partners/}{https://vectorinstitute.ai/partnerships/current-partners/}}.
\hfill\\
\noindent
P.\ Zimmermann gratefully acknowledges the support of the Swiss National Science Foundation (SNSF) under Grant No.\ 214500.

\subsection{Thanks}
\label{s:Funding_ss:Thanks}
A.\ Kratsios would like to thank Charles Fefferman for his short remarks which greatly reshaped the trajectory of the paper.  The authors would also like to thank Carlos Mudarra for his amazing feedback and explanations of the ideas behind the $C^{1,1}$ extension formulae.

\subsection{AI Declaration}
\label{s:Funding_ss:AI}
The following is aligned to the \href{https://leidendeclaration.ai/}{Leiden Declaration on Artificial Intelligence and Mathematics}.

\paragraph{Tool and computational resource disclosure}
The authors used \texttt{Chat GPT 6, Astra} to polish the manuscript, to assist in refining the original variational formula, and to proofread mathematical arguments, and formalizing the proof of Theorem~\ref{thrm:complexity__Assembly}. 
AI tools were also used to compile the \texttt{LEAN} certificate of their results.
The authors take full responsibility for the correctness and adequacy of the mathematical arguments and results.

\section{Proofs}
\label{s:Proofs}

We trace the following \textit{proof strategy:} 
We first establish an equivalent dual representation and construct geometric weights whose doubling constant depends only on the dimension. The cases $s=0,1$ are treated directly. For $s\ge2$, we fix an arbitrary $\vartheta\in[0,1]$ and carry the family $F_\vartheta$ through the local argument without differentiating the parameter. The main difficulty is to control derivatives of the optimized value near the measurement sites, where the inverse-distance factors are singular. We subtract a fixed Taylor polynomial and normalize by the local distance scale, obtaining estimates uniform in $\vartheta$, the number of sites, and their separation. Lemma~\ref{lem:three-local-errors} gives the resulting derivative-error bounds, and Lemma~\ref{lem:three-site-gluing} turns them into exact interpolation. A further derivative bound yields global Lipschitz control of $D^sF_\vartheta$. The specialization $\vartheta=1$ proves the higher-order part of Theorem~\ref{thrm:main}, while $\vartheta=0$ supplies the interpolation and regularity conclusions of Proposition~\ref{prop:entropy-only-reconstruction}. Taylor expansion at a nearest site then gives the approximation estimates, while the final calibration chooses the regularity parameter from the measured data.

\subsection{{Proof Theorem~\ref{thrm:main} (i):  The Jet Interpolation Property}}
\label{s:Proofs__ss:Interpolation}

\paragraph{Proof setting and notation.}
The data and the global witness satisfy \eqref{eq:training_data}; the geometry weights and their certificate remain fixed as in \eqref{eq:doubling}. 
For the higher-order part, fix an arbitrary $\vartheta\in[0,1]$ and use $F_\vartheta$ from~\eqref{eq:entropy-variance-extension__primal___SITES}--\eqref{eq:entropy-variance-extension__primal}. The parameter is fixed throughout spatial differentiation, and every estimate below is uniform in this fixed $\vartheta$. The original reconstruction is recovered at $\vartheta=1$ by~\eqref{eq:entropy-variance-extension__specialization}.
For a map $u$ on a set $\Omega$ into a normed finite-dimensional space, $\Lip_\Omega(u)$ denotes the supremum of $\|u(x)-u(y)\|/|x-y|$ over distinct $x,y\in\Omega$. Derivatives of vector- or tensor-valued maps carry the operator norm induced by the specified codomain norm. Partial tensor contractions retain the unfilled slots. Write $\mathrm{e}=\exp(1)$ and, for $a=X_i$, $P_a=P_i$. The distance functions $r_i$ and $\rho_{\mathbb{X}^N}$ below are defined for all $x\in\R^d$:
\begin{equation}
\label{eq:three-aux-geometry}
 m \eqdef s+1,\qquad r_i(x) \eqdef |x-X_i|,\qquad
 \rho_{\mathbb{X}^N}(x) \eqdef \min_{i\in[N]}r_i(x),\qquad
 S_\gamma(x) \eqdef \sum_{i=1}^N\mu_i r_i(x)^{-\gamma}
\end{equation}
(where, clearly, the sum $S_\gamma$ is defined for $x\notin\mathbb{X}^N$ and $\gamma\ge0$).
For $a\in\mathbb{X}^N$ and $r>0$, let $B_{\mathbb{X}^N}(a,r) \eqdef \{X_i:|X_i-a|<r\}$ and identify the fixed vector $\mu$ with the atomic measure $\mu(A) \eqdef \sum_{X_i\in A}\mu_i$. The definitions in \S\ref{s:TheFormula} implies that:
if $R^p=\frac{S_{q-p}}{S_q}$, $m-p=-\beta$, and $s\ge2 $ then 
\begin{equation}
\label{eq:three-aux-parameters}
(\beta,p,q)=\bigl(s-1,2s,2s+\lceil\log_2D\rceil+2\bigr).
\end{equation}
Here $q-p\ge2$, all distances are positive off the sites, and $R>0$. In particular the entropy convention makes the objective continuous on the compact simplex: each summand $\lambda_i\log\lambda_i$ extends continuously at zero, and the other terms are polynomials in $\lambda$. Indeed, writing $t=\mathrm{e}^{-u}$ gives $|t\log t|=u\mathrm{e}^{-u}\to0$ as $t\downarrow0$, since $\mathrm{e}^u\ge u^2/2$ for $u>0$. A coordinatewise convergent subsequence of a minimizing sequence therefore attains the minimum. No value of a singular helper function is assigned at a site. In the higher-order proof $\mathfrak h_\vartheta(x)$ will denote the normalized optimal value associated with $F_\vartheta$; it is distinct from the fill radius $h$.

To obtain constants independent of how the sites are distributed, we construct positive weights for which doubling a ball's radius increases its total weight by at most a factor depending only on $d$. The next two lemmas achieve this by transferring controlled weights between successively finer subsets of the sites.

\begin{lemma}
\label{lem:three-weight-transfer}
Fix $d\in\mathbb N_+$ and $r>0$, and put $A \eqdef 16$ and $B \eqdef A^{2d}$. Let $S,T\subset\mathbb{R}^d$ be finite nonempty sets, with distinct points of $S$ at distance at least $r$ and distinct points of $T$ at distance at least $r/A$. Suppose that $\varphi:T\to S$ is onto and $|t-\varphi(t)|<r$ for every $t\in T$. Let $(\nu_e)_{e\in S}$ be positive weights of sum one satisfying
\begin{equation}
\label{eq:three-parent-weight-comparison}
 \nu_e\le B^2\nu_{e'}\qquad
 (e,e'\in S,\ |e-e'|<4r).
\end{equation}
Then there are positive weights $(\nu'_t)_{t\in T}$ of sum one such that
\begin{equation}
\label{eq:three-child-weight-bounds}
 B^{-1}\nu_{\varphi(t)}\le\nu'_t\le\nu_{\varphi(t)},\qquad
 \nu'_t\le B^2\nu'_{t'}\quad
 (t,t'\in T,\ |t-t'|<4r/A).
\end{equation}
Moreover, there are numbers $\pi_{e,t}\ge0$ with row sums $\nu_e$, column sums $\nu'_t$, and $\pi_{e,t}>0$ implies that 
\begin{equation}
\label{eq:three-one-step-transport}
 |e-t|<(1+4/A)r.
\end{equation}
\end{lemma}
\begin{proof}
Write $\mathbb B(z,\rho)\eqdef\{y\in\mathbb R^d:|y-z|<\rho\}$, $v_d\eqdef\mathcal L^d(\mathbb B(0,1))>0$, $T_e\eqdef\varphi^{-1}(\{e\})$, $n_e\eqdef|T_e|$, and $w_t\eqdef\nu_{\varphi(t)}$. For $e\in S$, surjectivity and separation give
\[
\begin{gathered}
1\le n_e,\qquad \mathbb B(t,r/(2A))\cap\mathbb B(t',r/(2A))=\varnothing\quad(t,t'\in T_e,\ t\ne t'),\\
\bigcup_{t\in T_e}\mathbb B(t,r/(2A))\subseteq\mathbb B(e,r+r/(2A)),\\
n_ev_d(r/(2A))^d\le v_d(r+r/(2A))^d,\qquad n_e\le(2A+1)^d<A^{2d}=B.
\end{gathered}
\]
For $a,b\in T$, \eqref{eq:three-parent-weight-comparison} yields
\[
\begin{aligned}
|a-b|<8r/A&\Longrightarrow |\varphi(a)-\varphi(b)|\le|\varphi(a)-a|+|a-b|+|b-\varphi(b)|\\
&\hspace{37mm}<(2+8/A)r<4r\Longrightarrow w_a\le B^2w_b.
\end{aligned}
\]

Enumerate $\mathcal E\eqdef\{\{a,b\}:a,b\in T,\ a\ne b,\ |a-b|<4r/A\}=\{E_1,\ldots,E_L\}$ without repetition, allowing $L=0$, and put $u_t^0\eqdef w_t/n_{\varphi(t)}$. Recursively, for $1\le k\le L$ and $E_k=\{a,b\}$, set $u^k=u^{k-1}$ if $\max(u_a^{k-1},u_b^{k-1})\le B^2\min(u_a^{k-1},u_b^{k-1})$. Otherwise, let $(g_k,h_k)$ be the ordering of $E_k$ satisfying $u_{g_k}^{k-1}>B^2u_{h_k}^{k-1}$ and set
\begin{equation}
\label{eq:three-pair-weight-update}
\begin{aligned}
\delta_k&\eqdef\frac{u_{g_k}^{k-1}-B^2u_{h_k}^{k-1}}{B^2+1},&u_t^k&\eqdef u_t^{k-1}+\delta_k\bigl(\mathbf1_{\{t=h_k\}}-\mathbf1_{\{t=g_k\}}\bigr),\\
u_{g_k}^k&=\frac{B^2(u_{g_k}^{k-1}+u_{h_k}^{k-1})}{B^2+1},&u_{h_k}^k&=\frac{u_{g_k}^{k-1}+u_{h_k}^{k-1}}{B^2+1}.
\end{aligned}
\end{equation}
Let $\mathcal I\subseteq\{1,\ldots,L\}$ consist of the nonidentity updates. We prove by induction that
\begin{equation}
\label{eq:three-transfer-intervals}
0<w_t/B\le u_t^k\le w_t\quad(t\in T),\qquad \sum_{t\in T}u_t^k=1\qquad(0\le k\le L).
\end{equation}
At $k=0$, $1\le n_e<B$ and $\sum_{t\in T}u_t^0=\sum_{e\in S}n_e\nu_e/n_e=1$. For $k\in\mathcal I$, abbreviate $g=g_k$, $h=h_k$. The inductive hypothesis and \eqref{eq:three-pair-weight-update} imply $0<\delta_k<u_g^{k-1}$, $u_g^k=B^2u_h^k$, and
\[
\begin{aligned}
w_g/B&\le Bw_h\le B^2u_h^{k-1}\le B^2u_h^k=u_g^k\le u_g^{k-1}\le w_g,\\
w_h/B&\le u_h^{k-1}\le u_h^k=u_g^k/B^2\le u_g^{k-1}/B^2\le w_g/B^2\le w_h,\\
\sum_{t\in T}u_t^k&=\sum_{t\in T}u_t^{k-1}+\delta_k(1-1)=1.
\end{aligned}
\]
The remaining coordinates and the steps $k\notin\mathcal I$ are unchanged.

Put $\mathcal Q_k\eqdef\{\{a,b\}\in\mathcal E:\max(u_a^k,u_b^k)\le B^2\min(u_a^k,u_b^k)\}$. Fix $k\in\mathcal I$, $g=g_k$, $h=h_k$, and $t\in T\setminus\{g,h\}$. For $\{g,t\}\in\mathcal Q_{k-1}$, $|t-h|\le|t-g|+|g-h|<8r/A$, so
\[
u_g^k\le u_g^{k-1}\le B^2u_t^{k-1}=B^2u_t^k,\qquad u_t^k\le w_t\le B^2w_h\le Bu_g^k\le B^2u_g^k.
\]
For $\{h,t\}\in\mathcal Q_{k-1}$, $|g-t|\le|g-h|+|h-t|<8r/A$, so
\[
u_t^k=u_t^{k-1}\le B^2u_h^{k-1}\le B^2u_h^k,\qquad u_h^k=u_g^k/B^2\le w_g/B^2\le w_t\le Bu_t^k\le B^2u_t^k.
\]
Moreover, $u_g^k=B^2u_h^k$ and $E\cap\{g,h\}=\varnothing$ implies $(u_t^k)_{t\in E}=(u_t^{k-1})_{t\in E}$. Consequently,
\[
k\notin\mathcal I\ \Longrightarrow\ \mathcal Q_k=\mathcal Q_{k-1}\ni E_k,\qquad \mathcal Q_k\supseteq\mathcal Q_{k-1}\cup\{E_k\}\quad(1\le k\le L),\qquad \mathcal Q_L\supseteq\bigcup_{k=1}^L\{E_k\}=\mathcal E.
\]
Set $\nu'_t\eqdef u_t^L$. Then $\mathcal Q_L=\mathcal E$, $B\ge256>1$, and~\eqref{eq:three-transfer-intervals} give
\[
\sum_{t\in T}\nu'_t=1,\qquad 0<B^{-1}\nu_{\varphi(t)}\le\nu'_t\le\nu_{\varphi(t)},\qquad \nu'_t\le B^2\nu'_{t'}\quad(|t-t'|<4r/A),
\]
where $t=t'$ uses $B^2\ge1$.

Define $a_{gh}\eqdef\sum_{k\in\mathcal I:(g_k,h_k)=(g,h)}\delta_k$ for $g,h\in T$ and $\mathcal D\eqdef\{(g,h)\in T^2:a_{gh}>0\}$. For every $k\in\mathcal I$, \eqref{eq:three-transfer-intervals} gives
\begin{equation}
\label{eq:three-transfer-parent-decrease}
w_{g_k}\ge u_{g_k}^{k-1}>B^2u_{h_k}^{k-1}\ge Bw_{h_k}.
\end{equation}
If $(g,h),(h,t)\in\mathcal D$, then $|g-t|\le|g-h|+|h-t|<8r/A$ and $w_g>Bw_h>B^2w_t$, contradicting $w_g\le B^2w_t$. Hence
\[
a_{gh}a_{ht}=0\quad(g,h,t\in T),\qquad O_t\eqdef\sum_{h\in T}a_{th},\quad I_t\eqdef\sum_{g\in T}a_{gt},\quad O_tI_t=\sum_{g,h\in T}a_{gt}a_{th}=0.
\]
Summing \eqref{eq:three-pair-weight-update} over $k$ gives $u_t^L=u_t^0-O_t+I_t$. In particular,
\[
c_t\eqdef u_t^0-O_t=\begin{cases}u_t^0,&O_t=0,\\u_t^L,&O_t>0,\end{cases}\qquad c_t>0,\qquad \pi_{e,t}\eqdef\mathbf1_{\{e=\varphi(t)\}}c_t+\sum_{g\in T:\varphi(g)=e}a_{gt}\ge0.
\]
For $e\in S$ and $t\in T$,
\[
\begin{aligned}
\sum_{t\in T}\pi_{e,t}&=\sum_{t\in T_e}(u_t^0-O_t)+\sum_{g\in T_e}\sum_{t\in T}a_{gt}=\sum_{t\in T_e}u_t^0=\nu_e,\\
\sum_{e\in S}\pi_{e,t}&=c_t+\sum_{g\in T}a_{gt}=u_t^0-O_t+I_t=u_t^L=\nu'_t.
\end{aligned}
\]
Finally,
\[
\begin{aligned}
\pi_{e,t}>0&\ \Longrightarrow\ (e=\varphi(t))\ \lor\ (\exists g\in T:\ e=\varphi(g),\ (g,t)\in\mathcal D),\\
e=\varphi(t)&\ \Longrightarrow\ |e-t|<r<(1+4/A)r,\\
e=\varphi(g),\ (g,t)\in\mathcal D&\ \Longrightarrow\ |e-t|\le|\varphi(g)-g|+|g-t|<r+4r/A=(1+4/A)r.
\end{aligned}
\]
\end{proof}

\begin{lemma}[Positive Geometry Weights with a Dimension-Only Certificate]
\label{lem:three-doubling-weights}
For every $d,N\in\mathbb N_+$ and distinct sites $X_1,\ldots,X_N\in\mathbb{R}^d$, there exists $\mu\in\Delta_N^\circ$ satisfying \eqref{eq:doubling} with $D=2^{46d}$.
\end{lemma}
\begin{proof}
We use the finite redistribution construction of~\cite{VolbergKonyagin1988}. If $N=1$, set $\mu_1=1$; for every $r>0$, both sums in~\eqref{eq:doubling} equal $1\le2^{46d}$. Suppose $N\ge2$. Put $c\eqdef4(1+\operatorname{diam}(\mathbb X^N))$, $Y_i\eqdef(X_i-X_1)/c$, and $X\eqdef\{Y_i:i\in[N]\}$. Then
\[
c>0,\qquad |Y_i-Y_k|=c^{-1}|X_i-X_k|,\qquad \operatorname{diam}(X)=\frac{\operatorname{diam}(\mathbb X^N)}{4(1+\operatorname{diam}(\mathbb X^N))}<\frac14.
\]
Write $\mathcal B(z,t)\eqdef\{y\in\mathbb R^d:|y-z|<t\}$ and $B_X(z,t)\eqdef X\cap\mathcal B(z,t)$. Set $A\eqdef16$, $B\eqdef A^{2d}$, $r_j\eqdef A^{-j}$ for $j\in\mathbb N_0$, $\sigma\eqdef\min_{i\ne k}|Y_i-Y_k|>0$, and $J\eqdef\min\{j\in\mathbb N_0:r_j<\sigma\}<\infty$.

For $0\le j\le J$, define $S_j^{(0)}\eqdef\varnothing$ and, with $\operatorname{dist}(z,\varnothing)\eqdef\infty$,
\[
S_j^{(i)}\eqdef\begin{cases}S_j^{(i-1)}\cup\{Y_i\},&\operatorname{dist}(Y_i,S_j^{(i-1)})\ge r_j,\\S_j^{(i-1)},&\operatorname{dist}(Y_i,S_j^{(i-1)})<r_j,\end{cases}\qquad i\in[N],\qquad S_j\eqdef S_j^{(N)}.
\]
For $i<k$ with $Y_i,Y_k\in S_j$, the insertion rule gives $|Y_i-Y_k|\ge r_j$; $Y_i\in S_j$ gives $\operatorname{dist}(Y_i,S_j)=0<r_j$, whereas $Y_i\notin S_j$ gives $\operatorname{dist}(Y_i,S_j)\le\operatorname{dist}(Y_i,S_j^{(i-1)})<r_j$. Consequently,
\[
\min_{z\in S_j}|y-z|<r_j\quad(y\in X),\qquad S_0=\{Y_1\},\qquad r_J<\sigma\Longrightarrow S_J=X.
\]
Set $S_j\eqdef X$ for $j>J$. For every $j\in\mathbb N_0$, let $\varphi_j(t)$ be the least-indexed minimizer of $e\mapsto|t-e|$ on $S_j$, for $t\in S_{j+1}$. Thus $|t-\varphi_j(t)|<r_j$, and $\varphi_j=\operatorname{Id}_X$ for $j\ge J$. For $e\in S_j$, choose $t\in S_{j+1}$ with $|t-e|<r_{j+1}$; then
\[
|e-\varphi_j(t)|\le|e-t|+|t-\varphi_j(t)|\le2|e-t|<2r_{j+1}=\frac{r_j}{8}<r_j\Longrightarrow\varphi_j(t)=e\Longrightarrow\varphi_j(S_{j+1})=S_j.
\]

Set $\nu_0(Y_1)\eqdef1$; then $\nu_0(Y_1)=1\le B^2\nu_0(Y_1)$. Apply Lemma~\ref{lem:three-weight-transfer} inductively with
\[
(S,T,r,\varphi,\nu)=(S_j,S_{j+1},r_j,\varphi_j,\nu_j)\quad(0\le j<J),\qquad 4r_j/A=4r_{j+1}.
\]
Its conclusions give positive probability weights $\nu_j$ on $S_j$, $0\le j\le J$, satisfying
\begin{equation}
\label{eq:three-net-weight-properties}
\begin{gathered}
\sum_{e\in S_j}\nu_j(e)=1,\qquad \nu_j(e)>0,\qquad \nu_j(e)\le B^2\nu_j(e')\quad(e,e'\in S_j,\ |e-e'|<4r_j),\\
B^{-1}\nu_j(\varphi_j(t))\le\nu_{j+1}(t)\le\nu_j(\varphi_j(t))\quad(0\le j<J,\ t\in S_{j+1}),
\end{gathered}
\end{equation}
and, for $0\le j<J$, numbers $\pi_j(e,t)\ge0$ satisfying
\[
\begin{gathered}
\sum_{t\in S_{j+1}}\pi_j(e,t)=\nu_j(e),\qquad \sum_{e\in S_j}\pi_j(e,t)=\nu_{j+1}(t),\\
\pi_j(e,t)>0\Longrightarrow|e-t|<Cr_j,\qquad C\eqdef1+4/A.
\end{gathered}
\]
Put $\mu(y)\eqdef\nu_J(y)$ for $y\in X$ and $\nu_j\eqdef\mu$ for $j\ge J$. Since $r_j\le r_J$, for $e,e'\in X$,
\[
\begin{gathered}
|e-e'|<4r_j\Longrightarrow|e-e'|<4r_J\Longrightarrow\nu_j(e)=\nu_J(e)\le B^2\nu_J(e')=B^2\nu_j(e'),\\
B^{-1}\nu_j(\varphi_j(t))\le\nu_{j+1}(t)=\nu_j(\varphi_j(t)).
\end{gathered}
\]

Set
\begin{equation}
\label{eq:three-net-transport-constants}
C_1\eqdef\frac{1+4/A}{1-1/A}=\frac43,\qquad \frac1{1-1/A}=\frac{16}{15},\qquad C+C_1/A=C_1.
\end{equation}
For $j\ge J$, put $\Pi_j(e,y)\eqdef\mu(e)\mathbf1_{\{e=y\}}$. For $j=J-1,\ldots,0$, recursively put
\[
\Pi_j(e,y)\eqdef\sum_{t\in S_{j+1}}\pi_j(e,t)\frac{\Pi_{j+1}(t,y)}{\nu_{j+1}(t)}\qquad(e\in S_j,\ y\in X).
\]
For $j\ge J$,
\[
\begin{gathered}
\sum_{y\in X}\Pi_j(e,y)=\mu(e)=\nu_j(e),\qquad\sum_{e\in X}\Pi_j(e,y)=\mu(y),\\
\Pi_j(e,y)>0\Longrightarrow e=y\Longrightarrow|e-y|=0<C_1r_j.
\end{gathered}
\]
Since every $\nu_{j+1}(t)>0$, the backward induction step is
\[
\begin{aligned}
\sum_{y\in X}\Pi_j(e,y)&=\sum_{t\in S_{j+1}}\frac{\pi_j(e,t)}{\nu_{j+1}(t)}\sum_{y\in X}\Pi_{j+1}(t,y)=\sum_{t\in S_{j+1}}\pi_j(e,t)=\nu_j(e),\\
\sum_{e\in S_j}\Pi_j(e,y)&=\sum_{t\in S_{j+1}}\frac{\Pi_{j+1}(t,y)}{\nu_{j+1}(t)}\sum_{e\in S_j}\pi_j(e,t)=\sum_{t\in S_{j+1}}\Pi_{j+1}(t,y)=\mu(y),\\
\Pi_j(e,y)>0&\Longrightarrow\exists t\in S_{j+1}:\ \pi_j(e,t)\Pi_{j+1}(t,y)>0\\
&\Longrightarrow|e-y|\le|e-t|+|t-y|<Cr_j+C_1r_{j+1}=C_1r_j.
\end{aligned}
\]

Fix $x\in X$ and $r>0$, and write $\mu(E)\eqdef\sum_{y\in E}\mu(y)$ for $E\subseteq X$. If $r\ge1/2$, then $|y-x|\le\operatorname{diam}(X)<1/4<r$ for every $y\in X$, so $\mu(B_X(x,2r))=\mu(B_X(x,r))=1$. If $0<r<1/2$, put $j_0\eqdef\max\{j\in\mathbb N_0:r_j\ge2r\}$ and $j_1\eqdef\min\{j\in\mathbb N_0:r_j\le r\}$. Since $r_0=1>2r>r>0$ and $r_j\downarrow0$, these indices exist with $j_1\ge1$; their extremal properties give
\begin{equation}
\label{eq:three-net-scale-indices}
\begin{gathered}
r_{j_0}\ge2r>r_{j_0+1}=r_{j_0}/A,\qquad r_{j_1}\le r<r_{j_1-1}=Ar_{j_1},\qquad j_0<j_1,\\
2r\le r_{j_0}<2Ar,\qquad r/A<r_{j_1}\le r,\qquad A^{j_1-j_0}=r_{j_0}/r_{j_1}<2A^2.
\end{gathered}
\end{equation}
Choose $g\in S_{j_1+1}$ with $|g-x|<r_{j_1+1}$, set $g_{j_1+1}\eqdef g$, and recursively set $g_j\eqdef\varphi_j(g_{j+1})$ for $j_0\le j\le j_1$. Then
\begin{equation}
\label{eq:three-net-ancestor-distance}
|g_{j_0}-x|\le|g-x|+\sum_{j=j_0}^{j_1}|g_j-g_{j+1}|<\sum_{j=j_0}^{j_1+1}r_j\le\frac{r_{j_0}}{1-1/A}.
\end{equation}
Writing $h\eqdef j_1+1$, we have
\[
\Pi_h(g,y)>0\Longrightarrow|y-x|\le|y-g|+|g-x|<(C_1+1)r_h\le\frac{C_1+1}{A}r=\frac7{48}r<r.
\]
Hence
\begin{equation}
\label{eq:three-net-inner-mass}
\mu(B_X(x,r))=\sum_{e\in S_h}\sum_{y\in B_X(x,r)}\Pi_h(e,y)\ge\sum_{y\in B_X(x,r)}\Pi_h(g,y)=\sum_{y\in X}\Pi_h(g,y)=\nu_h(g)>0.
\end{equation}

Set $E_0\eqdef\{e\in S_{j_0}:|e-x|<2r+C_1r_{j_0}\}$. For $y\in B_X(x,2r)$,
\[
\Pi_{j_0}(e,y)>0\Longrightarrow|e-x|\le|e-y|+|y-x|<C_1r_{j_0}+2r\Longrightarrow e\in E_0.
\]
For $e\in E_0$, \eqref{eq:three-net-scale-indices} and~\eqref{eq:three-net-ancestor-distance} give
\[
|e-g_{j_0}|\le|e-x|+|x-g_{j_0}|<\left(1+C_1+\frac1{1-1/A}\right)r_{j_0}=\frac{17}{5}r_{j_0}<4r_{j_0}.
\]
Since $|e-e'|\ge r_{j_0}$ for distinct $e,e'\in E_0$, the balls $\mathcal B(e,r_{j_0}/2)$ are pairwise disjoint and their union lies in $\mathcal B(g_{j_0},9r_{j_0}/2)$. Thus, with $\omega_d\eqdef|\mathcal B(0,1)|>0$ denoting Lebesgue volume,
\[
|E_0|\omega_d(r_{j_0}/2)^d=\left|\bigcup_{e\in E_0}\mathcal B(e,r_{j_0}/2)\right|\le\omega_d(9r_{j_0}/2)^d\Longrightarrow|E_0|\le9^d\le12^d.
\]
Consequently, \eqref{eq:three-net-weight-properties} gives
\begin{equation}
\label{eq:three-net-outer-mass}
\mu(B_X(x,2r))=\sum_{e\in E_0}\sum_{y\in B_X(x,2r)}\Pi_{j_0}(e,y)\le\sum_{e\in E_0}\nu_{j_0}(e)\le12^dB^2\nu_{j_0}(g_{j_0}).
\end{equation}
The parent--child bounds and~\eqref{eq:three-net-scale-indices} imply
\[
\nu_{j_0}(g_{j_0})=\nu_h(g)\prod_{j=j_0}^{h-1}\frac{\nu_j(g_j)}{\nu_{j+1}(g_{j+1})}\le B^{h-j_0}\nu_h(g),\qquad B^{j_1-j_0}=\bigl(A^{j_1-j_0}\bigr)^{2d}\le2^{2d}B^2.
\]
Using~\eqref{eq:three-net-inner-mass} and~\eqref{eq:three-net-outer-mass}, we obtain
\begin{equation}
\label{eq:three-dimension-doubling-certificate}
\begin{aligned}
\mu(B_X(x,2r))&\le12^dB^{j_1-j_0+3}\nu_h(g)\\
&\le12^dB^5 2^{2d}\mu(B_X(x,r))=\left(\frac34\right)^d2^{46d}\mu(B_X(x,r))\le2^{46d}\mu(B_X(x,r)).
\end{aligned}
\end{equation}
Finally, set $\mu_i\eqdef\mu(Y_i)$ for $i\in[N]$. Then $\mu_i>0$, $\sum_i\mu_i=1$, and, for every $i\in[N]$ and $\rho>0$,
\[
\sum_{\substack{k\in[N]\\|X_k-X_i|<2\rho}}\mu_k=\mu(B_X(Y_i,2\rho/c))\le2^{46d}\mu(B_X(Y_i,\rho/c))=2^{46d}\sum_{\substack{k\in[N]\\|X_k-X_i|<\rho}}\mu_k.
\]
\end{proof}

\paragraph{Taylor estimates.}

\begin{lemma}[Taylor Remainders and Polarization]
\label{lem:three-taylor}
Fix $d\in\mathbb{N}_+$, $s\in\N$, and $M\ge0$. Let $\Omega\subset\R^d$ be open and convex, and let $u\in C^s(\Omega)$ satisfy $\Lip_\Omega(D^su)\le M$. For $a,x\in\Omega$ and $0\le j\le s$,
\begin{equation}
\label{eq:three-taylor-bound}
 \norm{D^ju(x)-D^jJ_a^su(x)}_{\op}
 \le \frac{M}{(s+1-j)!}|x-a|^{s+1-j}.
\end{equation}
If a symmetric $j$-linear form $T$, $j\ge1$, satisfies $|T[\xi^j]|\le B$ for every $|\xi|=1$, then
\begin{equation}
\label{eq:three-polarization-bound}
 \norm{T}_{\op}\le\frac{j^j}{j!}B.
\end{equation}
\end{lemma}
\begin{proof}
Put $\eta=x-a$ and $k=s-j$. If $k=0$, \eqref{eq:three-taylor-bound} is the defining Lipschitz bound on $D^su$. If $k\ge1$, applying the fundamental theorem of calculus $k$ times to the tensor-valued map $t\mapsto D^ju(a+t\eta)$ gives
\begin{equation}
\label{eq:three-taylor-integral}
 D^ju(a+\eta)-D^jJ_a^su(a+\eta)
 =\frac1{(k-1)!}\int_0^1(1-t)^{k-1}
 \bigl(D^su(a+t\eta)-D^su(a)\bigr)[\eta^k] \,\dd t.
\end{equation}
For completeness, the iterated integral formula follows inductively from $g(1)-g(0)=\int_0^1g'(t)\,\dd t$: integration over $0<t_k<\cdots<t_1<1$ reduces to integration against $(1-t_k)^{k-1}/(k-1)!$. Subtracting the constant value $D^su(a)$ accounts for the last Taylor term. The norm of the integrand in \eqref{eq:three-taylor-integral} is at most $Mt|\eta|^{k+1}$, and $\int_0^1t(1-t)^{k-1}\,\dd t=1/(k(k+1))$. This proves \eqref{eq:three-taylor-bound}.

For \eqref{eq:three-polarization-bound}, multi-linearity gives
\begin{equation}
\label{eq:three-polarization-identity}
 T[v_1,\ldots,v_j]
 =\frac1{2^jj!}\sum_{\varepsilon\in\{-1,1\}^j}
 \left(\prod_{i=1}^j\varepsilon_i\right)
 T\left[\left(\sum_{i=1}^j\varepsilon_i v_i\right)^j\right].
\end{equation}
Indeed, after expansion, summation in $\varepsilon_i$ kills every term in which $v_i$ occurs an even number of times. Since the total degree is $j$, the surviving terms contain every $v_i$ exactly once; there are $j!$ such terms. For $|v_i|\le1$, each vector in \eqref{eq:three-polarization-identity} has norm at most $j$. The triangle inequality now gives \eqref{eq:three-polarization-bound}.
\end{proof}

The measured Taylor polynomials impose a minimum amount of variation on the highest derivative of any compatible extension. Lemma~\ref{lem:finite-whitney-seminorm}, proved below, quantifies this requirement and gives an explicit calibration constant.
We remark that, $W_{\mathbb{D}^N}\le M$ where
    %%
    %%
    %%
    %%% HERE IS THE MOVED REPLACEMENT --- Unless we put backthe claibration from finite data thing
    %%
    %%
    %%
     \begin{equation}
    \label{eq:finite-whitney-seminorm}
     W_{\mathbb{D}^N}
     \eqdef
     \begin{cases}
     \displaystyle
     \max_{\substack{i,j\in[N]\\i\ne j}}
     \max_{\substack{\alpha\in\mathbb N^d\\|\alpha|\le s}}
     \frac{|\partial^\alpha(P_i-P_j)(X_j)|}
     {|X_i-X_j|^{s+1-|\alpha|}},&N\ge2,\\[2mm]
     0,&N=1.
     \end{cases}
    \end{equation}

    \begin{lemma}[Finite-Jet seminorm]
    \label{lem:finite-whitney-seminorm}
    Fix $d,N\in\mathbb N_+$ and $s\in\mathbb N$. Let
    $X_1,\ldots,X_N\in\mathbb{R}^d$ be distinct, and let
    $P_1,\ldots,P_N$ be real polynomials of degree at most $s$.
    For these data $\mathbb{D}^N\eqdef\{(X_i,P_i)\}_{i=1}^N$, define $W_{\mathbb{D}^N}$ as in~\eqref{eq:finite-whitney-seminorm}.
    There is a constant $\kappa_{d,s}\ge1$, depending only on $d,s$,
    such that:
    \begin{enumerate}
    \item[(i)] Every $f\in\dot{C}^{s,1}(\mathbb{R}^d)$ with
    $J_{X_i}^sf=P_i$ for all $i\in[N]$ satisfies
    $W_{\mathbb{D}^N}\le[f]_{s,1}$.
    \item[(ii)] There exists $g\in\dot{C}^{s,1}(\mathbb{R}^d)$ with
    $J_{X_i}^sg=P_i$ for all $i\in[N]$ and
    $[g]_{s,1}\le\kappa_{d,s}W_{\mathbb{D}^N}$.
    \item[(iii)] $W_{\mathbb{D}^N}=0$ if and only if all $P_i$
    are the same polynomial. In that case one may take $g=P_1$.
    \end{enumerate}
    \end{lemma}
    \begin{proof}
    For (i), the assertion is immediate from the definition when $N=1$.
    For $i\ne j$ and $|\alpha|\le s$, apply
    Lemma~\ref{lem:three-taylor} to $f$ on $\mathbb{R}^d$, with base
    point $X_i$, evaluation point $X_j$, and derivative order $|\alpha|$.
    Since $D^{|\alpha|}P_j(X_j)=D^{|\alpha|}f(X_j)$ and a coordinate
    component is bounded by the operator norm, this gives
    \[
     |\partial^\alpha(P_i-P_j)(X_j)|
     \le \frac{[f]_{s,1}}{(s+1-|\alpha|)!}
           |X_i-X_j|^{s+1-|\alpha|}
     \le [f]_{s,1}|X_i-X_j|^{s+1-|\alpha|}.
    \]
    Division by the positive denominator and maximization prove (i).
    
    If $W_{\mathbb{D}^N}=0$ and $N\ge2$, every derivative through
    order $s$ of $P_i-P_j$ vanishes at $X_j$. The exact Taylor
    expansion of this polynomial, whose degree is at most $s$, is
    therefore zero. Hence $P_i=P_j$. The assertion for $N=1$ is
    automatic, and the converse follows directly from the definition.
    A polynomial of degree at most $s$ has constant order-$s$
    derivative and thus has seminorm zero. This proves (iii) and
    also proves (ii) when $W_{\mathbb{D}^N}=0$.
    
    Suppose $W\eqdef W_{\mathbb{D}^N}>0$. Then $N\ge2$, and
    \[
     a\eqdef X_1,\qquad
     R\eqdef\max_{i\in[N]}|X_i-a|>0,\qquad
     z_i\eqdef\frac{X_i-a}{R},\qquad
     Q_i(z)\eqdef\frac{P_i(a+Rz)-P_1(a+Rz)}{WR^{s+1}}
    \]
    are well defined. Regard $Q_i$, of degree at most $s$, as an
    $(s+1)$-jet at $z_i$; its derivatives of order $s+1$ are zero.
    For $i\ne1$ and $|\alpha|\le s$, the ordered pair $(1,i)$ in~\eqref{eq:finite-whitney-seminorm} yields
    \[
     |\partial^\alpha Q_i(z_i)|
     =\frac{|\partial^\alpha(P_i-P_1)(X_i)|}
            {WR^{s+1-|\alpha|}}
     \le\left(\frac{|X_i-a|}{R}\right)^{s+1-|\alpha|}
     \le1.
    \]
    For $i=1$, this inequality holds because $Q_1=0$; for
    $|\alpha|=s+1$, it holds because the derivative is zero.
    For $i\ne j$ and $|\alpha|\le s$, one also has
    \[
     |\partial^\alpha(Q_i-Q_j)(z_j)|
     =\frac{|\partial^\alpha(P_i-P_j)(X_j)|}
            {WR^{s+1-|\alpha|}}
     \le |z_i-z_j|^{s+1-|\alpha|}.
    \]
    When $i=j$ the difference is zero; when $|\alpha|=s+1$
    both derivatives vanish. Thus all the hypotheses of
    \cite[Lemma~3.6, p.~327]{Fefferman2006} hold with
    $n=d$, $m=s+1$, $S=\{z_i:i\in[N]\}$, $P^{z_i}=Q_i$,
    and $A=1$. The sites remain distinct under the rescaling.
    That lemma supplies $H\in C^{s+1}(\mathbb{R}^d)$ satisfying
    \[
     J_{z_i}^{s+1}H=Q_i\quad(i\in[N]),\qquad
     \max_{|\alpha|\le s+1}
     \sup_{z\in\mathbb{R}^d}|\partial^\alpha H(z)|
     \le B_{d,s},
    \]
    where $B_{d,s}\ge1$ depends only on $d,s$. Here any fixed
    conversion from the source's $C^{s+1}$ norm to the displayed
    maximum is absorbed in $B_{d,s}$.
    For unit vectors $u_1,\ldots,u_{s+1}$, expansion in coordinates
    and the inequality $\sum_{k=1}^d|(u_\ell)_k|\le\sqrt d$ give
    \[
     |D^{s+1}H(z)[u_1,\ldots,u_{s+1}]|
     \le B_{d,s}\prod_{\ell=1}^{s+1}
                  \sum_{k=1}^d|(u_\ell)_k|
     \le d^{(s+1)/2}B_{d,s}.
    \]
    Set $\kappa_{d,s}\eqdef d^{(s+1)/2}B_{d,s}$ and
    $
     g(x)\eqdef P_1(x)+WR^{s+1}H\!\left(\frac{x-a}{R}\right)
    $.
    For every $i\in[N]$ and $|\alpha|\le s$, the chain rule and
    the jet identity for $H$ give
    \[
     \partial^\alpha g(X_i)
     =\partial^\alpha P_1(X_i)
      +WR^{s+1-|\alpha|}\partial^\alpha Q_i(z_i)
     =\partial^\alpha P_i(X_i).
    \]
    Thus $J_{X_i}^sg=P_i$. Since $D^{s+1}P_1=0$, the same
    chain rule yields
    $\sup_x\|D^{s+1}g(x)\|_{\op}\le\kappa_{d,s}W$.
    For $x,y\in\mathbb{R}^d$, integration along their segment gives
    \[
     D^sg(y)-D^sg(x)
     =\int_0^1D^{s+1}g(x+t(y-x))[y-x,\,\cdot\,]\,\mathrm dt,
    \]
    where the expression is scalar when $s=0$. Taking operator
    norms gives
    \[
     \|D^sg(y)-D^sg(x)\|_{\op}\le\kappa_{d,s}W|y-x|.
    \]
    This proves (ii). 
    % Neither the number of sites nor their diameter or minimum separation enters $\kappa_{d,s}$.
    \end{proof}
We next treat $s=0,1$ separately, since the entropy term vanishes at these orders. The following lemmas identify the resulting classical formulas and establish interpolation together with the unchanged regularity bound $M$.

\begin{lemma}[Order-Zero Values and Lipschitz Bound]
\label{lem:three-zero-order}
Under \eqref{eq:training_data}, if $s=0$, the defined extension satisfies
\begin{equation}
\label{eq:three-zero-envelope}
 F(x)=\min_{i\in[N]}\{f(X_i)+M|x-X_i|\},\qquad x\in\R^d.
\end{equation}
It is continuous, $F(X_i)=P_i(X_i)$, and $\Lip(F)\le M$.
\end{lemma}
\begin{proof}
For $s=0$ one has $\mathcal I_0=\{1\}$, $\beta=0$, $p=m=1$, $A_s=1$, and $P_i(x)=f(X_i)$. Hence $F=F_1$, the variance coefficient vanishes because it contains the factor $s=0$, and the entropy coefficient vanishes because $\beta=0$. 
% Off the sites, the objective in \eqref{eq:entropy-variance-extension} is consequently a convex combination 
Off the sites $\{X_i\}_{i=1}^N$, the objective defining $F_1$ in~\eqref{eq:entropy-variance-extension__primal} is consequently a convex combination
of the finitely many numbers $f(X_i)+M|x-X_i|$. Every such combination is at least their minimum, and a simplex vertex attaining that minimum realizes equality. At $x=X_n$, \eqref{eq:training_data} gives $f(X_i)+M|X_n-X_i|\ge f(X_n)$ for every $i$, with equality at $i=n$. Hence \eqref{eq:three-zero-envelope} also equals the assigned value at each site.
For arbitrary $x,y\in\R^d$, choose an index $i$ minimizing \eqref{eq:three-zero-envelope} at $y$. Then
\[
 F(x)-F(y)\le M\bigl(|x-X_i|-|y-X_i|\bigr)\le M|x-y|.
\]
Choosing instead an index minimizing at $x$ gives $F(y)-F(x)\le M|x-y|$. Thus $|F(x)-F(y)|\le M|x-y|$, which proves the asserted Lipschitz bound and continuity, including at the sites.
\end{proof}

\begin{lemma}[Quadratic Envelope as Variance]
\label{lem:three-quadratic-envelope}
Fix $d,N\in\mathbb{N}_+$ and $M>0$. Let $X_i,G_i\in\R^d$ and $f_i\in\R$ for $i=1,\ldots,N$, and set $P_i(x)=f_i+G_i\cdot(x-X_i)$. Define
\begin{equation}
\label{eq:three-quadratic-data}
g_i(y)\eqdef P_i(y)+\frac M2|y-X_i|^2+\frac M2|y|^2,\qquad g(y)\eqdef\min_i g_i(y),\qquad v_i(x)\eqdef G_i+M(x-X_i).
\end{equation}
Then, for every $x\in\R^d$,
\begin{equation}
\label{eq:three-variance-envelope-identity}
\conv(g)(x)-\frac M2|x|^2=\min_{\lambda\in\Delta_N}\left\{\sum_i\lambda_i\left(P_i(x)+\frac M2|x-X_i|^2\right)-\frac1{4M}\Var_\lambda(v(x))\right\}.
\end{equation}
\end{lemma}
\begin{proof}
Define
\begin{equation}
\label{eq:three-grouped-envelope}
H(x)\eqdef\inf\left\{\sum_i\lambda_i g_i(y_i):\lambda\in\Delta_N,\ (y_i)_{i=1}^N\in(\R^d)^N,\ \sum_i\lambda_i y_i=x\right\}.
\end{equation}
Put $\ell_i\eqdef G_i-MX_i$, $c_i\eqdef f_i-G_i\cdot X_i+\frac M2|X_i|^2$, and $B\eqdef\max\{0,\max_i(|\ell_i|^2/(2M)-c_i)\}$. Then
\[
g_i(y)=M|y|^2+\ell_i\cdot y+c_i=\frac M2|y|^2+\frac M2\left|y+\frac{\ell_i}{M}\right|^2+c_i-\frac{|\ell_i|^2}{2M}\ge\frac M2|y|^2-B.
\]
For every admissible $(\lambda,(y_i))$, the identity $\sum_i\lambda_i|y_i|^2=|x|^2+\sum_i\lambda_i|y_i-x|^2$ gives
\[
\sum_i\lambda_i g_i(y_i)\ge\frac M2|x|^2+\frac M2\sum_i\lambda_i|y_i-x|^2-B\ge\frac M2|x|^2-B.
\]
Taking $\lambda=e_j$ and $y_i=x$ for every $i$ also gives $H(x)\le g_j(x)$; hence $\frac M2|x|^2-B\le H(x)\le g(x)<\infty$.
Fix $x,x'\in\R^d$, $t\in(0,1)$, and $\varepsilon>0$. Choose admissible $(\lambda,(y_i))$ and $(\lambda',(y_i'))$ at $x$ and $x'$, respectively, with $\sum_i\lambda_i g_i(y_i)\le H(x)+\varepsilon$ and $\sum_i\lambda_i'g_i(y_i')\le H(x')+\varepsilon$. Set
\[
\widehat\lambda_i\eqdef t\lambda_i+(1-t)\lambda_i',\qquad
\widehat y_i\eqdef
\begin{cases}
\displaystyle\frac{t\lambda_i y_i+(1-t)\lambda_i'y_i'}{\widehat\lambda_i},&\widehat\lambda_i>0,\\
0,&\widehat\lambda_i=0.
\end{cases}
\]
Then $\widehat\lambda_i\ge0$, $\sum_i\widehat\lambda_i=t+(1-t)=1$, and
\[
\sum_i\widehat\lambda_i\widehat y_i=t\sum_i\lambda_i y_i+(1-t)\sum_i\lambda_i'y_i'=tx+(1-t)x'.
\]
For every $i\in[N]$,
\[
t\lambda_i g_i(y_i)+(1-t)\lambda_i'g_i(y_i')-\widehat\lambda_i g_i(\widehat y_i)
=M\bigl(t\lambda_i|y_i-\widehat y_i|^2+(1-t)\lambda_i'|y_i'-\widehat y_i|^2\bigr)\ge0.
\]
Consequently,
\[
\begin{aligned}
H(tx+(1-t)x')&\le\sum_i\widehat\lambda_i g_i(\widehat y_i)\\
&\le t\sum_i\lambda_i g_i(y_i)+(1-t)\sum_i\lambda_i'g_i(y_i')\\
&\le tH(x)+(1-t)H(x')+\varepsilon.
\end{aligned}
\]
Letting $\varepsilon\downarrow0$ proves convexity; the cases $t=0,1$ are equalities.

Fix $x\in\R^d$ and $r>0$. Set $V\eqdef x+r\{-1,1\}^d$ and $A\eqdef\max_{v\in V}H(v)$. Since $x+[-r,r]^d$ is the convex hull of $V$, convexity gives $H\le A$ on this cube. For $0<\|h\|_\infty<r$, put $t\eqdef\|h\|_\infty/r$ and $u_\pm\eqdef x\pm h/t$. Then $u_\pm\in x+[-r,r]^d$ and
$
x+h=(1-t)x+tu_+$ implies that $H(x+h)\le(1-t)H(x)+tA$; while 
$x=\frac{x+h}{1+t}+\frac{t}{1+t}u_-$ implies that $H(x)\le\frac{H(x+h)+tA}{1+t}$.  
Thus $|H(x+h)-H(x)|\le r^{-1}(A-H(x))\|h\|_\infty\to0$ as $h\to0$, so $H$ is continuous. If $u:\R^d\to\R$ is any lower-semicontinuous convex minorant of $g$, every admissible representation gives
\[
u(x)\le\sum_i\lambda_i u(y_i)\le\sum_i\lambda_i g(y_i)\le\sum_i\lambda_i g_i(y_i).
\]
Taking the infimum yields $u\le H$. Since $H$ is itself a continuous convex minorant of $g$, $H=\conv(g)$.

Fix $x\in\R^d$ and $\lambda\in\Delta_N$. Write $y_i=x+h_i$ and $\bar v\eqdef\sum_i\lambda_i v_i(x)$, so that $\sum_i\lambda_i y_i=x$ is equivalent to $\sum_i\lambda_i h_i=0$. Under this constraint,
\begin{equation}
\label{eq:three-quadratic-expansion}
\begin{aligned}
\sum_i\lambda_i g_i(x+h_i)
&=\sum_i\lambda_i g_i(x)+\sum_i\lambda_i\bigl((v_i(x)+Mx)\cdot h_i+M|h_i|^2\bigr)\\
&=\sum_i\lambda_i\left(P_i(x)+\frac M2|x-X_i|^2\right)+\frac M2|x|^2+\sum_i\lambda_i\bigl(v_i(x)\cdot h_i+M|h_i|^2\bigr).
\end{aligned}
\end{equation}
The identities $\sum_i\lambda_i(v_i(x)-\bar v)=0$, $\sum_i\lambda_i\bar v\cdot h_i=0$, and $\Var_\lambda(v(x))=\sum_i\lambda_i|v_i(x)-\bar v|^2$ give
\begin{equation}
\label{eq:three-variance-completion}
\begin{aligned}
\sum_i\lambda_i\bigl(v_i(x)\cdot h_i+M|h_i|^2\bigr)
&=M\sum_i\lambda_i\left|h_i+\frac{v_i(x)-\bar v}{2M}\right|^2-\frac1{4M}\Var_\lambda(v(x)).
\end{aligned}
\end{equation}
For $h_i^*\eqdef-(v_i(x)-\bar v)/(2M)$, one has $\sum_i\lambda_i h_i^*=-(\sum_i\lambda_i v_i(x)-\bar v)/(2M)=0$, and every square in \eqref{eq:three-variance-completion} vanishes. Therefore,
\[
\inf_{\substack{(h_i)\in(\R^d)^N\\ \sum_i\lambda_i h_i=0}}\sum_i\lambda_i g_i(x+h_i)
=\sum_i\lambda_i\left(P_i(x)+\frac M2|x-X_i|^2\right)+\frac M2|x|^2-\frac1{4M}\Var_\lambda(v(x)).
\]
The right-hand side is polynomial in $\lambda$, and $\Delta_N$ is nonempty and compact. Taking its minimum over $\Delta_N$ in \eqref{eq:three-grouped-envelope}, and using $H=\conv(g)$, proves \eqref{eq:three-variance-envelope-identity}.
\end{proof}

\begin{lemma}[First-Order Jet and Gradient Control]
\label{lem:three-first-order}
Under \eqref{eq:training_data}, if $s=1$, then $F\in C^1(\R^d)$,
\begin{equation}
\label{eq:three-first-order-conclusions}
F(X_i)=P_i(X_i),\qquad \nabla F(X_i)=\nabla P_i(X_i),\qquad |\nabla F(x)-\nabla F(y)|\le M|x-y|
\end{equation}
for $i\in[N]$ and $x,y\in\R^d$.
\end{lemma}
\begin{proof}
For $s=1$, one has $\mathcal I_1=\{1\}$, $F=F_1$, $p=2$, $\beta=0$, and $A_s=1/2$. Put $f_i\eqdef f(X_i)$ and $G_i\eqdef\nabla f(X_i)$; then \eqref{eq:training_data} gives $P_i(x)=f_i+G_i\cdot(x-X_i)$. Let $g_i,g,v_i$ be as in Lemma~\ref{lem:three-quadratic-envelope}, and define $\widetilde F(x)\eqdef\conv(g)(x)-M|x|^2/2$ for $x\in\R^d$. By \eqref{eq:entropy-variance-extension__primal} and \eqref{eq:three-variance-envelope-identity},
\[
F(x)=\min_{\lambda\in\Delta_N}\left\{\sum_i\lambda_i\left(P_i(x)+\frac M2|x-X_i|^2\right)-\frac1{4M}\Var_\lambda(v(x))\right\}=\widetilde F(x),\qquad x\notin\mathbb X^N.
\]

Fix $a,b\in\mathbb X^N$, and put $G_a\eqdef\nabla f(a)$, $G_b\eqdef\nabla f(b)$, $\eta\eqdef b-a$, $\Delta\eqdef G_b-G_a$, and $z\eqdef(a+b)/2+\Delta/(2M)$. Applying Lemma~\ref{lem:three-taylor} with $\Omega=\R^d$, $u=f$, $s=1$, and $j=0$ gives
\[
f(b)+G_b\cdot(z-b)-\frac M2|z-b|^2\le f(z)\le f(a)+G_a\cdot(z-a)+\frac M2|z-a|^2.
\]
Since $z-a=\eta/2+\Delta/(2M)$ and $z-b=-\eta/2+\Delta/(2M)$,
\begin{equation}
\label{eq:three-sharp-compatibility}
\begin{aligned}
f(b)-f(a)
&\le G_a\cdot(z-a)-G_b\cdot(z-b)+\frac M2\bigl(|z-a|^2+|z-b|^2\bigr)\\
&=\frac12(G_a+G_b)\cdot\eta-\frac{|\Delta|^2}{2M}+\frac M4|\eta|^2+\frac{|\Delta|^2}{4M}\\
&=\frac12(G_a+G_b)\cdot(b-a)+\frac M4|a-b|^2-\frac1{4M}|G_a-G_b|^2.
\end{aligned}
\end{equation}

Put $b_i\eqdef G_i-MX_i$, $d_i\eqdef f_i-G_i\cdot X_i+\frac M2|X_i|^2-\frac{|b_i|^2}{4M}$, and $\bar b(\lambda)\eqdef\sum_i\lambda_i b_i$. For $x\in\R^d$ and $\lambda\in\Delta_N$,
\[
\begin{aligned}
P_i(x)+\frac M2|x-X_i|^2&=\frac M2|x|^2+x\cdot b_i+d_i+\frac{|b_i|^2}{4M},\\
\Var_\lambda(v(x))&=\sum_i\lambda_i|b_i+Mx|^2-|\bar b(\lambda)+Mx|^2=\sum_i\lambda_i|b_i|^2-|\bar b(\lambda)|^2.
\end{aligned}
\]
Substitution in \eqref{eq:three-variance-envelope-identity} gives
\begin{equation}
\label{eq:three-first-order-potential}
\begin{aligned}
\widetilde F(x)&=\frac M2|x|^2+\psi(x),\qquad \psi(x)\eqdef\min_{\lambda\in\Delta_N}\Phi_x(\lambda),\\
\Phi_x(\lambda)&\eqdef\sum_i\lambda_i d_i+x\cdot\bar b(\lambda)+\frac{|\bar b(\lambda)|^2}{4M}.
\end{aligned}
\end{equation}
Since $\Phi_x$ is continuous and $\Delta_N$ is nonempty and compact, $\Lambda_x\eqdef\operatorname*{argmin}_{\lambda\in\Delta_N}\Phi_x(\lambda)\ne\varnothing$. For $\lambda,\lambda'\in\Lambda_x$,
\[
\psi(x)\le\Phi_x\left(\frac{\lambda+\lambda'}2\right)=\frac{\Phi_x(\lambda)+\Phi_x(\lambda')}2-\frac{|\bar b(\lambda)-\bar b(\lambda')|^2}{16M}=\psi(x)-\frac{|\bar b(\lambda)-\bar b(\lambda')|^2}{16M}.
\]
Thus $\bar b(\lambda)=\bar b(\lambda')$, and $b_x\eqdef\bar b(\lambda)$, $\lambda\in\Lambda_x$, is well defined.

Fix $x,y\in\R^d$ and choose $\lambda_x\in\Lambda_x$, $\lambda_y\in\Lambda_y$. For $t\in(0,1]$, $(1-t)\lambda_x+t\lambda_y\in\Delta_N$ and
\[
\begin{aligned}
0&\le\frac{\Phi_x((1-t)\lambda_x+t\lambda_y)-\Phi_x(\lambda_x)}{t}\\
&=\sum_i(\lambda_{y,i}-\lambda_{x,i})d_i+\left(x+\frac{b_x}{2M}\right)\cdot(b_y-b_x)+\frac{t}{4M}|b_y-b_x|^2.
\end{aligned}
\]
Letting $t\downarrow0$ and interchanging $x,y$ gives
\[
\begin{aligned}
0&\le\sum_i(\lambda_{y,i}-\lambda_{x,i})d_i+\left(x+\frac{b_x}{2M}\right)\cdot(b_y-b_x),\\
0&\le\sum_i(\lambda_{x,i}-\lambda_{y,i})d_i+\left(y+\frac{b_y}{2M}\right)\cdot(b_x-b_y).
\end{aligned}
\]
Adding yields
\begin{equation}
\label{eq:three-first-order-monotonicity}
(x-y)\cdot(b_x-b_y)\le-\frac{|b_x-b_y|^2}{2M}.
\end{equation}
Consequently,
\[
|b_x-b_y|^2\le-2M(x-y)\cdot(b_x-b_y)\le2M|x-y|\,|b_x-b_y|,
\qquad |b_x-b_y|\le2M|x-y|.
\]
Moreover,
\[
\begin{aligned}
\psi(y)&\le\Phi_y(\lambda_x)=\psi(x)+b_x\cdot(y-x),\\
\psi(x)&\le\Phi_x(\lambda_y)=\psi(y)+b_y\cdot(x-y).
\end{aligned}
\]
Hence, for $h\in\R^d$,
\[
-2M|h|^2\le(b_{x+h}-b_x)\cdot h\le\psi(x+h)-\psi(x)-b_x\cdot h\le0.
\]
It follows that $\lim_{h\to0,\ h\ne0}|\psi(x+h)-\psi(x)-b_x\cdot h|/|h|=0$, so $\nabla\psi(x)=b_x$. The Lipschitz estimate for $x\mapsto b_x$ gives $\psi,\widetilde F\in C^1(\R^d)$ and $\nabla\widetilde F(x)=Mx+b_x$. By \eqref{eq:three-first-order-monotonicity},
\begin{equation}
\label{eq:three-first-order-sharp-gradient}
\begin{aligned}
|\nabla\widetilde F(x)-\nabla\widetilde F(y)|^2
&=M^2|x-y|^2+2M(x-y)\cdot(b_x-b_y)+|b_x-b_y|^2\\
&\le M^2|x-y|^2.
\end{aligned}
\end{equation}

Fix $i\in[N]$ and define $\ell_i(z)\eqdef f_i+\frac M2|X_i|^2+(G_i+MX_i)\cdot(z-X_i)$. For $j\in[N]$, put $z_{ij}\eqdef(X_i+X_j)/2+(G_i-G_j)/(2M)$ and $C_{ij}\eqdef g_j(z_{ij})-\ell_i(z_{ij})$. Completing the square gives, for $z\in\R^d$,
\begin{equation}
\label{eq:three-first-order-contact}
\begin{aligned}
g_j(z)-\ell_i(z)&=M|z-z_{ij}|^2+C_{ij},\\
\min_{z\in\R^d}(g_j(z)-\ell_i(z))
&=C_{ij}=f_j-f_i+\frac12(G_i+G_j)\cdot(X_i-X_j)\\
&\quad+\frac M4|X_i-X_j|^2-\frac1{4M}|G_i-G_j|^2\ge0,
\end{aligned}
\end{equation}
where the inequality is \eqref{eq:three-sharp-compatibility} with $(a,b)=(X_j,X_i)$. Therefore,
\[
\ell_i\le\conv(g)\le g_i,\qquad
\ell_i(X_i)=g_i(X_i)=f_i+\frac M2|X_i|^2,
\qquad \widetilde F(X_i)=f_i.
\]
Set $\varphi_i\eqdef\conv(g)-\ell_i$. Since $\conv(g)=\widetilde F+M|\cdot|^2/2$, one has $\varphi_i\in C^1(\R^d)$, $\varphi_i\ge0$, and $\varphi_i(X_i)=0$. For every $e\in\R^d$ and both signs,
\[
0\le\lim_{t\downarrow0}\frac{\varphi_i(X_i\pm te)-\varphi_i(X_i)}{t}
=\pm\nabla\varphi_i(X_i)\cdot e.
\]
Thus $\nabla\varphi_i(X_i)=0$ and $\nabla\widetilde F(X_i)=\nabla\ell_i(X_i)-MX_i=G_i$. The site assignment \eqref{eq:entropy-variance-extension__primal___SITES} and the off-site identity give
\begin{equation}
\label{eq:three-first-order-envelope-global}
F(x)=\widetilde F(x)=\conv(g)(x)-\frac M2|x|^2,\qquad x\in\R^d.
\end{equation}
Consequently, $F\in C^1(\R^d)$, $F(X_i)=f_i=P_i(X_i)$, $\nabla F(X_i)=G_i=\nabla P_i(X_i)$, and \eqref{eq:three-first-order-sharp-gradient} yields $|\nabla F(x)-\nabla F(y)|\le M|x-y|$ for all $x,y\in\R^d$.
\end{proof}

\paragraph{Higher-order geometry and entropy.}

For $s\ge2$, the next two lemmas supply the geometric and optimization controls needed for differentiation. The first compares $R(x)$ with the nearest-site distance; the second gives a unique optimizer that varies smoothly away from the sites, together with a bounded inverse for the derivative of its optimality equation.

\begin{lemma}[Distance Scale from Doubling]
\label{lem:three-scale}
Under \eqref{eq:training_data} and~\eqref{eq:doubling}, suppose $s\ge2$ and put
\begin{equation}
\label{eq:three-L-definition}
 L_{s,D} \eqdef \left(\frac43\,2^qD^2\right)^{1/(2s)}.
\end{equation}
Then, for $x\notin \mathbb{X}^N$,
\begin{equation}
\label{eq:three-R-comparison}
 \rho_{\mathbb{X}^N}(x)\le R(x)\le L_{s,D}\rho_{\mathbb{X}^N}(x),
 \qquad
 L_{s,D}^{s+1}\le8\,2^{s+1}D^3.
\end{equation}
\end{lemma}
\begin{proof}
Fix $x\notin \mathbb{X}^N$, set $\rho=\rho_{\mathbb{X}^N}(x)$, and choose a site $a$ at distance $\rho$. Write $\mu_0=\mu(B_{\mathbb{X}^N}(a,\rho))>0$. Every point of this ball is at distance less than $2\rho$ from $x$, whence $S_q(x)\ge(2\rho)^{-q}\mu_0$. For $j\ge0$, the sites with $2^j\rho\le r_i(x)<2^{j+1}\rho$ belong to $B_{\mathbb{X}^N}(a,2^{j+2}\rho)$ and have total mass at most $D^{j+2}\mu_0$. Since $q>2s$,
\begin{equation}
\label{eq:three-annulus}
 S_{q-2s}(x)
 \le D^2\mu_0\rho^{2s-q}
 \sum_{j=0}^\infty(D2^{2s-q})^j
 \le\frac43D^2\mu_0\rho^{2s-q}.
\end{equation}
The last step uses $D2^{2s-q}\le1/4$, which follows from \eqref{eq:three-aux-parameters}. Divide \eqref{eq:three-annulus} by the lower bound for $S_q$ and use $R^{2s}=S_{q-2s}/S_q$. The lower bound for $R$ follows from $r_i\ge\rho$ and $\sum_i\theta_i=1$.
Finally, $2^q\le2^{2s+3}D$ and $(s+1)/(2s)\le3/4$. Consequently,
\begin{equation}
\label{eq:three-L-simplification}
 L_{s,D}^{s+1}
 \le2^{s+1}(32/3)^{(s+1)/(2s)}D^{3(s+1)/(2s)}
 \le8\,2^{s+1}D^3,
\end{equation}
since $32/3<16$ and $D\ge1$.
\end{proof}

\begin{lemma}[Entropy Duality and Smoothness]
\label{lem:three-entropy}
Fix $N,d\in\mathbb{N}_+$. Let $w_i\in\R$, $b_i\in\R^d$, $\theta_i>0$ with $\sum_i\theta_i=1$, and $c>0$. The function
\begin{equation}
\label{eq:three-abstract-primal}
 \mathcal L(\lambda) \eqdef \sum_i\lambda_i w_i+c\left|\sum_i\lambda_i b_i\right|^2
 +\KL(\lambda\mid\theta)
\end{equation}
has a unique minimizer $\lambda^*\in\Delta_N$, and every coordinate of $\lambda^*$ is positive. If $z_*=\sum_i\lambda_i^*b_i$, then
\begin{equation}
\label{eq:three-gibbs-dual}
 \begin{aligned}
 \lambda_i^*&=\frac{\theta_i \mathrm{e}^{-w_i-2c z_*\cdot b_i}}
 {\sum_j\theta_j \mathrm{e}^{-w_j-2c z_*\cdot b_j}},\\
 \min_{\lambda\in\Delta_N}\mathcal L(\lambda)
 &=\max_{z\in\R^d}\left\{-c|z|^2-
 \log\sum_i\theta_i \mathrm{e}^{-w_i-2c z\cdot b_i}\right\}.
 \end{aligned}
\end{equation}
The maximizer is unique. These optimizers and the optimal value depend smoothly on any finite-dimensional parameter on which $w_i,b_i,\theta_i$ depend smoothly, while the weights remain positive.
\end{lemma}
\begin{proof}
The objective is continuous on $\Delta_N$. A minimizing sequence has a coordinatewise convergent subsequence in $[0,1]^N$; its limit belongs to $\Delta_N$ and attains the infimum by continuity. Suppose a minimizer has $\lambda_j=0$ and $\lambda_k=L>0$. Put $z=\sum_i\lambda_i b_i$ and $d_b=b_j-b_k$. Replace these coordinates by $\varepsilon$ and $L-\varepsilon$, with $0<\varepsilon<L/2$. The change in the objective is
\begin{equation}
\label{eq:three-boundary-variation}
 \varepsilon(w_j-w_k)+2c\varepsilon z\cdot d_b+c\varepsilon^2|d_b|^2
 +\varepsilon\log(\varepsilon/\theta_j)
 +(L-\varepsilon)\log((L-\varepsilon)/\theta_k)-L\log(L/\theta_k).
\end{equation}
Set $B=\max_{r\in[L/2,L]}|1+\log(r/\theta_k)|$ and
$C=|w_j-w_k|+2c|z||d_b|+c|d_b|^2+|\log\theta_j|+B$.
The fundamental theorem of calculus bounds \eqref{eq:three-boundary-variation} above by $\varepsilon\log\varepsilon+C\varepsilon$ whenever $\varepsilon<1$. Taking $\varepsilon<\min\{L/2,1,\mathrm{e}^{-C-1}\}$ contradicts minimality. On the interior, the Hessian of the entropy is $\operatorname{diag}(1/\lambda_i)$, while the quadratic term is convex. Strict convexity along every nonconstant segment in the interior proves uniqueness.

At an interior minimizer the derivatives in every feasible direction $e_i-e_j$ vanish. Thus all partial derivatives of the objective are equal, giving
$ w_i+2cz_*\cdot b_i+\log(\lambda_i^*/\theta_i)+1=\text{constant}$,
which proves the first line of \eqref{eq:three-gibbs-dual}. To justify the second line without a minimax interchange, let
$Z(z)=\sum_i\theta_i \mathrm{e}^{-w_i-2cz\cdot b_i}$ and
$\pi_i(z)=\theta_i \mathrm{e}^{-w_i-2cz\cdot b_i}/Z(z)$.
For every $\lambda\in\Delta_N$,
\begin{equation}
\label{eq:three-duality-gap}
 \sum_i\lambda_i w_i+c\left|\sum_i\lambda_i b_i\right|^2
 +\KL(\lambda\mid\theta)
 =-c|z|^2-\log Z(z)
 +c\left|\sum_i\lambda_i b_i-z\right|^2
 +\KL(\lambda\mid\pi(z)).
\end{equation}
The relative entropy is nonnegative: $\log t\le t-1$ gives
$\sum_{\lambda_i>0}\lambda_i\log(\pi_i/\lambda_i)
\le\sum_{\lambda_i>0}\pi_i-1\le0$.
Equality in \eqref{eq:three-duality-gap} holds for $z=z_*$ and $\lambda=\lambda^*$. This proves duality and attainment.

Write $\bar b(z)=\sum_i\pi_i(z)b_i$ and define
$\Cov_{\pi(z)}(b)=\sum_i\pi_i(z)(b_i-\bar b(z))(b_i-\bar b(z))^\top$.
Differentiating the finite sum gives
\begin{equation}
\label{eq:three-dual-hessian}
 \nabla_z(-c|z|^2-\log Z(z))=-2c(z-\bar b(z)),\qquad
 D_z^2(-c|z|^2-\log Z(z))=-2cI-4c^2\Cov_{\pi(z)}(b).
\end{equation}
Here $v^\top\Cov_\pi(b)v=\sum_i\pi_i(v\cdot(b_i-\bar b))^2\ge0$. Thus the dual is strictly concave and its maximizer is unique. The stationarity map $K=z-\bar b$ has $D_zK=I+2c\Cov_\pi(b)$. For every vector $u$, $u^\top D_zK\,u\ge|u|^2$, so Cauchy-Schwarz gives $|D_zK\,u|\ge|u|$. This square matrix is injective and therefore invertible, and its inverse has norm at most one.

We record the local smoothness argument used here and later. Let $K(x_0,z_0)=0$ and let $J=D_zK(x_0,z_0)$ be invertible. Write $\Phi_x(z)=z-J^{-1}K(x,z)$. By continuity of $D_zK$, there are convex parameter and $z$ neighborhoods with compact closures on which $\norm{D_z\Phi_x(z)}_{\op}\le1/2$. Choose a closed ball of radius $r>0$ about $z_0$ in the latter neighborhood, and shrink the parameter neighborhood until $|\Phi_x(z_0)-z_0|\le r/2$. The derivative bound, integrated on a segment, gives $|\Phi_x(z)-\Phi_x(z')|\le|z-z'|/2$ in the ball. Hence $\Phi_x$ maps the ball to itself. Its iterates from $z_0$ have successive differences at most $2^{-n}$ times the first difference, so they are Cauchy. Their limit is a fixed point by continuity and is unique in the ball by the same contraction inequality.

Let $A$ bound $\norm{D_x\Phi_x(z)}_{\op}$ on this compact product. For two parameters in the chosen convex neighborhood, their fixed points satisfy
\begin{equation}
\label{eq:three-implicit-local-Lipschitz}
 |z(x)-z(x')|\le\tfrac12|z(x)-z(x')|+A|x-x'|,
 \qquad |z(x)-z(x')|\le2A|x-x'|.
\end{equation}
Put $\Delta z=z(x+\eta)-z(x)$. Differentiability of $K$ and \eqref{eq:three-implicit-local-Lipschitz} give
\[
 \begin{aligned}
 0&=D_xK(x,z(x))\eta+D_zK(x,z(x))\Delta z+o(|\eta|+|\Delta z|)\\
  &=D_xK(x,z(x))\eta+D_zK(x,z(x))\Delta z+o(|\eta|).
 \end{aligned}
\]
Here $D_zK$ remains invertible: $J^{-1}D_zK=I-D_z\Phi_x$ and the inverse of $I-D_z\Phi_x$ is its convergent geometric matrix series. Solving the last expansion proves
\begin{equation}
\label{eq:three-implicit-derivative}
 Dz(x)=-(D_zK(x,z(x)))^{-1}D_xK(x,z(x)).
\end{equation}
The right side is continuous, so $z$ is $C^1$. Matrix inversion is smooth where the determinant is nonzero, by the adjugate formula. If $z$ is $C^k$, the right side of \eqref{eq:three-implicit-derivative} is $C^k$, since $K$ is smooth; hence $z$ is $C^{k+1}$. Induction gives smoothness of every finite order. The local fixed points solve the stationarity equation and therefore equal the unique global maximizer by \eqref{eq:three-dual-hessian}. The Gibbs formula then proves smoothness of the weights and optimal value.
\end{proof}

We now subtract one fixed Taylor polynomial at the evaluation point and express the remaining quantities in units determined by the local distance scale. The resulting exponential decay of the optimal weights controls the contribution of distant sites, including weighted powers of their relative distances, uniformly in $N$.

\begin{lemma}[Uniform Normalization and Tails]
\label{lem:three-normalization}
Under \eqref{eq:training_data} and~\eqref{eq:doubling}, suppose $s\ge2$ and fix $\vartheta\in[0,1]$. Fix a witness $f$, fix $x_0\notin \mathbb{X}^N$, put $\rho=\rho_{\mathbb{X}^N}(x_0)$, and fix the polynomial $Q=J_{x_0}^sf$. In a neighborhood of $x_0$, define
\begin{equation}
\label{eq:three-normalized-data}
 \begin{aligned}
 c&\eqdef\frac{s}{4(s-1)},
 &\delta_i(x)&\eqdef\frac{P_i(x)-Q(x)}{MR(x)^m},\\
 b_i^{(\vartheta)}(x)&\eqdef
 \sqrt{\vartheta}\,\frac{v_i(x)-\nabla Q(x)}{MR(x)^s},
 &w_i^{(\vartheta)}(x)&\eqdef
 \frac{\delta_i(x)}\beta+
 \frac{A_s}\beta\left(\frac{r_i(x)}{R(x)}\right)^{2s}
 -c\left|b_i^{(\vartheta)}(x)\right|^2.
 \end{aligned}
\end{equation}
Let $\mathfrak h_\vartheta(x)$ be the minimum in~\eqref{eq:three-abstract-primal} with costs $w_i^{(\vartheta)}(x)$, vectors $b_i^{(\vartheta)}(x)$, weights $\theta_i(x)$, and the fixed coefficient $c=s/[4(s-1)]$. Then
\begin{equation}
\label{eq:three-F-normalized}
 F_\vartheta(x)=Q(x)+\beta MR(x)^m\mathfrak h_\vartheta(x).
\end{equation}
At $x_0$, abbreviate $u_i=(r_i(x_0)/R(x_0))^s$ and $V_i=1+u_i^2$, and let $\pi_i^{(\vartheta)}$ be the minimizing weights and
\[
 z_{0,\vartheta}\eqdef
 \sum_i\pi_i^{(\vartheta)}b_i^{(\vartheta)}(x_0).
\]
Then
\begin{equation}
\label{eq:three-normalized-bounds}
 \begin{gathered}
 \left|b_i^{(\vartheta)}(x_0)\right|\le u_i,
 \qquad
 \frac{u_i^2}{12}-1\le w_i^{(\vartheta)}(x_0)\le1+u_i^2,
 \qquad
 \sum_i\theta_i(x_0)u_i^2=1,\\
 -1\le\mathfrak h_\vartheta(x_0)\le\frac56,
 \qquad
 |z_{0,\vartheta}|\le3,
 \qquad
 \pi_i^{(\vartheta)}
 \le\mathrm e^{64}\theta_i(x_0)\mathrm e^{-u_i^2/24}.
 \end{gathered}
\end{equation}
With $K_0=\mathrm e^{65}$ and $\kappa=24\mathrm e$, for every integer $k\ge1$,
\begin{equation}
\label{eq:three-moment-bound}
 \sum_i\pi_i^{(\vartheta)}V_i^k\le K_0\kappa^k k!.
\end{equation}
At the endpoint $\vartheta=0$,
\begin{equation}
\label{eq:three-zero-parameter-normalization}
 b_i^{(0)}\equiv0,
 \qquad
 z_{0,0}=0.
\end{equation}
All constants in these conclusions are independent of the fixed parameter $\vartheta\in[0,1]$.
\end{lemma}
\begin{proof}
Fix $x$ in the neighborhood under consideration and $\lambda\in\Delta_N$. Translation invariance of weighted variance follows by expanding its definition:
\begin{equation}
\label{eq:three-variance-translation}
 \operatorname{Var}_\lambda(v-\nabla Q)
 =\sum_i\lambda_i|v_i-\nabla Q|^2
 -\left|\sum_i\lambda_i(v_i-\nabla Q)\right|^2
 =\operatorname{Var}_\lambda(v),
\end{equation}
where $\sum_i\lambda_i=1$. By~\eqref{eq:three-normalized-data},
\begin{equation}
\label{eq:three-variance-scaling}
 \operatorname{Var}_\lambda\bigl(b^{(\vartheta)}(x)\bigr)
 =\frac{\vartheta}{M^2R(x)^{2s}}
 \operatorname{Var}_\lambda(v(x)).
\end{equation}
Moreover,
$
 \sum_i\lambda_i|b_i^{(\vartheta)}|^2
 -\left|\sum_i\lambda_i b_i^{(\vartheta)}\right|^2
 =\operatorname{Var}_\lambda\bigl(b^{(\vartheta)}\bigr)$.  
Multiplying the objective in~\eqref{eq:three-abstract-primal} by $\beta MR(x)^m$ and adding $Q(x)$ gives
\begin{align*}
 &Q(x)+\beta MR(x)^m
 \left[
 \sum_i\lambda_iw_i^{(\vartheta)}(x)
 +c\left|\sum_i\lambda_i b_i^{(\vartheta)}(x)\right|^2
 +\operatorname{KL}(\lambda\mid\theta(x))
 \right]\\
 &\quad=Q(x)+\sum_i\lambda_i(P_i(x)-Q(x))
 +\sum_i\lambda_i\frac{A_sM r_i(x)^{2s}}{R(x)^{s-1}}\\
 &\qquad\quad
 -\beta cMR(x)^m
 \operatorname{Var}_\lambda\bigl(b^{(\vartheta)}(x)\bigr)
 +\beta MR(x)^m\operatorname{KL}(\lambda\mid\theta(x))\\
 &\quad=\sum_i\lambda_i
 \left[P_i(x)+\frac{A_sM r_i(x)^{2s}}{R(x)^{s-1}}\right]
 -\frac{\vartheta s}{4MR(x)^{s-1}}
 \operatorname{Var}_\lambda(v(x))
 +(s-1)MR(x)^{s+1}\operatorname{KL}(\lambda\mid\theta(x)).
\end{align*}
Here $m=s+1$, $\beta=s-1$, $\beta c=s/4$, and~\eqref{eq:three-variance-scaling} were used in the last equality. Taking the minimum over $\lambda\in\Delta_N$ proves~\eqref{eq:three-F-normalized}. No division by $\vartheta$ or $\sqrt\vartheta$ occurs. Lemma~\ref{lem:three-entropy} applies because the auxiliary coefficient $c=s/[4(s-1)]$ remains strictly positive and every $\theta_i$ is positive.

At $x_0$, Lemma~\ref{lem:three-taylor}, first with derivative order zero and then with order one, gives
\begin{equation}
\label{eq:three-base-remainders}
 |\delta_i(x_0)|\le\frac{t_i^m}{m!},
 \qquad
 \left|b_i^{(\vartheta)}(x_0)\right|
 \le\sqrt\vartheta\,\frac{2t_i^s}{s!}
 \le\frac{2t_i^s}{s!},
 \qquad
 t_i\eqdef\frac{r_i(x_0)}{R(x_0)}.
\end{equation}
Since $2/s!\le1$, the vector bound in~\eqref{eq:three-normalized-bounds} follows. Also
\begin{equation}
\label{eq:three-a-lower-general}
 w_i^{(\vartheta)}(x_0)
 \ge
 \left(
 \frac12+\frac1{\beta m!}
 -\frac{\vartheta s}{\beta(s!)^2}
 \right)t_i^{2s}
 -\frac{t_i^m}{\beta m!}.
\end{equation}
When $s=2$, this gives the endpoint estimate
\begin{equation}
\label{eq:three-s-two-parameter-bound}
 w_i^{(\vartheta)}(x_0)
 \ge-\frac{t_i^3}{6}
 +\left(\frac23-\frac\vartheta2\right)t_i^4
 \ge-\frac{t_i^3}{6}+\frac{t_i^4}{6}.
\end{equation}
The polynomial $t^4/12-t^3/6$ has minimum $-9/64$ on $[0,\infty)$, attained at $t=3/2$; its derivative is $t^2(2t-3)/6$. Hence~\eqref{eq:three-s-two-parameter-bound} yields $w_i^{(\vartheta)}\ge t_i^4/12-9/64\ge u_i^2/12-1$. For $s\ge3$, the coefficient in~\eqref{eq:three-a-lower-general} is bounded below by its value at $\vartheta=1$. The inequality $t_i^m\le1+t_i^{2s}$ follows by treating $t_i\le1$ and $t_i\ge1$ separately. The positive sequences $s/((s-1)(s!)^2)$ and $1/((s-1)(s+1)!)$ decrease: their consecutive ratios are $(s-1)/(s^2(s+1))$ and $(s-1)/(s(s+2))$, both less than one. At $s=3$ their values are $1/24$ and $1/48$, respectively. Substitution into~\eqref{eq:three-a-lower-general} gives $w_i^{(\vartheta)}\ge(11/24)t_i^{2s}-1/48\ge u_i^2/12-1$. Dropping the nonpositive term $-c|b_i^{(\vartheta)}|^2$ and using $t_i^m\le1+t_i^{2s}$ gives
\[
 w_i^{(\vartheta)}(x_0)
 \le\frac1{\beta m!}
 +\left(\frac12+\frac2{\beta m!}\right)t_i^{2s}
 \le1+t_i^{2s}.
\]
The moment identity for $\theta$ follows from the definition of $R$.

The lower cost bound, the nonnegativity of the quadratic term, and the nonnegativity of relative entropy give $\mathfrak h_\vartheta(x_0)\ge-1$. At $\lambda=\theta(x_0)$, evaluate the unexpanded normalized family objective and retain its nonpositive variance term. The inequality
\[
 \sum_i\theta_i t_i^m
 \le\left(\sum_i\theta_i t_i^{2s}\right)^{m/(2s)}=1
\]
follows from concavity of $t\mapsto t^{m/(2s)}$ and finite Jensen. Therefore
\[
 \mathfrak h_\vartheta(x_0)
 \le\frac{A_s}{\beta}+\frac1{\beta m!}
 =\frac12+\frac2{\beta m!}\le\frac56.
\]
At the optimizer,
\[
 c|z_{0,\vartheta}|^2
 \le\mathfrak h_\vartheta(x_0)+1\le\frac{11}{6}.
\]
Since $c\ge1/4$, this proves $|z_{0,\vartheta}|<3$. At $\vartheta=0$, the definition gives $b_i^{(0)}\equiv0$, and hence its weighted mean is $z_{0,0}=0$.

The Gibbs exponent satisfies
\begin{equation}
\label{eq:three-exponent-tail}
 w_i^{(\vartheta)}(x_0)
 +2c z_{0,\vartheta}\cdot b_i^{(\vartheta)}(x_0)
 \ge\frac{u_i^2}{12}-1-3u_i
 \ge\frac{u_i^2}{24}-55.
\end{equation}
Indeed, $2c\le1$, $|z_{0,\vartheta}|<3$, and $|b_i^{(\vartheta)}|\le u_i$; the last inequality uses $u^2/24-3u\ge-54$. The identity $\sum_i\theta_i u_i^2=1$ implies that the set $u_i^2\le2$ has $\theta$-mass at least $1/2$. On that set the exponent is at most $3+3\sqrt2<8$. Therefore its partition function is at least $\mathrm e^{-8}/2$. The first line of~\eqref{eq:three-gibbs-dual} and~\eqref{eq:three-exponent-tail} give
\[
 \pi_i^{(\vartheta)}
 \le2\mathrm e^{63}\theta_i\mathrm e^{-u_i^2/24}
 \le\mathrm e^{64}\theta_i\mathrm e^{-u_i^2/24}.
\]
Finally, for $k\ge1$, the maximum of $(1+v)^k\mathrm e^{-v/24}$ on $v\ge0$ is $(24k)^k\mathrm e^{-k+1/24}\le\mathrm e(24k)^k$. Thus
\[
 \sum_i\pi_i^{(\vartheta)}V_i^k
 \le\mathrm e^{65}(24k)^k.
\]
The bound $k!\ge(k/\mathrm e)^k$ follows from
$\sum_{j=1}^k\log j\ge\int_1^k\log t\,\dd t=k\log k-k+1$. This proves~\eqref{eq:three-moment-bound} uniformly in $\vartheta$.
\end{proof}

\paragraph{Finite Taylor-jet estimates.}
The remaining task is to control repeated derivatives of both the geometric quantities and the optimizer. We organize the product and chain-rule estimates by bounding Taylor coefficients with scalar series having nonnegative coefficients; bounds on these series then control each derivative needed below, using only finite Taylor orders.

The following notation is used only in the proof. A formal series with nonnegative coefficients can be evaluated at positive arguments whenever the corresponding sum is finite. All comparisons below are coefficientwise, rather than pointwise comparisons of functions.

\begin{definition}[Taylor Coefficient Majorization]
\label{def:three-majorization}
Let $(E,\|\cdot\|_E)$ be a finite-dimensional normed real vector space, let $U\subseteq\R\times\R^d$ be an open neighborhood of $(0,0)$, and let $G\in C^\infty(U;E)$. For $j,k\in\Nzero$, the mixed derivative $\partial_t^jD_y^kG(0,0)$ is an $E$-valued symmetric $k$-linear map on $(\R^d)^k$. For $k\ge1$, equip this map with the operator norm
\[
 \bigl\|\partial_t^jD_y^kG(0,0)\bigr\|_{\op}
 \eqdef
 \sup_{\substack{h_1,\ldots,h_k\in\R^d\\|h_1|,\ldots,|h_k|\le1}}
 \bigl\|\partial_t^jD_y^kG(0,0)[h_1,\ldots,h_k]\bigr\|_E.
\]
For $k=0$, set $D_y^0G\eqdef G$ and interpret $\|\partial_t^jD_y^0G(0,0)\|_{\op}$ as $\|\partial_t^jG(0,0)\|_E$. Thus the same notation covers scalar-valued, vector-valued, and tensor-valued maps by choosing the corresponding finite-dimensional normed codomain $E$.

An $E$-valued formal series in two indeterminates is a family $(A_{jk})_{j,k\in\Nzero}\subseteq E$, written
\[
 A(t,v)=\sum_{j,k\ge0}A_{jk}t^jv^k.
\]
A scalar formal majorant is a series
\[
 B(t,v)=\sum_{j,k\ge0}B_{jk}t^jv^k,
 \qquad B_{jk}\in[0,\infty).
\]
For either series, write $[t^jv^k]A\eqdef A_{jk}$ and $[t^jv^k]B\eqdef B_{jk}$. In these expressions, $t$ and $v$ are formal indeterminates; in particular, no convergence of either series is assumed.

For $n\in\Nzero$, define the truncated coefficientwise relations
\begin{equation}
\label{eq:three-majorization-definition}
\begin{aligned}
 A\preccurlyeq_n B
 &\quad\text{when}\quad
 \bigl\|[t^jv^k]A\bigr\|_E
 \le [t^jv^k]B
 &&\text{for every }j,k\in\Nzero\text{ with }j+k\le n,\\
 G\preccurlyeq_n B
 &\quad\text{when}\quad
 \frac{\bigl\|\partial_t^jD_y^kG(0,0)\bigr\|_{\op}}{j!\,k!}
 \le [t^jv^k]B
 &&\text{for every }j,k\in\Nzero\text{ with }j+k\le n.
\end{aligned}
\end{equation}
We write $A\preccurlyeq B$ when $A\preccurlyeq_n B$ for every $n\in\Nzero$, and we write $G\preccurlyeq B$ when $G\preccurlyeq_n B$ for every $n\in\Nzero$. When $E=\R$, the norm $\|\cdot\|_E$ in the formal-series relation is absolute value.

If $G$ depends only on the scalar variable $t$ and $B(t)=\sum_{j\ge0}B_jt^j$, then $G\preccurlyeq_n B$ means
\[
 \frac{\|\partial_t^jG(0)\|_E}{j!}\le B_j,
 \qquad 0\le j\le n,
\]
and $G\preccurlyeq B$ means that these inequalities hold for every $j\in\Nzero$. This is precisely the specialization of~\eqref{eq:three-majorization-definition} to $k=0$. The same convention applies to one-variable formal series.
\end{definition}
\begin{remark}
The relation $G\preccurlyeq_n B$ depends only on the order-$n$ Taylor jet of $G$ at $(0,0)$ and is therefore defined whenever $G$ is of class $C^n$ near $(0,0)$. The untruncated notation $G\preccurlyeq B$ requires derivatives of every finite order, but does not assert that the Taylor series of $G$ converges or represents $G$\footnote{Likewise, the formal relation $A\preccurlyeq B$ is purely coefficientwise; whenever a later argument evaluates $A$ or $B$ at positive numerical arguments, convergence at those arguments must be established separately.}. 
\end{remark}

\begin{lemma}[Algebra of Taylor Majorants]
\label{lem:three-majorant-calculus}
Majorization is preserved by addition, bounded bilinear products, and composition with a zero-constant-term majorant. In particular:
\begin{enumerate}
\item If $G\preccurlyeq B$, $H\preccurlyeq C$, and the bilinear multiplication has norm at most one, then $GH\preccurlyeq BC$.
\item If $G$ is scalar valued, $G(0,0)=0$, and $G\preccurlyeq B$ with $B(0,0)=0$, then
\begin{equation}
\label{eq:three-analytic-majorants}
 \mathrm{e}^G-1\preccurlyeq \mathrm{e}^B-1,\quad
 (1+G)^{-1}\preccurlyeq(1-B)^{-1},\quad
 \log(1+G)\preccurlyeq-\log(1-B).
\end{equation}
\item If $G\preccurlyeq B(t,v)$ and $y(0)=0$, $y(t)\preccurlyeq Y(t)$ with $Y(0)=0$, then $G(t,y(t))\preccurlyeq B(t,Y(t))$.
\end{enumerate}
These statements apply to finite jets without an analyticity hypothesis on $G,H,y$.
\end{lemma}
\begin{proof}
For a smooth map $K$ with values in a finite-dimensional normed space, and for $j,k\in\Nzero$, set
\[
 \mathfrak c_{j,k}(K)
 \eqdef
 \frac{\|\partial_t^jD_y^kK(0,0)\|_{\op}}{j!\,k!}.
\]
Thus $K\preccurlyeq A=\sum_{j,k\ge0}A_{jk}t^jv^k$ means precisely that $\mathfrak c_{j,k}(K)\le A_{jk}$ for every $j,k\in\Nzero$. We prove the three assertions coefficient by coefficient.

Suppose first that $G,H$ take values in the same normed space and satisfy $G\preccurlyeq B$ and $H\preccurlyeq C$. The triangle inequality gives $\mathfrak c_{j,k}(G+H)\le B_{jk}+C_{jk}=[t^jv^k](B+C)$ for every $j,k\in\Nzero$. Hence $G+H\preccurlyeq B+C$.

We next prove the bilinear-product assertion. Let $E_1,E_2,E_3$ be finite-dimensional normed spaces, let $\mathfrak m:E_1\times E_2\to E_3$ be bilinear, and suppose
\[
 \|\mathfrak m(x_1,x_2)\|_{E_3}
 \le \|x_1\|_{E_1}\|x_2\|_{E_2},
 \qquad x_1\in E_1,\quad x_2\in E_2.
\]
Write $K\eqdef\mathfrak m(G,H)$. Fix $j,k\in\Nzero$ and vectors $h_1,\ldots,h_k\in\R^d$. For a subset $I=\{i_1<\cdots<i_r\}\subseteq\{1,\ldots,k\}$, write $h_I\eqdef(h_{i_1},\ldots,h_{i_r})$; for $I=\varnothing$, the corresponding derivative is the order-zero derivative. The mixed Leibniz rule gives
\begin{equation}
\label{eq:three-majorant-mixed-Leibniz}
 \partial_t^jD_y^kK(0,0)[h_1,\ldots,h_k]
 =
 \sum_{a=0}^j\binom ja
 \sum_{I\subseteq\{1,\ldots,k\}}
 \mathfrak m\!\left(
 \partial_t^aD_y^{|I|}G(0,0)[h_I],
 \partial_t^{j-a}D_y^{k-|I|}H(0,0)[h_{I^c}]
 \right).
\end{equation}
Indeed, one first distributes the $j$ derivatives in the scalar variable between the two factors, producing the factor $\binom ja$, and then distributes the labelled $y$-directions $h_1,\ldots,h_k$ between the two factors, producing the sum over subsets $I$.

Assume now that $|h_1|,\ldots,|h_k|\le1$. Applying the bilinear bound in~\eqref{eq:three-majorant-mixed-Leibniz}, taking the supremum over these vectors, grouping the subsets according to $b=|I|$, and using $G\preccurlyeq B$ and $H\preccurlyeq C$, we obtain
\begin{align*}
 \|\partial_t^jD_y^kK(0,0)\|_{\op}
 &\le
 \sum_{a=0}^j\binom ja
 \sum_{b=0}^k\binom kb
 \|\partial_t^aD_y^bG(0,0)\|_{\op}
 \|\partial_t^{j-a}D_y^{k-b}H(0,0)\|_{\op}\\
 &\le
 \sum_{a=0}^j\sum_{b=0}^k
 \binom ja\binom kb
 a!\,b!\,B_{ab}
 (j-a)!\,(k-b)!\,C_{j-a,k-b}\\
 &=j!\,k!
 \sum_{a=0}^j\sum_{b=0}^k
 B_{ab}C_{j-a,k-b}.
\end{align*}
Consequently,
\[
 \mathfrak c_{j,k}(K)
 \le
 \sum_{a=0}^j\sum_{b=0}^kB_{ab}C_{j-a,k-b}
 =
 [t^jv^k](BC).
\]
Since $j,k$ were arbitrary, $\mathfrak m(G,H)\preccurlyeq BC$. This proves the addition and bilinear-product assertions.

We next prove~\eqref{eq:three-analytic-majorants}. Let $G$ be scalar valued, assume $G(0,0)=0$, and suppose $G\preccurlyeq B$ with $B_{00}=0$. Repeated application of the bilinear-product estimate gives
\begin{equation}
\label{eq:three-majorant-powers}
 G^m\preccurlyeq B^m,
 \qquad m\in\Nzero,
\end{equation}
where $G^0=B^0=1$. Moreover,
\begin{equation}
\label{eq:three-majorant-power-vanishing}
 \partial_t^jD_y^k(G^m)(0,0)=0
 \qquad\text{whenever}\qquad j+k<m.
\end{equation}
To verify~\eqref{eq:three-majorant-power-vanishing}, apply the mixed Leibniz rule to the product of the $m$ factors of $G$. Every resulting term distributes the total number $j+k$ of derivatives among these $m$ factors. If $j+k<m$, at least one factor receives no derivative and is therefore evaluated as $G(0,0)=0$.

Fix $j,k\in\Nzero$ and put $r\eqdef j+k$. Since $G(0,0)=0$ and $G$ is continuous, after restricting the neighborhood of $(0,0)$ we may assume $|G|<1/2$. Thus $(1+G)^{-1}$ and $\log(1+G)$ are well defined and smooth there. Taylor's theorem in one real variable yields smooth functions $\psi_{\exp,r}$, $\psi_{\mathrm{inv},r}$, and $\psi_{\log,r}$ near zero such that
\begin{align*}
 \mathrm e^x-1
 &=\sum_{m=1}^{r}\frac{x^m}{m!}
 +x^{r+1}\psi_{\exp,r}(x),\\
 (1+x)^{-1}
 &=\sum_{m=0}^{r}(-1)^m x^m
 +x^{r+1}\psi_{\mathrm{inv},r}(x),\\
 \log(1+x)
 &=\sum_{m=1}^{r}\frac{(-1)^{m+1}}{m}x^m
 +x^{r+1}\psi_{\log,r}(x).
\end{align*}
Empty sums are understood as zero. By~\eqref{eq:three-majorant-power-vanishing}, every mixed derivative of total order at most $r$ of
\[
 G^{r+1}\psi_{\exp,r}(G),\qquad
 G^{r+1}\psi_{\mathrm{inv},r}(G),\qquad
 G^{r+1}\psi_{\log,r}(G)
\]
vanishes at $(0,0)$. Indeed, in every term produced by the product and chain rules, at most $r$ derivatives are distributed among the $r+1$ explicit factors of $G$, so at least one of those factors remains undifferentiated and vanishes at $(0,0)$.

Using~\eqref{eq:three-majorant-powers}, the triangle inequality, and the preceding finite Taylor expansions, we therefore obtain
\begin{align*}
 \mathfrak c_{j,k}(\mathrm e^G-1)
 &\le
 \sum_{m=1}^{r}\frac1{m!}[t^jv^k]B^m,\\
 \mathfrak c_{j,k}((1+G)^{-1})
 &\le
 \sum_{m=0}^{r}[t^jv^k]B^m,\\
 \mathfrak c_{j,k}(\log(1+G))
 &\le
 \sum_{m=1}^{r}\frac1m[t^jv^k]B^m.
\end{align*}
Because $B_{00}=0$, every monomial occurring in $B^m$ has total degree at least $m$. Hence $[t^jv^k]B^m=0$ whenever $m>j+k=r$. It follows that the finite sums above are exactly the corresponding coefficients of the formal series
\[
 \mathrm e^B-1
 \eqdef\sum_{m=1}^{\infty}\frac{B^m}{m!},
 \qquad
 (1-B)^{-1}
 \eqdef\sum_{m=0}^{\infty}B^m,
 \qquad
 -\log(1-B)
 \eqdef\sum_{m=1}^{\infty}\frac{B^m}{m}.
\]
For every fixed coefficient, each series on the right contains only finitely many nonzero contributions. We have therefore proved
\[
 \mathrm e^G-1\preccurlyeq\mathrm e^B-1,
 \qquad
 (1+G)^{-1}\preccurlyeq(1-B)^{-1},
 \qquad
 \log(1+G)\preccurlyeq-\log(1-B).
\]

It remains to prove the composition assertion. To avoid confusing the curve with the vector variable of $G$, write $\eta(t)$ for the map denoted by $y(t)$ in the statement. Thus $\eta:(-\varepsilon,\varepsilon)\to\R^d$ is smooth, $\eta(0)=0$, and
\[
 \eta\preccurlyeq Y(t)=\sum_{\ell\ge0}Y_\ell t^\ell,
 \qquad Y_\ell\ge0,
 \qquad Y_0=0.
\]
Define the normalized Taylor coefficients
\[
 \mathcal G_{a,b}
 \eqdef
 \frac1{a!\,b!}\partial_t^aD_y^bG(0,0),
 \qquad
 \eta_\ell
 \eqdef
 \frac1{\ell!}\eta^{(\ell)}(0).
\]
Then $\mathcal G_{a,b}$ is an output-valued symmetric $b$-linear map,
\[
 \|\mathcal G_{a,b}\|_{\op}\le B_{ab},
 \qquad
 |\eta_\ell|\le Y_\ell,
\]
and $\eta_0=0$.
Fix $n\in\Nzero$ and let
\[
 T_nG(t,y)
 \eqdef
 \sum_{\substack{a,b\in\Nzero\\a+b\le n}}
 t^a\mathcal G_{a,b}[y^b]
\]
be the total-degree-$n$ Taylor polynomial of $G$ at $(0,0)$, where $\mathcal G_{a,0}[y^0]\eqdef\mathcal G_{a,0}$. The multivariate Peano remainder formula gives
\[
 \|G(t,y)-T_nG(t,y)\|
 =o\!\left((|t|+|y|)^n\right)
 \qquad\text{as }(t,y)\to(0,0).
\]
Since $\eta(0)=0$ and $\eta$ is differentiable at zero, $\eta(t)=O(|t|)$. Consequently,
\[
 G(t,\eta(t))-T_nG(t,\eta(t))=o(|t|^n).
\]
Let $\eta^{[n]}(t)\eqdef\sum_{\ell=1}^n\eta_\ell t^\ell$. Taylor's theorem gives $\eta(t)-\eta^{[n]}(t)=o(|t|^n)$. For every $b\ge1$, multilinearity yields the telescoping identity
\begin{align*}
 \mathcal G_{a,b}[\eta(t)^b]
 -\mathcal G_{a,b}[\eta^{[n]}(t)^b]
 &=
 \sum_{q=1}^b
 \mathcal G_{a,b}\!\left[
 \eta(t)^{q-1},
 \eta(t)-\eta^{[n]}(t),
 \eta^{[n]}(t)^{b-q}
 \right].
\end{align*}
Both $\eta(t)$ and $\eta^{[n]}(t)$ are $O(|t|)$, so the right-hand side is $o(|t|^{n+b-1})$, and hence is $o(|t|^n)$. Since $T_nG$ contains only finitely many terms,
\[
 G(t,\eta(t))
 -
 T_nG(t,\eta^{[n]}(t))
 =o(|t|^n).
\]
The left-hand side is of class $C^n$ in $t$. Its one-variable Taylor formula therefore shows that its derivatives at zero through order $n$ vanish. Hence the normalized coefficient of order $n$ of the composite is the coefficient obtained by formally substituting $\eta^{[n]}$ into $T_nG$:
\begin{equation}
\label{eq:three-majorant-composition-coefficient}
 \frac1{n!}\frac{\dd^n}{\dd t^n}G(t,\eta(t))\bigg|_{t=0}
 =
 \sum_{\substack{a,b\in\Nzero,\ \ell_1,\ldots,\ell_b\in\N_+\\
 a+\ell_1+\cdots+\ell_b=n}}
 \mathcal G_{a,b}
 [\eta_{\ell_1},\ldots,\eta_{\ell_b}].
\end{equation}
For $b=0$, the inner tuple is empty, its product is one, and the constraint in~\eqref{eq:three-majorant-composition-coefficient} reduces to $a=n$.

Taking norms in~\eqref{eq:three-majorant-composition-coefficient} gives
\begin{align*}
 \frac1{n!}
 \left\|
 \frac{\dd^n}{\dd t^n}G(t,\eta(t))\bigg|_{t=0}
 \right\|
 &\le
 \sum_{\substack{a,b\in\Nzero,\ \ell_1,\ldots,\ell_b\in\N_+\\
 a+\ell_1+\cdots+\ell_b=n}}
 \|\mathcal G_{a,b}\|_{\op}
 \prod_{q=1}^b|\eta_{\ell_q}|\\
 &\le
 \sum_{\substack{a,b\in\Nzero,\ \ell_1,\ldots,\ell_b\in\N_+\\
 a+\ell_1+\cdots+\ell_b=n}}
 B_{ab}\prod_{q=1}^bY_{\ell_q}\\
 &=[t^n]B(t,Y(t)).
\end{align*}
The final equality is the coefficient formula for formal composition. It is well defined because $Y_0=0$: for a fixed $n$, the constraint $a+\ell_1+\cdots+\ell_b=n$ forces $0\le a\le n$ and $0\le b\le n$, so only finitely many terms occur. Since $n$ was arbitrary, $G(t,\eta(t))\preccurlyeq B(t,Y(t))$.
\hfill\\
\noindent
Finally, every coefficient estimate of total degree at most $n$ obtained above uses only derivatives and majorant coefficients of total degree at most $n$. Thus the same arguments prove the corresponding assertions for the truncated relations $\preccurlyeq_n$, without any analyticity assumption or convergence of the Taylor series of $G,H,$ or $\eta$. This completes the proof.
\end{proof}

\begin{lemma}[Majorants for Geometric Primitives]
\label{lem:three-primitives}
Use the setting and notation of Lemma~\ref{lem:three-normalization}, including a fixed $\vartheta\in[0,1]$. Fix a unit vector $\xi\in\R^d$ and substitute $x=x_0+\rho t\xi$, keeping $Q=J_{x_0}^sf$ fixed. Define the nonnegative formal series
\begin{equation}
\label{eq:three-geometric-majorants}
\begin{aligned}
\ell(t)&\eqdef-\tfrac12\log(1-2t-t^2), &g(t)&\eqdef-\log\bigl(2-\mathrm e^{q\ell(t)}\bigr),\\
B(t)&\eqdef g(t)/s, &T(t)&\eqdef q\ell(t)+g(t),\\
b(t)&\eqdef\frac{(1+t)^s+(1+t)\mathrm e^{(s-1)\ell(t)}}{s!}\,\mathrm e^{sB(t)},\\
a(t)&\eqdef\frac{(1+t)^m}{\beta m!}\mathrm e^{mB(t)}+\frac{A_s}{\beta}\mathrm e^{2s(\ell(t)+B(t))}+c\,b(t)^2.
\end{aligned}
\end{equation}
Then, for every $i\in[N]$, in the sense of Definition~\ref{def:three-majorization},
\begin{equation}
\label{eq:three-primitive-bounds}
\begin{aligned}
\log\frac{R(x_0+\rho t\xi)}{R(x_0)}&\preccurlyeq B(t), &\log\frac{\theta_i(x_0+\rho t\xi)}{\theta_i(x_0)}&\preccurlyeq T(t),\\
b_i^{(\vartheta)}(x_0+\rho t\xi)&\preccurlyeq u_i b(t), &w_i^{(\vartheta)}(x_0+\rho t\xi)&\preccurlyeq V_i a(t).
\end{aligned}
\end{equation}
The scalar functions $b(t)$ and $a(t)$ are majorants and are not the normalized vectors or costs. They are independent of $\vartheta$; in particular, they are not set to zero at $\vartheta=0$. Moreover,
\begin{equation}
\label{eq:three-primitive-constants}
b(0)=2/s!\le1,\qquad a(0)=\frac12+\frac2{\beta m!}+\frac{s}{\beta(s!)^2}\le\frac43.
\end{equation}
\end{lemma}
\begin{proof}
Put $x(t)\eqdef x_0+\rho t\xi$, $r_{i0}\eqdef r_i(x_0)$, and $R_0\eqdef R(x_0)$. For $i\in[N]$ and $|t|<1$, $r_i(x(t))\ge r_{i0}-\rho|t|\ge\rho(1-|t|)>0$. Set $L_i(t)\eqdef\log(r_i(x(t))/r_{i0})$, $\alpha_i\eqdef\rho\langle x_0-X_i,\xi\rangle/r_{i0}^2$, and $\eta_i\eqdef\rho^2/r_{i0}^2$. Then
\begin{equation}
\label{eq:three-distance-log-majorant}
\frac{r_i(x(t))^2}{r_{i0}^2}=1+2\alpha_i t+\eta_i t^2,\qquad |\alpha_i|\le\frac{\rho}{r_{i0}}\le1,\qquad 0<\eta_i\le1.
\end{equation}
Since $2\alpha_i t+\eta_i t^2\preccurlyeq2t+t^2$ and both polynomials vanish at zero, Lemma~\ref{lem:three-majorant-calculus} gives $L_i(t)=\frac12\log(1+2\alpha_i t+\eta_i t^2)\preccurlyeq\ell(t)$, with $L_i(0)=\ell(0)=0$. The formal identities $\ell(t)=\frac12\sum_{k\ge1}(2t+t^2)^k/k$ and $g(t)=\sum_{k\ge1}(\mathrm e^{q\ell(t)}-1)^k/k$ give $\ell,g,B,T\in t\R_{\ge0}[[t]]$ and $a,b\in\R_{\ge0}[[t]]$.

For $\gamma\in[0,q]$, put $\omega_{i,\gamma}\eqdef\mu_i r_{i0}^{-\gamma}/S_\gamma(x_0)>0$ and $G_\gamma(t)\eqdef S_\gamma(x(t))/S_\gamma(x_0)-1$. Since $\sum_{i=1}^N\omega_{i,\gamma}=1$, one has $G_\gamma(0)=0$ and, by~\eqref{eq:three-analytic-majorants},
\[
G_\gamma(t)=\sum_{i=1}^N\omega_{i,\gamma}\bigl(\mathrm e^{-\gamma L_i(t)}-1\bigr)\preccurlyeq\sum_{i=1}^N\omega_{i,\gamma}\bigl(\mathrm e^{\gamma\ell(t)}-1\bigr)=\mathrm e^{\gamma\ell(t)}-1.
\]
The majorant on the right has zero constant coefficient, and
\[
\mathrm e^{\gamma\ell}-1=\sum_{j\ge1}\frac{\gamma^j\ell^j}{j!}\preccurlyeq\sum_{j\ge1}\frac{q^j\ell^j}{j!}=\mathrm e^{q\ell}-1.
\]
A second application of~\eqref{eq:three-analytic-majorants} yields
\begin{equation}
\label{eq:three-sum-log-majorant}
\begin{aligned}
\log\frac{S_\gamma(x(t))}{S_\gamma(x_0)}
&=\log(1+G_\gamma(t))\preccurlyeq-\log\bigl(2-\mathrm e^{\gamma\ell(t)}\bigr)\\
&=\sum_{k\ge1}\frac{(\mathrm e^{\gamma\ell(t)}-1)^k}{k}\preccurlyeq\sum_{k\ge1}\frac{(\mathrm e^{q\ell(t)}-1)^k}{k}=g(t).
\end{aligned}
\end{equation}
Define $\Lambda(t)\eqdef\log(R(x(t))/R_0)$. Since $q-2s\in[0,q]$, $R^{2s}=S_{q-2s}/S_q$, and $\theta_i=\mu_i r_i^{-q}/S_q$, equation~\eqref{eq:three-sum-log-majorant} gives
\[
\begin{aligned}
\Lambda(t)&=\frac1{2s}\left(\log\frac{S_{q-2s}(x(t))}{S_{q-2s}(x_0)}-\log\frac{S_q(x(t))}{S_q(x_0)}\right)\preccurlyeq\frac{g(t)+g(t)}{2s}=B(t),\\
\log\frac{\theta_i(x(t))}{\theta_i(x_0)}&=-qL_i(t)-\log\frac{S_q(x(t))}{S_q(x_0)}\preccurlyeq q\ell(t)+g(t)=T(t).
\end{aligned}
\]
In particular, $\Lambda(0)=B(0)=0$, and for every $\sigma\ge0$, $\mathrm e^{-\sigma\Lambda}\preccurlyeq\mathrm e^{\sigma B}$ and $\mathrm e^{\sigma L_i}\preccurlyeq\mathrm e^{\sigma\ell}$.

By~\eqref{eq:training_data}, $P_i=J_{X_i}^sf$ and $D^jQ(x_0)=D^jf(x_0)$ for $0\le j\le s$. Lemma~\ref{lem:three-taylor}, applied with base point $X_i$ and evaluation point $x_0$, therefore gives
\begin{equation}
\label{eq:three-polynomial-coefficients}
\norm{D^j(P_i-Q)(x_0)}_{\op}=\norm{D^jJ_{X_i}^sf(x_0)-D^jf(x_0)}_{\op}\le\frac{Mr_{i0}^{m-j}}{(m-j)!},\qquad 0\le j\le s.
\end{equation}
Since $\rho\le r_{i0}$ and $|\xi|=1$, equation~\eqref{eq:three-polynomial-coefficients} implies
\[
\begin{aligned}
\bigl|[t^j](P_i-Q)(x(t))\bigr|&=\frac{\rho^j}{j!}\bigl|D^j(P_i-Q)(x_0)[\xi^j]\bigr|\le\frac{Mr_{i0}^m}{j!(m-j)!},&&0\le j\le s,\\
\bigl|[t^j]\nabla(P_i-Q)(x(t))\bigr|&=\frac{\rho^j}{j!}\sup_{|h|\le1}\bigl|D^{j+1}(P_i-Q)(x_0)[\xi^j,h]\bigr|\\
&\le\frac{Mr_{i0}^s}{j!(s-j)!},&&0\le j\le s-1.
\end{aligned}
\]
As $\deg(P_i-Q)\le s=m-1$ and $\deg\nabla(P_i-Q)\le s-1$, the remaining coefficients vanish. Thus
\[
\begin{aligned}
(P_i-Q)(x(t))&\preccurlyeq Mr_{i0}^m\sum_{j=0}^s\frac{t^j}{j!(m-j)!}\preccurlyeq\frac{Mr_{i0}^m}{m!}(1+t)^m,\\
\nabla(P_i-Q)(x(t))&\preccurlyeq Mr_{i0}^s\sum_{j=0}^{s-1}\frac{t^j}{j!(s-j)!}\preccurlyeq\frac{Mr_{i0}^s}{s!}(1+t)^s.
\end{aligned}
\]

Set $t_i\eqdef r_{i0}/R_0$, so $u_i=t_i^s$ and $V_i=1+t_i^{2s}$. Since $x(t)-X_i\preccurlyeq r_{i0}+\rho t\preccurlyeq r_{i0}(1+t)$, the identity $v_i(x)=\nabla P_i(x)+(M/s!)r_i(x)^{s-1}(x-X_i)$ and~\eqref{eq:three-normalized-data} give
\begin{equation}
\label{eq:three-parameter-vector-majorant}
\begin{aligned}
b_i^{(\vartheta)}(x(t))&=\sqrt\vartheta\,b_i^{(1)}(x(t))\\
&=\frac{\sqrt\vartheta\,\mathrm e^{-s\Lambda(t)}}{MR_0^s}\left(\nabla(P_i-Q)(x(t))+\frac{M r_{i0}^{s-1}}{s!}\mathrm e^{(s-1)L_i(t)}(x(t)-X_i)\right)\\
&\preccurlyeq\sqrt\vartheta\,\frac{t_i^s}{s!}\bigl((1+t)^s+(1+t)\mathrm e^{(s-1)\ell(t)}\bigr)\mathrm e^{sB(t)}\\
&=\sqrt\vartheta\,u_i b(t)\preccurlyeq u_i b(t).
\end{aligned}
\end{equation}
For the quadratic contribution, the bilinear-product assertion of Lemma~\ref{lem:three-majorant-calculus}, with $|\langle u,v\rangle|\le|u|\,|v|$, gives
\[
c|b_i^{(\vartheta)}(x(t))|^2=c\langle b_i^{(\vartheta)}(x(t)),b_i^{(\vartheta)}(x(t))\rangle\preccurlyeq\vartheta c u_i^2b(t)^2\preccurlyeq c u_i^2b(t)^2.
\]
Since $m=s+1\le2s$, one has $t_i^m\le1$ for $0<t_i\le1$ and $t_i^m\le t_i^{2s}$ for $t_i\ge1$; hence $t_i^m\le V_i$ and $u_i^2\le V_i$. Using the signed cost in~\eqref{eq:three-normalized-data} and coefficientwise addition, we obtain
\[
\begin{aligned}
w_i^{(\vartheta)}(x(t))&=\frac{(P_i-Q)(x(t))}{\beta MR_0^m}\mathrm e^{-m\Lambda(t)}+\frac{A_s}{\beta}t_i^{2s}\mathrm e^{2s(L_i(t)-\Lambda(t))}-c|b_i^{(\vartheta)}(x(t))|^2\\
&\preccurlyeq\frac{t_i^m}{\beta m!}(1+t)^m\mathrm e^{mB(t)}+\frac{A_s}{\beta}u_i^2\mathrm e^{2s(\ell(t)+B(t))}+c u_i^2b(t)^2\\
&\preccurlyeq V_i a(t).
\end{aligned}
\]
Finally, $\ell(0)=g(0)=B(0)=0$, $A_s=1/m!+\beta/2$, and $c=s/(4\beta)$ give $b(0)=2/s!\le1$ and
\[
a(0)=\frac1{\beta m!}+\frac{A_s}{\beta}+\frac{4c}{(s!)^2}=\frac12+\frac2{\beta m!}+\frac{s}{\beta(s!)^2}\le\frac12+\frac13+\frac12=\frac43,
\]
where $\beta\ge1$, $m!\ge6$, $s!\ge2$, and $s\le2\beta$.
\end{proof}

\begin{lemma}[Gibbs Quotient Taylor Majorants]
\label{lem:three-gibbs-majorants}
In Lemma~\ref{lem:three-primitives}, let $z_{0,\vartheta}$ and $\pi_i^{(\vartheta)}$ be as in Lemma~\ref{lem:three-normalization}. Define, with $K_0=\mathrm e^{65}$ and $\kappa=24\mathrm e$,
\begin{equation}
\label{eq:three-gibbs-majorants}
 \begin{aligned}
 \alpha(t)&\eqdef T(t)+a(t)-a(0)+6c\bigl(b(t)-b(0)\bigr),\\
 X(t,v)&\eqdef\alpha(t)+2cb(t)v,\\
 Q_*(t,v)&\eqdef\frac{K_0\kappa X(t,v)}{1-\kappa X(t,v)},\\
 N_*(t,v)&\eqdef K_0\kappa\bigl(b(t)-b(0)\bigr)
 +K_0\kappa b(t)\left((1-\kappa X(t,v))^{-2}-1\right),\\
 J_*(t,v)&\eqdef\frac{N_*(t,v)+3Q_*(t,v)}{1-Q_*(t,v)},\\
 \mathcal R(t,v)&\eqdef J_*(t,v)-[v]J_*(0,v)\,v.
 \end{aligned}
\end{equation}
All coefficients in these series are nonnegative, and their constant terms vanish. In particular $\partial_v\mathcal R(0,0)=0$.
Let
\[
 \mathcal K_\vartheta(t,z)
 \eqdef z-\sum_i\lambda_i^{(\vartheta)}(t,z)
 b_i^{(\vartheta)}(x_0+\rho t\xi)
\]
be the stationarity map, where $\lambda^{(\vartheta)}(t,z)$ is the Gibbs vector in~\eqref{eq:three-gibbs-dual} for the costs $w_i^{(\vartheta)}(x_0+\rho t\xi)$, vectors $b_i^{(\vartheta)}(x_0+\rho t\xi)$, and weights $\theta_i(x_0+\rho t\xi)$. Put
$J_\vartheta=D_z\mathcal K_\vartheta(0,z_{0,\vartheta})$. Then
\begin{equation}
\label{eq:three-stationarity-remainder}
 \norm{J_\vartheta^{-1}}_{\op}\le1,
 \qquad
 \mathcal K_\vartheta(t,z_{0,\vartheta}+y)-J_\vartheta y
 \preccurlyeq\mathcal R(t,v).
\end{equation}
If $W(t)=\sum_{n\ge1}W_nt^n$ is the unique formal solution with zero constant term of
\begin{equation}
\label{eq:three-majorant-equation}
 W(t)=\mathcal R(t,W(t)),
\end{equation}
then
\[
 z_\vartheta(x_0+\rho t\xi)-z_{0,\vartheta}\preccurlyeq W(t),
\]
where $z_\vartheta(x)$ denotes the Gibbs mean at the optimizer. Moreover,
\begin{equation}
\label{eq:three-h-majorant}
 \begin{aligned}
 \mathfrak h_\vartheta(x_0+\rho t\xi)-\mathfrak h_\vartheta(x_0)
 &\preccurlyeq
 c\bigl(6W(t)+W(t)^2\bigr)
 -\log(1-Q_*(t,W(t))),\\
 \frac{F_\vartheta(x_0+\rho t\xi)-Q(x_0+\rho t\xi)}
 {\beta MR(x_0)^m}
 &\preccurlyeq\mathcal H(t),\\
 \mathcal H(t)&\eqdef
 \mathrm e^{mB(t)}
 \left[1+c(6W(t)+W(t)^2)-\log(1-Q_*(t,W(t)))\right].
 \end{aligned}
\end{equation}
At $\vartheta=0$ one has, identically in $(t,z)$,
\begin{equation}
\label{eq:three-zero-parameter-stationarity}
 \mathcal K_0(t,z)=z,
 \qquad
 J_0=I,
 \qquad
 \norm{J_0^{-1}}_{\op}=1.
\end{equation}
All displayed majorants are uniform in the fixed parameter $\vartheta\in[0,1]$.
\end{lemma}
\begin{proof}
Fix $\vartheta\in[0,1]$. Write $x(t)\eqdef x_0+\rho t\xi$, $b_i(t)\eqdef b_i^{(\vartheta)}(x(t))$, $w_i(t)\eqdef w_i^{(\vartheta)}(x(t))$, $\theta_i(t)\eqdef\theta_i(x(t))$, $z_0\eqdef z_{0,\vartheta}$, and $\pi_i\eqdef\pi_i^{(\vartheta)}$; a subscript $0$ denotes evaluation at $t=0$. For sufficiently small $\varepsilon\in(0,1)$ and $|t|<\varepsilon$,
\[
\begin{gathered}
|x(t)-X_i|\ge\rho(1-|t|)>0,\qquad \theta_i(t)>0,\qquad \sum_i\theta_i(t)=1,\qquad c=\frac{s}{4(s-1)}>0,\\
R(x(t))^{2s}=\sum_i\theta_i(t)|x(t)-X_i|^{2s}\ge\rho^{2s}(1-|t|)^{2s}>0.
\end{gathered}
\]
The geometric definitions and~\eqref{eq:three-normalized-data} give $b_i\in C^\infty((-\varepsilon,\varepsilon);\R^d)$ and $w_i,\theta_i\in C^\infty((-\varepsilon,\varepsilon);\R)$. Set $Z_\vartheta(t,z)\eqdef\sum_i\theta_i(t)\mathrm e^{-w_i(t)-2cz\cdot b_i(t)}>0$. By~\eqref{eq:three-gibbs-dual},
\[
\lambda_i^{(\vartheta)}(t,z)=\frac{\theta_i(t)\mathrm e^{-w_i(t)-2cz\cdot b_i(t)}}{Z_\vartheta(t,z)},\qquad \pi_i=\lambda_i^{(\vartheta)}(0,z_0)>0,\qquad \sum_i\pi_i=1,\qquad \sum_i\pi_i b_{i0}=z_0.
\]
Define
\begin{equation}
\label{eq:three-log-gibbs-change}
\zeta_i^{(\vartheta)}(t,y)\eqdef\log\frac{\theta_i(t)}{\theta_{i0}}-(w_i(t)-w_{i0})-2cz_0\cdot(b_i(t)-b_{i0})-2cy\cdot b_i(t).
\end{equation}
Lemmas~\ref{lem:three-normalization} and~\ref{lem:three-primitives} give $|z_0|\le3$, $V_i\ge\max\{1,u_i\}$, and
\[
\begin{aligned}
b_i-b_{i0}&\preccurlyeq u_i(b-b(0)),&w_i-w_{i0}&\preccurlyeq V_i(a-a(0)),\\
\zeta_i^{(\vartheta)}&\preccurlyeq T+V_i(a-a(0))+6cu_i(b-b(0))+2cu_ibv\preccurlyeq V_iX(t,v),&\zeta_i^{(\vartheta)}(0,0)&=0.
\end{aligned}
\]
Here the product bounds are those of Lemma~\ref{lem:three-majorant-calculus}. Since $T(0)=0$ and $a,b,T$ have nonnegative coefficients, $\alpha,X$ have nonnegative coefficients and $\alpha(0)=X(0,0)=0$. In particular, $[t^jv^k]X^\ell=0$ whenever $\ell>j+k$.

Put $Q_{0,\vartheta}(t,y)\eqdef\sum_i\pi_i\mathrm e^{\zeta_i^{(\vartheta)}(t,y)}$ and $N_{0,\vartheta}(t,y)\eqdef\sum_i\pi_i b_i(t)\mathrm e^{\zeta_i^{(\vartheta)}(t,y)}$. Substitution in~\eqref{eq:three-log-gibbs-change} yields
\[
\begin{gathered}
\theta_i(t)\mathrm e^{-w_i(t)-2c(z_0+y)\cdot b_i(t)}=\theta_{i0}\mathrm e^{-w_{i0}-2cz_0\cdot b_{i0}}\mathrm e^{\zeta_i^{(\vartheta)}(t,y)},\\
Q_{0,\vartheta}(t,y)=\frac{Z_\vartheta(t,z_0+y)}{Z_\vartheta(0,z_0)}>0,\qquad \frac{N_{0,\vartheta}(t,y)}{Q_{0,\vartheta}(t,y)}=\sum_i\lambda_i^{(\vartheta)}(t,z_0+y)b_i(t),\\
Q_{0,\vartheta}(0,0)=1,\qquad N_{0,\vartheta}(0,0)=z_0.
\end{gathered}
\]
By~\eqref{eq:three-moment-bound}, coefficientwise in $(t,v)$,
\begin{equation}
\label{eq:three-moment-series}
\begin{aligned}
\sum_i\pi_i(\mathrm e^{V_iX}-1)&=\sum_{\ell\ge1}\frac{X^\ell}{\ell!}\sum_i\pi_iV_i^\ell\preccurlyeq K_0\sum_{\ell\ge1}(\kappa X)^\ell=Q_*,\\
\sum_i\pi_iV_i(\mathrm e^{V_iX}-1)&=\sum_{\ell\ge1}\frac{X^\ell}{\ell!}\sum_i\pi_iV_i^{\ell+1}\\
&\preccurlyeq K_0\kappa\sum_{\ell\ge1}(\ell+1)(\kappa X)^\ell=K_0\kappa\bigl((1-\kappa X)^{-2}-1\bigr).
\end{aligned}
\end{equation}
Thus $Q_{0,\vartheta}-1\preccurlyeq Q_*$, and $\sum_i\pi_i u_i\le\sum_i\pi_iV_i\le K_0\kappa$ gives
\[
\begin{aligned}
N_{0,\vartheta}-z_0&=\sum_i\pi_i(b_i-b_{i0})+\sum_i\pi_i b_i(\mathrm e^{\zeta_i^{(\vartheta)}}-1)\\
&\preccurlyeq(b-b(0))\sum_i\pi_i u_i+b\sum_i\pi_i u_i(\mathrm e^{V_iX}-1)\\
&\preccurlyeq K_0\kappa(b-b(0))+K_0\kappa b\bigl((1-\kappa X)^{-2}-1\bigr)=N_*.
\end{aligned}
\]
In particular, $Q_*,N_*$ have nonnegative coefficients and $Q_*(0,0)=N_*(0,0)=0$. By~\eqref{eq:three-analytic-majorants}, $Q_{0,\vartheta}^{-1}\preccurlyeq(1-Q_*)^{-1}$. Therefore
\begin{equation}
\label{eq:three-gibbs-mean-ratio}
\begin{aligned}
G_\vartheta(t,y)&\eqdef\frac{N_{0,\vartheta}(t,y)}{Q_{0,\vartheta}(t,y)}-z_0=\frac{N_{0,\vartheta}-z_0-z_0(Q_{0,\vartheta}-1)}{1+(Q_{0,\vartheta}-1)},\\
G_\vartheta&\preccurlyeq\frac{N_*+3Q_*}{1-Q_*}=J_*,\qquad G_\vartheta(0,0)=0.
\end{aligned}
\end{equation}
The expansion $J_*=(N_*+3Q_*)\sum_{\ell\ge0}Q_*^\ell$ gives nonnegative coefficients and $J_*(0,0)=0$. Hence, with $R_{jk}\eqdef[t^jv^k]\mathcal R$,
\[
R_{00}=R_{01}=0,\qquad R_{jk}=[t^jv^k]J_*\ge0\quad\bigl((j,k)\notin\{(0,0),(0,1)\}\bigr).
\]

For $h\in\R^d$, differentiation of the Gibbs quotient gives
\[
\begin{aligned}
D_z\lambda_i^{(\vartheta)}(0,z_0)[h]&=-2c\pi_i(b_{i0}-z_0)\cdot h,\\
J_\vartheta h&=h-\sum_iD_z\lambda_i^{(\vartheta)}(0,z_0)[h]b_{i0}=h+2c\sum_i\pi_i\bigl((b_{i0}-z_0)\cdot h\bigr)b_{i0}\\
&=h+2c\sum_i\pi_i\bigl((b_{i0}-z_0)\cdot h\bigr)(b_{i0}-z_0),\\
|h|^2&\le |h|^2+2c\sum_i\pi_i\bigl((b_{i0}-z_0)\cdot h\bigr)^2=h\cdot J_\vartheta h\le|h|\,|J_\vartheta h|.
\end{aligned}
\]
Thus $|J_\vartheta h|\ge|h|$, $\ker J_\vartheta=\{0\}$, and $\|J_\vartheta^{-1}\|_{\op}\le1$. Since $\mathcal K_\vartheta(t,z_0+y)=y-G_\vartheta(t,y)$ and $J_\vartheta=I-D_yG_\vartheta(0,0)$,
\begin{equation}
\label{eq:three-linear-cancellation}
\mathcal E_\vartheta(t,y)\eqdef\mathcal K_\vartheta(t,z_0+y)-J_\vartheta y=-G_\vartheta(t,y)+D_yG_\vartheta(0,0)y.
\end{equation}
For $j,k\in\Nzero$, put $E_{jk}\eqdef\partial_t^jD_y^k\mathcal E_\vartheta(0,0)/(j!k!)$. Then
\[
\begin{gathered}
E_{00}=-G_\vartheta(0,0)=0,\qquad E_{01}=-D_yG_\vartheta(0,0)+D_yG_\vartheta(0,0)=0,\\
E_{jk}=-\frac{\partial_t^jD_y^kG_\vartheta(0,0)}{j!k!}\quad\bigl((j,k)\notin\{(0,0),(0,1)\}\bigr),\\
\|E_{jk}\|_{\op}\le R_{jk}\quad(j,k\in\Nzero),\qquad \mathcal E_\vartheta\preccurlyeq\mathcal R.
\end{gathered}
\]
At $\vartheta=0$, \eqref{eq:three-zero-parameter-normalization} gives, for every $(t,z)\in(-\varepsilon,\varepsilon)\times\R^d$,
\[
\mathcal K_0(t,z)=z-\sum_i\lambda_i^{(0)}(t,z)\,0=z,\qquad J_0=I,\qquad \|J_0^{-1}\|_{\op}=1.
\]

For $n\ge1$, define the finite index set
\[
\mathcal I_n\eqdef\left\{(j,k,\boldsymbol\ell):\begin{array}{l}j,k\in\Nzero,\quad (j,k)\notin\{(0,0),(0,1)\},\\ \boldsymbol\ell=(\ell_1,\ldots,\ell_k)\in\mathbb N_+^k,\quad j+\sum_{a=1}^k\ell_a=n\end{array}\right\}.
\]
For $k=0$, the tuple is empty, its sum is $0$, its product is $1$, and $E_{j0}[\,]\eqdef E_{j0}$. For $(j,k,\boldsymbol\ell)\in\mathcal I_n$,
\[
j+k\le n,\qquad k\ge1\Longrightarrow 1\le\ell_a\le n-j-k+1\le n-1\quad(1\le a\le k).
\]
Starting with $W_1=R_{10}$, define recursively
\begin{equation}
\label{eq:three-W-recursion}
W_n\eqdef\sum_{(j,k,\boldsymbol\ell)\in\mathcal I_n}R_{jk}\prod_{a=1}^kW_{\ell_a}=[t^n]\mathcal R\left(t,\sum_{\ell=1}^{n-1}W_\ell t^\ell\right)\ge0\qquad(n\ge1).
\end{equation}
For $W\eqdef\sum_{n\ge1}W_nt^n$, one has $[t^0]\mathcal R(t,W)=R_{00}=0$ and $[t^n]\mathcal R(t,W)=W_n$ for $n\ge1$, so $W=\mathcal R(t,W)$. If $\widetilde W(0)=0$ and $\widetilde W=\mathcal R(t,\widetilde W)$, then $\widetilde W_1=R_{10}=W_1$ and
\[
\bigl(\widetilde W_\ell=W_\ell\text{ for }1\le\ell<n\bigr)\Longrightarrow\widetilde W_n=\sum_{(j,k,\boldsymbol\ell)\in\mathcal I_n}R_{jk}\prod_{a=1}^k\widetilde W_{\ell_a}=W_n.
\]

Lemma~\ref{lem:three-entropy}, applied to the smooth data $(w_i(t),b_i(t),\theta_i(t))$ and the fixed $c>0$, gives a smooth curve $\eta(t)\eqdef z_\vartheta(x(t))-z_0$ with $\eta(0)=0$ and $\mathcal K_\vartheta(t,z_0+\eta(t))=0$. Put $\eta_n\eqdef\eta^{(n)}(0)/n!$. The finite Taylor coefficient formula~\eqref{eq:three-majorant-composition-coefficient} and~\eqref{eq:three-linear-cancellation} give, for every $n\ge1$,
\[
J_\vartheta\eta_n+\sum_{(j,k,\boldsymbol\ell)\in\mathcal I_n}E_{jk}[\eta_{\ell_1},\ldots,\eta_{\ell_k}]=0.
\]
At $n=1$, $|\eta_1|=|J_\vartheta^{-1}E_{10}|\le R_{10}=W_1$. If $n\ge2$ and $|\eta_\ell|\le W_\ell$ for $1\le\ell<n$, then
\[
\begin{aligned}
|\eta_n|&\le\|J_\vartheta^{-1}\|_{\op}\sum_{(j,k,\boldsymbol\ell)\in\mathcal I_n}\|E_{jk}\|_{\op}\prod_{a=1}^k|\eta_{\ell_a}|\\
&\le\sum_{(j,k,\boldsymbol\ell)\in\mathcal I_n}R_{jk}\prod_{a=1}^kW_{\ell_a}=W_n.
\end{aligned}
\]
Consequently $|\eta^{(n)}(0)|/n!\le W_n$ for every $n\ge1$, and $\eta\preccurlyeq W$ in Definition~\ref{def:three-majorization}.

Abbreviate $\mathfrak h_\vartheta(x(t))$ by $\mathfrak h_\vartheta(t)$. The dual identity~\eqref{eq:three-gibbs-dual} and the partition-function ratio give
\[
\begin{aligned}
\mathfrak h_\vartheta(t)-\mathfrak h_\vartheta(0)&=-c\bigl(|z_0+\eta(t)|^2-|z_0|^2\bigr)-\log\frac{Z_\vartheta(t,z_0+\eta(t))}{Z_\vartheta(0,z_0)}\\
&=-c\bigl(2z_0\cdot\eta(t)+|\eta(t)|^2\bigr)-\log Q_{0,\vartheta}(t,\eta(t)).
\end{aligned}
\]
Since $\eta(0)=W(0)=0$, Lemma~\ref{lem:three-majorant-calculus} gives
\[
\begin{gathered}
Q_{0,\vartheta}(t,\eta(t))-1\preccurlyeq Q_*(t,W(t)),\qquad [t^0]Q_*(t,W(t))=0,\\
-c\bigl(2z_0\cdot\eta+|\eta|^2\bigr)\preccurlyeq c(6W+W^2),\\
\mathfrak h_\vartheta(t)-\mathfrak h_\vartheta(0)\preccurlyeq c(6W+W^2)+\sum_{\ell\ge1}\frac{Q_*(t,W(t))^\ell}{\ell}=c(6W+W^2)-\log(1-Q_*(t,W(t))).
\end{gathered}
\]
Finally, put $L(t)\eqdef\log(R(x(t))/R(x_0))$. By~\eqref{eq:three-primitive-bounds}, $L(0)=B(0)=0$, $L\preccurlyeq B$, and $\mathrm e^{mL}\preccurlyeq\mathrm e^{mB}$. Since~\eqref{eq:three-normalized-bounds} gives $|\mathfrak h_\vartheta(0)|\le1$, \eqref{eq:three-F-normalized} yields
\[
\begin{aligned}
\frac{F_\vartheta(x(t))-Q(x(t))}{\beta MR(x_0)^m}&=\mathrm e^{mL(t)}\mathfrak h_\vartheta(t)\\
&\preccurlyeq\mathrm e^{mB(t)}\left[1+c(6W(t)+W(t)^2)-\log(1-Q_*(t,W(t)))\right]=\mathcal H(t).
\end{aligned}
\]
The data $a,b,B,T,c,K_0,\kappa$ in~\eqref{eq:three-gibbs-majorants} are independent of $\vartheta$; \eqref{eq:three-W-recursion} therefore gives the same $W$ and $\mathcal H$ for every $\vartheta\in[0,1]$.
\end{proof}

\begin{lemma}[Uniform Radius for Majorants]
\label{lem:three-majorant-radius}
For every integer $s\ge2$ and real $q\ge2s+2$, the series in Lemmas~\ref{lem:three-primitives} and \ref{lem:three-gibbs-majorants} satisfy
\begin{equation}
\label{eq:three-J-geometric-bound}
 J_*(t,v)\preccurlyeq\frac{C_0(qt+v)}{1-C_0(qt+v)},
 \qquad C_0 \eqdef  \mathrm{e}^{100}.
\end{equation}
The solution $W$ of \eqref{eq:three-majorant-equation} has nonnegative coefficients and satisfies
\begin{equation}
\label{eq:three-radius-bound}
 \tau \eqdef \frac1{64C_0^3q},\qquad
 r_* \eqdef \frac1{16C_0^2},\qquad
 W(\tau)\le r_*,\qquad \mathcal H(\tau)\le8.
\end{equation}
In particular, $[t^n]\mathcal H(t)\le8\tau^{-n}$ for every $n\in\N$. Since the scalar series in Lemmas~\ref{lem:three-primitives} and~\ref{lem:three-gibbs-majorants} are unchanged, every conclusion here holds uniformly for each fixed $\vartheta\in[0,1]$.
\end{lemma}
The proof of this lemma uses the following notation:
we write $\R[[t]]$ for the ring of formal power series in $t$ with real coefficients and $\R_{\ge0}[[t]]$ for its subset consisting of series with nonnegative coefficients:
\[
\R[[t]]\eqdef\left\{\sum_{n\ge0}a_nt^n:a_n\in\R\right\},
\qquad
\R_{\ge0}[[t]]\eqdef\left\{\sum_{n\ge0}a_nt^n:a_n\ge0\right\}.
\]
For every integer $j\ge0$, membership $A\in t^j\R[[t]]$ means that $[t^n]A=0$ for $0\le n<j$; membership $A\in t^j\R_{\ge0}[[t]]$ additionally requires all coefficients of $A$ to be nonnegative. These notations impose no convergence assumption.
\begin{proof}
Put $\varepsilon\eqdef C_0^{-1}=\mathrm e^{-100}$, $t_0\eqdef\varepsilon/q$, and $v_0\eqdef\varepsilon$. Recall $m=s+1$, $\beta=s-1$, $c=s/(4\beta)\le1/2$, $K_0=\mathrm e^{65}$, and $\kappa=24\mathrm e$. The series in~\eqref{eq:three-geometric-majorants} and~\eqref{eq:three-gibbs-majorants} have nonnegative coefficients. Numerical evaluations of such series are understood as increasing limits of finite partial sums. Since $n!\ge2^{n-1}$ for $n\ge1$, with strict inequality for $n\ge3$,
\[
\frac83=1+1+\frac1{2!}+\frac1{3!}<\mathrm e=\sum_{n=0}^{\infty}\frac1{n!}<1+\sum_{n=1}^{\infty}2^{1-n}=3.
\]
In particular, $0<\varepsilon<2^{-100}$, $q\ge6$, $s/q\le1/2$, and $m/q\le1/2$. For $0\le y\le1/2$,
\begin{equation}
\label{eq:three-elementary-series-bounds}
\begin{aligned}
-\log(1-y)&=\sum_{n\ge1}\frac{y^n}{n}\le\sum_{n\ge1}y^n=\frac{y}{1-y}\le2y,\\
\mathrm e^y-1&=\sum_{n\ge1}\frac{y^n}{n!}\le\sum_{n\ge1}y^n=\frac{y}{1-y}\le2y.
\end{aligned}
\end{equation}
For $0\le t\le1/8$, $2t+t^2\le17/64$ and $2t+t^2\le(17/8)t$, so
\[
\ell(t)=\frac12\sum_{n\ge1}\frac{(2t+t^2)^n}{n}\le\frac{2t+t^2}{2(1-2t-t^2)}\le\frac{68}{47}t\le2t.
\]
Since $t_0<1/8$, equations~\eqref{eq:three-geometric-majorants} and~\eqref{eq:three-elementary-series-bounds} give
\begin{equation}
\label{eq:three-small-primitive-values}
\begin{aligned}
q\ell(t_0)&\le2\varepsilon,&\mathrm e^{q\ell(t_0)}-1&\le4\varepsilon,\\
2-\mathrm e^{q\ell(t_0)}&\ge1-4\varepsilon>\frac12,&g(t_0)&\le2\bigl(\mathrm e^{q\ell(t_0)}-1\bigr)\le8\varepsilon,\\
T(t_0)&=q\ell(t_0)+g(t_0)\le10\varepsilon,&sB(t_0)&=g(t_0)\le8\varepsilon,\\
mB(t_0)&=\frac ms g(t_0)\le12\varepsilon.&&
\end{aligned}
\end{equation}
Here $m/s\le3/2$, and the bound $\mathrm e^{q\ell(t_0)}-1<1$ gives convergence of the defining logarithmic series for $g(t_0)$.

The inequalities $\log(1+t_0)\le t_0$, $s/q\le1/2$, and $(2s-1)/q\le1$ imply
\[
\begin{aligned}
(1+t_0)^s&\le\mathrm e^{st_0}\le\mathrm e^{\varepsilon/2},\\
(1+t_0)\mathrm e^{(s-1)\ell(t_0)}&\le\mathrm e^{t_0+2(s-1)t_0}=\mathrm e^{(2s-1)\varepsilon/q}\le\mathrm e^\varepsilon.
\end{aligned}
\]
Thus, using $b(0)=2/s!\le1$ and~\eqref{eq:three-elementary-series-bounds},
\begin{equation}
\label{eq:three-b-small-value}
\begin{aligned}
b(t_0)&\le\frac{2}{s!}\mathrm e^{9\varepsilon}=b(0)\mathrm e^{9\varepsilon}\le1+18\varepsilon\le2,\\
0\le b(t_0)-b(0)&\le b(0)(\mathrm e^{9\varepsilon}-1)\le18\varepsilon.
\end{aligned}
\end{equation}
For the three terms defining $a$, the respective factors satisfy
\[
\begin{aligned}
m\log(1+t_0)+mB(t_0)&\le\frac m q\varepsilon+12\varepsilon\le13\varepsilon,\\
2s\bigl(\ell(t_0)+B(t_0)\bigr)&\le\frac{4s}{q}\varepsilon+16\varepsilon\le18\varepsilon,\\
b(t_0)^2&\le b(0)^2\mathrm e^{18\varepsilon}.
\end{aligned}
\]
Consequently, by~\eqref{eq:three-primitive-constants} and $18\varepsilon<1/2$,
\begin{equation}
\label{eq:three-a-small-value}
\begin{aligned}
a(t_0)&\le\left(\frac1{\beta m!}+\frac{A_s}{\beta}+cb(0)^2\right)\mathrm e^{18\varepsilon}=a(0)\mathrm e^{18\varepsilon},\\
0\le a(t_0)-a(0)&\le\frac43(\mathrm e^{18\varepsilon}-1)\le48\varepsilon.
\end{aligned}
\end{equation}
Substitution in~\eqref{eq:three-gibbs-majorants}, with $6c\le3$ and $2c\le1$, gives
\begin{equation}
\label{eq:three-X-small-value}
\begin{aligned}
\alpha(t_0)&\le10\varepsilon+48\varepsilon+3\cdot18\varepsilon=112\varepsilon,\\
X(t_0,v_0)&=\alpha(t_0)+2cb(t_0)v_0\le112\varepsilon+2\varepsilon=114\varepsilon\le128\varepsilon.
\end{aligned}
\end{equation}
Set $Y\eqdef\kappa X(t_0,v_0)$. Then $0\le Y\le72\cdot128\varepsilon=9216\varepsilon<1/4$, so $1-Y>3/4$. Since $18432<(8/3)^{15}<\mathrm e^{15}$,
\begin{equation}
\label{eq:three-Q-small-value}
\begin{aligned}
Q_*(t_0,v_0)&=K_0\sum_{n\ge1}Y^n=\frac{K_0Y}{1-Y}\le2K_0Y\\
&\le18432\mathrm e^{65}\varepsilon\le\mathrm e^{80}\varepsilon=\mathrm e^{-20}<\frac12.
\end{aligned}
\end{equation}
For $0\le Y\le1/4$, $\sum_{n\ge1}(n+1)Y^n=(1-Y)^{-2}-1=Y(2-Y)/(1-Y)^2\le(32/9)Y\le4Y$. Hence, using $5{,}309{,}712<(8/3)^{20}<\mathrm e^{20}$,
\begin{equation}
\label{eq:three-N-small-value}
\begin{aligned}
N_*(t_0,v_0)&=K_0\kappa\bigl(b(t_0)-b(0)\bigr)+K_0\kappa b(t_0)\sum_{n\ge1}(n+1)Y^n\\
&\le\mathrm e^{65}\bigl(72\cdot18+72\cdot2\cdot4\cdot9216\bigr)\varepsilon\\
&=5{,}309{,}712\mathrm e^{65}\varepsilon\le\mathrm e^{85}\varepsilon.
\end{aligned}
\end{equation}
In particular, $1-Q_*(t_0,v_0)\ge1-\mathrm e^{-20}>1/2$. Since $3<\mathrm e^5$ and $4<\mathrm e^2$, equations~\eqref{eq:three-Q-small-value}--\eqref{eq:three-N-small-value} yield
\begin{equation}
\label{eq:three-J-small-value}
\begin{aligned}
J_*(t_0,v_0)=\bigl(N_*(t_0,v_0)+3Q_*(t_0,v_0)\bigr)\sum_{n\ge0}Q_*(t_0,v_0)^n
\le2\bigl(\mathrm e^{85}+3\mathrm e^{80}\bigr)\varepsilon\le4\mathrm e^{85}\varepsilon<\mathrm e^{87}\varepsilon=\mathrm e^{-13}<1.
\end{aligned}
\end{equation}
Write $J_*(t,v)=\sum_{j,k\ge0}c_{jk}t^jv^k$, where $c_{jk}\ge0$ and $c_{00}=0$. For $j+k\ge1$, equation~\eqref{eq:three-J-small-value} gives
\[
\begin{aligned}
0\le c_{jk}t_0^jv_0^k&\le\sum_{a,b\ge0}c_{ab}t_0^av_0^b=J_*(t_0,v_0)<1,\\
c_{jk}&\le t_0^{-j}v_0^{-k}=C_0^{j+k}q^j\le C_0^{j+k}\binom{j+k}{j}q^j\\
&=[t^jv^k]\sum_{n\ge1}C_0^n(qt+v)^n=[t^jv^k]\frac{C_0(qt+v)}{1-C_0(qt+v)}.
\end{aligned}
\]
This proves~\eqref{eq:three-J-geometric-bound}, including its zero constant coefficient. Write $\mathcal R(t,v)=\sum_{a,b\ge0}r_{ab}t^av^b$. By~\eqref{eq:three-gibbs-majorants}, $r_{01}=0$ and $r_{ab}=c_{ab}$ for $(a,b)\ne(0,1)$. Deleting the coefficient of $v$ on both sides of~\eqref{eq:three-J-geometric-bound} therefore gives
\begin{equation}
\label{eq:three-R-scalar-majorant}
\mathcal R(t,v)\preccurlyeq C_0qt+\sum_{n\ge2}C_0^n(qt+v)^n=C_0qt+\frac{C_0^2(qt+v)^2}{1-C_0(qt+v)}.
\end{equation}
For $\tau$ and $r_*$ in~\eqref{eq:three-radius-bound},
\[
q\tau+r_*\le\frac5{64C_0^2},\qquad C_0(q\tau+r_*)\le\frac5{64C_0}\le\frac5{64},\qquad 1-C_0(q\tau+r_*)\ge\frac{59}{64}.
\]
The convergent geometric series in~\eqref{eq:three-R-scalar-majorant} consequently yields
\begin{equation}
\label{eq:three-majorant-invariance}
\mathcal R(\tau,r_*)\le\frac1{64C_0^2}+\frac{25}{3776C_0^2}=\frac{21}{59}r_*<r_*.
\end{equation}

Define $W^{(0)}(t)\eqdef0$ and $W^{(n+1)}(t)\eqdef\mathcal R(t,W^{(n)}(t))$ for $n\ge0$. For $U,V\in t\R_{\ge0}[[t]]$, $U\preccurlyeq V$ implies $U^b\preccurlyeq V^b$ for every $b\ge0$, and hence $\mathcal R(t,U)\preccurlyeq\mathcal R(t,V)$ because $r_{ab}\ge0$. Since $r_{00}=0$ and $0=W^{(0)}\preccurlyeq\mathcal R(t,0)=W^{(1)}$, induction gives $W^{(n)}\in t\R_{\ge0}[[t]]$ and $0\preccurlyeq W^{(n)}\preccurlyeq W^{(n+1)}$. For $U,V\in t\R[[t]]$,
\[
\mathcal R(t,U)-\mathcal R(t,V)=(U-V)\sum_{\substack{a\ge0\\b\ge1}}r_{ab}t^a\sum_{\ell=0}^{b-1}U^\ell V^{b-1-\ell}.
\]
Every summand in the last factor belongs to $t^{a+b-1}\R[[t]]$. For each $k\ge0$, only the finitely many pairs $(a,b)$ with $a+b\le k+1$ can contribute to its coefficient of $t^k$. The factor is therefore a well-defined formal series, and $r_{01}=0$ gives its membership in $t\R[[t]]$. Consequently, for every $j\ge1$,
$
U-V\in t^j\R[[t]]$ implies that $\mathcal R(t,U)-\mathcal R(t,V)\in t^{j+1}\R[[t]]$.  
Starting from $W^{(1)}-W^{(0)}=\mathcal R(t,0)\in t\R[[t]]$, induction yields
\[
W^{(n+1)}-W^{(n)}\in t^{n+1}\R[[t]],\qquad [t^j]W^{(n)}=[t^j]W^{(j)}\quad(n\ge j\ge1).
\]
Put $W_j\eqdef[t^j]W^{(j)}\ge0$. The same order estimate gives
\[
W_j=[t^j]\mathcal R(t,W^{(j-1)}(t))=[t^j]\mathcal R\left(t,\sum_{k=1}^{j-1}W_kt^k\right),\qquad j\ge1.
\]
These are the coefficients of the unique formal solution $W$ by~\eqref{eq:three-W-recursion}.
Let $U(t)=\sum_{j\ge1}u_jt^j$ have $u_j\ge0$ and $U(\tau)\le r_*$. Rearrangement of nonnegative sums gives
\[
\begin{aligned}
\sum_{n\ge0}\tau^n[t^n]\mathcal R(t,U(t))&=\sum_{a,b\ge0}r_{ab}\tau^a\left(\sum_{j\ge1}u_j\tau^j\right)^b\\
&=\mathcal R(\tau,U(\tau))\le\mathcal R(\tau,r_*)<r_*.
\end{aligned}
\]
In particular, the composed series converges at $\tau$. Applying this identity inductively from $W^{(0)}(\tau)=0$ yields $W^{(n)}(\tau)\le r_*$ for every $n\ge0$. For every $L\ge1$, stabilization therefore gives
\[
0\le\sum_{j=1}^{L}W_j\tau^j=\sum_{j=1}^{L}\bigl([t^j]W^{(L)}(t)\bigr)\tau^j\le W^{(L)}(\tau)\le r_*.
\]
Taking the increasing limit in $L$ proves $W(\tau)=\sum_{j\ge1}W_j\tau^j\le r_*$.

Since $\tau/t_0=1/(64C_0^2)\le1$ and $r_*/v_0=1/(16C_0)\le1$, nonnegativity of coefficients and~\eqref{eq:three-small-primitive-values}--\eqref{eq:three-Q-small-value} imply
$
mB(\tau)\le mB(t_0)\le12\varepsilon$ and $Q_*(\tau,W(\tau))\le Q_*(t_0,v_0)\le\mathrm e^{-20}<\frac12$.  
The same nonnegative-sum identity and $Q_*(0,W(0))=Q_*(0,0)=0$ give
\[
\begin{aligned}
\sum_{n\ge0}\tau^n[t^n]\sum_{k\ge1}\frac{Q_*(t,W(t))^k}{k}&=\sum_{k\ge1}\frac{Q_*(\tau,W(\tau))^k}{k}=-\log(1-Q_*(\tau,W(\tau)))<\infty.
\end{aligned}
\]
Equation~\eqref{eq:three-h-majorant}, together with $c\le1/2$, now gives
\begin{equation}
\label{eq:three-H-small-value}
\begin{aligned}
\mathcal H(\tau)&\le\mathrm e^{12\varepsilon}\left[1+\frac12(6r_*+r_*^2)-\log(1-\mathrm e^{-20})\right]\\
&\le2\left(1+\frac{97}{512}+2\mathrm e^{-20}\right)<2\left(1+\frac{97}{512}+\frac12\right)<4<8.
\end{aligned}
\end{equation}
Here $r_*=(16C_0^2)^{-1}\le1/16$ and $\frac12(6/16+1/16^2)=97/512$, while~\eqref{eq:three-elementary-series-bounds} gives $\mathrm e^{12\varepsilon}\le1+24\varepsilon\le2$ and $-\log(1-\mathrm e^{-20})\le2\mathrm e^{-20}<1/2$. Finally, writing $\mathcal H_n\eqdef[t^n]\mathcal H(t)\ge0$,
\[
0\le\mathcal H_n\tau^n\le\sum_{k\ge0}\mathcal H_k\tau^k=\mathcal H(\tau)<8,\qquad\mathcal H_n\le8\tau^{-n}\quad(n\ge0)
\]
which completes the proof.
\end{proof}

\begin{lemma}[Local Derivative Errors Off Sites]
\label{lem:three-local-errors}
Under \eqref{eq:training_data} and~\eqref{eq:doubling}, suppose $s\ge2$ and fix $\vartheta\in[0,1]$. Set
\begin{equation}
\label{eq:three-constant-D}
 C_s(D)\eqdef D^3\left[\mathrm e^{2000}(s+1)
 \bigl(2s+\lceil\log_2D\rceil+2\bigr)\right]^{s+1}.
\end{equation}
Then $F_\vartheta\in C^\infty(\R^d\setminus\mathbb X^N)$ and, for $x\notin\mathbb X^N$ and $0\le j\le s$,
\begin{equation}
\label{eq:three-local-errors}
 \|D^jF_\vartheta(x)-D^jf(x)\|_{\op}
 \le C_s(D)M\rho_{\mathbb X^N}(x)^{s+1-j}.
\end{equation}
The constant is uniform in the fixed parameter $\vartheta\in[0,1]$. In particular, at $D=2^{46d}$,
\begin{equation}
\label{eq:three-constant-ds}
 C_s(2^{46d})=2^{138d}
 \left[\mathrm e^{2000}(s+1)(2s+46d+2)\right]^{s+1}.
\end{equation}
\end{lemma}
\begin{proof}
Fix $x_0\notin\mathbb X^N$ and the polynomial $Q=J_{x_0}^sf$. The positive distances, finite sums, quotients, and positive real powers defining $R$, $\theta$, and $v$ are smooth off the sites. In a neighborhood of $x_0$, Lemma~\ref{lem:three-normalization} gives the normalized identity for $F_\vartheta$ with this fixed $Q$. The coefficient $c=s/(4(s-1))$ is positive for every parameter value and every $\theta_i$ is positive, so Lemma~\ref{lem:three-entropy} gives a smooth optimal value in that identity. Hence $F_\vartheta$ is smooth near $x_0$; arbitrariness of $x_0$ proves off-site smoothness, including at $\vartheta=0$.

Let $\rho=\rho_{\mathbb X^N}(x_0)>0$ and keep $Q=J_{x_0}^sf$ fixed. For every unit vector $\xi$, Lemmas~\ref{lem:three-gibbs-majorants} and~\ref{lem:three-majorant-radius} give, for $0\le j\le s$,
\begin{equation}
\label{eq:three-directional-estimates}
 |D^j(F_\vartheta-Q)(x_0)[\xi^j]|
 \le8\beta MR(x_0)^m j!
 \left(\frac{64\mathrm e^{300}q}{\rho}\right)^j.
\end{equation}
At $j=0$ the left side means the value. Indeed, the coefficient of $t^j$ in
$F_\vartheta(x_0+\rho t\xi)-Q(x_0+\rho t\xi)$ is
$\rho^jD^j(F_\vartheta-Q)(x_0)[\xi^j]/j!$, and the majorant coefficient is at most $8\tau^{-j}$, with $\tau^{-1}=64\mathrm e^{300}q$.
For $j\ge1$ apply the polarization part of Lemma~\ref{lem:three-taylor}; for $j=0$ no polarization is needed. Lemma~\ref{lem:three-scale} then gives, with $0^0=1$ only for the displayed zero-order factor,
\begin{equation}
\label{eq:three-tensor-estimates}
 \|D^j(F_\vartheta-Q)(x_0)\|_{\op}
 \le8\beta L_{s,D}^m j^j(64\mathrm e^{300}q)^j
 M\rho^{m-j},
 \qquad 0\le j\le s.
\end{equation}
Since $j^j\le m^m$ and $64\mathrm e^{300}q>1$, the constants in this finite range are bounded by
\begin{equation}
\label{eq:three-intermediate-constant}
 \begin{aligned}
 8\beta L_{s,D}^m(64\mathrm e^{300}qm)^m
 &\le64\beta D^3(128\mathrm e^{300}qm)^m\\
 &\le D^3(\mathrm e^{315}qm)^m
 \le D^3(\mathrm e^{2000}qm)^m=C_s(D).
 \end{aligned}
\end{equation}
The first inequality uses $L_{s,D}^m\le8\,2^mD^3$ from Lemma~\ref{lem:three-scale}. For the second, $128<\mathrm e^5$ and $64\beta\le\mathrm e^{10m}$: here $\beta=m-2<m$, $64<\mathrm e^5$, and $m\le\mathrm e^m$, so $64\beta<\mathrm e^{m+5}\le\mathrm e^{10m}$ for $m\ge3$. Finally $D^jQ(x_0)=D^jf(x_0)$ for $j\le s$, so~\eqref{eq:three-tensor-estimates} proves~\eqref{eq:three-local-errors}. Since $d$ is a positive integer, $\log_2(2^{46d})=46d$ is integral, and substitution gives~\eqref{eq:three-constant-ds}. The distance comparison cancels the same positive local mass, and the moment and majorant bounds contain no factor involving $\vartheta$, $N$, the minimum atomic weight, or the sites' separation; thus neither does the stated constant.
\end{proof}
Let us briefly add some details on the sub-components of the \textit{fill radius} of $\mathbb{D}^N$; namely, write
\[
    \rho_{\mathbb X^N}(x)\eqdef\min_{i\in[N]}|x-X_i|
    \quad\mbox{ and }\quad
    h\eqdef\sup_{x\in[0,1]^d}\rho_{\mathbb X^N}(x).
\]
The local estimates show that the error in derivative order $j$ vanishes like $\rho_{\mathbb X^N}(x)^{s+1-j}$ as $x$ approaches a measurement site. The next lemma turns this decay into actual differentiability at the sites, showing that the continuously extended derivative formulas are successive derivatives of $F$.

\begin{lemma}[Differentiation Through the Finite Sites]
\label{lem:three-site-gluing}
Let $s\ge1$, $u\in C^s(\R^d)$, and $G:\R^d\to\R$ satisfy $G(X_i)=u(X_i)$ for $i\in[N]$ and $G\in C^s(\R^d\setminus\mathbb{X}^N)$. Suppose that a finite number $B\ge0$ satisfies
\begin{equation}
\label{eq:three-gluing-hypothesis}
 \|D^jG(x)-D^ju(x)\|_{\op}
 \le B\rho_{\mathbb{X}^N}(x)^{s+1-j},
 \qquad x\notin\mathbb{X}^N,\quad 0\le j\le s.
\end{equation}
Then $G\in C^s(\R^d)$ and $D^jG(X_i)=D^ju(X_i)$ for $i\in[N]$ and $0\le j\le s$.
\end{lemma}
\begin{proof}
For $0\le j\le s$, define a scalar field at $j=0$ and a symmetric tensor field at $j\ge1$ by
\begin{equation}
\label{eq:three-extended-error}
 H_j(x) \eqdef 
 \begin{cases}
 D^jG(x)-D^ju(x),&x\notin\mathbb{X}^N,\\
 0,&x\in\mathbb{X}^N.
 \end{cases}
\end{equation}
For a site $a$, $\rho_{\mathbb{X}^N}(x)\le|x-a|$. Thus \eqref{eq:three-gluing-hypothesis} gives $\|H_j(x)\|_{\op}\le B|x-a|^{s+1-j}$ off the sites. The exponent is at least one, and $H_j$ is zero on the sites, so $H_j$ is continuous at $a$. Its off-site continuity follows from $G,u\in C^s$ there. Hence every $H_j$ is continuous on $\R^d$.

We prove $D^jH_0=H_j$ by induction for $0\le j\le s$. The assertion at $j=0$ is the definition. Suppose it holds at $j-1$, where $1\le j\le s$. Off the sites $\{X_i\}_{i=1}^N$, differentiation gives $D H_{j-1}=H_j$. At a site $a$, if $\eta\ne0$ and $a+\eta\notin\mathbb{X}^N$, the preceding-order bound gives
$\|H_{j-1}(a+\eta)-H_{j-1}(a)\|_{\op}/|\eta|\le B|\eta|^{s+1-j}$. Consequently,
\begin{equation}
\label{eq:three-site-differentiation}
 \lim_{\substack{\eta\to0\\\eta\ne0}}
 \frac{\|H_{j-1}(a+\eta)-H_{j-1}(a)\|_{\op}}{|\eta|}
 =0.
\end{equation}
If $a+\eta$ is another site, its numerator is zero. Since $s+1-j\ge1$, this is the full Frechet difference-quotient estimate, uniform in the direction of $\eta$. The derivative at $a$ is the zero linear map into the preceding tensor space, identified with $H_j(a)=0$. Continuity of $H_j$ shows that $H_{j-1}$ is continuously differentiable, completing the induction. Hence $H_0\in C^s(\R^d)$. The site-value hypothesis and \eqref{eq:three-extended-error} give $G=u+H_0$ everywhere. Differentiating and evaluating at a site proves $D^jG(X_i)=D^ju(X_i)$ throughout the required range.
\end{proof}

We are now ready to prove the first point in  Theorem~\ref{thrm:main}; namely, jet interpolation.

\begin{proof}[{Proof of Theorem~\ref{thrm:main} - Point (i)}]
The geometry stipulated in Section~\ref{s:MainResults} can be chosen by Lemma~\ref{lem:three-doubling-weights}.
If $s=0$, Lemma~\ref{lem:three-zero-order} proves continuity and the required value equality, which is the only derivative order in the statement. If $s=1$, Lemma~\ref{lem:three-first-order} proves the value and gradient equalities and their existence.
Suppose $s\ge2$ and fix an arbitrary $\vartheta\in[0,1]$. The witness in~\eqref{eq:training_data} belongs to $C^s(\R^d)$, and~\eqref{eq:entropy-variance-extension__primal___SITES} gives
$F_\vartheta(X_i)=P_i(X_i)=f(X_i)$. Lemma~\ref{lem:three-local-errors} proves off-site smoothness and~\eqref{eq:three-gluing-hypothesis} for
\[
 G=F_\vartheta,
 \qquad u=f,
 \qquad B=C_s(D)M<\infty.
\]
All hypotheses of Lemma~\ref{lem:three-site-gluing} are therefore met. That lemma proves
$F_\vartheta\in C^s(\R^d)$ and
\begin{equation}
\label{eq:three-family-interpolation}
 D^jF_\vartheta(X_n)
 =D^jf(X_n)
 =D^jP_n(X_n),
 \qquad n\in[N],\quad0\le j\le s,
\end{equation}
where the last equality follows from $P_n=J_{X_n}^sf$. Since $\vartheta$ was arbitrary, this is uniform over all fixed $\vartheta\in[0,1]$. Taking $\vartheta=1$ and using~\eqref{eq:entropy-variance-extension__specialization} proves the higher-order case of Theorem~\ref{thrm:main}(i). Together with the $s=0,1$ arguments, this covers every nonnegative integer $s$. The particular value conclusion is the case $j=0$.
\end{proof}

\subsection{Proof Theorem \texorpdfstring{\ref{thrm:main} (ii):  $C^{s,1}$-Smoothness}{Main Theorem Property (ii)}}
\label{s:Proofs__ss:Cs1_Smoothness}

We now apply the coefficient estimates at order $s+1$ to bound $D^{s+1}F$ away from the sites; the fixed comparison polynomial has degree at most $s$, so its contribution vanishes. Since $D^sF$ is already continuous everywhere, integration along line segments extends this bound to a global Lipschitz estimate, including across the measurement sites.

\begin{lemma}[Highest Derivative Away from Sites]
\label{lem:smoothness-top-off-sites}
Under \eqref{eq:training_data} and~\eqref{eq:doubling}, suppose $s\ge2$ and fix $\vartheta\in[0,1]$. With $C_s(D)$ defined by~\eqref{eq:three-constant-D},
\begin{equation}
\label{eq:smoothness-top-off-sites}
\|D^{s+1}F_\vartheta(x)\|_{\op}\le C_s(D)M,\qquad x\in\R^d\setminus\mathbb X^N.
\end{equation}
The estimate is uniform in the fixed parameter $\vartheta$.
\end{lemma}
\begin{proof}
Fix $x_0\notin\mathbb X^N$ and a witness $f$ satisfying~\eqref{eq:training_data}. Put $\rho\eqdef\rho_{\mathbb X^N}(x_0)>0$, $m\eqdef s+1\ge3$, $\beta\eqdef s-1=m-2>0$, and $q\eqdef2s+\lceil\log_2D\rceil+2\ge2s+2$, and keep $Q\eqdef J_{x_0}^sf$ fixed. Lemmas~\ref{lem:three-local-errors} and~\ref{lem:three-scale} give $F_\vartheta\in C^\infty(\R^d\setminus\mathbb X^N)$ and $\rho\le R(x_0)\le L_{s,D}\rho$. For every $\xi\in\R^d$ with $|\xi|=1$, $i\in[N]$, and $|t|<1$, one has $|x_0+\rho t\xi-X_i|\ge|x_0-X_i|-\rho|t|\ge\rho(1-|t|)>0$. Hence
\begin{equation}
\label{eq:smoothness-normalized-line}
\Phi_{\xi,\vartheta}(t)\eqdef\frac{F_\vartheta(x_0+\rho t\xi)-Q(x_0+\rho t\xi)}{\beta MR(x_0)^m},\qquad |t|<1,
\end{equation}
satisfies $\Phi_{\xi,\vartheta}\in C^\infty((-1,1))$, with $\beta MR(x_0)^m\ge\beta M\rho^m>0$.

With this fixed $f,x_0,Q,\vartheta,\xi$, Lemma~\ref{lem:three-gibbs-majorants}, specifically~\eqref{eq:three-h-majorant}, gives $\Phi_{\xi,\vartheta}\preccurlyeq_n\mathcal H$ for every $n\in\Nzero$. Lemma~\ref{lem:three-majorant-radius}, with $\tau\eqdef(64\mathrm e^{300}q)^{-1}>0$, gives $[t^n]\mathcal H(t)\ge0$ for every $n\in\Nzero$ and $\mathcal H(\tau)\le8$. Thus Definition~\ref{def:three-majorization}, at order $m$, yields
\begin{equation}
\label{eq:smoothness-coefficient-bound}
\begin{aligned}
\frac{|\Phi_{\xi,\vartheta}^{(m)}(0)|}{m!}
&\le[t^m]\mathcal H(t)\le\tau^{-m}\sum_{n=0}^{\infty}\bigl([t^n]\mathcal H(t)\bigr)\tau^n\\
&=\tau^{-m}\mathcal H(\tau)\le8\tau^{-m}.
\end{aligned}
\end{equation}
Since $\deg Q\le s=m-1$, $D^mQ=0$. Differentiating~\eqref{eq:smoothness-normalized-line} and using~\eqref{eq:smoothness-coefficient-bound} gives
\begin{equation}
\label{eq:smoothness-diagonal-bound}
\begin{aligned}
|D^mF_\vartheta(x_0)[\xi^m]|
&=\beta MR(x_0)^m\rho^{-m}|\Phi_{\xi,\vartheta}^{(m)}(0)|\\
&\le8\beta M\left(\frac{R(x_0)}{\rho}\right)^m m!(64\mathrm e^{300}q)^m\\
&\le8\beta ML_{s,D}^m m!(64\mathrm e^{300}q)^m.
\end{aligned}
\end{equation}
The symmetric $m$-linear form $D^mF_\vartheta(x_0)$ therefore satisfies, by~\eqref{eq:three-polarization-bound} with $j=m$,
\begin{equation}
\label{eq:smoothness-tensor-bound}
\begin{aligned}
\|D^mF_\vartheta(x_0)\|_{\op}
&\le\frac{m^m}{m!}\sup_{|\xi|=1}|D^mF_\vartheta(x_0)[\xi^m]|\\
&\le\frac{m^m}{m!}\,8\beta ML_{s,D}^m m!(64\mathrm e^{300}q)^m\\
&=8\beta ML_{s,D}^m(64\mathrm e^{300}qm)^m.
\end{aligned}
\end{equation}

By~\eqref{eq:three-R-comparison}, $L_{s,D}^m\le8\,2^mD^3$. Moreover, $128<\mathrm e^5$ and $64\beta<64m<\mathrm e^{m+5}\le\mathrm e^{10m}$, since $m\ge3$ and $m\le\mathrm e^m$. Hence
\begin{equation}
\label{eq:smoothness-constant-bound}
\begin{aligned}
8\beta L_{s,D}^m(64\mathrm e^{300}qm)^m
&\le64\beta D^3(128\mathrm e^{300}qm)^m\\
&\le\mathrm e^{10m}D^3(\mathrm e^{305}qm)^m\\
&=D^3(\mathrm e^{315}qm)^m\\
&\le D^3(\mathrm e^{2000}qm)^m=C_s(D).
\end{aligned}
\end{equation}
Since $x_0\notin\mathbb X^N$ and $\vartheta\in[0,1]$ were arbitrary, \eqref{eq:smoothness-tensor-bound}--\eqref{eq:smoothness-constant-bound} give
\[
\sup_{\vartheta\in[0,1]}\sup_{x\notin\mathbb X^N}\|D^{s+1}F_\vartheta(x)\|_{\op}\le C_s(D)M.
\]
\end{proof}

\begin{lemma}[Lipschitz Control Across Finite Sites]
\label{lem:smoothness-finite-sites}
Let $V$ be a finite-dimensional normed space and let $H:\R^d\to V$ be continuous on $\R^d$ and continuously differentiable on $\R^d\setminus\mathbb{X}^N$. If $B\ge0$ satisfies $\|DH(x)\|_{\op}\le B$ for every $x\notin\mathbb{X}^N$, then
\begin{equation}
\label{eq:smoothness-finite-sites}
\|H(y)-H(x)\|_V\le B|y-x|,\qquad x,y\in\R^d.
\end{equation}
\end{lemma}
\begin{proof}
For $x=y$, both sides of~\eqref{eq:smoothness-finite-sites} vanish. Fix $x\ne y$ and put $\gamma(t)\eqdef x+t(y-x)$ for $t\in[0,1]$. Since $|\gamma(t)-\gamma(u)|=|t-u|\,|y-x|$, the map $\gamma$ is injective. Consequently, there exist $1\le J\le N+1$ and $0=t_0<t_1<\cdots<t_J=1$ such that
\[
\{t_0,\ldots,t_J\}=\{0,1\}\cup\{t\in(0,1):\gamma(t)\in\mathbb X^N\},\qquad \gamma((t_{k-1},t_k))\cap\mathbb X^N=\varnothing\quad(k\in[J]).
\]
For $k\in[J]$ and $t_{k-1}<a<b<t_k$, the map $H\circ\gamma$ is continuously differentiable on a neighborhood of $[a,b]$, with $(H\circ\gamma)'(t)=DH(\gamma(t))[y-x]$. The fundamental theorem of calculus and the integral triangle inequality give
\begin{equation}
\label{eq:smoothness-subsegment}
\begin{aligned}
H(\gamma(b))-H(\gamma(a))
&=\int_a^b DH(\gamma(t))[y-x]\,\dd t,\\
\|H(\gamma(b))-H(\gamma(a))\|_V
&\le\int_a^b\|DH(\gamma(t))[y-x]\|_V\,\dd t\\
&\le\int_a^b\|DH(\gamma(t))\|_{\op}|y-x|\,\dd t
\le B(b-a)|y-x|.
\end{aligned}
\end{equation}

Set $\Delta_k\eqdef t_k-t_{k-1}>0$. For $0<\varepsilon<1/2$, put $a_{k,\varepsilon}\eqdef t_{k-1}+\varepsilon\Delta_k$ and $b_{k,\varepsilon}\eqdef t_k-\varepsilon\Delta_k$. Then $t_{k-1}<a_{k,\varepsilon}<b_{k,\varepsilon}<t_k$ and $b_{k,\varepsilon}-a_{k,\varepsilon}=(1-2\varepsilon)\Delta_k$. By continuity of $H\circ\gamma$ on $[0,1]$ and~\eqref{eq:smoothness-subsegment},
\[
\begin{aligned}
\|H(\gamma(t_k))-H(\gamma(t_{k-1}))\|_V
&=\lim_{\varepsilon\downarrow0}\|H(\gamma(b_{k,\varepsilon}))-H(\gamma(a_{k,\varepsilon}))\|_V\\
&\le\lim_{\varepsilon\downarrow0}B(1-2\varepsilon)\Delta_k|y-x|
=B\Delta_k|y-x|.
\end{aligned}
\]
Therefore,
\[
\begin{aligned}
\|H(y)-H(x)\|_V
&=\left\|\sum_{k=1}^J\bigl(H(\gamma(t_k))-H(\gamma(t_{k-1}))\bigr)\right\|_V\\
&\le\sum_{k=1}^J\|H(\gamma(t_k))-H(\gamma(t_{k-1}))\|_V\\
&\le B|y-x|\sum_{k=1}^J(t_k-t_{k-1})=B|y-x|.
\end{aligned}
\]
\end{proof}

We first estimate the numerical multiplier used below.
\begin{lemma}[Growth of the $C^{s,1}$-Constant]
\label{lem:Cs1_constant_growth}
For $d\in\mathbb N_+$ and $s\in\mathbb N$, define
\begin{equation}
\label{eq:Kds-exact-growth}
K_{d,s}\eqdef
\begin{cases}
1,&s=0,1,\\[1mm]
\displaystyle 2^{138d}\left[\mathrm{e}^{2000}(s+1)(2s+46d+2)\right]^{s+1},&s\ge2.
\end{cases}
\end{equation}
Then $K_{d,s}\ge1$ and $K_{d,s}\le\mathrm{e}^{4400d}(d+s)^{4400s}$.
\end{lemma}
\begin{proof}
If $s\in\{0,1\}$, then $d+s\ge1$ and $1=K_{d,s}\le\mathrm e^{4400d}(d+s)^{4400s}$. Suppose $s\ge2$. Since $(s+1)(2s+46d+2)\ge1$, \eqref{eq:Kds-exact-growth} gives $K_{d,s}\ge2^{138d}\mathrm e^{2000(s+1)}>1$.

Put $\ell\eqdef\log(d+s+1)\ge\log4$. Using $s+1\le d+s+1$, $2s+46d+2\le46(d+s+1)$, $138\log2<96$, and $\log46<4$, we obtain
\[
\begin{aligned}
\log K_{d,s}
&=138d\log2+(s+1)\bigl(2000+\log(s+1)+\log(2s+46d+2)\bigr)\\
&\le138d\log2+(s+1)\bigl(2000+\log46+2\ell\bigr)\\
&\le96d+(s+1)(2004+2\ell)\\
&\le96d+3006s+3s\ell,
\end{aligned}
\]
where $s+1\le3s/2$ was used in the last line. Since $3006\le2197\log4\le2197\ell$, one has $3006s+3s\ell\le2200s\ell$. Moreover, $d+s\ge3$ implies $d+s+1\le(d+s)^2$, hence $\ell\le2\log(d+s)$. Therefore,
\begin{equation}
\label{eq:de_estimatoski}
\begin{aligned}
\log K_{d,s}
&\le96d+2200s\ell\\
&\le2200d+4400s\log(d+s)\\
&\le4400d+4400s\log(d+s).
\end{aligned}
\end{equation}
Exponentiating~\eqref{eq:de_estimatoski} yields $K_{d,s}\le\mathrm e^{4400d}\mathrm e^{4400s\log(d+s)}=\mathrm e^{4400d}(d+s)^{4400s}$.
\end{proof}

\noindent
We are now ready to derive our main result.

\begin{proof}[{Proof of Theorem~\ref{thrm:main} (ii)}]
Recall from~\eqref{eq:entropy-variance-extension__specialization} that the calibrated reconstruction is $F=F_1$. We treat separately the cases $s=0$, $s=1$, and $s\ge2$.

Suppose first that $s=0$. By Lemma~\ref{lem:three-zero-order}, $F\in C(\R^d)$ and $|F(x)-F(y)|\le M|x-y|$ for all $x,y\in\R^d$. Since $D^0F=F$, Definition~\eqref{eq:Lip_Seminorm} yields $[F]_{0,1}\le M$, and therefore $F\in\dot C^{0,1}(\R^d)$.

Suppose next that $s=1$. Lemma~\ref{lem:three-first-order} yields $F\in C^1(\R^d)$ and $|\nabla F(x)-\nabla F(y)|\le M|x-y|$ for all $x,y\in\R^d$. For fixed $x,y\in\R^d$, put $g_{x,y}\eqdef\nabla F(x)-\nabla F(y)$. Since $(DF(x)-DF(y))[v]=g_{x,y}\cdot v$, Cauchy-Schwarz and, when $g_{x,y}\ne0$, the choice $v=g_{x,y}/|g_{x,y}|$ give $\|DF(x)-DF(y)\|_{\op}=|g_{x,y}|$. The same identity is immediate when $g_{x,y}=0$. Hence $[F]_{1,1}\le M$, and therefore $F\in\dot C^{1,1}(\R^d)$.

Assume henceforth that $s\ge2$ and fix an arbitrary $\vartheta\in[0,1]$. By Setting~\ref{TheSetting}, $f\in\dot C^{s,1}(\R^d)\subset C^s(\R^d)$ and $P_i=J_{X_i}^sf$ for every $i\in[N]$. Moreover,~\eqref{eq:entropy-variance-extension__primal___SITES} gives $F_\vartheta(X_i)=P_i(X_i)=f(X_i)$ for every $i\in[N]$. Lemma~\ref{lem:three-local-errors} gives
\begin{equation}
\label{eq:smoothness-family-local-errors}
F_\vartheta\in C^\infty(\R^d\setminus\mathbb X^N),\qquad \|D^jF_\vartheta(x)-D^jf(x)\|_{\op}\le C_s(D)M\rho_{\mathbb X^N}(x)^{s+1-j}
\end{equation}
for every $x\in\R^d\setminus\mathbb X^N$ and $0\le j\le s$. Thus the hypotheses of Lemma~\ref{lem:three-site-gluing} hold with $G=F_\vartheta$, $u=f$, and $B=C_s(D)M$. Consequently,
\begin{equation}
\label{eq:smoothness-family-Cs-and-jets}
F_\vartheta\in C^s(\R^d),\qquad D^jF_\vartheta(X_i)=D^jf(X_i)=D^jP_i(X_i),\qquad i\in[N],\quad0\le j\le s.
\end{equation}

Let $V_s\eqdef\operatorname{Sym}^s(\R^d;\R)$ be the finite-dimensional normed space of symmetric $s$-linear forms on $\R^d$, equipped with $\|\cdot\|_{\op}$, and set $H_\vartheta\eqdef D^sF_\vartheta:\R^d\to V_s$. By~\eqref{eq:smoothness-family-Cs-and-jets}, $H_\vartheta$ is continuous on $\R^d$, while~\eqref{eq:smoothness-family-local-errors} implies that it is continuously differentiable on $\R^d\setminus\mathbb X^N$. For $x\notin\mathbb X^N$ and $u,v_1,\ldots,v_s\in\R^d$, one has $(DH_\vartheta(x)[u])[v_1,\ldots,v_s]=D^{s+1}F_\vartheta(x)[u,v_1,\ldots,v_s]$. Hence, by Lemma~\ref{lem:smoothness-top-off-sites},
$$
\begin{aligned}
\|DH_\vartheta(x)\|_{\op}&=\sup_{|u|\le1}\|DH_\vartheta(x)[u]\|_{\op}\\
&=\sup_{|u|\le1}\ \sup_{|v_1|,\ldots,|v_s|\le1}|D^{s+1}F_\vartheta(x)[u,v_1,\ldots,v_s]|\\
&=\|D^{s+1}F_\vartheta(x)\|_{\op}\le C_s(D)M,
\end{aligned}
\qquad x\in\R^d\setminus\mathbb X^N.
$$
Lemma~\ref{lem:smoothness-finite-sites}, applied to $H_\vartheta$ with $B=C_s(D)M$, therefore gives $\|D^sF_\vartheta(y)-D^sF_\vartheta(x)\|_{\op}\le C_s(D)M|y-x|$ for all $x,y\in\R^d$. Combining this estimate with~\eqref{eq:smoothness-family-Cs-and-jets} and Definition~\eqref{eq:Lip_Seminorm} yields
\begin{equation}
\label{eq:smoothness-family-bound}
F_\vartheta\in\dot C^{s,1}(\R^d),\qquad [F_\vartheta]_{s,1}\le C_s(D)M.
\end{equation}
The constants in Lemmas~\ref{lem:three-local-errors} and~\ref{lem:smoothness-top-off-sites} are independent of the fixed parameter $\vartheta$. Since $\vartheta\in[0,1]$ was arbitrary,
\begin{equation}
\label{eq:smoothness-family-uniform-bound}
\sup_{\vartheta\in[0,1]}[F_\vartheta]_{s,1}\le C_s(D)M.
\end{equation}
The geometry fixed in Theorem~\ref{thrm:main} has $D=2^{46d}$. Since $46d\in\mathbb N$, Definition~\eqref{eq:three-constant-D} gives
\begin{equation}
\label{eq:smoothness-constant-specialization}
\begin{aligned}
C_s(2^{46d})&=(2^{46d})^3\left[\mathrm e^{2000}(s+1)\bigl(2s+\lceil\log_2(2^{46d})\rceil+2\bigr)\right]^{s+1}\\
&=2^{138d}\left[\mathrm e^{2000}(s+1)(2s+46d+2)\right]^{s+1}.
\end{aligned}
\end{equation}
Taking $\vartheta=1$ in~\eqref{eq:smoothness-family-bound} and using $F=F_1$, then combining with the cases $s=0,1$, proves
\begin{equation}
\label{eq:smoothness-all-orders}
[F]_{s,1}\le\begin{cases}M,&s=0,1,\\[1mm]2^{138d}\left[\mathrm e^{2000}(s+1)(2s+46d+2)\right]^{s+1}M,&s\ge2,\end{cases}=K_{d,s}M.
\end{equation}
The constant in~\eqref{eq:smoothness-all-orders} is precisely the one defined in~\eqref{eq:Kds-exact-growth}. Lemma~\ref{lem:Cs1_constant_growth} further gives $1\le K_{d,s}\le\mathrm e^{4400d}(d+s)^{4400s}$, which proves the asserted dependence on $d,s$ and completes the proof.
\end{proof}

\subsection{The Fixed-Parameter Family and the Entropy-Only Endpoint}
\label{s:Proofs__ss:EntropyOnly}

We first prove this intermediate result.
\begin{proposition}[Uniform Fixed-Parameter Higher-Order Family]
\label{prop:variance-parameter-family}
Under Setting~\ref{TheSetting}, suppose $s\ge2$, let $D=2^{46d}$, and choose $\mu\in\Delta_{N:D}$. For every fixed $\vartheta\in[0,1]$, the function $F_\vartheta$ defined by~\eqref{eq:entropy-variance-extension__primal___SITES} and~\eqref{eq:entropy-variance-extension__primal} satisfies
\begin{enumerate}
\item[(i)] $F_\vartheta\in C^s(\mathbb R^d)$ and
\[
 D^jF_\vartheta(X_i)=D^jP_i(X_i),
 \qquad i\in[N],\quad 0\le j\le s;
\]
\item[(ii)] $F_\vartheta\in\dot C^{s,1}(\mathbb R^d)$ and
\[
 [F_\vartheta]_{s,1}\le K_{d,s}M,
\]
where $K_{d,s}$ is exactly the constant in Theorem~\ref{thrm:main} and the estimate is uniform in the fixed parameter $\vartheta\in[0,1]$.
\end{enumerate}
\end{proposition}
\begin{proof}[{Proof of Proposition~\ref{prop:variance-parameter-family}}]
Fix an arbitrary $\vartheta\in[0,1]$. Equation~\eqref{eq:smoothness-family-Cs-and-jets} gives $F_\vartheta\in C^s(\R^d)$ and $D^jF_\vartheta(X_i)=D^jP_i(X_i)$ for every $i\in[N]$ and $0\le j\le s$, which proves~(i). Equation~\eqref{eq:smoothness-family-bound} gives $F_\vartheta\in\dot C^{s,1}(\R^d)$ and $[F_\vartheta]_{s,1}\le C_s(D)M$. For $D=2^{46d}$, equations~\eqref{eq:smoothness-constant-specialization} and~\eqref{eq:Kds-exact-growth} give $C_s(D)=K_{d,s}$, which proves~(ii). Finally,~\eqref{eq:smoothness-family-uniform-bound} gives $\sup_{\vartheta\in[0,1]}[F_\vartheta]_{s,1}\le K_{d,s}M$, so the estimate is uniform in the fixed parameter.
\end{proof}

% \begin{proof}[{Proof of Corollary~\ref{cor:entropy-only-reconstruction}}]
\begin{proof}[{Proof of Proposition~\ref{prop:entropy-only-reconstruction}}]
Write $\widetilde F\eqdef F_{\operatorname{ext}}$.
Put $U\eqdef\R^d\setminus\mathbb X^N$ and fix $x\in U$. Since $s\ge2$, $M>0$, and $R(x)>0$, define
\begin{equation}
\label{eq:entropy-only-data}
a_i(x)\eqdef P_i(x)+\frac{A_sM|x-X_i|^{2s}}{R(x)^{s-1}},\qquad \tau(x)\eqdef(s-1)MR(x)^{s+1}>0.
\end{equation}
Because $0\in\mathcal I_s$, Definition~\eqref{eq:entropy-variance-extension__primal} yields
\begin{equation}
\label{eq:entropy-only-zero-primal}
F_0(x)=\min_{\lambda\in\Delta_N}\left\{\sum_{i=1}^N\lambda_i a_i(x)+\tau(x)\operatorname{KL}(\lambda\mid\theta(x))\right\}.
\end{equation}
Since $\theta_i(x)>0$ and $\sum_i\theta_i(x)=1$, define
\begin{equation}
\label{eq:entropy-only-gibbs-measure}
Z(x)\eqdef\sum_{i=1}^N\theta_i(x)\exp\!\left(-\frac{a_i(x)}{\tau(x)}\right)>0,\qquad \nu_i(x)\eqdef\frac{\theta_i(x)\exp(-a_i(x)/\tau(x))}{Z(x)}.
\end{equation}
Then $\nu_i(x)>0$ for every $i$ and $\sum_i\nu_i(x)=1$, so $\nu(x)\in\Delta_N^\circ$.

Let $\lambda\in\Delta_N$ and put $I_\lambda\eqdef\{i\in[N]:\lambda_i>0\}$. Since $\theta_i(x),\nu_i(x)>0$, one has $\log(\nu_i(x)/\theta_i(x))=-a_i(x)/\tau(x)-\log Z(x)$ for every $i$. Coordinates outside $I_\lambda$ contribute zero to both relative entropies by the convention $0\log0=0$. Therefore,
\begin{equation}
\label{eq:entropy-only-KL-decomposition}
\begin{aligned}
\tau(x)\operatorname{KL}(\lambda\mid\nu(x))&=\tau(x)\sum_{i\in I_\lambda}\lambda_i\left[\log\frac{\lambda_i}{\theta_i(x)}-\log\frac{\nu_i(x)}{\theta_i(x)}\right]\\
&=\tau(x)\operatorname{KL}(\lambda\mid\theta(x))+\sum_{i=1}^N\lambda_i a_i(x)+\tau(x)\log Z(x).
\end{aligned}
\end{equation}
To treat boundary points of $\Delta_N$, put $t_i\eqdef\lambda_i/\nu_i(x)\in[0,\infty)$ and define
$$
\phi(t)\eqdef\begin{cases}t\log t-t+1,&t>0,\\1,&t=0.\end{cases}
$$
The convention at zero agrees with $\lim_{t\downarrow0}t\log t=0$. On $(0,\infty)$, $\phi'(t)=\log t$ and $\phi''(t)=1/t>0$, so $\phi(t)\ge\phi(1)=0$, with equality exactly at $t=1$; moreover, $\phi(0)=1>0$. Since $t\log t=\phi(t)+t-1$ on $[0,\infty)$ under the convention $0\log0=0$, one has
\begin{equation}
\label{eq:entropy-only-KL-nonnegative}
\begin{aligned}
\operatorname{KL}(\lambda\mid\nu(x))&=\sum_{i=1}^N\nu_i(x)t_i\log t_i\\
&=\sum_{i=1}^N\nu_i(x)\phi(t_i)+\sum_{i=1}^N\lambda_i-\sum_{i=1}^N\nu_i(x)\\
&=\sum_{i=1}^N\nu_i(x)\phi(t_i)\ge0.
\end{aligned}
\end{equation}
Because every $\nu_i(x)>0$, equality in~\eqref{eq:entropy-only-KL-nonnegative} holds precisely when $t_i=1$ for every $i$, that is, when $\lambda=\nu(x)$. Since $\tau(x)>0$, equations~\eqref{eq:entropy-only-zero-primal} and~\eqref{eq:entropy-only-KL-decomposition} show that $\nu(x)$ is the unique minimizer of the $\vartheta=0$ objective and
\begin{equation}
\label{eq:entropy-only-value}
F_0(x)=-\tau(x)\log Z(x).
\end{equation}

By~\eqref{eq:entropy-only-data},
$$
\begin{aligned}
\frac{a_i(x)}{\tau(x)}&=\frac{P_i(x)}{(s-1)MR(x)^{s+1}}+\frac{A_sM|x-X_i|^{2s}}{(s-1)MR(x)^{s-1}R(x)^{s+1}}=\frac1{s-1}\left[\frac{P_i(x)}{MR(x)^{s+1}}+A_s\left(\frac{|x-X_i|}{R(x)}\right)^{2s}\right].
\end{aligned}
$$
Substituting this identity and $\tau(x)=(s-1)MR(x)^{s+1}$ into~\eqref{eq:entropy-only-value} gives exactly~\eqref{eq:ExtenstionFormula_Expression} on $U$. At every site $X_i$, the site definitions give $F_0(X_i)=P_i(X_i)=\widetilde F(X_i)$. Hence
\begin{equation}
\label{eq:entropy-only-equals-family}
\widetilde F=F_0\qquad\text{on }\R^d.
\end{equation}
If $D=2^{46d}$ and $\mu\in\Delta_{N:D}$, Proposition~\ref{prop:variance-parameter-family}, applied with $\vartheta=0$, gives $D^j\widetilde F(X_i)=D^jP_i(X_i)$ for every $i\in[N]$ and $0\le j\le s$, together with $\widetilde F\in\dot C^{s,1}(\mathbb R^d)$ and $[\widetilde F]_{s,1}\le K_{d,s}M$.

Under the original hypotheses, we next verify off-site real analyticity.  
For $\gamma\ge0$, set $S_\gamma(x)\eqdef\sum_{i=1}^N\mu_i|x-X_i|^{-\gamma}$ for $x\in U$. For every $i$, the polynomial $\sigma_i(x)\eqdef\sum_{k=1}^d(x_k-X_{i,k})^2$ is strictly positive on $U$. Since $t\mapsto t^{1/2}$ is real analytic on $(0,\infty)$, $r_i(x)\eqdef|x-X_i|=\sigma_i(x)^{1/2}$ is positive and real analytic on $U$. Hence $S_q$ and $S_{q-p}$ are positive real-analytic functions on $U$. From the definitions of $\theta$ and $R$,
\begin{equation}
\label{eq:entropy-only-R-ratio}
R(x)^p=\sum_{i=1}^N\theta_i(x)r_i(x)^p=\frac{\sum_{i=1}^N\mu_i r_i(x)^{p-q}}{S_q(x)}=\frac{S_{q-p}(x)}{S_q(x)}>0.
\end{equation}
Since $t\mapsto t^{1/p}$ is real analytic on $(0,\infty)$,~\eqref{eq:entropy-only-R-ratio} implies that $R$ is positive and real analytic on $U$. The identity $\theta_i(x)=\mu_i r_i(x)^{-q}/S_q(x)$ then shows that every $\theta_i$ is real analytic on $U$. It follows from~\eqref{eq:entropy-only-data} that each $a_i$ and $\tau$ is real analytic on $U$, with $\tau>0$. Consequently, each function $x\mapsto\theta_i(x)\exp(-a_i(x)/\tau(x))$ is positive and real analytic on $U$, so their finite sum $Z$ is positive and real analytic. Since $\log$ is real analytic on $(0,\infty)$,~\eqref{eq:entropy-only-value} yields $\widetilde F|_U=-\tau\log Z\in C^\omega(U)$.

We finally prove that $\widetilde F$ is definable in $\R_{\operatorname{exp}}$. The set $U=\R^d\setminus\{X_1,\ldots,X_N\}$ is semialgebraic and hence definable in $\R_{\operatorname{exp}}$. The graph of the positive distance $r_i$ on $U$ is the semialgebraic set
$$
\left\{(x,r)\in U\times(0,\infty):r^2=\sum_{k=1}^d(x_k-X_{i,k})^2\right\},
$$
and is therefore definable. Because $p,q,q-p\in\mathbb N_+$ and each $r_i$ is positive on $U$, the functions $r_i^{-q}$, $S_q$, $S_{q-p}$, and $\theta_i$ are definable there. By~\eqref{eq:entropy-only-R-ratio}, the graph of $R$ on $U$ is the definable set
$$
\left\{(x,r)\in U\times(0,\infty):r^pS_q(x)=S_{q-p}(x)\right\}.
$$
Hence the functions $a_i$ and $\tau$ in~\eqref{eq:entropy-only-data} are definable on $U$. Since the exponential map belongs to the language of $\R_{\operatorname{exp}}$, the function $Z$ in~\eqref{eq:entropy-only-gibbs-measure} is definable. The graph of $\log$ on $(0,\infty)$ is the definable set $\{(z,y)\in(0,\infty)\times\R:\exp(y)=z\}$, so $\log$ is definable there. It follows that
$$
\Gamma_U\eqdef\{(x,y)\in U\times\R:y=-\tau(x)\log Z(x)\}
$$
is definable. The finite set $\Gamma_{\mathbb X^N}\eqdef\bigcup_{i=1}^N\{(X_i,P_i(X_i))\}$ is also definable, and~\eqref{eq:entropy-only-equals-family} gives $\operatorname{graph}(\widetilde F)=\Gamma_U\cup\Gamma_{\mathbb X^N}$. Thus $\widetilde F$ is definable in $\R_{\operatorname{exp}}$. Every parameter appearing in its graph is one of the site coordinates, a coefficient of one of the polynomials $P_i$, $M$, or one of the weights $\mu_i$; the witness $f$ does not enter the defining formula.
\end{proof}

\begin{lemma}[Explicit finite-jet seminorm bound]
\label{lem:finite-whitney-seminorm__bound}
Fix $d,N\in\mathbb N_+$ and $s\in\mathbb N$. Let $X_1,\ldots,X_N\in\mathbb R^d$ be distinct, and let $P_1,\ldots,P_N$ be real polynomials of degree at most $s$. For $\mathbb D^N=\{(X_i,P_i)\}_{i=1}^N$, 
let $W_{\mathbb D^N}$ be defined by~\eqref{eq:finite-whitney-seminorm}, and set
\begin{equation}
\label{eq:explicit-whitney-calibration-constant}
\kappa_{d,s}\eqdef
\begin{cases}
1 & \mbox{ : if }s=0,\\[1mm]
2^{138d+3}d^{s/2}
\bigl[128(s+1)(2s+46d+2)\bigr]^{s+1}&  \mbox{ : if } s\ge1.
\end{cases}
\end{equation}
Then:
\begin{enumerate}
\item[(i)] Every $f\in\dot C^{s,1}(\mathbb R^d)$ satisfying
$J_{X_i}^sf=P_i$ for all $i\in[N]$ obeys
$W_{\mathbb D^N}\le[f]_{s,1}$.
\item[(ii)] There exists $g\in\dot C^{s,1}(\mathbb R^d)$ such that
$J_{X_i}^sg=P_i$ for all $i\in[N]$ and
$[g]_{s,1}\le\kappa_{d,s}W_{\mathbb D^N}$.
\item[(iii)] $W_{\mathbb D^N}=0$ if and only if all $P_i$
coincide. In this case one may take $g=P_1$.
\end{enumerate}
Moreover, for $s\ge1$, the constant satisfies $1\le\kappa_{d,s}\le\mathrm{e}^{100d}(d+s)^{32s}$.
\end{lemma}
\begin{proof}[{Proof of Lemma~\ref{lem:finite-whitney-seminorm__bound}}]
Write $W=W_{\mathbb D^N}$ and $\mathbb X^N=\{X_1,\ldots,X_N\}$. The convention for $N=1$ gives $W=0$; otherwise, $W$ is a finite maximum with positive denominators. Let $f\in\dot C^{s,1}(\mathbb R^d)$ satisfy the prescribed jets. For $i\ne j$ and $|\alpha|\le s$, Lemma~\ref{lem:three-taylor}, with base point $X_i$ and evaluation point $X_j$, gives $|\partial^\alpha(P_i-P_j)(X_j)|\le [f]_{s,1}|X_i-X_j|^{s+1-|\alpha|}/(s+1-|\alpha|)!$, since each coordinate component of a derivative tensor is bounded by its operator norm. Dividing by $|X_i-X_j|^{s+1-|\alpha|}$ and taking the maximum proves (i), the case $N=1$ being immediate.

Suppose that $W=0$. If $N\ge2$, every derivative of $P_i-P_j$ of order at most $s$ vanishes at $X_j$; the exact Taylor expansion of this polynomial therefore gives $P_i=P_j$. For $N=1$, the assertion is automatic. Conversely, identical polynomials give $W=0$. Since the order-$s$ derivative of a polynomial of degree at most $s$ is constant, $g=P_1$ has seminorm zero. This proves (iii) and also (ii) when $W=0$.

We may henceforth assume that $W>0$, and therefore $N\ge2$. If $s=0$, write $P_i=c_i$ and define $g(x)=\min_i\{c_i+W|x-X_i|\}$. The definition of $W$ gives $c_i+W|X_j-X_i|\ge c_j$, with equality for $i=j$, so $g(X_j)=c_j$. Choosing an index attaining the minimum at $y$ gives $g(x)\le g(y)+W|x-y|$; interchanging $x$ and $y$ proves $[g]_{0,1}\le W$. Thus (ii) holds for $s=0$.

Assume now that $s\ge1$, and put $m=s+1$, $D=2^{46d}$ and $q=2m+46d$. By Lemma~\ref{lem:three-doubling-weights}, there are positive probability weights $\mu_i$ satisfying~\eqref{eq:doubling} with this value of $D$. Define
\begin{equation}\label{eq:explicit-whitney-witness}
g(x)=\frac{\sum_i\mu_i|x-X_i|^{-q}P_i(x)}{\sum_i\mu_i|x-X_i|^{-q}}\quad(x\notin\mathbb X^N),\qquad g(X_i)=P_i(X_i).
\end{equation}
The evenness of $q$ implies that $A_i(x)=\mu_i\prod_{k\ne i}|x-X_k|^q$ and $A(x)=\sum_iA_i(x)$ are polynomials. Off the sites, every $A_i$ is positive, while $A(X_j)=\mu_j\prod_{k\ne j}|X_j-X_k|^q>0$. Clearing denominators in~\eqref{eq:explicit-whitney-witness} therefore gives $g=\sum_iA_iP_i/A$ on $\mathbb R^d$, so $g\in C^\infty(\mathbb R^d)$. For each $j$, the numerator of $g-P_j=\sum_{i\ne j}A_i(P_i-P_j)/A$ is divisible by the polynomial $|x-X_j|^q$. Since $q>s$ and $A$ is nonzero at $X_j$, all derivatives of $g-P_j$ through order $s$ vanish there. Hence $J_{X_j}^sg=P_j$.

Fix $x\notin\mathbb X^N$, set $r_i=|x-X_i|$, $\rho=\min_i r_i>0$ and $\theta_i=\mu_i r_i^{-q}/\sum_k\mu_k r_k^{-q}$, and choose a nearest site $a=X_j$. Write $B_{\mathbb X^N}(a,r)=\{X_i:|X_i-a|<r\}$ and $\mu_0=\sum_{X_i\in B_{\mathbb X^N}(a,\rho)}\mu_i$. The ball contains $a$, so $\mu_0>0$, and its sites have distance less than $2\rho$ from $x$. Consequently, $\sum_i\mu_i r_i^{-q}\ge(2\rho)^{-q}\mu_0$. For each $\ell\ge0$, the shell $2^\ell\rho\le r_i<2^{\ell+1}\rho$ lies in $B_{\mathbb X^N}(a,2^{\ell+2}\rho)$ and has mass at most $D^{\ell+2}\mu_0$ by repeated application of~\eqref{eq:doubling}. Since $q>m$ and $D2^{m-q}=2^{-m}\le1/4$, we obtain
\[
\sum_i\mu_i r_i^{m-q}\le D^2\mu_0\rho^{m-q}\sum_{\ell=0}^{\infty}(D2^{m-q})^\ell\le\frac43D^2\mu_0\rho^{m-q}.
\]
Dividing by the preceding lower bound for the denominator gives
\begin{equation}\label{eq:explicit-whitney-distance-moment}
\sum_i\theta_i r_i^m\le\frac43\,2^qD^2\rho^m=\frac43\,4^mD^3\rho^m.
\end{equation}

Let $\xi$ be a unit vector, put $P_a=P_j$ and $h(t)=g(x+\rho t\xi)-P_a(x+\rho t\xi)$, and keep $x,\rho,a$ fixed throughout the following coefficient calculations. For scalar Taylor series $B,C$, write $B\preccurlyeq C$ if $C$ has nonnegative coefficients and $|[t^k]B|\le[t^k]C$ for every $k\ge0$. The triangle inequality and coefficient convolution show that these bounds are preserved under sums and products. Set $a_i=\rho\langle x-X_i,\xi\rangle/r_i^2$ and $b_i=\rho^2/r_i^2$, so $|a_i|\le1$ and $0<b_i\le1$. The normalized kernel $\gamma_i(t)=(|x+\rho t\xi-X_i|/r_i)^{-q}$ equals $(1+2a_it+b_it^2)^{-q/2}$. Thus the binomial expansion gives $\gamma_i\preccurlyeq U$ and $\gamma_i-1\preccurlyeq U-1$, where $U(t)=(1-2t-t^2)^{-q/2}$. For $Z(t)=\sum_i\theta_i\gamma_i(t)$, we have $Z(0)=1$ and $Z-1\preccurlyeq U-1$. Since $Z-1$ has zero constant coefficient, only finitely many powers contribute to each coefficient of its geometric inverse; hence
\begin{equation}\label{eq:explicit-whitney-reciprocal}
Z^{-1}=\sum_{\ell\ge0}(-1)^\ell(Z-1)^\ell\preccurlyeq\sum_{\ell\ge0}(U-1)^\ell=(2-U)^{-1}.
\end{equation}

For $i\ne j$ and $|\alpha|\le s$, the ordered pair $(i,j)$ in the definition of $W$ gives $|\partial^\alpha(P_i-P_a)(a)|\le W|X_i-a|^{m-|\alpha|}\le W(2r_i)^{m-|\alpha|}$, because $|X_i-a|\le r_i+\rho\le2r_i$. The same bound holds for $i=j$, when the polynomial difference is zero. Put $V(t)=\sum_{k=0}^s[\sqrt d(1+t)/2]^k/k!$. Expanding $P_i-P_a$ at $a$, grouping equal degrees by the multinomial formula, and using $\|x-a\|_1\le\sqrt d\,\rho$, $\|\xi\|_1\le\sqrt d$ and $\rho\le r_i$, we obtain
\[
\begin{aligned}
P_i(x+\rho t\xi)-P_a(x+\rho t\xi)&\preccurlyeq W\sum_{k=0}^s\frac{(2r_i)^{m-k}[\sqrt d\,\rho(1+t)]^k}{k!}\\
&\preccurlyeq W(2r_i)^mV(t).
\end{aligned}
\]
The identity $h=Z^{-1}\sum_i\theta_i\gamma_i\bigl(P_i(x+\rho t\xi)-P_a(x+\rho t\xi)\bigr)$ and~\eqref{eq:explicit-whitney-reciprocal} therefore imply
\begin{equation}\label{eq:explicit-whitney-value-majorant}
h(t)\preccurlyeq B_x(t)\eqdef W2^m\left(\sum_i\theta_i r_i^m\right)\frac{U(t)V(t)}{2-U(t)}.
\end{equation}

Choose $\delta=(16q)^{-1}$ and set $u=2\delta+\delta^2$. Since $0\le u\le3/(16q)\le3/16<1/2$, the inequality $-\log(1-u)\le2u$ yields $U(\delta)\le\mathrm{e}^{3/16}\le16/13$, where the last inequality follows by comparing the exponential and geometric series. In particular, $U(\delta)/(2-U(\delta))\le8/5<2$. Moreover, $d\ge1$ and $\delta\le1$ give $V(\delta)\le d^{s/2}\sum_{k=0}^s[(1+\delta)/2]^k/k!\le\mathrm{e}\,d^{s/2}$. These estimates ensure convergence of the nonnegative majorant at $\delta$, and~\eqref{eq:explicit-whitney-distance-moment} gives $B_x(\delta)\le(8\mathrm{e}/3)8^mD^3d^{s/2}W\rho^m$. Since $[t^m]B_x\le\delta^{-m}B_x(\delta)$ and $D^mP_a=0$, it follows from~\eqref{eq:explicit-whitney-value-majorant} that
\[
|D^mg(x)[\xi^m]|=m!\rho^{-m}|[t^m]h|\le\frac{8\mathrm{e}}{3}\,8^mD^3d^{s/2}m!(16q)^mW.
\]
The polarization estimate in Lemma~\ref{lem:three-taylor}, applied at order $m$, has factor $m^m/m!$. Therefore, using $\mathrm{e}<3$,
\[
\|D^mg(x)\|_{\mathrm{op}}\le\frac{8\mathrm{e}}{3}D^3d^{s/2}[128mq]^mW\le8D^3d^{s/2}[128mq]^mW=\kappa_{d,s}W.
\]

By continuity of $D^mg$, this bound also holds at the sites. For arbitrary $x,y\in\mathbb R^d$, the fundamental theorem of calculus gives
\[
D^sg(y)-D^sg(x)=\int_0^1D^{s+1}g(x+t(y-x))[y-x,\,\cdot\,]\,\mathrm dt.
\]
Taking operator norms yields $[g]_{s,1}\le\kappa_{d,s}W$, proving (ii). The estimates are independent of $N$, the individual atomic masses and the separation of the sites: the local mass cancels in~\eqref{eq:explicit-whitney-distance-moment}, and the factor $\rho^m$ cancels upon differentiation.

Finally, put $t=d+s\ge2$. Then $m\le t$, $m\le2s$, $q\le48t$, $d\le t$ and $6144\le t^{13}$. Also, $\log2<7/10$ and $d\ge1$ give $2^{138d+3}\le\mathrm{e}^{100d}$. Consequently,
$
\kappa_{d,s}\le\mathrm{e}^{100d}t^{s/2}[6144t^2]^m\le\mathrm{e}^{100d}t^{s/2+15m}\le\mathrm{e}^{100d}t^{32s},
$ and we are done.
\end{proof}

\subsection{{Proof of Analyticity Away from the Data; Theorem~\ref{thrm:main} (iv)}}
\label{s:Proof_Analyticity}

We now prove the analyticity of our formula away from the sites $\{X_n\}_{n=1}^N$ when $s\ge 2$.  

It will be convenient to prove the following dual form of our formula; which is only valid when $s\ge 2$.
We now \textbf{temporarily} write $G(x)$ for the $\vartheta=1$ primal value in~\eqref{eq:entropy-variance-extension__primal}. By~\eqref{eq:entropy-variance-extension__specialization}, this is the calibrated reconstruction $F$; the temporary symbol $G$ is used only to distinguish the primal representation from the dual expression introduced below.
At the sites, set $G(X_n)\eqdef P_n(X_n)$ for every $n\in[N]$.
For $s\ge2$ and $x\in\mathbb{R}^d\setminus\mathbb{X}^N$, and for each $i\in[N]$, define
\begin{equation}
\label{eq:gibbs-dual-quantities}
\begin{aligned}
    a_i(x)
    &  \eqdef 
        P_i(x)
        +A_sM\frac{|x-X_i|^p}{R(x)^\beta}
\qquad
    &
    c(x)
    \eqdef 
        \frac{s}{4MR(x)^{s-1}},
\\
    \tau(x)
    &  \eqdef 
        \beta MR(x)^{s+1}
\qquad
    & 
    b_i(x)
    \eqdef 
        a_i(x)-c(x)|v_i(x)|^2
.
\end{aligned}
\end{equation}
In particular, $c(x)>0$ and $\tau(x)>0$.
Consistently with~\eqref{eq:entropy-variance-extension__primal___SITES} and~\eqref{eq:entropy-variance-extension__specialization}, retain the site values
$
    F(X_n) \eqdef P_n(X_n)
$ for each $n\in[N]$.
For any $s\ge2$ and $x\in\mathbb{R}^d\setminus\mathbb{X}^N$ define
\begin{extensionbox}
\begin{equation}
\label{eq:entropy-variance-extension}
\begin{aligned}
    F(x)
     \eqdef 
    \max_{y\in\mathbb{R}^d}
    \Biggl\{
        -\frac{|y|^2}{4c(x)}
        -\tau(x)
        \log\!\Biggl(
            \sum_{i=1}^{N}
            \theta_i(x)
            \exp\!\left(
                -\frac{
                    b_i(x)+\langle y,v_i(x)\rangle
                }{\tau(x)}
            \right)
        \Biggr)
    \Biggr\}.
\end{aligned}
\end{equation}
\end{extensionbox}
The dual representation replaces optimization over the mixing weights by optimization over a vector in $\mathbb R^d$. Equality of the two optimal values follows from an explicit identity: their objective values differ by a squared norm and a relative entropy, both of which vanish for weights recovered from a dual maximizer.

\begin{proposition}[Improved Scalability for Large Sample Sizes via Dual Program]
\label{prop:dualization}
% In the setting of Theorem~\ref{thrm:main}, for each $x\in\mathbb{R}^d$ we have
In Setting~\ref{TheSetting}, let $s\ge2$, $D=2^{46d}$, and $\mu\in\Delta_{N:D}$. If $G$ denotes the $\vartheta=1$ right-hand side of~\eqref{eq:entropy-variance-extension__primal} and $F$ denotes the right-hand side of~\eqref{eq:entropy-variance-extension}, then, for every $x\in\mathbb{R}^d$,
\[
    G(x)=F(x).
\]
\end{proposition}

\begin{proof}[{Proof of Proposition~\ref{prop:dualization}}]
Fix $x\in\mathbb{R}^d\setminus\mathbb{X}^N$ and abbreviate
\[
    a_i \eqdef a_i(x),\qquad
    b_i \eqdef b_i(x),\qquad
    c \eqdef c(x),\qquad
    \tau \eqdef \tau(x),\qquad
    \theta_i \eqdef \theta_i(x),\qquad
    v_i \eqdef v_i(x).
\]
% Since $s\ge2$ in the setting of Theorem~\ref{thrm:main}, one has
Since $s\ge2$ by the hypothesis of this proposition, one has
$\beta=s-1>0$. Moreover $M>0$, $R(x)>0$, and $\theta_i>0$ for every
$i\in[N]$. Consequently
\[
    c=\frac{s}{4MR(x)^{s-1}}>0,
    \qquad
    \tau=\beta MR(x)^{s+1}>0.
\]
For a function $\Xi:\mathbb{R}^k\to(-\infty,+\infty]$, where $k\in\mathbb N_+$, define its (convex) conjugate by 
\[
\Xi^*(u)\eqdef\sup_{z\in\mathbb{R}^k} \{\langle u,z\rangle-\Xi(z)\}.
\]We compute the conjugates below and prove the equality of the optimal values by an explicit nonnegative gap. Since $b_i=a_i-c|v_i|^2$, the primal formula for $G$ becomes
\begin{equation}
\label{eq:dualization-primal-rewrite}
    G(x)
    =
    \min_{\lambda\in\Delta_N}
    \left\{
        \sum_{i=1}^N\lambda_i b_i
        +
        c\left|\sum_{i=1}^N\lambda_i v_i\right|^2
        +
        \tau\operatorname{KL}(\lambda\mid\theta)
    \right\}.
\end{equation}
Define the linear map
$
    A:\mathbb{R}^N\to\mathbb{R}^d
$ by $A\lambda \eqdef \sum_{i=1}^N\lambda_i v_i$ 
and define the extended-real-valued functions
\[
    \Phi(\lambda)
     \eqdef 
    \begin{cases}
        \displaystyle
        \sum_{i=1}^N\lambda_i b_i
        +
        \tau\operatorname{KL}(\lambda\mid\theta),
        &\lambda\in\Delta_N,\\[2mm]
        +\infty,
        &\lambda\notin\Delta_N,
    \end{cases}
    \qquad
    \Psi(z) \eqdef c|z|^2.
\]
Then \eqref{eq:dualization-primal-rewrite} can thus be re-written as
\begin{equation}
\label{eq:dualization-composite-primal}
    G(x)
    =
    \min_{\lambda\in\mathbb{R}^N}
    \bigl\{\Phi(\lambda)+\Psi(A\lambda)\bigr\}.
\end{equation}
The function $\Psi$ is finite, continuous and convex on $\mathbb{R}^d$.
The simplex $\Delta_N$ is nonempty, compact and convex, and
$t\mapsto t\log t$, with $0\log0 \eqdef 0$, is convex and continuous on
$[0,\infty)$. Hence $\Phi$ is proper, lower semicontinuous and convex.
The primal minimum is attained, since its objective is continuous on the nonempty compact simplex. We next compute the two conjugates directly. By definition,
$
    \Psi^*(y)
    =
    \sup_{z\in\mathbb{R}^d}
    \,
    \{\langle y,z\rangle-c|z|^2\}$.  
Completing the square,
$
    \langle y,z\rangle-c|z|^2
    =
    -c|z-\tfrac{y}{2c}|^2
    +
    \tfrac{|y|^2}{4c}
$; thus, for each $y\in\mathbb{R}^d$ we have
\begin{equation}
\label{eq:dualization-quadratic-conjugate}
    \Psi^*(y)=\frac{|y|^2}{4c}.
\end{equation}
Now let
$u=(u_i)_{i=1}^N\in\mathbb{R}^N$. Again by the definition of
the convex conjugate,
\[
\begin{aligned}
    \Phi^*(u)
    &=
    \sup_{\lambda\in\Delta_N}
    \left\{
        \sum_{i=1}^N\lambda_i(u_i-b_i)
        -
        \tau\operatorname{KL}(\lambda\mid\theta)
    \right\}.
\end{aligned}
\]
Set
$
    Z(u)
     \eqdef 
    \sum_{i=1}^N
    \theta_i
    \exp\!\left(\frac{u_i-b_i}{\tau}\right)
    >0
$
and
$
    \pi_i(u)
     \eqdef 
    \tfrac{
        \theta_i\exp((u_i-b_i)/\tau)
    }{
        Z(u)
    }
$ for each $i\in[N]$.  
Then $\pi(u)\in\Delta_N^\circ$. Since
$
    \log\tfrac{\pi_i(u)}{\theta_i}
    =
    \tfrac{u_i-b_i}{\tau}
    -
    \log Z(u)$ 
for every $\lambda\in\Delta_N$ we have
\begin{equation}
\label{eq:dualization-entropy-identity}
\begin{aligned}
    \sum_{i=1}^N\lambda_i(u_i-b_i)
    -
    \tau\operatorname{KL}(\lambda\mid\theta)
    &=
    \tau\log Z(u)
    -
    \tau\operatorname{KL}(\lambda\mid\pi(u)).
\end{aligned}
\end{equation}
For completeness,
$\operatorname{KL}(\lambda\mid\pi(u))\ge0$. Indeed,
$\log r\le r-1$ for $r>0$, and hence
\[
\begin{aligned}
    -\operatorname{KL}(\lambda\mid\pi(u))
=
    \sum_{\lambda_i>0}
    \lambda_i
    \log\frac{\pi_i(u)}{\lambda_i}
\le
    \sum_{\lambda_i>0}
    \lambda_i
    \left(
        \frac{\pi_i(u)}{\lambda_i}-1
    \right)
=
    \sum_{\lambda_i>0}\pi_i(u)-1
    \le 0.
\end{aligned}
\]
Equality holds at $\lambda=\pi(u)$. Thus
\begin{equation}
\label{eq:dualization-entropy-conjugate}
    \Phi^*(u)
    =
    \tau
    \log\!\left(
        \sum_{i=1}^N
        \theta_i
        \exp\!\left(
            \frac{u_i-b_i}{\tau}
        \right)
    \right).
\end{equation}
Retain the fixed point $x$ and the coefficients introduced above. Throughout the following argument these coefficients are held fixed; in particular, $c>0$, $\tau>0$, $b_i\in\mathbb{R}$, $v_i\in\mathbb{R}^d$, and $\theta_i>0$ with $\sum_{i=1}^N\theta_i=1$. Recall that
    \[
     A\lambda=\sum_{i=1}^N\lambda_i v_i,\qquad \Psi(z)=c|z|^2,\qquad \Phi(\lambda)=\sum_{i=1}^N\lambda_i b_i+\tau\sum_{i=1}^N\lambda_i\log\frac{\lambda_i}{\theta_i}\quad(\lambda\in\Delta_N),
    \]
    where each summand with $\lambda_i=0$ is interpreted as zero. The spaces $\mathbb{R}^N$ and $\mathbb{R}^d$ are equipped with their Euclidean inner products. Hence, for every $\lambda\in\mathbb{R}^N$ and $y\in\mathbb{R}^d$,
    \[
     \langle A\lambda,y\rangle=\sum_{i=1}^N\lambda_i\langle v_i,y\rangle=\left\langle\lambda,\bigl(\langle y,v_i\rangle\bigr)_{i=1}^N\right\rangle.
    \]
    By the defining property of the adjoint, this proves $(A^*y)_i=\langle y,v_i\rangle$ for every $i\in[N]$.  For $y\in\mathbb{R}^d$, put
    $
     \mathcal Z(y)\eqdef\sum_{i=1}^N\theta_i\exp\!\left(-\frac{b_i+\langle y,v_i\rangle}{\tau}\right)$.
    Every summand is finite and strictly positive. Consequently, $\mathcal Z(y)>0$, and the previously defined Gibbs vector $\pi(-A^*y)$ has coordinates
    \[
     \pi_i(-A^*y)=\frac{\theta_i\exp\!\left(-\frac{b_i+\langle y,v_i\rangle}{\tau}\right)}{\mathcal Z(y)},\qquad i\in[N].
    \]
    In particular, $\pi_i(-A^*y)>0$ for every $i$, and summing the numerators gives $\sum_{i=1}^N\pi_i(-A^*y)=1$. Thus $\pi(-A^*y)\in\Delta_N^\circ$.  Define the dual objective by
    $
     \mathcal D(y)\eqdef-\Phi^*(-A^*y)-\Psi^*(y)$ for $y\in\mathbb{R}^d$.  
    The computed formula for $\Phi^*$ and \eqref{eq:dualization-quadratic-conjugate} give
    \[
     \Phi^*(-A^*y)=\tau\log\mathcal Z(y),\qquad \mathcal D(y)=-\tau\log\mathcal Z(y)-\frac{|y|^2}{4c}.
    \]
    Since $\mathcal Z$ is a finite sum of smooth, strictly positive functions, $\mathcal D$ is finite and smooth on $\mathbb{R}^d$.
    
    We next establish the gap identity, including its validity on the boundary of the simplex. First, the non-negativity of relative entropy can be seen directly from the function
    $
     \varphi(t)\eqdef t\log t-t+1$ for each $t>0$ and $\varphi(0)\eqdef1$; and not that the extension at zero is continuous (since $\lim_{t\downarrow0}t\log t=0$). 
     For $t>0$, one has $\varphi'(t)=\log t$; thus $\varphi$ decreases on $(0,1)$ and increases on $(1,\infty)$. It follows that $\varphi(t)\ge0$ for every $t\ge0$, with equality precisely when $t=1$. For $\lambda\in\Delta_N$, the positivity of the Gibbs coordinates therefore gives
    \[
     \operatorname{KL}(\lambda\mid\pi(-A^*y))=\sum_{i=1}^N\pi_i(-A^*y)\varphi\!\left(\frac{\lambda_i}{\pi_i(-A^*y)}\right)\ge0.
    \]
    Indeed, expanding the sum on the right gives the relative entropy plus $-\sum_i\lambda_i+\sum_i\pi_i(-A^*y)$, and these last two terms cancel. This calculation uses the convention $0\log(0/a)=0$ for $a>0$, so it includes zero coordinates of $\lambda$. Moreover, equality holds if and only if $\lambda=\pi(-A^*y)$.  Taking logarithms gives
    $
     \log\pi_i(-A^*y)=\log\theta_i-\frac{b_i+\langle y,v_i\rangle}{\tau}-\log\mathcal Z(y)$. 
    For each index with $\lambda_i>0$, substitute this expression into $\lambda_i\log(\lambda_i/\pi_i(-A^*y))$. Indices with $\lambda_i=0$ contribute zero throughout. Summing and using $\sum_i\lambda_i=1$, we obtain
    \[
     \begin{aligned}
     \tau\operatorname{KL}(\lambda\mid\pi(-A^*y))
     &=\tau\sum_{i=1}^N\lambda_i\log\frac{\lambda_i}{\theta_i}+\sum_{i=1}^N\lambda_i\bigl(b_i+\langle y,v_i\rangle\bigr)+\tau\log\mathcal Z(y)\\
     &=\Phi(\lambda)+\langle y,A\lambda\rangle+\Phi^*(-A^*y).
     \end{aligned}
    \]
    This is \eqref{eq:dualization-entropy-identity} evaluated at $u=-A^*y$. Rearranging it and using $\Psi(A\lambda)=c|A\lambda|^2$ yields
    \[
     \begin{aligned}
     \Phi(\lambda)+\Psi(A\lambda)-\mathcal D(y)
     &=\Phi(\lambda)+c|A\lambda|^2+\Phi^*(-A^*y)+\frac{|y|^2}{4c}\\
     &=\tau\operatorname{KL}(\lambda\mid\pi(-A^*y))+c|A\lambda|^2-\langle y,A\lambda\rangle+\frac{|y|^2}{4c}.
     \end{aligned}
    \]
    Completing the square in the last three terms proves that, for every $\lambda\in\Delta_N$ and $y\in\mathbb{R}^d$,
    \begin{equation}
    \label{eq:dualization-direct-gap}
     \Phi(\lambda)+\Psi(A\lambda)-\mathcal D(y)=c\left|A\lambda-\frac{y}{2c}\right|^2+\tau\operatorname{KL}(\lambda\mid\pi(-A^*y))\ge0.
    \end{equation}
    
    The restriction of $\Phi+\Psi\circ A$ to $\Delta_N$ is continuous: the only boundary terms requiring attention are $\lambda_i\log\lambda_i$, and these extend continuously by zero at $\lambda_i=0$. Since $\Delta_N$ is nonempty and compact, its minimum is finite and attained. By \eqref{eq:dualization-composite-primal}, this minimum equals $G(x)$. Taking the minimum over $\lambda\in\Delta_N$ in \eqref{eq:dualization-direct-gap} therefore gives
    \[
     \mathcal D(y)\le\min_{\lambda\in\Delta_N}\{\Phi(\lambda)+\Psi(A\lambda)\}=G(x)
    \]
     for each $y\in\mathbb{R}^d$.  
    We now prove that the dual supremum is attained. Since $N\ge1$, the first summand in $\mathcal Z(y)$ gives
    $
     \mathcal Z(y)\ge\theta_1\,e^{-\frac{b_1+\langle y,v_1\rangle}{\tau}}.
    $
    The logarithm is increasing and $\tau>0$. Hence
    \[
     \begin{aligned}
     \mathcal D(y)\le-\frac{|y|^2}{4c}+b_1+\langle y,v_1\rangle-\tau\log\theta_1\le-\frac{|y|^2}{4c}+|v_1||y|+b_1-\tau\log\theta_1.
     \end{aligned}
    \]
    Because $c>0$, the quadratic term dominates the linear term as $|y|\to\infty$. Consequently, $\mathcal D(y)\to-\infty$ as $|y|\to\infty$. In particular, there exists $R>0$ such that
    if $
     |y|\ge R$ then $\mathcal D(y)<\mathcal D(0)-1$.  
    The continuous function $\mathcal D$ attains its maximum on the compact ball $\overline B(0,R)$ at some point $y_*$. Since $0\in\overline B(0,R)$, one has $\mathcal D(y_*)\ge\mathcal D(0)$. Every point outside that ball has strictly smaller value. Thus $y_*$ is a global maximizer of $\mathcal D$ on $\mathbb{R}^d$.
    
    It remains to show that this maximizer yields equality in \eqref{eq:dualization-direct-gap}. Term-wise differentiation is legitimate because the defining sum for $\mathcal Z$ is finite, and gives
    $
     \nabla\mathcal Z(y)=-\frac1\tau\sum_{i=1}^N\theta_i\,e^{-\frac{b_i+\langle y,v_i\rangle}{\tau}}v_i$.   Since $\mathcal Z(y)>0$, the chain rule now yields
    \[
     \begin{aligned}
     \nabla\mathcal D(y)=-\frac{y}{2c}-\tau\frac{\nabla\mathcal Z(y)}{\mathcal Z(y)}=-\frac{y}{2c}+\sum_{i=1}^N\frac{\theta_i\exp\!\left(-\frac{b_i+\langle y,v_i\rangle}{\tau}\right)}{\mathcal Z(y)}\,v_i=-\frac{y}{2c}+A\pi(-A^*y).
     \end{aligned}
    \]
    For each $h\in\mathbb{R}^d$, the differentiable function $t\mapsto\mathcal D(y_*+th)$ has a global maximum at $t=0$. Its derivative there is therefore zero, so $\langle\nabla\mathcal D(y_*),h\rangle=0$ for every $h$. Consequently $\nabla\mathcal D(y_*)=0$.
    Set $\lambda_*\eqdef\pi(-A^*y_*)$. The positivity and normalization already proved give $\lambda_*\in\Delta_N^\circ$, and the gradient identity implies
    $
     A\lambda_*=\tfrac{y_*}{2c}$.  
    Thus the squared term in \eqref{eq:dualization-direct-gap} vanishes at $(\lambda_*,y_*)$. The relative-entropy term also vanishes, since $\lambda_*=\pi(-A^*y_*)$. It follows that
    \[
    \Phi(\lambda_*)+\Psi(A\lambda_*)=\mathcal D(y_*)
    .
    \]
    Combining this equality with the previously established bound $\mathcal D(y_*)\le G(x)$ and the admissibility of $\lambda_*$ gives
    \[
     \mathcal D(y_*)\le G(x)\le\Phi(\lambda_*)+\Psi(A\lambda_*)=\mathcal D(y_*).
    \]
    Every inequality is therefore an equality. In particular, $\lambda_*$ attains the primal minimum, $y_*$ attains the dual maximum, and their values coincide:
    \begin{equation}
    \label{eq:dualization-fenchel-rockafellar}
     G(x)=\max_{y\in\mathbb{R}^d}\{-\Phi^*(-A^*y)-\Psi^*(y)\}.
    \end{equation}
Substituting $u=-A^*y$ into
\eqref{eq:dualization-entropy-conjugate} and using
\eqref{eq:dualization-quadratic-conjugate} therefore yields
\[
\begin{aligned}
    G(x)
    &=
    \max_{y\in\mathbb{R}^d}
    \Biggl\{
        -\frac{|y|^2}{4c}
        -
        \tau
        \log\!\Biggl(
            \sum_{i=1}^N
            \theta_i
            \exp\!\left(
                -\frac{
                    b_i+\langle y,v_i\rangle
                }{\tau}
            \right)
        \Biggr)
    \Biggr\}.
\end{aligned}
\]
Restoring the $x$-dependence gives
\[
\begin{aligned}
    G(x)
=
    \max_{y\in\mathbb{R}^d}
    \Biggl\{
        -\frac{|y|^2}{4c(x)}
        -
        \tau(x)
        \log\!\Biggl(
            \sum_{i=1}^N
            \theta_i(x)
            \exp\!\left(
                -\frac{
                    b_i(x)+\langle y,v_i(x)\rangle
                }{\tau(x)}
            \right)
        \Biggr)
    \Biggr\}
=
    F(x)
,
\end{aligned}
\]
where the last equality is exactly
\eqref{eq:entropy-variance-extension}.
Finally, if $x=X_n$ for some $n\in[N]$, the prescribed on-site
definitions give $
    G(X_n)=P_n(X_n)=F(X_n)$.  
Thus $G(x)=F(x)$ for every $x\in\mathbb{R}^d$.
\end{proof}

\begin{lemma}[real-analyticity Away from the Measurement Sites]
\label{lem:offsite-analyticity}
Under the assumptions of Theorem~\ref{thrm:main}, suppose that $s\ge2$. Then $F$ is real-analytic on $\R^d\setminus\mathbb X^N$
\end{lemma}
\begin{proof}[{Proof of Lemma~\ref{lem:offsite-analyticity}}]
Put $U\eqdef\R^d\setminus\mathbb X^N$ and retain $r_i(x)=|x-X_i|$. For every $t\in\R$ and each $x\in U$
\[
 r_i(x)^t=\exp\!\left(\frac t2\log\!\left(\sum_{k=1}^d(x_k-X_{i,k})^2\right)\right)
 .
\]
Consequently all the powers of $r_i$ used in the reconstruction are real-analytic on $U$. Since $\mu_i>0$, one has
\[
 \sum_{j=1}^N\mu_jr_j^{-q}>0,\qquad \theta_i>0,\qquad \sum_{i=1}^N\theta_i=1,\qquad R^p=\sum_{i=1}^N\theta_i r_i^p>0.
\]
The formulas defining $\theta_i$, the positive root $R$, and $v_i$ therefore give real-analytic functions on $U$; here each $P_i$ is a polynomial. Retain $a_i,b_i,c,\tau$ from~\eqref{eq:gibbs-dual-quantities}. They are real-analytic on $U$, and $c,\tau>0$ because $M,R>0$ and $\beta=s-1\ge1$.
\hfill\\
\noindent
For $x\in U$ and $\lambda\in\Delta_N$, set
\begin{equation}
\label{eq:offsite-analytic-objective}
 \begin{aligned}
 m_\lambda(x)&\eqdef\sum_{i=1}^N\lambda_i v_i(x),\\
 \mathcal J(x,\lambda)&\eqdef\sum_{i=1}^N\lambda_i b_i(x)+c(x)|m_\lambda(x)|^2+\tau(x)\KL(\lambda\mid\theta(x)).
 \end{aligned}
\end{equation}
Specializing~\eqref{eq:entropy-variance-extension__primal} to $\vartheta=1$ and expanding the variance, then using $b_i=a_i-c|v_i|^2$, gives $F(x)=\min_{\lambda\in\Delta_N}\mathcal J(x,\lambda)$. For each fixed $x$, apply Lemma~\ref{lem:three-entropy} to $\mathcal J(x,\lambda)/\tau(x)$, with scalar costs $b_i(x)/\tau(x)$, vectors $v_i(x)$, weights $\theta_i(x)$, and quadratic coefficient $c(x)/\tau(x)>0$. Its existence, boundary-variation, and strict-convexity arguments give a unique minimizer $\lambda^*(x)\in\Delta_N^\circ$. Write $z_*(x)\eqdef m_{\lambda^*(x)}(x)$. The Gibbs identity in that lemma yields
\begin{equation}
\label{eq:offsite-analytic-gibbs}
 \begin{aligned}
 Z(x,z)&\eqdef\sum_{i=1}^N\theta_i(x)\exp\!\left(-\frac{b_i(x)+2c(x)\langle z,v_i(x)\rangle}{\tau(x)}\right)>0,\\
 \pi_i(x,z)&\eqdef\frac{\theta_i(x)\exp\!\left(-\frac{b_i(x)+2c(x)\langle z,v_i(x)\rangle}{\tau(x)}\right)}{Z(x,z)},\\
 \bar v(x,z)&\eqdef\sum_{i=1}^N\pi_i(x,z)v_i(x),\qquad H(x,z)\eqdef z-\bar v(x,z),\\
 \lambda_i^*(x)&=\pi_i(x,z_*(x)),\qquad H(x,z_*(x))=0.
 \end{aligned}
\end{equation}
The functions $Z$, $\pi_i$, $\bar v$, and $H$ are jointly real-analytic on $U\times\R^d$, and $\pi(x,z)\in\Delta_N^\circ$.
Now, every zero of $H(x,\cdot)$ determines the actual minimizer. Indeed, substitution of
\[
 \log\frac{\pi_i(x,z)}{\theta_i(x)}=-\frac{b_i(x)+2c(x)\langle z,v_i(x)\rangle}{\tau(x)}-\log Z(x,z)
\]
into~\eqref{eq:offsite-analytic-objective} gives, for every $\lambda\in\Delta_N$, we have
\begin{equation}
\label{eq:offsite-analytic-gap}
 \mathcal J(x,\lambda)=-c(x)|z|^2-\tau(x)\log Z(x,z)+c(x)|m_\lambda(x)-z|^2+\tau(x)\KL(\lambda\mid\pi(x,z)).
\end{equation}
This identity includes zero coordinates of $\lambda$, with $0\log(0/b)=0$ for $b>0$. If $H(x,z)=0$, then $m_{\pi(x,z)}(x)=z$, and hence
\[
 \mathcal J(x,\lambda)-\mathcal J(x,\pi(x,z))=c(x)|m_\lambda(x)-z|^2+\tau(x)\KL(\lambda\mid\pi(x,z))\ge0.
\]
For the last inequality, $u\log(u/v)\ge u-v$ for $u\ge0$, $v>0$; summing over two probability vectors gives non-negativity of relative entropy. Uniqueness of the minimizer therefore proves that: if $H(x,z)=0$ then $\pi(x,z)=\lambda^*(x)$, $ z=z_*(x)$.
Differentiating the finite sums with respect to $z$, for $u\in\R^d$, gives
\[
 \begin{aligned}
 D_z\log Z(x,z)[u]&=-\frac{2c(x)}{\tau(x)}\langle\bar v(x,z),u\rangle,\\
 D_z\pi_i(x,z)[u]&=-\frac{2c(x)}{\tau(x)}\pi_i(x,z)\langle v_i(x)-\bar v(x,z),u\rangle,\\
 D_zH(x,z)&=I+\frac{2c(x)}{\tau(x)}\sum_{i=1}^N\pi_i(x,z)(v_i(x)-\bar v(x,z))\otimes(v_i(x)-\bar v(x,z)).
 \end{aligned}
\]
Here $(a\otimes b)u\eqdef\langle b,u\rangle a$. Thus
\begin{equation}
\label{eq:offsite-analytic-jacobian}
 \langle D_zH(x,z)u,u\rangle=|u|^2+\frac{2c(x)}{\tau(x)}\sum_{i=1}^N\pi_i(x,z)\langle v_i(x)-\bar v(x,z),u\rangle^2\ge|u|^2.
\end{equation}
Consequently $D_zH(x,z)$ has trivial kernel and is invertible.

Fix $x_0\in U$, put $z_0=z_*(x_0)$, and let $B=D_zH(x_0,z_0)$. We give the analytic implicit-function step directly. The real-analytic map
\[
 T(\xi,w)\eqdef w-B^{-1}H(x_0+\xi,z_0+w)
\]
satisfies $T(0,0)=0$ and $D_wT(0,0)=0$. Write its power series as $T(\xi,w)=\sum_{\alpha,\eta\in\mathbb N^d}t_{\alpha\eta}\xi^\alpha w^\eta$, with $t_{\alpha\eta}\in\R^d$, and choose convergence radii $a_0,b_0>0$ such that
\[
 \sum_{\alpha,\eta}|t_{\alpha\eta}|_\infty a_0^{|\alpha|}b_0^{|\eta|}<\infty.
\]
Here $|w|_\infty\eqdef\max_j|w_j|$, $t_{00}=0$, and $t_{0,e_j}=0$ for each coordinate vector $e_j$. Choose $0<b<b_0/2$ and then $0<a<a_0$ such that $x_0+[-a,a]^d\subset U$ and
\begin{equation}
\label{eq:offsite-analytic-radii}
 A(a)\eqdef\sum_\alpha|t_{\alpha0}|_\infty a^{|\alpha|}\le\frac b2,\qquad Q(a,b)\eqdef\sum_{\alpha,\,|\eta|\ge1}|\eta|\,|t_{\alpha\eta}|_\infty a^{|\alpha|}b^{|\eta|-1}\le\frac12.
\end{equation}
Indeed, $|\eta|(b/b_0)^{|\eta|-1}\le1$ when $|\eta|\ge1$ and $b<b_0/2$, so the series defining $Q(a,b)$ is dominated by $b_0^{-1}\sum_{\alpha,\eta}|t_{\alpha\eta}|_\infty a_0^{|\alpha|}b_0^{|\eta|}$. Its part with $\alpha=0$ tends to zero as $b\downarrow0$, since $t_{0,e_j}=0$; choose $b$ so that this part is at most $1/4$. Holding $b$ fixed and decreasing $a$, the part with $|\alpha|\ge1$ becomes at most $1/4$, while $A(a)\le b/2$ because $t_{00}=0$. These choices may also be made with $a<\rho_{\mathbb X^N}(x_0)/(2\sqrt d)$, which ensures the stated inclusion in $U$.

Let $\mathcal A_a$ be the space of scalar power series $g(\xi)=\sum_\alpha g_\alpha\xi^\alpha$ with $\|g\|_a\eqdef\sum_\alpha|g_\alpha|a^{|\alpha|}<\infty$. It is complete, being a weighted $\ell^1$ space, and coefficient convolution gives $\|gh\|_a\le\|g\|_a\|h\|_a$. On $\mathcal A_a^d$ use $\|g\|_{a,\infty}\eqdef\max_j\|g_j\|_a$. For $\|g\|_{a,\infty},\|h\|_{a,\infty}\le b$, telescoping each monomial gives
\[
 \|g^\eta-h^\eta\|_a\le|\eta|b^{|\eta|-1}\|g-h\|_{a,\infty}\qquad(|\eta|\ge1).
\]
Consequently the substitution $\mathcal T(g)(\xi)\eqdef T(\xi,g(\xi))$ is well defined and satisfies
\[
 \begin{aligned}
 \|\mathcal T(g)-\mathcal T(h)\|_{a,\infty}&\le Q(a,b)\|g-h\|_{a,\infty}\le\tfrac12\|g-h\|_{a,\infty},\\
 \|\mathcal T(g)\|_{a,\infty}&\le A(a)+\tfrac12\|g\|_{a,\infty}\le b.
 \end{aligned}
\]
Starting from $g^{(0)}=0$, the iterates $g^{(k+1)}=\mathcal T(g^{(k)})$ satisfy
\[
 \|g^{(k+1)}-g^{(k)}\|_{a,\infty}\le2^{-k}\frac b2.
\]
They therefore converge in $\mathcal A_a^d$ to a fixed point $g$. Since $T(0,0)=0$, every iterate vanishes at zero, and hence $g(0)=0$. Every scalar $h=\sum_\alpha h_\alpha\xi^\alpha\in\mathcal A_a$ is real-analytic on $(-a,a)^d$: for $\xi_0\in(-a,a)^d$ and $0<\rho<a-|\xi_0|_\infty$, its binomial re-expansion about $\xi_0$ is absolutely convergent, bounded by $\sum_\alpha|h_\alpha|(|\xi_0|_\infty+\rho)^{|\alpha|}\le\|h\|_a$. Hence $\zeta(x)\eqdef z_0+g(x-x_0)$ is real-analytic on $V\eqdef x_0+(-a,a)^d$ and satisfies $\zeta(x_0)=z_0$ and $H(x,\zeta(x))=0$. 
Since we knew that: if $H(x,z)=0$ then $\pi(x,z)=\lambda^*(x),\qquad z=z_*(x)$; then we deduce that $\zeta(x)=z_*(x)$ on $V$; no prior continuity of the minimizer is required.
Evaluating~\eqref{eq:offsite-analytic-gap} at $z=z_*(x)$ and $\lambda=\lambda^*(x)$ now gives: for every $x\in V$
\[
 F(x)=-c(x)|z_*(x)|^2-\tau(x)\log Z(x,z_*(x)).
\]
The right-hand side is real-analytic, since all its constituents are real-analytic on $V$ and $Z>0$. Since $x_0\in U$ was arbitrary and $V$ is a neighborhood of $x_0$ contained in $U$, it follows that $F\in C^\omega(U)$.
\end{proof}

\subsection{Proof of Definability in \texorpdfstring{$\mathbb{R}_{\operatorname{exp}}$}{The Real Ordered Field Expanded by the Exponential}}
\label{s:Proofs__ss:Definability}

\begin{lemma}[Definability of the Reconstruction]
\label{lem:reconstruction-definability}
Under Setting~\ref{TheSetting}, let $D=2^{46d}$ and $\mu\in\Delta_{N:D}$. The calibrated reconstruction $F=F_1:\mathbb R^d\to\mathbb R$ defined by~\eqref{eq:entropy-variance-extension__primal___SITES}--\eqref{eq:entropy-variance-extension__specialization} is definable in $\mathbb R_{\operatorname{exp}}$, with the coordinates of the sites $X_i$, the coefficients of $P_i$, $M$, and $\mu_i$ as real parameters. This holds for every fixed $d,N\in\mathbb N_+$ and $s\in\mathbb N$.
\end{lemma}
\begin{proof}
We retain the notation of $\S$\ref{s:TheFormula} and~\eqref{eq:three-aux-geometry}; definability below uses only the parameters in the statement. Since $d,N,s$ are fixed, \eqref{eq:fancy_constants} gives $\beta\in\mathbb N$, $p\in\mathbb N_+$, $q=p+46d+2\in\mathbb N_+$, and $A_s\in\mathbb Q$. Moreover,
\[
\begin{aligned}
 x\notin\mathbb X^N&\quad\mbox{ if and only if } \quad\bigwedge_{i=1}^N\left(\sum_{k=1}^d(x_k-X_{i,k})^2>0\right),\\
 \lambda\in\Delta_N&\quad\mbox{ if and only if } \quad\left(\bigwedge_{i=1}^N\lambda_i\ge0\right)\land\left(\sum_{i=1}^N\lambda_i=1\right).
\end{aligned}
\]
Thus both domains are definable. For $x\notin\mathbb X^N$ and $i\in[N]$, the relations
\[
\begin{aligned}
 r_i(x)=a&\quad\mbox{ if and only if } \quad a>0\ \land\ a^2=\sum_{k=1}^d(x_k-X_{i,k})^2,\\
 b=a^{-1}&\quad\mbox{ if and only if } \quad ab=1\qquad(a>0)
\end{aligned}
\]
give definable distances and positive reciprocals. Hence $S_q$ is definable, and
\[
 S_q(x)>0,\qquad \theta_i(x)>0,\qquad \sum_{i=1}^N\theta_i(x)=1,\qquad R(x)^p=\sum_{i=1}^N\theta_i(x)r_i(x)^p>0.
\]
In particular,
\[
\begin{aligned}
 \theta_i(x)=a&\quad\mbox{ if and only if } \quad aS_q(x)=\mu_i r_i(x)^{-q},\\
 R(x)=a&\quad\mbox{ if and only if } \quad a>0\ \land\ a^p=\sum_{i=1}^N\theta_i(x)r_i(x)^p.
\end{aligned}
\]
These identities make $\theta_i$ and $R$ definable. The components of $\nabla P_i$ are polynomials in $x$ with coefficients obtained from those of $P_i$ by multiplication by integers. Thus $v_i$ is definable as well. Here each integer power is fixed, and compositions are expressed by real quantifiers over intermediate values. When $s=0$, $\nabla P_i=0$ and $v_i(x)=Mr_i(x)^{-1}(x-X_i)$; the reciprocal is defined because $r_i(x)>0$.

For $a>0$, $\log a=b$ is equivalent to $\exp(b)=a$. With the entropy convention in $\S$\ref{s:Notation}, the graph of its scalar summand is therefore
\[
 \left\{(t,u,z)\in\mathbb R^3:\ u>0\ \land\
 \left[\begin{array}{l}
 (t=0\land z=0)\ \lor{}\\
 \bigl(t>0\land\exists w\in\mathbb R:\ u\exp(w)=t\land z=tw\bigr)
 \end{array}\right]\right\}.
\]
For $u>0$ and $0<t\le1$,
\[
\begin{aligned}
 |t\log(t/u)|
 &\le -t\log t+t|\log u|
 =t\int_t^1\frac{\dd a}{a}+t|\log u|\\
 &\le t\int_t^1 a^{-3/2}\,\dd a+t|\log u|
 =2\sqrt t-2t+t|\log u|
 % \xrightarrow[t\downarrow0]{}0.
\end{aligned}
\]
the right-hand side of which converges to $0$ as $t$ tends to $0$ from the right.
Consequently, $\lambda\mapsto\KL(\lambda\mid\theta(x))$ is finite and continuous on all of $\Delta_N$ for each $x\notin\mathbb X^N$.

Let $\mathcal E_1(x,\lambda)$ denote the objective in~\eqref{eq:entropy-variance-extension__primal} specialized to $\vartheta=1$. By~\eqref{eq:weighted_variance}, the variance is polynomial in $\lambda$ and the coordinates of $v_i$. Since $M,R(x)>0$, the preceding identities imply that $\mathcal E_1:(\mathbb R^d\setminus\mathbb X^N)\times\Delta_N\to\mathbb R$ is definable and  $
 \mathcal E_1(x,\cdot)\in C(\Delta_N)$.  
In particular, at $s=0$ the variance denominator is $4M/R(x)>0$, whereas at $s=1$ it is $4M>0$; the entropy is finite before multiplication by $\beta=0$ in either case.

Fix $x\notin\mathbb X^N$. For any sequence $(\lambda^k)_{k\ge1}\subseteq\Delta_N$, successive dyadic subdivisions of $[0,1]^N$ give nested closed cubes $Q_j$ of side length $2^{-j}$ containing infinitely many sequence terms. Choose $k_j\uparrow\infty$ with $\lambda^{k_j}\in Q_j$. For $\ell\ge j$, $|\lambda^{k_\ell}-\lambda^{k_j}|\le\sqrt N\,2^{-j}$; hence 
\[
 \lim\limits_{j\uparrow \infty} \, \lambda^{k_j} =\bar\lambda,\qquad
 \bar\lambda_i=\lim_{j\to\infty}\lambda_i^{k_j}\ge0,\qquad
 \sum_{i=1}^N\bar\lambda_i=\lim_{j\to\infty}\sum_{i=1}^N\lambda_i^{k_j}=1.
\]
Thus $\bar\lambda\in\Delta_N$. Put $a_x\eqdef\inf_{\lambda\in\Delta_N}\mathcal E_1(x,\lambda)$. Nonemptiness of $\Delta_N$ and finiteness of $\mathcal E_1$ give $a_x<+\infty$. If $a_x=-\infty$, choose $\mathcal E_1(x,\lambda^k)\le-k$ and extract the preceding subsequence. Continuity would give both $\mathcal E_1(x,\lambda^{k_j})\to\mathcal E_1(x,\bar\lambda)\in\mathbb R$ and $\mathcal E_1(x,\lambda^{k_j})\le-k_j\to-\infty$, a contradiction. Hence $a_x\in\mathbb R$. Choose $0\le\mathcal E_1(x,\lambda^k)-a_x\le k^{-1}$ and extract $\lambda^{k_j}\to\lambda^*\in\Delta_N$. Then
\[
\begin{aligned}
 0\le\mathcal E_1(x,\lambda^*)-a_x
 &=\lim_{j\to\infty}\bigl(\mathcal E_1(x,\lambda^{k_j})-a_x\bigr)
 \le\lim_{j\to\infty}k_j^{-1}=0,\\
 F(x)&=\mathcal E_1(x,\lambda^*)=a_x.
\end{aligned}
\]

Let $G_{\mathcal E_1}(x,\lambda,t)$ be the definable graph relation for $\mathcal E_1$ obtained above. For $(x,t)\in\mathbb R^d\times\mathbb R$,
\begin{equation}
\label{eq:reconstruction-definable-offsite-graph}
\begin{aligned}
 &(x,t)\in\operatorname{graph}\!\left(F|_{\mathbb R^d\setminus\mathbb X^N}\right)\\
 &\quad\mbox{ if and only if } \quad x\notin\mathbb X^N\ \land\ \exists\lambda\in\mathbb R^N\,\Bigl[
 \lambda\in\Delta_N\land G_{\mathcal E_1}(x,\lambda,t)\\
 &\qquad\land\ \forall(\eta,z)\in\mathbb R^N\times\mathbb R,\
 \bigl([\eta\in\Delta_N\land G_{\mathcal E_1}(x,\eta,z)]\Longrightarrow z\ge t\bigr)\Bigr].
\end{aligned}
\end{equation}
For the forward implication, take $\lambda=\lambda^*$. Conversely, the quantified conditions give $\mathcal E_1(x,\lambda)=t\le\mathcal E_1(x,\eta)$ for every $\eta\in\Delta_N$, hence $t=F(x)$. The simplex constraints and $G_{\mathcal E_1}$ are finite first-order formulas; hence so is the right-hand side of~\eqref{eq:reconstruction-definable-offsite-graph}. Finally,~\eqref{eq:entropy-variance-extension__primal___SITES} gives
\begin{equation}
\label{eq:reconstruction-definable-global-graph}
 \operatorname{graph}(F)=\operatorname{graph}\!\left(F|_{\mathbb R^d\setminus\mathbb X^N}\right)
 \cup\bigcup_{i=1}^N\{(X_i,P_i(X_i))\}.
\end{equation}
Each added singleton has defining formula $\bigwedge_{k=1}^d(x_k=X_{i,k})\land(t=P_i(X_i))$, using the same parameters. Finite disjunction proves global definability. The argument also applies to $\Delta_1=\{(1)\}$, coincident polynomial jets at distinct sites, and vanishing coordinates of a minimizer; it uses neither uniqueness of $\lambda^*$ nor definability of $f$.
\end{proof}

\subsection{Proof of Corollary~\ref{cor:full-Cs1-control}}
\label{s:Proofs__cor:full-Cs1-control}
\begin{proof}[{Proof of Corollary~\ref{cor:full-Cs1-control}}]
Fix $i\in[N]$. Theorem~\ref{thrm:main}(i)-(ii) gives $J_{X_i}^sF=P_i$ and $[F]_{s,1}\le K_{d,s}M$. Lemma~\ref{lem:three-taylor}, applied on $\mathbb R^d$, and the finite Taylor expansion of $D^rP_i$ at $X_i$ therefore give, for $x\in[0,1]^d$ and $0\le r\le s$,
\begin{equation}
\label{eq:full-Cs1-derivative-control}
\begin{aligned}
 \|D^rF(x)\|_{\op}
 &\le \|D^rP_i(x)\|_{\op}
       +\frac{K_{d,s}M}{(s+1-r)!}|x-X_i|^{s+1-r}\\
 &\le \sum_{k=0}^{s-r}\frac{\|D^{r+k}P_i(X_i)\|_{\op}}{k!}|x-X_i|^k
       +\frac{K_{d,s}M}{(s+1-r)!}|x-X_i|^{s+1-r}.
\end{aligned}
\end{equation}
Since $|x-X_i|\le\sqrt d$ and $\Lip_{[0,1]^d}(D^sF)\le[F]_{s,1}$, summing in $r$ and collecting terms with the same derivative order yields
\[
\begin{aligned}
 \|F\|_{C^{s,1}([0,1]^d)}
 &\le \sum_{j=0}^s\|D^jP_i(X_i)\|_{\op}
           \sum_{k=0}^j\frac{d^{k/2}}{k!}
       +K_{d,s}M\sum_{k=0}^{s+1}\frac{d^{k/2}}{k!}\\
 &\le \mathrm e^{\sqrt d}\left(
       \sum_{j=0}^s\|D^jP_i(X_i)\|_{\op}+K_{d,s}M\right).
\end{aligned}
\]
Taking the minimum over $i\in[N]$ proves~\eqref{eq:full-Cs1-control}. For $s=0$, the first estimate is the Lipschitz inequality based at $X_i$; for $r=s$, its remainder is $K_{d,s}M|x-X_i|$. The proof also covers $N=1$ and sites on the boundary of the cube.
\end{proof}

\subsubsection{Proof of Lipschitz Regularity}
\label{s:Proofs__ss:LipschitzRegularity}

We now first bound the Lipschitz constant, as it is still on the same theme as bonding the $C^{s,1}$-norm. 

\begin{lemma}[Gradients and Lipschitz Constants]
\label{lem:three-gradient-lipschitz}
Let $\Omega\subset\R^d$ be nonempty, open and convex and let $u\in C^1(\Omega)$. Then, with both sides allowed to be infinite,
\begin{equation}
\label{eq:three-gradient-lipschitz}
 \Lip_\Omega(u)=\sup_{x\in\Omega}|\nabla u(x)|.
\end{equation}
If $u$ is $C^1$ on an open neighborhood of the compact set $\overline\Omega$, then the quantities in \eqref{eq:three-gradient-lipschitz} are finite.
\end{lemma}
\begin{proof}
If $B \eqdef \sup_\Omega|\nabla u|<\infty$, the segment joining $x,y\in\Omega$ lies in $\Omega$. The one-variable fundamental theorem of calculus applied to $t\mapsto u(x+t(y-x))$ gives
\[
 u(y)-u(x)=\int_0^1\nabla u(x+t(y-x))\cdot(y-x)\,\dd t,
 \qquad |u(y)-u(x)|\le B|y-x|.
\]
Hence $\Lip_\Omega(u)\le B$; this inequality is automatic if $B=\infty$. Conversely, if $L_u \eqdef \Lip_\Omega(u)<\infty$, each $x\in\Omega$ and unit vector $\xi$ satisfy $x+t\xi\in\Omega$ for all sufficiently small $|t|$, because $\Omega$ is open. Therefore
\[
 |\nabla u(x)\cdot\xi|
 =\left|\lim_{t\to0}\frac{u(x+t\xi)-u(x)}{t}\right|\le L_u.
\]
For a nonzero gradient choose its own unit direction; for a zero gradient the norm is zero. Taking the supremum proves $B\le L_u$, and if $L_u=\infty$ this inequality is automatic. This proves the equality, including its infinite cases. Under the last hypothesis the continuous gradient is bounded on the compact closure, which makes $B$ finite; the segment estimate then makes $L_u$ finite.
\end{proof}

\section{Proof of Complexity Theoretic Results}
\label{s:Proofs_Complexity}

We now prove Theorems~\ref{thrm:complexity} and~\ref{thrm:complexity__Assembly}.

\subsection{Proof of Complexity Theoretic Result\texorpdfstring{ Theorem~\ref{thrm:complexity}}{}}
\label{s:Proofs_Complexity__ss:CircuitComplexity}

\begin{proof}[{Proof of Theorem~\ref{thrm:complexity}}]
We exhibit a circuit computing $F_{\operatorname{ext}}$ by first exhibiting a long list of sub-circuits which execute ``sub-routines'' of the overall algorithm.  These are then ``wired together'' into a larger more complicated algorithm performing the computation.  We first compile a circuit computing $F_{\operatorname{ext}}$ on $U\eqdef \mathbb{R}^d\setminus \{X_n\}_{n=1}^N$, which we then ``wire into'' a larger circuit computing the whole extension on all of $\mathbb{R}^d$.

Naturally, parallelism is of course a core component of achieving the logarithmic run-time of our algorithm.  We exhibit illustrations of the more complicated sub-circuits in our construction to streamline the communication of our proof; but we do not illustrate the most standard sub-circuits such as (padded) complete binary trees implementing $2\le m$-fold products.

Every displayed affine map is expanded into constant and binary arithmetic gates. Likewise, squaring on $\mathbb R$ is implemented by three gates computing $0=\mathsf c_0(t)$, $t+0$, and $t(t+0)$, respectively.

\paragraph{Computing: $m$-ary summation, products, and the Euclidean norm:}

We first make some standard observations.  We note that our proof is not geared towards maximal efficiency; but rather ``humanly-comprehensible clarity''.  

For $m\in\mathbb N_+$, put $h_m\eqdef\lceil\log_2m\rceil$ and $\bar m\eqdef2^{h_m}$. Let $(V,E_T)$ be a complete binary tree of height $h_m$, with leaves $v_1<\cdots<v_{\bar m}$ and edges directed towards its root. Define
$$
E\eqdef E_T\cup\{(v_1,v_i):m<i\le\bar m\}
\qquad\mbox{ and }\qquad
g_v\eqdef
\begin{cases}
\operatorname{Id}_{\mathbb R},&v\in\{v_1,\ldots,v_m\},\\
\mathsf c_0,&v\in\{v_i:m<i\le\bar m\},\\
+_2,&v\in V\setminus\{v_1,\ldots,v_{\bar m}\}.
\end{cases}
$$
Consider the $(+_2,\mathsf c_0,\operatorname{Id}_{\mathbb{R}})$-circuit 
$\mathcal C_{+,m}\eqdef(V,E,(g_v)_{v\in V})$, equipped with $\Pi\eqdef I_m$, has exactly $m$ input vertices. Its tree leaves evaluate to $u_i$ for $i\le m$ and to zero otherwise, so it computes $\operatorname{SUM}_m(u)\eqdef\sum_{i=1}^m u_i$ on $\mathbb R^m$. Moreover,
$$
|\mathcal C_{+,m}|=2\bar m-m-1,
\qquad
\operatorname{depth}(\mathcal C_{+,m})=h_m+\mathbf 1_{\{m<\bar m\}},
\qquad
\operatorname{width}(\mathcal C_{+,m})\le\bar m.
$$
Since $m\le\bar m<2m$, size and width are $O(m)$, and depth is $O(\log m)$ for $m\ge2$; for $m=1$, size and depth are zero.
Note that, replacing $+_2$ by $\times_2$ and $\mathsf c_0$ by $\mathsf c_1$ gives a 
$(\times_2,\mathsf c_1,\operatorname{Id}_{\mathbb{R}})$-circuit computing $\operatorname{Prod}_m(u)\eqdef\prod_{i=1}^m u_i$ on $\mathbb R^m$, with the same lifting channel and complexity bounds.

A very similar construction also exhibits a $(\cdot^2,\cdot^{1/2},+_2)$-circuit computing $\|\cdot\|_2$ on all of $\mathbb{R}^m$; with depth $\mathcal{O}(\log(m))$, width $\mathcal{O}(m)$, and size $\mathcal{O}(m)$.  Note that, in particular $|\cdot|$ is computed by a $(\cdot^2,\cdot^{1/2})$-circuit of depth, width, and size $\mathcal{O}(1)$.

\paragraph{Computing point indicators:}
Fix $m\in\mathbb N_+$ and $b\in\mathbb R^m$. Let $\mathcal C_{\times,m}\eqdef(V_\times,E_\times,(g_v^\times)_{v\in V_\times})$ be the product circuit constructed above, with input vertices $v_1,\ldots,v_m$. Adjoin fresh vertices $w_1<\cdots<w_m$ and define
\[
V\eqdef V_\times\cup\{w_1,\ldots,w_m\}
\qquad
\mbox{ and }
\qquad
E\eqdef E_\times\cup\{(w_i,v_i):i\in[m]\}
\]
with gates given by 
\[
g_v\eqdef
\begin{cases}
\operatorname{Id}_{\mathbb R},&v=w_i,\ i\in[m],\\
I_{\cdot=b_i},&v=v_i,\ i\in[m],\\
g_v^\times,&v\in V_\times\setminus\{v_1,\ldots,v_m\}.
\end{cases}
\]
Equip $\mathcal C_{=,b}\eqdef(V,E,(g_v)_{v\in V})$ with $\Pi\eqdef I_m$ on the input vertices $(w_i)_{i=1}^m$. Its original input leaves now evaluate to $I_{\cdot=b_i}(x_i)$; whence for every $x\in\mathbb R^m$ we have
$$
f_{\mathcal C_{=,b},I_m}(x)
=\prod_{i=1}^m I_{\cdot=b_i}(x_i)
=I_{\cdot=b}(x).
$$
Moreover, $\mathcal C_{=,b}$ has size 
$2\bar m-1\in \mathcal{O}(m)$, width at-most $\bar m\in \mathcal{O}(m)$, and depth $
=h_m+1+\mathbf 1_{\{m<\bar m\}} \in \mathcal{O}(\log(m))$.

\paragraph{Computing the polynomials $P_i$:}
For $s\ge2$, put $\Lambda\eqdef\{\alpha\in\mathbb N_0^d:|\alpha|\le s\}$, $B\eqdef|\Lambda|=\binom{d+s}{s}$, and $\bar B\eqdef2^{h_B}$. Write $P_i(x)=\sum_{\alpha\in\Lambda}c_{i,\alpha}x^\alpha$. For each $\alpha\in\Lambda$ with $|\alpha|\ge2$, choose $\beta_\alpha,\gamma_\alpha\in\Lambda$ satisfying $\beta_\alpha+\gamma_\alpha=\alpha$, $|\beta_\alpha|=\lfloor|\alpha|/2\rfloor$, and $|\gamma_\alpha|=\lceil|\alpha|/2\rceil$. Define
$$
V_0\eqdef\{v_\alpha:\alpha\in\Lambda\setminus\{0\}\}
\qquad\mbox{ and }\qquad
E_0\eqdef\bigcup_{\substack{\alpha\in\Lambda\\|\alpha|\ge2}}
\{(v_{\beta_\alpha},v_\alpha),(v_{\gamma_\alpha},v_\alpha)\},
$$
with gates given by
\[
g_{v_\alpha}\eqdef
\begin{cases}
\operatorname{Id}_{\mathbb R},&|\alpha|=1,\\
\cdot^2,&|\alpha|\ge2,\ \beta_\alpha=\gamma_\alpha,\\
\times_2,&|\alpha|\ge2,\ \beta_\alpha\ne\gamma_\alpha.
\end{cases}
\]
Equip the input vertices $v_{e_1}<\cdots<v_{e_d}$ with $\Pi\eqdef I_d$, taking squaring on all of $\mathbb R$. Induction on $|\alpha|$ shows that $v_\alpha$ evaluates to $x^\alpha$ and, after scalar expansion, has depth at most $3h_{|\alpha|}$.
For each $\alpha\in\Lambda$, attach a constant vertex $\kappa_\alpha$ carrying $\mathsf c_{c_{i,\alpha}}$ to $v_{e_1}$. For $\alpha\ne0$, add a $\times_2$-vertex $t_\alpha$ with parents $\kappa_\alpha$ and $v_\alpha$, and set $t_0\eqdef\kappa_0$. Attach the summation subcircuit $\mathcal C_{+,B}$ to $(t_\alpha)_{\alpha\in\Lambda}$, ordering $0$ first. The resulting circuit $\mathcal C_{P_i}$ computes $P_i$ on $\mathbb R^d$\footnote{A similar construction holds for non-constant univariate polynomials.}.
Each of the $B-d-1$ monomial operations uses at most three gates. There are also $B$ constant gates, $B-1$ coefficient products, and $2\bar B-B-1$ summation gates. Consequently, after scalar expansion, $\mathcal C_{P_i}$ has size $\mathcal O(B)$, depth at most $3h_s+h_B+1=\mathcal O(\log B)$, and width at most $d+|\mathcal C_{P_i}|=\mathcal O(B)$.

\begin{figure}[H]
\centering
\resizebox{0.75\linewidth}{!}{%
\begin{tikzpicture}[
  x=1cm,y=1cm,
  >={Stealth[length=1.7mm,width=1.25mm]},
  gate/.style={draw=black!75,rounded corners=2pt,fill=white,
    line width=.55pt,minimum width=.78cm,minimum height=.64cm,
    inner xsep=5pt,inner ysep=4pt,font=\small},
  wire/.style={draw=black!80,line width=.6pt},
  flow/.style={wire,->},
  junction/.style={circle,fill=black!80,inner sep=0pt,
    minimum size=2.2pt},
  terminal/.style={inner sep=2pt,font=\small},
  every node/.style={outer sep=1pt}
]

% Shared input and summation subcircuit.
\node[terminal] (x) at (-.25,0) {$x$};
\node[gate,fill=black!4] (sum) at (8.55,0)
  {$\operatorname{SUM}_N$};
\coordinate (xfan) at (.5,0);
\coordinate (denfan) at (9.8,0);

% First and last site branches.
\foreach \i/\a/\y in {1/first/2.1,N/last/-2.1}{
  \node[gate] (shift-\a) at (1.65,\y) {$\cdot-X_{\i}$};
  \node[gate] (norm-\a) at (3.15,\y) {$|\cdot|$};
  \node[gate] (power-\a) at (4.65,\y) {$(\cdot)^{-q}$};
  \node[gate] (weight-\a) at (6.15,\y) {$\mu_{\i}\,\cdot$};
  \node[gate] (divide-\a) at (11.2,\y) {$\div_2$};
  \node[terminal,anchor=west] (out-\a) at (12.05,\y)
    {$\theta_{\i}(x)$};
  \coordinate (atap-\a) at (7.05,\y);
}

% Supply the shared input to every site branch.
\draw[flow] (x.east) -- (xfan);
\draw[wire] (.5,-2.1) -- (.5,2.1);
\node[junction] at (xfan) {};

\foreach \a/\y in {first/2.1,last/-2.1}{
  \draw[flow] (.5,\y) -- (shift-\a.west);
  \node[junction] at (.5,\y) {};
  \draw[flow] (shift-\a.east) -- (norm-\a.west);
  \draw[flow] (norm-\a.east) -- (power-\a.west);
  \draw[flow] (power-\a.east) -- (weight-\a.west);
  \draw[flow] (weight-\a.east) -- (divide-\a.west);
  \draw[flow] (divide-\a.east) -- (out-\a.west);
  \node[junction] at (atap-\a) {};
}

% Each unnormalized weight also enters the shared summation.
\draw[flow] (atap-first) -- (7.05,.8) -| (sum.north);
\draw[flow] (atap-last) -- (7.05,-.8) -| (sum.south);

\node[terminal] at (8.55,2.43) {$a_1(x)$};
\node[terminal] at (8.55,-2.43) {$a_N(x)$};

% Broadcast the shared denominator to every final division.
\draw[wire] (sum.east) -- (denfan);
\draw[flow] (denfan) -- (9.8,.9) -| (divide-first.south);
\draw[flow] (denfan) -- (9.8,-.9) -| (divide-last.north);
\node[junction] at (denfan) {};
\node[terminal,anchor=south west] at (9.8,0) {$S_q(x)$};

% Intermediate branches are omitted.
\foreach \xx in {1.65,3.15,4.65,6.15,7.05,11.2,12.55}{
  \node at (\xx,0) {$\vdots$};
}

\end{tikzpicture}%
}
\caption{The circuit $\mathcal C_{\theta:N}$ computing $\theta(x)$ for $N\ge2$ and $x\in U$. }
\label{fig:theta_circuit}
\end{figure}
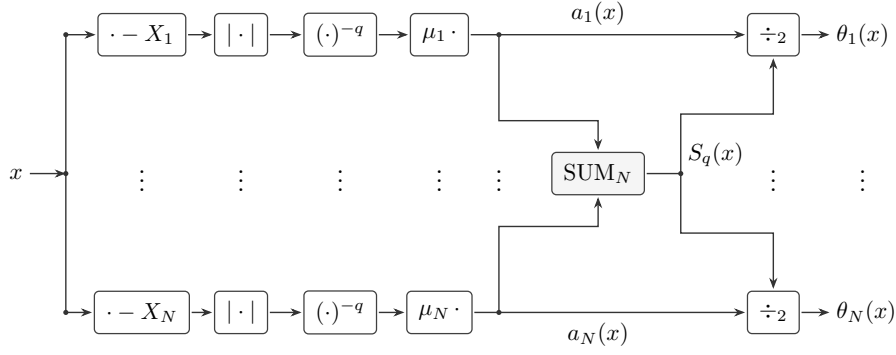

\paragraph{Computing the ``Location'' Factors $\theta$, $R$:}
Next, we construct $
(\cdot+k,\cdot^p,k\cdot,+_2,\div_2,\operatorname{Id}_{\mathbb{R}})_{k\in \mathbb{R},p>0}
$-circuits computing $\theta$ and $R(x)$, cf.~\eqref{eq:entropy-variance-extension__primal__THEVARIABLES}. We first compute $\theta$. 

We now formally compile the circuit sketched in Figure~\ref{fig:theta_circuit}.  
For $N=1$, simply use the constant gate $\mathsf c_1$. 
For $N\ge2$, consider the DAG of subcircuits $\mathcal C_{\theta:N}\eqdef(V,E,(g_v)_{v\in V})$, with input block $v_0$ carrying $x\in\mathbb R^d$ and designated outputs $(v_{6,j})_{j=1}^N$, where
$$
\begin{aligned}
V
&=\{v_0\}
\cup\{v_{i,j}:i\in[4],\,j\in[N]\}
\cup\{v_{5,1}\}
\cup\{v_{6,j}:j\in[N]\},\\
E
&=\{(v_0,v_{1,j}):j\in[N]\}
\cup\{(v_{i,j},v_{i+1,j}):i\in[3],\,j\in[N]\}\\
&\quad\cup\{(v_{4,j},v_{5,1}):j\in[N]\}
\cup\{(v_{4,j},v_{6,j}):j\in[N]\}
\cup\{(v_{5,1},v_{6,j}):j\in[N]\}.
\end{aligned}
$$
Assign the maps
\[
g_v=
\begin{cases}
\operatorname{Id}_{\mathbb R^d},&v=v_0,\\
\cdot-X_j,&v=v_{1,j},\\
|\cdot|,&v=v_{2,j},\\
\cdot^{-q},&v=v_{3,j},\\
\mu_j\cdot,&v=v_{4,j},\\
\operatorname{SUM}_N,&v=v_{5,1},\\
\div_2,&v=v_{6,j}.
\end{cases}
\]
Each displayed operation is understood as its scalar $
(\cdot+k,\cdot^p,k\cdot,+_2,\div_2,\operatorname{Id}_{\mathbb{R}})_{k\in \mathbb{R},p>0}
$-subcircuit; the input block expands into $d$ identity vertices with lifting channel $\Pi=I_d$%
\footnote{Here squaring is understood on all of $\mathbb R$, and $r^{-q}$ is implemented as $1/r^q$ for $r>0$ and we choose numerical labels with $v_{4,j}<v_{5,1}$, so the final division has the prescribed argument order.}%
.
Then, on $U\eqdef\mathbb R^d\setminus\{X_1,\ldots,X_N\}$, the designated outputs satisfy
\[
f_{\mathcal C_{\theta:N}}(x)
=\left(\frac{\mu_j|x-X_j|^{-q}}{\sum_{k=1}^N\mu_k|x-X_k|^{-q}}\right)_{j=1}^N
=\theta(x).
\]
Moreover, the sub-circuits (computing the coordinate-functions of $\theta$) of $\mathcal C_{\theta:N}$ has depth $6$, width $N$, and size $5N+1$. After scalar expansion, its size and width are $\mathcal{O}(dN)$, and its depth is $\mathcal{O}(\log d+\log N)$\footnote{For $N=1$, simply use the constant sub-circuit $\mathsf c_1(t)\eqdef(0t)+1$.}.

A similar construction shows that $R$ can be computed by a $(\cdot+k,\cdot^p,k\cdot,+_2,\div_2,\operatorname{Id}_{\mathbb R})_{k\in\mathbb R,\,p>0}$-subcircuit of depth $\mathcal O(\log d+\log N)$, with width and size $\mathcal O(dN)$.

\begin{figure}[H]%[bh!]%[H]
    \centering
\resizebox{0.75\linewidth}{!}{%
    \begin{tikzpicture}[
      x=1cm,y=1cm,
      >={Stealth[length=1.7mm,width=1.25mm]},
      gate/.style={draw=black!75,rounded corners=2pt,fill=white,
        line width=.55pt,minimum width=.78cm,minimum height=.64cm,
        inner xsep=5pt,inner ysep=4pt,font=\small},
      wire/.style={draw=black!80,line width=.6pt},
      flow/.style={wire,->},
      crossing/.style={preaction={draw=white,line width=2.5pt}},
      junction/.style={circle,fill=black!80,inner sep=0pt,minimum size=2.2pt},
      terminal/.style={inner sep=2pt,font=\small},
      every node/.style={outer sep=1pt}
    ]
    
    % Shared input and the two shared reciprocal branches.
    \node[terminal] (x) at (-.25,0) {$x$};
    \node[gate,fill=black!4] (R) at (1.65,5.5) {$R$};
    \node[gate] (invR) at (3.3,4.7) {$\dfrac{1}{\cdot}$};
    \node[gate] (MR) at (5.75,5.5) {$M(\cdot)^{s+1}$};
    \node[gate] (invMR) at (7.8,5.5) {$\dfrac{1}{\cdot}$};
    \coordinate (xfan) at (.5,0);
    \coordinate (Rfan) at (2.35,5.5);
    \coordinate (RbusTop) at (4.1,4.7);
    \coordinate (MRbusTop) at (8.65,5.5);
    
    % A branch for i=1 and a branch for i=N.
    \foreach \i/\a/\y in {1/first/2,N/last/-3.1}{
      \node[terminal] (u-\a) at (-.25,\y+1.3) {$u_{\i}$};
      \node[gate] (shift-\a) at (1.65,\y) {$\cdot-X_{\i}$};
      \node[gate] (norm-\a) at (3.1,\y) {$|\cdot|$};
      \node[gate] (ratio-\a) at (4.7,\y) {$\times_2$};
      \node[gate] (power-\a) at (6.3,\y) {$(\cdot)^{2s}$};
      \node[gate] (As-\a) at (7.8,\y) {$A_s\,\cdot$};
      \node[gate] (energy-\a) at (7.8,\y+1.3) {$\times_2$};
      \node[gate] (sum-\a) at (9.6,\y) {$+_2$};
      \node[gate] (scale-\a) at (11.1,\y) {$\dfrac{1}{1-s}\,\cdot$};
      \node[gate] (exp-\a) at (12.8,\y) {$\exp$};
      \node[terminal,anchor=west] (out-\a) at (13.75,\y)
        {$E_{\operatorname{exp},\i}(u,x)$};
      \coordinate (Rtap-\a) at (4.1,\y+.75);
      \coordinate (MRtap-\a) at (8.65,\y+1.95);
    }
    
    % Vertical fan-out wires. Only filled dots indicate junctions.
    \draw[wire] (.5,-3.1) -- (.5,5.5);
    \draw[flow] (x.east) -- (xfan);
    \draw[flow] (.5,5.5) -- (R.west);
    \draw[flow] (R.east) -- (MR.west);
    \draw[flow] (Rfan) |- (invR.west);
    \draw[flow] (MR.east) -- (invMR.west);
    \draw[wire] (invR.east) -- (RbusTop) -- (Rtap-last);
    \draw[wire] (invMR.east) -- (MRbusTop) -- (MRtap-last);
    
    \foreach \a/\y in {first/2,last/-3.1}{
      \draw[flow] (.5,\y) -- (shift-\a.west);
      \draw[flow] (Rtap-\a) -| (ratio-\a.north);
      \draw[flow] (MRtap-\a) -| (energy-\a.north);
      \node[junction] at (.5,\y) {};
      \node[junction] at (Rtap-\a) {};
      \node[junction] at (MRtap-\a) {};
    }
    \node[junction] at (xfan) {};
    \node[junction] at (Rfan) {};
    
    % Site computations. White underlays distinguish crossings from junctions.
    \foreach \a in {first,last}{
      \draw[flow] (shift-\a.east) -- (norm-\a.west);
      \draw[flow,crossing] (norm-\a.east) -- (ratio-\a.west);
      \draw[flow] (ratio-\a.east) -- (power-\a.west);
      \draw[flow] (power-\a.east) -- (As-\a.west);
      \draw[flow,crossing] (As-\a.east) -- (sum-\a.west);
      \draw[flow,crossing] (u-\a.east) -- (energy-\a.west);
      \draw[flow,crossing] (energy-\a.east) -| (sum-\a.north);
      \draw[flow] (sum-\a.east) -- (scale-\a.west);
      \draw[flow] (scale-\a.east) -- (exp-\a.west);
      \draw[flow] (exp-\a.east) -- (out-\a.west);
    }
    
    \node at (1.65,-.5) {$\vdots$};
    \node at (6.3,-.5) {$\vdots$};
    \node at (12.8,-.5) {$\vdots$};
    \end{tikzpicture}
}
    \caption{The circuit constructed to compute the ``exponential factor''~\eqref{eq:exp_factor} in~\eqref{eq:ExtenstionFormula_Expression}.
\hfill\\
\textit{Note:} Shared wire segments represent separate directed edges carrying the same value.
    }
    \label{fig:ExpFactor}
\end{figure}
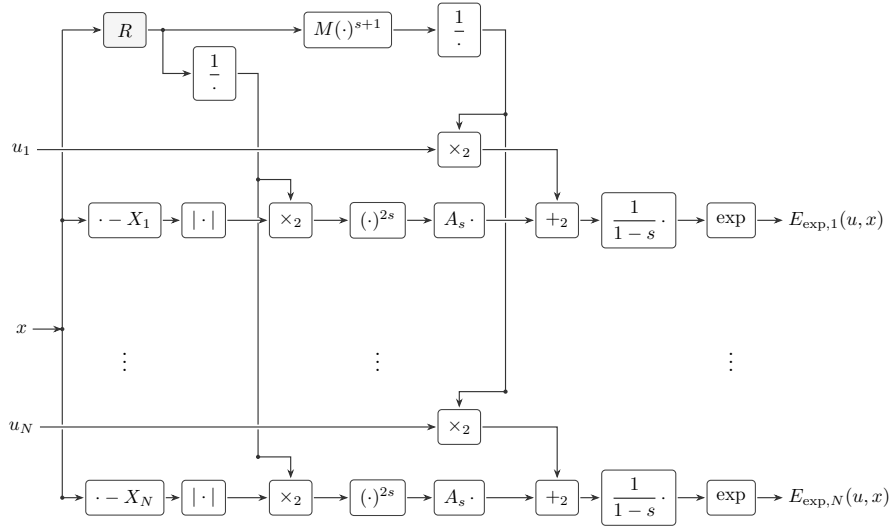

\paragraph{Exponential Factor:}
We now compute the ``exponential factor'' in~\eqref{eq:ExtenstionFormula_Expression}, which we denote by
\begin{equation}
\label{eq:exp_factor}
\begin{aligned}
E_{\operatorname{exp}}:\mathbb R^N\times U&\to\mathbb R^N,\\
E_{\operatorname{exp}}(u_1,\ldots,u_N,x)
&\eqdef
\left(
\exp\left[
\frac{1}{1-s}
\left(
\frac{u_i}{MR(x)^{s+1}}
+
A_s\left(\frac{|x-X_i|}{R(x)}\right)^{2s}
\right)
\right]
\right)_{i=1}^N,
\end{aligned}
\end{equation}
where $U\eqdef\mathbb R^d\setminus\{X_1,\ldots,X_N\}$.

We now formalize the $\mathbb{G}$-circuit in~\eqref{fig:ExpFactor}.  
Consider the DAG of subcircuits $\mathcal C_{\exp:N}\eqdef(V,E,(g_v)_{v\in V})$, with input vertices $w_j$ carrying $u_j$, input block $v_0$ carrying $x$, and designated outputs $(v_{9,j})_{j=1}^N$, where
$$
\begin{aligned}
V
&=\{v_0,r_0,r_1,r_2,r_3\}
\cup\{w_j:j\in[N]\}
\cup\{v_{i,j}:i\in[9],\,j\in[N]\},\\
E
&=\{(v_0,r_0),(r_0,r_1),(r_0,r_2),(r_2,r_3)\}\\
&\quad\cup\{(v_{i,j},v_{i+1,j}):i\in[4],\,j\in[N]\}\\
&\quad\cup\bigcup_{j=1}^N
\bigl\{
(v_0,v_{1,j}),(r_1,v_{3,j}),
(w_j,v_{6,j}),(r_3,v_{6,j}),
(v_{5,j},v_{7,j}),(v_{6,j},v_{7,j}),
(v_{7,j},v_{8,j}),(v_{8,j},v_{9,j})
\bigr\}.
\end{aligned}
$$
Assign the maps
$$
g_v=
\begin{cases}
\operatorname{Id}_{\mathbb R^d},&v=v_0,\\
\operatorname{Id}_{\mathbb R},&v=w_j,\\
R,&v=r_0,\\
1/\cdot,&v\in\{r_1,r_3\},\\
M(\cdot)^{s+1},&v=r_2,\\
\cdot-X_j,&v=v_{1,j},\\
|\cdot|,&v=v_{2,j},\\
\end{cases}
\quad \dots \mbox{ and } \dots \quad
g_v=
\begin{cases}
\times_2,&v\in\{v_{3,j},v_{6,j}\},\\
\cdot^{2s},&v=v_{4,j},\\
A_s\cdot,&v=v_{5,j},\\
+_2,&v=v_{7,j},\\
\frac{1}{1-s}\cdot,&v=v_{8,j},\\
\exp,&v=v_{9,j}.
\end{cases}
$$
The shared vertices $r_1$ and $r_3$ evaluate to $1/R(x)$ and $1/(MR(x)^{s+1})$, respectively. Thus $v_{5,j}$ evaluates to $A_s(|x-X_j|/R(x))^{2s}$, while $v_{6,j}$ evaluates to $u_j/(MR(x)^{s+1})$. The final addition, scaling, and exponential therefore give $f_{\mathcal C_{\exp:N}}=E_{\operatorname{exp}}$ on $\mathbb R^N\times U$. Every division is admissible since $R(x)>0$ on $U$. After expanding the sub-circuits, the depth is $\mathcal O(\log d+\log N)$, and the width and size are $\mathcal O(dN)$.

\paragraph{Compiling the Circuit Computing $F_{\operatorname{ext}}|_U$ on $U$:}
We now combine the previously constructed subcircuits.  The circuit which we compile is summarized in Figure~\ref{fig:F_ext__on_U}.

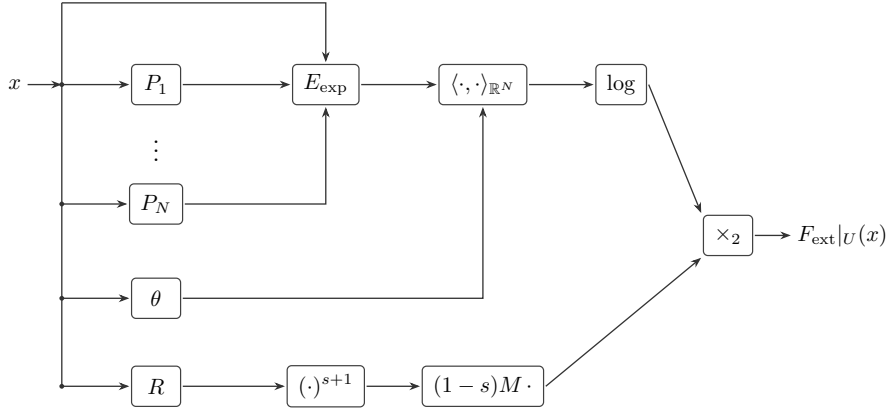
\begin{figure}[H]
    \centering
\resizebox{0.75\linewidth}{!}{%
    \begin{tikzpicture}[
      x=1cm,y=1cm,
      >={Stealth[length=1.7mm,width=1.25mm]},
      gate/.style={draw=black!75,rounded corners=2pt,fill=white,
        line width=.55pt,minimum width=.78cm,minimum height=.64cm,
        inner xsep=5pt,inner ysep=4pt,font=\small},
      wire/.style={draw=black!80,line width=.6pt},
      flow/.style={wire,->},
      junction/.style={circle,fill=black!80,inner sep=0pt,minimum size=2.2pt},
      terminal/.style={inner sep=2pt,font=\small},
      every node/.style={outer sep=1pt}
    ]

    % Input and previously constructed subcircuits.
    \node[terminal] (x) at (-.25,2.4) {$x$};
    \node[gate] (Pfirst) at (2,2.4) {$P_1$};
    \node[gate] (Plast) at (2,.5) {$P_N$};
    \node at (2,1.45) {$\vdots$};
    \node[gate] (Eexp) at (4.7,2.4) {$E_{\operatorname{exp}}$};
    \node[gate] (theta) at (2,-1) {$\theta$};
    \node[gate] (R) at (2,-2.4) {$R$};

    % Inner product, logarithm, and prefactor.
    \node[gate] (inner) at (7.2,2.4)
      {$\langle\cdot,\cdot\rangle_{\mathbb R^N}$};
    \node[gate] (log) at (9.4,2.4) {$\log$};
    \node[gate] (power) at (4.7,-2.4) {$(\cdot)^{s+1}$};
    \node[gate] (scale) at (7.2,-2.4) {$(1-s)M\,\cdot$};
    \node[gate] (prod) at (11.1,0) {$\times_2$};
    \node[terminal,anchor=west] (out) at (12.1,0)
      {$F_{\operatorname{ext}}|_{U}(x)$};

    % Shared input.
    \draw[wire] (.5,-2.4) -- (.5,3.7);
    \draw[flow] (x.east) -- (.5,2.4);
    \foreach \a/\y in {Pfirst/2.4,Plast/.5,theta/-1,R/-2.4}{
      \draw[flow] (.5,\y) -- (\a.west);
      \node[junction] at (.5,\y) {};
    }
    \draw[flow] (.5,3.7) -| (Eexp.north);

    % Polynomial outputs enter the exponential block separately.
    \draw[flow] (Pfirst.east) -- (Eexp.west);
    \draw[flow] (Plast.east) -| (Eexp.south);

    % Weighted exponential sum, followed by the logarithm.
    \draw[flow] (Eexp.east) -- (inner.west);
    \draw[flow] (theta.east) -| (inner.south);
    \draw[flow] (inner.east) -- (log.west);

    % Prefactor and final multiplication.
    \draw[flow] (R.east) -- (power.west);
    \draw[flow] (power.east) -- (scale.west);
    \draw[flow] (log.east) -- (prod.north west);
    \draw[flow] (scale.east) -- (prod.south west);
    \draw[flow] (prod.east) -- (out.west);
    \end{tikzpicture}
}
    \caption{The circuit computing $F_{\operatorname{ext}}|_{U}$ on $U$ from the previously constructed subcircuits.
    }
    \label{fig:F_ext__on_U}
\end{figure}

Consider the DAG of subcircuits $\mathcal C_{\operatorname{ext}}\eqdef(V,E,(g_v)_{v\in V})$, with input block $v_0$ carrying $x$ and designated output $v_{5,1}$, where
$$
\begin{aligned}
V
&=\{v_0,v_{5,1}\}
\cup\{v_{1,j}:j\in[N]\}
\cup\{v_{2,j}:j\in[3]\}
\cup\{v_{3,j},v_{4,j}:j\in[2]\},\\
E
&=\{(v_0,v_{1,j}),(v_{1,j},v_{2,1}):j\in[N]\}
\cup\{(v_0,v_{2,j}):j\in[3]\}\\
&\quad\cup\{(v_{2,1},v_{3,1}),(v_{2,2},v_{3,1}),
(v_{2,3},v_{3,2})\}\\
&\quad\cup\{(v_{3,j},v_{4,j}),(v_{4,j},v_{5,1}):j\in[2]\}.
\end{aligned}
$$
Assign the maps
$$
g_v=
\begin{cases}
\operatorname{Id}_{\mathbb R^d},&v=v_0,\\
P_j,&v=v_{1,j},\\
E_{\operatorname{exp}},&v=v_{2,1},\\
\theta,&v=v_{2,2},\\
R,&v=v_{2,3},\\
\langle\cdot,\cdot\rangle_{\mathbb R^N},&v=v_{3,1},\\
\cdot^{s+1},&v=v_{3,2},\\
\log,&v=v_{4,1},\\
(1-s)M\cdot,&v=v_{4,2},\\
\times_2,&v=v_{5,1}.
\end{cases}
$$
Each displayed map is implemented by its scalar subcircuit, with the inner product using $N$ parallel $\times_2$ gates followed by $\mathcal C_{+,N}$. The inputs of $v_{2,1}$ are ordered as $(P_1(x),\ldots,P_N(x),x)$. The input block expands into $d$ identity vertices with lifting channel $\Pi=I_d$.
Upon writing $P(x)\eqdef(P_1(x),\ldots,P_N(x))$, we obtain
$$
f_{\mathcal C_{\operatorname{ext}},I_d}(x)
=(1-s)MR(x)^{s+1}
\log\left(
\sum_{j=1}^N\theta_j(x)
E_{\operatorname{exp},j}(P(x),x)
\right)
=F_{\operatorname{ext}}(x),
\qquad x\in U.
$$
The logarithm is well-defined since $\theta(x)\in\Delta_N$ and every coordinate of $E_{\operatorname{exp}}(P(x),x)$ is positive. For fixed $d,s$, the expanded circuit has size and width $\mathcal O(N)$, and depth $\mathcal O(\log N)$ for $N\ge2$; for $N=1$, all three are $\mathcal O(1)$.

\paragraph{Compiling the Circuit Computing $F_{\operatorname{ext}}$:}
Put $H\eqdef F_{\operatorname{ext}}|_U$, computed by the preceding circuit, and choose
$
a\eqdef\left(1+\sum_{i=1}^N |(X_i)_1|\right)e_1\in U.
$
Since the sites are distinct,
$
\chi(x)\eqdef1-\sum_{i=1}^N I_{\cdot=X_i}(x)=I_U(x)
$.
Define $S_a(x,t)\eqdef a+t(x-a)$, implemented coordinatewise using affine gates and $\times_2$. Then $S_a(x,\chi(x))=x$ on $U$ and equals $a$ at every site. Consequently, every evaluation of $H$ in the identity
\begin{equation}
\label{eq:basic_idenity}
F_{\operatorname{ext}}(x)
=
\sum_{i=1}^N I_{\cdot=X_i}(x)P_i(x)
+
\chi(x)\,H\bigl(S_a(x,\chi(x))\bigr)
\end{equation}
is well-defined for each $x\in\mathbb R^d$.

\begin{figure}[ht]%[H]
    \centering
% \includegraphics[width=0.5\linewidth]{0_Remove_at_END/AnnieArts/FinalAssembly.jpg}
%%%%    
    \resizebox{0.75\linewidth}{!}{%
    \begin{tikzpicture}[
      x=1cm,y=1cm,
      >={Stealth[length=1.7mm,width=1.25mm]},
      gate/.style={draw=black!75,rounded corners=2pt,fill=white,
        line width=.55pt,minimum width=.78cm,minimum height=.64cm,
        inner xsep=5pt,inner ysep=4pt,font=\small},
      wire/.style={draw=black!80,line width=.6pt},
      flow/.style={wire,->},
      crossing/.style={preaction={draw=white,line width=2.5pt}},
      junction/.style={circle,fill=black!80,inner sep=0pt,minimum size=2.2pt},
      terminal/.style={inner sep=2pt,font=\small},
      every node/.style={outer sep=1pt}
    ]

    % Shared input and the branch using H.
    \node[terminal] (x) at (-.25,1.6) {$x$};
    \node[gate] (maskSum) at (4.6,3.7) {$\operatorname{SUM}_N$};
    \node[gate] (complement) at (6.3,3.7) {$1-\cdot$};
    \node[gate] (safe) at (8,3.7) {$S_a$};
    \node[gate] (H) at (9.8,3.7) {$H$};
    \node[gate] (maskedH) at (11.5,3.7) {$\times_2$};
    \coordinate (chiFan) at (7.15,3.7);

    % First and last site contributions.
    \foreach \i/\a/\y/\tapx in {1/first/1.6/3.05,N/last/-1.6/3.7}{
      \node[gate] (indicator-\a) at (2,\y) {$I_{\cdot=X_{\i}}$};
      \node[gate] (P-\a) at (2,\y-.95) {$P_{\i}$};
      \node[gate] (mul-\a) at (5.2,\y) {$\times_2$};
      \coordinate (tap-\a) at (\tapx,\y);
    }
    \node at (2,-.45) {$\vdots$};
    \node at (5.2,0) {$\vdots$};

    \node[gate] (finalSum) at (13.3,0) {$\operatorname{SUM}_{N+1}$};
    \node[terminal,anchor=west] (out) at (14.6,0)
      {$F_{\operatorname{ext}}(x)$};

    % Shared input: also supply x directly to S_a.
    \draw[wire] (.5,-2.55) -- (.5,5);
    \draw[flow] (x.east) -- (.5,1.6);
    \draw[flow] (.5,5) -| (safe.north);
    \foreach \a/\y in {first/1.6,last/-1.6}{
      \draw[flow] (.5,\y) -- (indicator-\a.west);
      \draw[flow] (.5,\y-.95) -- (P-\a.west);
      \node[junction] at (.5,\y) {};
      \node[junction] at (.5,\y-.95) {};
    }

    % Sum the point indicators, then form their complement.
    \draw[flow] (tap-first) |- (maskSum.west);
    \draw[flow] (tap-last) -- (3.7,2.75) -| (maskSum.south);
    \draw[flow] (maskSum.east) -- (complement.west);
    \draw[flow] (complement.east) -- (safe.west);

    % Evaluate H at S_a(x,chi(x)), then multiply by chi(x).
    \draw[flow] (safe.east) -- (H.west);
    \draw[flow] (H.east) -- (maskedH.west);
    \draw[flow] (chiFan) -- (7.15,2.55)
      -- (11.5,2.55) -- (maskedH.south);
    \node[junction] at (chiFan) {};

    % Site contributions. White gaps indicate unconnected crossings.
    \foreach \a in {first,last}{
      \draw[flow,crossing] (indicator-\a.east) -- (mul-\a.west);
      \draw[flow,crossing] (P-\a.east) -| (mul-\a.south);
      \node[junction] at (tap-\a) {};
    }

    % Final summation.
    \draw[flow] (maskedH.east) -- (finalSum.north);
    \draw[flow] (mul-first.east) -- (finalSum.west);
    \draw[flow] (mul-last.east) -- (finalSum.south west);
    \draw[flow] (finalSum.east) -- (out.west);
    \end{tikzpicture}
}
    \caption{The master circuit computing $F_{\operatorname{ext}}$ on $\mathbb R^d$, where $H=F_{\operatorname{ext}}|_U$, $a\in U$, and $S_a(x,t)=a+t(x-a)$.
\hfill\\
\textit{Note:} Shared wire segments represent separate directed edges carrying the same value.
    }
%%%%%
    \label{fig:F_ext__finished}
\end{figure}
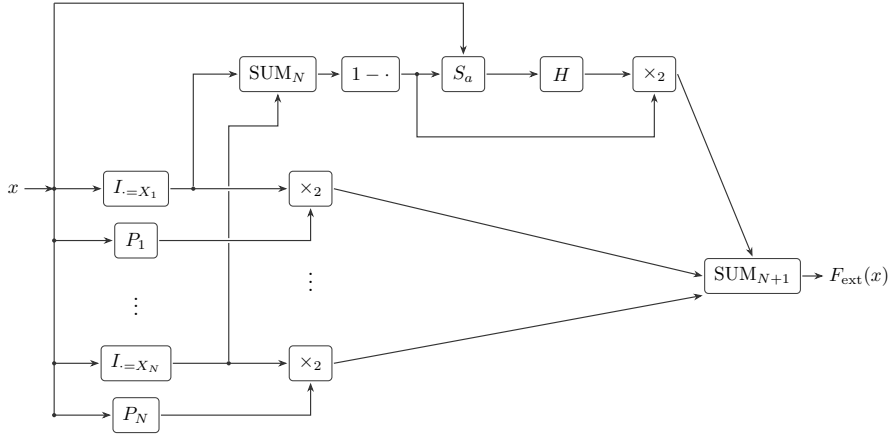

We now implement the ``master'' circuit in Figure~\ref{fig:F_ext__finished} using the identity in~\eqref{eq:basic_idenity} and all our sub-circuits constructed thus far to compute $F_{\operatorname{ext}}$ on all of $\mathbb{R}^d$.
Consider the DAG of subcircuits $\widehat{\mathcal C}_{\operatorname{ext}}\eqdef(V,E,(g_v)_{v\in V})$, with input block $v_0$ carrying $x$ and designated output $r_5$, where
$$
\begin{aligned}
V
&=\{v_0,r_0,\ldots,r_5\}
\cup\{v_{i,j}:i\in[3],\,j\in[N]\},\\
E
&=\bigcup_{j=1}^N
\bigl\{
(v_0,v_{1,j}),(v_0,v_{2,j}),
(v_{1,j},v_{3,j}),(v_{2,j},v_{3,j}),
(v_{1,j},r_0),(v_{3,j},r_5)
\bigr\}\\
&\quad\cup
\bigl\{
(r_0,r_1),(v_0,r_2),(r_1,r_2),(r_2,r_3),
(r_1,r_4),(r_3,r_4),(r_4,r_5)
\bigr\}.
\end{aligned}
$$
Assign the maps
$$
g_v=
\begin{cases}
\operatorname{Id}_{\mathbb R^d},&v=v_0,\\
I_{\cdot=X_j},&v=v_{1,j},\\
P_j,&v=v_{2,j},\\
\times_2,&v\in\{v_{3,j}:j\in[N]\}\cup\{r_4\},\\
\operatorname{SUM}_N,&v=r_0,\\
1-\cdot,&v=r_1,\\
S_a,&v=r_2,\\
H,&v=r_3,\\
\operatorname{SUM}_{N+1},&v=r_5.
\end{cases}
$$
Each displayed map uses its scalar subcircuit, with the inputs of $r_2$ ordered as $(x,\chi(x))$. The input block expands into $d$ identity vertices with lifting channel $\Pi=I_d$. The vertices $v_{3,j}$ and $r_4$ evaluate to $I_{\cdot=X_j}(x)P_j(x)$ and $\chi(x)H(S_a(x,\chi(x)))$, respectively. Hence~\eqref{eq:basic_idenity} gives
$$
f_{\widehat{\mathcal C}_{\operatorname{ext}},I_d}(x)
=F_{\operatorname{ext}}(x),
\qquad x\in\mathbb R^d.
$$
For fixed $d,s$, the expanded circuit has size and width $\mathcal O(N)$, and depth $\mathcal O(\log N)$ for $N\ge2$; for $N=1$, all three are $\mathcal O(1)$.

\paragraph{The Final Tally:}
Put $B_{d,s}\eqdef\binom{d+s}{s}=\mathcal O_s(d^s)$. The master circuit contains $2N$ polynomial subcircuits, each having size $\mathcal O(B_{d,s})$ and depth $\mathcal O(\log B_{d,s})$ after scalar expansion. All remaining subcircuits together contribute $\mathcal O(Nd)$ gates and at most $\mathcal O(\log(d+N))$ further depth. Consequently,
$$
\begin{aligned}
|\widehat{\mathcal C}_{\operatorname{ext}}|
&=\mathcal O(NB_{d,s}+Nd)
=\mathcal O_s(Nd^s),\\
\operatorname{depth}(\widehat{\mathcal C}_{\operatorname{ext}})
&=\mathcal O(\log B_{d,s}+\log(d+N))
=\mathcal O_s(\log(d+N)),\\
\operatorname{width}(\widehat{\mathcal C}_{\operatorname{ext}})
&\le d+|\widehat{\mathcal C}_{\operatorname{ext}}|
=\mathcal O_s(Nd^s).
\end{aligned}
$$
These bounds include the point-indicator subcircuits and the input adjustment $S_a$, which ensure that the realization domain is all of $\mathbb R^d$.
\end{proof}

\subsection{Proof of \texorpdfstring{Theorem~\ref{thrm:complexity__Assembly}}{Efficiency of the circuit compiler}}
\label{s:compiler-proof}

The next proof proceeds in three steps: (i) we construct and verify the circuit computing $F_{\operatorname{ext}}$, following the construction in the proof of Theorem~\ref{thrm:complexity}; (ii) we store each operation, together with its parameters and input indices, in a representation compatible with the exact-real word-RAM model; and (iii) we prove that producing this representation requires $\mathcal O_s(Nd^s)$ one-time work.

\begin{remark}
The key distinction is that Theorem~\ref{thrm:complexity__Assembly} proves that a \textit{single fixed program}, depending only on $d,s$, constructs the stored circuit from the supplied data within this work bound, whereas Theorem~\ref{thrm:complexity} bounds the size, width, and depth of a circuit computing $F_{\operatorname{ext}}$.
\end{remark}

\paragraph{Compiler primitives.}
Fix $d,s$ as in Theorem~\ref{thrm:complexity__Assembly}. Put $\Lambda=\Lambda_{d,s}$, $B=\binom{d+s}{s}$, $h_m=\lceil\log_2m\rceil$, and $\bar m=2^{h_m}$ for $m\ge1$. Order $\Lambda=\{\alpha^{(0)},\ldots,\alpha^{(B-1)}\}$ by nondecreasing degree, with $\alpha^{(0)}=0$ and $\alpha^{(j)}=e_j$ for $j\in[d]$. Fix the splitting ranks
\[
\begin{aligned}
1\le\beta_k,\gamma_k&<k,\qquad \alpha^{(\beta_k)}+\alpha^{(\gamma_k)}=\alpha^{(k)},\\
|\alpha^{(\beta_k)}|&=\lfloor|\alpha^{(k)}|/2\rfloor,\qquad |\alpha^{(\gamma_k)}|=\lceil|\alpha^{(k)}|/2\rceil\qquad(d<k<B).
\end{aligned}
\]
The ranks and coefficient offsets are fixed literals of the program for $d,s$.

Let $\boldsymbol x=(\mathbf x_1,\ldots,\mathbf x_d)$ be the input handles, with $g_{\mathbf x_j}=\operatorname{Id}_{\mathbb R}$ for $j\in[d]$. For existing pairwise distinct handles $u_1,\ldots,u_k$, where $k=\operatorname{ar}(g)$, the operation $v\gets\operatorname{New}(g;u_1,\ldots,u_k)$ appends a fresh record with gate $g$ and ordered operand list $(u_1,\ldots,u_k)$; nested calls are evaluated from left to right. Define
\[
\begin{aligned}
\operatorname{Const}(a)&\eqdef\operatorname{New}(\mathsf c_a;\mathbf x_1),\\
\operatorname{Shift}_a(u)&\eqdef\operatorname{New}(+_2;u,\operatorname{Const}(a)),\\
\operatorname{Sq}(u)&\eqdef\operatorname{New}\bigl(\times_2;u,\operatorname{New}(+_2;u,\operatorname{New}(\mathsf c_0;u))\bigr),\\
\operatorname{Inv}(u)&\eqdef\operatorname{New}(\div_2;\operatorname{Const}(1),u).
\end{aligned}
\]
Write $\operatorname{Inst}(F;\boldsymbol v)$ for the joint scalar-DAG emission specified below, substituting the supplied handles for its formal inputs and returning its output handles. The handle $\mathbf x_1$ in $\operatorname{Const}$ is a fixed external anchor. Each computation vertex is emitted once per call; repeated references use its previously stored handle.

For $\star\in\{+_2,\times_2\}$, put $\varepsilon_{+_2}=0$ and $\varepsilon_{\times_2}=1$. For pairwise distinct $u_1,\ldots,u_m$, set $\rho_\star(u_1)=u_1$. For $m\ge2$, compute $b=\bar m$ by $b\gets1$ and $b\gets2b$ while $b<m$, and emit
\[
\begin{aligned}
T[b+j-1]&\gets u_j &&(1\le j\le m),\\
T[b+j-1]&\gets\operatorname{New}(\mathsf c_{\varepsilon_\star};u_1) &&(m<j\le b),\\
T[j]&\gets\operatorname{New}(\star;T[2j],T[2j+1]) &&(j=b-1,\ldots,1),\\
\rho_\star(u_1,\ldots,u_m)&\gets T[1].
\end{aligned}
\]

For $i\in[N]$, the centered polynomial emitter is
\[
\begin{aligned}
z_j&\gets\operatorname{Shift}_{-X_{i,j}}(v_j),\qquad m_j\gets z_j\qquad(j\in[d]),\\
m_k&\gets\begin{cases}\operatorname{New}(\times_2;m_{\beta_k},m_{\gamma_k}),&\beta_k\ne\gamma_k,\\\operatorname{Sq}(m_{\beta_k}),&\beta_k=\gamma_k,\end{cases} &&(k=d+1,\ldots,B-1),\\
\kappa_k&\gets\operatorname{New}(\mathsf c_{c_{i,\alpha^{(k)}}};z_1) &&(0\le k<B),\\
t_0&\gets\kappa_0,\qquad t_k\gets\operatorname{New}(\times_2;\kappa_k,m_k) &&(1\le k<B),\\
\operatorname{Poly}(i;\boldsymbol v)&\gets\rho_{+_2}(t_0,\ldots,t_{B-1}).
\end{aligned}
\]
The remaining scalar expansions are
\[
\begin{aligned}
\operatorname{Dist}_i(\boldsymbol v)&\eqdef\operatorname{New}\!\left(\mathsf{pow}_{1/2};\rho_{+_2}\bigl((\operatorname{Sq}(\operatorname{Shift}_{-X_{i,j}}(v_j)))_{j=1}^d\bigr)\right),\\
\operatorname{Inst}(I_{\cdot=X_i};\boldsymbol v)&\eqdef\rho_{\times_2}\bigl((\operatorname{New}(I_{\cdot=X_{i,j}};v_j))_{j=1}^d\bigr),\\
\operatorname{Inst}(S_a;\boldsymbol v,t)&\eqdef\bigl(\operatorname{Shift}_{a_j}(\operatorname{New}(\times_2;t,\operatorname{Shift}_{-a_j}(v_j)))\bigr)_{j=1}^d.
\end{aligned}
\]

For $N=1$, set $\operatorname{Inst}(\theta;\boldsymbol v)=\operatorname{New}(\mathsf c_1;v_1)$. For $N\ge2$, emit all $\omega_i$ before $Z$, and then emit all $\eta_i$:
\[
\begin{aligned}
\delta_i&\gets\operatorname{Dist}_i(\boldsymbol v),\qquad p_i\gets\operatorname{New}(\mathsf{pow}_q;\delta_i) &&(i\in[N]),\\
\omega_i&\gets\operatorname{New}(\times_2;\operatorname{Const}(\mu_i),\operatorname{Inv}(p_i)) &&(i\in[N]),\\
Z&\gets\rho_{+_2}(\omega_1,\ldots,\omega_N),\\
\eta_i&\gets\operatorname{New}(\div_2;\omega_i,Z) &&(i\in[N]),\\
\operatorname{Inst}(\theta;\boldsymbol v)&\gets(\eta_1,\ldots,\eta_N).
\end{aligned}
\]
The radius emitter is
\[
\begin{aligned}
(\eta_i)_{i=1}^N&\gets\operatorname{Inst}(\theta;\boldsymbol v),\\
\delta_i&\gets\operatorname{Dist}_i(\boldsymbol v),\qquad \beta_i\gets\operatorname{New}(\mathsf{pow}_{2s};\delta_i) &&(i\in[N]),\\
t_i&\gets\operatorname{New}(\times_2;\eta_i,\beta_i) &&(i\in[N]),\\
T&\gets\rho_{+_2}(t_1,\ldots,t_N),\\
r&\gets\operatorname{New}(\mathsf{pow}_{1/(2s)};T),\qquad\operatorname{Inst}(R;\boldsymbol v)\gets r.
\end{aligned}
\]
For $\operatorname{Inst}(E_{\operatorname{exp}};\boldsymbol u,\boldsymbol v)$, emit
\[
\begin{aligned}
r&\gets\operatorname{Inst}(R;\boldsymbol v),\qquad r^{-1}\gets\operatorname{Inv}(r),\\
p_R&\gets\operatorname{New}(\mathsf{pow}_{s+1};r),\\
p_M&\gets\operatorname{New}(\times_2;\operatorname{Const}(M),p_R),\qquad p_M^{-1}\gets\operatorname{Inv}(p_M),\\
\delta_i&\gets\operatorname{Dist}_i(\boldsymbol v),\qquad a_i\gets\operatorname{New}(\times_2;\delta_i,r^{-1}) &&(i\in[N]),\\
\beta_i&\gets\operatorname{New}(\mathsf{pow}_{2s};a_i),\qquad b_i\gets\operatorname{New}(\times_2;\operatorname{Const}(A_s),\beta_i) &&(i\in[N]),\\
c_i&\gets\operatorname{New}(\times_2;u_i,p_M^{-1}),\qquad d_i\gets\operatorname{New}(+_2;c_i,b_i) &&(i\in[N]),\\
e_i&\gets\operatorname{New}(\exp;\operatorname{New}(\times_2;\operatorname{Const}(1/(1-s)),d_i)) &&(i\in[N]),\\
\operatorname{return}&\ (e_1,\ldots,e_N).
\end{aligned}
\]
Symbols such as $r^{-1}$ in these emitter assignments denote handles. Each invocation uses fresh local records.

\paragraph{Decorated adjacency representation.}
Put $K=|\widetilde V|$. For a bijection $\iota:\widetilde V\to[K]$ with $\iota(\mathbf x_j)=j$, write $\{p_1(v)<\cdots<p_{\operatorname{ar}(v)}(v)\}=\iota(\operatorname{pa}_v)$ and set $p_j(v)=0$ for $j>\operatorname{ar}(v)$. Define
\begin{equation}
\label{eq:compiler-record}
\begin{aligned}
\operatorname{Rec}_\iota(v)&\eqdef\bigl(\operatorname{tag}(v),\operatorname{ar}(v),p_1(v),p_2(v);\lambda(v)\bigr),\\
\lambda(v)&\eqdef\begin{cases}z,&g_v=\mathsf c_z,\\r,&g_v=\mathsf{pow}_r,\\b,&g_v=I_{\cdot=b},\\0,&\text{otherwise},\end{cases}\qquad
\operatorname{ar}(v)=\begin{cases}0,&v\in\widetilde V_{\rm in},\\\operatorname{ar}(g_v),&v\notin\widetilde V_{\rm in}.
\end{cases}
\end{aligned}
\end{equation}
Here $\operatorname{tag}$ identifies the input or gate type. Set $A[\iota(v)]=\operatorname{Rec}_\iota(v)$ and $\operatorname{Adj}(\mathcal C)=(K,\iota(v_{\rm out}),A)$. The two headers and first four record fields are words; $\lambda(v)$ is one real cell. The representation occupies $5K+2$ cells; the fixed lifting channel $I_d$ is not stored. For the instance corresponding to $D\ge1$, supply $q=2s+\lceil\log_2D\rceil+2$.

\begin{algorithm}[t]
\caption{\texttt{Compile-Circuit}; the algorithm in Theorem~\ref{thrm:complexity__Assembly}}
\label{alg:compile-circuit}
\begin{algorithmic}[1]
\Require $N\in\mathbb N_+$ and integer $q\ge2s+2$, $M>0$, pairwise distinct $X_1,\ldots,X_N\in\mathbb R^d$, $\mu\in\Delta_N^\circ$, and $(c_{i,\alpha})_{i\in[N],\alpha\in\Lambda}$.
\Ensure $\operatorname{Adj}(\widehat{\mathcal C}_{\operatorname{ext}})$ with lifting channel $I_d$.
\State $(\widetilde V,\widetilde E) \gets(\{\mathbf x_1,\ldots,\mathbf x_d\},\varnothing)$.
\State Embed $q$ into a real register for the power-gate labels.
\State Compute $a\gets(1+\sum_{i=1}^N|X_{i,1}|)e_1$ using exact-real arithmetic and sign tests.
\For{$i=1,\ldots,N$}
\State $b_i\gets\operatorname{Inst}(I_{\cdot=X_i}; \boldsymbol x)$.
\State $p_i\gets\operatorname{Poly}(i;\boldsymbol x)$.
\State $t_i\gets\operatorname{New}(\times_2;b_i,p_i)$.
\EndFor
\State $r_0\gets\rho_{+_2}(b_1,\ldots,b_N)$.
\State $\chi\gets \operatorname{New}(+_2;\operatorname{Const}(1), \operatorname{New}(\times_2;\operatorname{Const}(-1),r_0))$.
\State $\boldsymbol y\gets \operatorname{Inst}(S_a;\boldsymbol x,\chi)$.
\For{$i=1,\ldots,N$}
\State $u_i\gets\operatorname{Poly}(i;\boldsymbol y)$.
\EndFor
\State $(\eta_1,\ldots,\eta_N)\gets \operatorname{Inst}(\theta;\boldsymbol y)$.
\State $(e_1,\ldots,e_N)\gets \operatorname{Inst}(E_{\operatorname{exp}}; u_1,\ldots,u_N,\boldsymbol y)$.
\State $r\gets\operatorname{Inst}(R;\boldsymbol y)$.
\State $w_i\gets\operatorname{New}(\times_2;\eta_i,e_i)$ for $i\in[N]$.
\State $\ell\gets \operatorname{New}(\log;\rho_{+_2}(w_1,\ldots,w_N))$.
\State $k\gets \operatorname{New}(\times_2;\operatorname{Const}((1-s)M), \operatorname{New}(\mathsf{pow}_{s+1};r))$.
\State $h\gets\operatorname{New}(\times_2;k,\ell)$.
\State $v_{\rm out}\gets \rho_{+_2}(t_1,\ldots,t_N, \operatorname{New}(\times_2;\chi,h))$.
\State $\mathcal V_{\rm c}\gets \{v\in\widetilde V:g_v=\mathsf c_z \text{ for some }z\in\mathbb R\}$.
\State Assign consecutive labels $\iota$ by the stable three-block partition
\Statex \hspace{\algorithmicindent} $\widetilde V_{\rm in},\quad \mathcal V_{\rm c},\quad \widetilde V\setminus (\widetilde V_{\rm in}\cup\mathcal V_{\rm c})$.
\State Write $A[\iota(v)]\gets\operatorname{Rec}_\iota(v)$ for every $v\in\widetilde V$.
\State \Return $(|\widetilde V|,\iota(v_{\rm out}),A)$.
\end{algorithmic}
\end{algorithm}

\begin{proof}[{Proof of Theorem~\ref{thrm:complexity__Assembly}}]
Fix an input satisfying Data Format~\ref{dataformat}. Put $U=\mathbb R^d\setminus\{X_1,\ldots,X_N\}$ and $H=F_{\operatorname{ext}}|_U$. Every constant denoted by $C$ below is absolute and may increase between occurrences. We first verify the emitted scalar graph, then its decorated adjacency representation, and finally the work of its uniform implementation.

\emph{1. Emission and realization.}

Write $\widetilde V_0=\{\mathbf x_1,\ldots,\mathbf x_d\}$, $\widetilde V_t=\widetilde V_0\sqcup\{v_1,\ldots,v_t\}$, $c(\mathbf x_j)=j$, and $c(v_t)=d+t$, where $v_t$ is the $t$th emitted vertex. For an instantiation with distinct supplied handles $(u_j)_{j=1}^m$ and formal computation vertices $(a_k)_{k=1}^L$, the substitution table satisfies
\[
\begin{aligned}
\sigma(\text{port}_j)&=u_j\qquad(j\in[m]),\\
\sigma(a_k)&\notin\widetilde V_{\rm before}\sqcup\{\sigma(a_\ell):\ell<k\}\qquad(k\in[L]),\\
\sigma(a_k)&\ne\sigma(\text{port}_j)\qquad(k\in[L],\ j\in[m]).
\end{aligned}
\]
For $k,\ell\in[L]$, $\sigma(a_k)=\sigma(a_\ell)$ if and only if $k=\ell$; for $i,j\in[m]$, $\sigma(\text{port}_i)=\sigma(\text{port}_j)$ if and only if $i=j$. The supplied tuples belong to $\{\boldsymbol x,\boldsymbol y,(\boldsymbol x,\chi),(u_1,\ldots,u_N,\boldsymbol y)\}$. Freshness in the coordinate and site iterations gives $y_j\ne y_k$ for distinct $j,k\in[d]$ and $u_i\ne u_k$ for distinct $i,k\in[N]$. Moreover,
\[
\begin{aligned}
\max_{j\in[d]}c(\mathbf x_j)&<c(\chi)<\min_{j\in[d]}c(y_j)\le\max_{j\in[d]}c(y_j)<\min_{i\in[N]}c(u_i).
\end{aligned}
\]
In a reduction, its leaf handles are distinct and every internal handle is fresh. Hence $T[j]\ne T[k]$ for distinct allocated positions $j,k$, in particular $T[2j]\ne T[2j+1]$. A repeated monomial operand $u$ is replaced by the three fresh vertices $z=\operatorname{New}(\mathsf c_0;u)$, $u'=\operatorname{New}(+_2;u,z)$, and $\operatorname{New}(\times_2;u,u')$, with $u\ne z$ and $u\ne u'$. Every remaining binary call in the emitter definitions has a fresh-constant pair, a distinct-monomial pair, a division pair treated below, or one of the following pairs of handles from distinct branches:
\[
\begin{gathered}
(\chi,x_j-a_j),\quad (b_i,p_i),\quad (\theta_i,\delta_i^{2s}),\quad (\delta_i,R^{-1}),\quad (u_i,(MR^{s+1})^{-1}),\\
\bigl(A_s(\delta_i/R)^{2s},u_i(MR^{s+1})^{-1}\bigr),\quad (\eta_i,e_i),\quad (k,\ell),\quad(\chi,h),
\end{gathered}
\]
Here each expression denotes its constructor handle; the division pairs are checked in Step~2. Thus
\[
\begin{gathered}
v_t\notin\widetilde V_{t-1},\qquad \operatorname{pa}_{v_t}\subseteq\widetilde V_{t-1},\\
\#\operatorname{pa}_{v_t}=\operatorname{ar}(g_{v_t})\in\{1,2\},\qquad \widetilde V_{\rm in}=\widetilde V_0.
\end{gathered}
\]
If $(u,v)\in\widetilde E$, then $c(u)<c(v)$. A directed cycle $v_0\to v_1\to\cdots\to v_m=v_0$ with $m\ge1$ would therefore give $c(v_0)<c(v_1)<\cdots<c(v_m)=c(v_0)$, a contradiction. Hence $\widetilde{\mathcal D}$ is acyclic.

Use the temporary record bases $b_{\mathbb W}^{\rm tmp},b_{\mathbb R}^{\rm tmp}$ from the layout in Lemma~\ref{lem:compiler-ram-implementation}. For $v\in\widetilde V$ with $k=c(v)$ and ordered parents $(u_{v,j})_{j=1}^{\operatorname{ar}(v)}$, put $\widetilde p_j(v)=c(u_{v,j})$ if $j\le\operatorname{ar}(v)$ and $\widetilde p_j(v)=0$ otherwise, for $j\in\{1,2\}$. Recording $v$ means executing
\[
\begin{aligned}
\mathsf W[b_{\mathbb W}^{\rm tmp}+4(k-1)]&\gets\operatorname{tag}(v),&\mathsf W[b_{\mathbb W}^{\rm tmp}+4(k-1)+1]&\gets\operatorname{ar}(v),\\
\mathsf W[b_{\mathbb W}^{\rm tmp}+4(k-1)+1+j]&\gets\widetilde p_j(v) &&(j\in\{1,2\}),\\
\mathsf R[b_{\mathbb R}^{\rm tmp}+k-1]&\gets\lambda(v).
\end{aligned}
\]
Each record append writes four word fields and one real field. Each address requires a bounded number of word operations. For distinct fixed scratch addresses $a,b$, an indirect word read or write is respectively $\mathsf W[b]\gets\mathsf W[\mathsf W[a]]$ or $\mathsf W[\mathsf W[a]]\gets\mathsf W[b]$; the corresponding real instructions are $\mathsf R[b]\gets\mathsf R[\mathsf W[a]]$ and $\mathsf R[\mathsf W[a]]\gets\mathsf R[b]$. A memory-to-memory copy uses an indirect load followed by an indirect store through a scratch register. The field writes, address computations, and counter update are charged to $W_{\rm emit}$ in Step~3.

For the stored operand tuple $(p_{t,1},\ldots,p_{t,a_t})$ of $v_t$, define
\[
\begin{aligned}
\Omega_0&=\mathbb R^d,\qquad E_{\mathbf x_j}(x)=x_j,\\
\Omega_t&=\{x\in\Omega_{t-1}:(E_{p_{t,1}}(x),\ldots,E_{p_{t,a_t}}(x))\in\operatorname{dom}(g_{v_t})\},\\
E_{v_t}(x)&=g_{v_t}(E_{p_{t,1}}(x),\ldots,E_{p_{t,a_t}}(x))\qquad(x\in\Omega_t),\\
E_{\operatorname{Sq}(u)}&=E_u\cdot(E_u+0)=E_u^2.
\end{aligned}
\]

For a singleton reduction, $E_{\rho_\star(u_1)}=E_{u_1}$. For $m\ge2$, identify heap positions $\bar m,\ldots,2\bar m-1$ with the ordered leaves of the padded reduction in the proof of Theorem~\ref{thrm:complexity}, and position $j<\bar m$ with the parent of positions $2j,2j+1$. The first $m$ leaves receive the supplied handles; every remaining leaf carries $\mathsf c_{\varepsilon_\star}$ with parent $u_1$. The reduction identity established there gives $E_{\rho_\star(u_1,\ldots,u_m)}=\mathop{\star}_{j=1}^m E_{u_j}$ on the common realization domain of the supplied handles.

Induction on $|\alpha^{(k)}|$ gives the centred polynomial identities
\[
\begin{aligned}
E_{m_j}&=E_{v_j}-X_{i,j}\qquad(j\in[d]),\\
E_{m_k}&=E_{m_{\beta_k}}E_{m_{\gamma_k}}=\prod_{j=1}^d(E_{v_j}-X_{i,j})^{\alpha_j^{(k)}}\qquad(d<k<B),\\
E_{\operatorname{Poly}(i;\boldsymbol v)}&=\sum_{k=0}^{B-1}c_{i,\alpha^{(k)}}\prod_{j=1}^d(E_{v_j}-X_{i,j})^{\alpha_j^{(k)}}=P_i((E_{v_j})_{j=1}^d).
\end{aligned}
\]
For every $x\in\mathbb R^d$, the polynomial identities and the indicator products give
\[
\begin{aligned}
b_i(x)&=\prod_{j=1}^dI_{\cdot=X_{i,j}}(x_j)=I_{\cdot=X_i}(x),\qquad p_i(x)=P_i(x),\\
\sum_{i=1}^Nb_i(x)&=\begin{cases}0,&x\in U,\\1,&x\in\{X_1,\ldots,X_N\},\end{cases}\qquad \chi(x)=I_U(x),\\
a_1-X_{i,1}&=1+\sum_{k=1}^N|X_{k,1}|-X_{i,1}\ge1+|X_{i,1}|-X_{i,1}\ge1,\\
y(x)&=a+\chi(x)(x-a)=\begin{cases}x,&x\in U,\\a,&x\in\{X_1,\ldots,X_N\},\end{cases}\qquad y(x)\in U.
\end{aligned}
\]
Fix $x\in\mathbb R^d$ and put $z=y(x)\in U$. For each $i\in[N]$, choose $j_i\in[d]$ such that $z_{j_i}\ne X_{i,j_i}$. Since $q\ge2s+2>0$, $\mu_i>0$, and $s\ge2$, we have
\[
\begin{aligned}
\delta_i^2\eqdef\sum_{j=1}^d(z_j-X_{i,j})^2&\ge(z_{j_i}-X_{i,j_i})^2>0,\qquad \delta_i=|z-X_i|>0,\\
\delta_i^q&>0,\qquad\omega_i\eqdef\mu_i/\delta_i^q>0,\qquad Z\eqdef\sum_i\omega_i\ge\omega_1>0,\\
\theta_i(z)&=\begin{cases}\omega_i/Z,&N\ge2,\\1,&N=1,\end{cases}>0,\qquad\sum_i\theta_i(z)=1,\\
T&\eqdef\sum_{i=1}^N\theta_i(z)\delta_i^{2s}\ge\theta_1(z)\delta_1^{2s}>0,\qquad r=T^{1/(2s)}=R(z)>0,\\
Mr^{s+1}&>0,\qquad r^{-1}>0,\qquad(Mr^{s+1})^{-1}>0,\qquad\delta_i/r>0,\\
1-s&\le-1,\qquad 2s>0,\\
\zeta_i&\eqdef\frac1{1-s}\left(\frac{P_i(z)}{Mr^{s+1}}+A_s\left(\frac{\delta_i}{r}\right)^{2s}\right)\in\mathbb R,\\
\xi_i&\eqdef\exp(\zeta_i)>0,\qquad L_x\eqdef\sum_i\theta_i(z)\xi_i\ge\theta_1(z)\xi_1>0.
\end{aligned}
\]
For $N=1$, $\mu_1=1$ and $\omega_1/Z=1$. The three $\theta$ instances are denoted by $\theta^0,\theta^E,\theta^H$: the direct call, the call within the exponential radius $R^E$, and the call within the final radius $R^H$. All their formal coordinate inputs have value $z$.

The following table lists the partial gates introduced directly by each constructor. Its nested calls contribute the rows for those constructors; $\operatorname{Poly}$, $I_{\cdot=X_i}$, $S_a$, $\rho_\star$, and the remaining master assignments introduce only total gates. The entries give mathematical input values whenever the parents have been evaluated, with $z=y(x)$ and the quantities in the preceding display.
\[
\begin{array}{c|c|l}
\text{constructor}&\text{nested partial-gate constructors}&\text{new partial operations}\\ \hline
\operatorname{Dist}_i&\varnothing&\mathsf{pow}_{1/2}(\delta_i^2)\\
\theta,\ N=1&\varnothing&\text{none}\\
\theta,\ N\ge2&\operatorname{Dist}_i&
\mathsf{pow}_q(\delta_i),\quad 1/\delta_i^q,\quad \omega_i/Z\\
R&\theta,\operatorname{Dist}_i&
\mathsf{pow}_{2s}(\delta_i),\quad \mathsf{pow}_{1/(2s)}(T)\\
E_{\operatorname{exp}}&R,\operatorname{Dist}_i&
\begin{gathered}
1/r,\quad \mathsf{pow}_{s+1}(r),\quad 1/(Mr^{s+1}),\\
\mathsf{pow}_{2s}(\delta_i/r)
\end{gathered}\\
\text{master}&\theta,E_{\operatorname{exp}},R&
\mathsf{pow}_{s+1}(r),\quad \log(L_x)
\end{array}
\]
Thus every division receives one of $(1,\delta_i^q)$, $(\omega_i,Z)$, $(1,r)$, $(1,Mr^{s+1})$ in $\mathbb R\times(0,\infty)$; every power receives one of $(\delta_i^2,1/2)$, $(\delta_i,q)$, $(\delta_i,2s)$, $(T,1/(2s))$, $(r,s+1)$, $(\delta_i/r,2s)$ in $(0,\infty)^2$; and the unique logarithm receives $L_x>0$. All remaining gates belong to $\{+_2,\times_2,\exp\}\cup\{\mathsf c_a,I_{\cdot=a}:a\in\mathbb R\}$ and have domain $\mathbb R^{\operatorname{ar}(g_v)}$. The table covers the three $\theta$ instances and both radius instances through their nested calls. There are $6N\mathbf1_{\{N\ge2\}}+2$ division vertices and one logarithm vertex. If $\Omega_{t-1}=\mathbb R^d$, then the preceding checks give $(E_{p_{t,j}}(x))_{j=1}^{a_t}\in\operatorname{dom}(g_{v_t})$ for every $x\in\mathbb R^d$, so $\Omega_t=\mathbb R^d$. Starting from $\Omega_0=\mathbb R^d$, induction in creation order therefore gives
\[
\begin{aligned}
\Omega_t&=\mathbb R^d\qquad(0\le t\le K-d),\\
E_{u_i}(x)&=P_i(z),\qquad E_{\eta_i}(x)=\theta_i(z),\qquad E_{e_i}(x)=\xi_i,\qquad E_r(x)=R(z),\\
E_h(x)&=(1-s)MR(z)^{s+1}\log\!\left(\sum_{i=1}^N\theta_i(z)\xi_i\right)=H(z),\\
E_{v_{\rm out}}(x)&=\sum_i I_{\cdot=X_i}(x)P_i(x)+\chi(x)H(y(x))\\
&=\begin{cases}H(x),&x\in U,\\P_i(X_i),&x=X_i,\ i\in[N],\end{cases}=F_{\operatorname{ext}}(x).
\end{aligned}
\]
The last identity is precisely~\eqref{eq:basic_idenity}, with the emitted safe-point map $y=S_a(x,\chi(x))$.

Let $u\leadsto v$ denote a directed path of length at least zero. For a reduction with $m\ge2$, put $b=\bar m$. Its heap edges satisfy $T[j]\to T[\lfloor j/2\rfloor]$ for $2\le j\le2b-1$, whence
\[
T[j]\leadsto T[1]\qquad(1\le j\le2b-1).
\]
For every polynomial template, write $m_{\alpha^{(k)}}=m_k$ for $1\le k<B$ and use monomial handles $m_\alpha$, coefficient handles $\kappa_\alpha$, and summands $t_0=\kappa_0$, $t_\alpha=\operatorname{New}(\times_2;\kappa_\alpha,m_\alpha)$ for $\alpha\ne0$,
\[
\begin{aligned}
\kappa_0&=t_0\leadsto\rho_{+_2}((t_\alpha)_{\alpha\in\Lambda}),\\
\kappa_\alpha&\to t_\alpha\leadsto\rho_{+_2}((t_\alpha)_{\alpha\in\Lambda})\quad(\alpha\ne0),\\
m_\alpha&\to t_\alpha\leadsto\rho_{+_2}((t_\alpha)_{\alpha\in\Lambda})\quad(\alpha\ne0),\\
\mathsf c_0(m_\beta)&\to m_\beta+0\to m_\beta(m_\beta+0)=m_{2\beta}.
\end{aligned}
\]
The affine, distance, geometric, and exponential templates have the following directed paths, where all products, sums, reciprocals, and powers denote the emitted handles:
\[
\begin{aligned}
v_j&\to v_j-X_{i,j}\leadsto(v_j-X_{i,j})^2\leadsto\delta_i,\\
\delta_i&\to\delta_i^q\to\delta_i^{-q}\to\omega_i\to\omega_i/Z\quad(N\ge2),\\
\omega_i&\leadsto Z\to\omega_i/Z\quad(N\ge2),\\
\theta_i,\ \delta_i^{2s}&\to\theta_i\delta_i^{2s}\leadsto\sum_k\theta_k\delta_k^{2s}\to R,\\
R&\to R^{-1}\to\delta_iR^{-1}\to(\delta_iR^{-1})^{2s}\to A_s(\delta_iR^{-1})^{2s}\leadsto e_i,\\
R&\to R^{s+1}\to MR^{s+1}\to(MR^{s+1})^{-1}\to u_i(MR^{s+1})^{-1}\leadsto e_i,\\
x_j-a_j,\ \chi&\to\chi(x_j-a_j)\to y_j.
\end{aligned}
\]
For $N=1$, the only vertex emitted by $\theta$ is its returned constant. Every newly emitted constant in these paths is a parent of the displayed affine, product, reciprocal, or square vertex. Thus each newly emitted vertex of a template reaches one of its returned handles. For the returned handles of the master algorithm,
\[
\begin{aligned}
\mathbf x_j&\to I_{\cdot=X_{1,j}}(\mathbf x_j)\leadsto b_1\to t_1\leadsto v_{\rm out},\\
b_i,\ p_i&\to t_i\leadsto v_{\rm out},\\
r_0&\leadsto\chi\to\chi h\leadsto v_{\rm out},\\
y_j&\leadsto u_1\leadsto e_1\to w_1\leadsto\ell\to h\to\chi h\leadsto v_{\rm out},\\
u_i&\leadsto e_i\to w_i\leadsto\ell\to h\to\chi h\leadsto v_{\rm out},\\
\eta_i,\ e_i&\to w_i\leadsto\ell\to h\to\chi h\leadsto v_{\rm out},\\
r&\to r^{s+1}\to k\to h\to\chi h\leadsto v_{\rm out}.
\end{aligned}
\]
Since $N+1\ge2$, the final reduction emits $v_{\rm out}$ last. If $v\in\widetilde V$ and $v\ne v_{\rm out}$, then a path from $v$ to $v_{\rm out}$ has a first edge $v\to u$, with $u\in\widetilde V$ and $u\leadsto v_{\rm out}$. Consequently,
\[
\widetilde V=\{v:v\leadsto v_{\rm out}\},\qquad \operatorname{ch}_{v_{\rm out}}=\varnothing,\qquad \widetilde V_{\rm out}=\{v_{\rm out}\}.
\]
For any $u,v\in\widetilde V$, both $u\leadsto v_{\rm out}$ and $v\leadsto v_{\rm out}$ hold, so an undirected path joins $u$ and $v$.

\emph{2. Ordered numerical representation.}

Numerical labels determine operand order, while a topological ordering determines evaluation; see Section~\ref{s:Background_Complexity__ss:Circuits}. We first check operand order. Put $\mathcal V_{\rm r}=\widetilde V\setminus(\widetilde V_{\rm in}\sqcup\mathcal V_{\rm c})$, $n_{\rm c}=|\mathcal V_{\rm c}|$, and $n_{\rm r}=|\mathcal V_{\rm r}|$. For $A\in\{\mathcal V_{\rm c},\mathcal V_{\rm r}\}$, set $r_A(v)=\#\{u\in A:c(u)\le c(v)\}$. If $u,v\in A$ and $c(u)<c(v)$, then
\[
r_A(v)-r_A(u)=\#\{z\in A:c(u)<c(z)\le c(v)\}\ge1.
\]
Listing $A$ in increasing creation order gives $r_A(A)=[|A|]$, with $[0]=\varnothing$. Define
\[
\begin{aligned}
\iota(v)&=\begin{cases}j,&v=\mathbf x_j,\\d+r_{\mathcal V_{\rm c}}(v),&v\in\mathcal V_{\rm c},\\d+n_{\rm c}+r_{\mathcal V_{\rm r}}(v),&v\in\mathcal V_{\rm r},\end{cases}\\
\iota(\widetilde V)&=[d]\sqcup(d+[n_{\rm c}])\sqcup(d+n_{\rm c}+[n_{\rm r}])=[K].
\end{aligned}
\]
If $\iota(u)=\iota(v)$, then the disjoint ranges place $u,v$ in the same block. Within each block, labels are strictly increasing in creation order; hence $c(u)=c(v)$ and $u=v$. Thus $\iota$ is a bijection.
The rank map is emitted by three scans. With $T=K-d$, set
\[
\begin{aligned}
\varepsilon_t&=\mathbf1_{\{v_t\in\mathcal V_{\rm c}\}},\qquad n_{\rm c}=\sum_{t=1}^T\varepsilon_t,\\
\iota(\mathbf x_j)&\gets j\qquad(j\in[d]),\\
c_0&=0,\qquad c_t=c_{t-1}+\varepsilon_t,\qquad r_t=t-c_t,\\
\iota(v_t)&=\begin{cases}d+c_t,&\varepsilon_t=1,\\d+n_{\rm c}+r_t,&\varepsilon_t=0,\end{cases}\qquad(t=1,\ldots,T),\\
A[\iota(v)]&\gets\operatorname{Rec}_\iota(v)\qquad(v\text{ in creation order}),\\
W_{\rm scan}&\le C(K+K+K)=CK.
\end{aligned}
\]
Each record uses at most two parent-label reads and one comparison.

Let $\zeta_\iota$ be the relabelling-map base and $b_{\mathbb W}^{\rm out},b_{\mathbb R}^{\rm out}$ the output bases from Lemma~\ref{lem:compiler-ram-implementation}. The first two scans read the temporary tag cells $\mathsf W[b_{\mathbb W}^{\rm tmp}+4(c(v)-1)]$; the second writes $\mathsf W[\zeta_\iota+c(v)-1]\gets\iota(v)$ using the displayed rank recurrence. The third reads each temporary record's four word fields and one real field, and obtains
\[
\begin{aligned}
\ell&\gets\mathsf W[\zeta_\iota+c(v)-1],\\
\rho_j&\gets\begin{cases}\mathsf W[\zeta_\iota+\widetilde p_j(v)-1],&j\le\operatorname{ar}(v),\\0,&j>\operatorname{ar}(v),\end{cases}\qquad j\in\{1,2\}.
\end{aligned}
\]
If $\operatorname{ar}(v)=2$, then one comparison gives $(p_1(v),p_2(v))=(\min\{\rho_1,\rho_2\},\max\{\rho_1,\rho_2\})$; otherwise $(p_1(v),p_2(v))=(\rho_1,\rho_2)$. Put $(f_0,f_1,f_2,f_3)=(\operatorname{tag}(v),\operatorname{ar}(v),p_1(v),p_2(v))$ and write
\[
\begin{aligned}
\mathsf W[b_{\mathbb W}^{\rm out}+2+4(\ell-1)+j]&\gets f_j\qquad(0\le j\le3),\\
\mathsf R[b_{\mathbb R}^{\rm out}+\ell-1]&\gets\mathsf R[b_{\mathbb R}^{\rm tmp}+c(v)-1].
\end{aligned}
\]
After the third scan, write the two headers once:
\[
\mathsf W[b_{\mathbb W}^{\rm out}]\gets K,\qquad \mathsf W[b_{\mathbb W}^{\rm out}+1]\gets\mathsf W[\zeta_\iota+c(v_{\rm out})-1].
\]
Unused parent fields cause no map lookup. Since $\iota$ is a bijection, these writes fill exactly $4K+2$ word cells and $K$ real cells. All scan reads, map and output writes, comparisons, address calculations, and loop updates are charged to $W_{\rm scan}\le C(3K+1)\le CK$.

For $N\ge2$, the division handles in the three joint $\theta$ instances $\tau\in\{0,E,H\}$ are
\[
\begin{aligned}
p_{\tau,i}&=\operatorname{New}(\mathsf{pow}_q;\delta_{\tau,i}),\\
\kappa_{\tau,i}&=\operatorname{New}(\mathsf c_1;\mathbf x_1),\\
\zeta_{\tau,i}&=\operatorname{New}(\div_2;\kappa_{\tau,i},p_{\tau,i}),\\
\omega_{\tau,i}&=\operatorname{New}(\times_2;\operatorname{Const}(\mu_i),\zeta_{\tau,i}),\\
Z_\tau&=\rho_{+_2}(\omega_{\tau,1},\ldots,\omega_{\tau,N}),\\
\eta_{\tau,i}&=\operatorname{New}(\div_2;\omega_{\tau,i},Z_\tau).
\end{aligned}
\]
For the exponential radius handle $r_E$, put
\[
\begin{aligned}
B_E&=\operatorname{New}(\times_2;\operatorname{Const}(M),\operatorname{New}(\mathsf{pow}_{s+1};r_E)),\\
\zeta_E&=\operatorname{New}(\div_2;\operatorname{Const}(1),r_E),\\
\zeta'_E&=\operatorname{New}(\div_2;\operatorname{Const}(1),B_E).
\end{aligned}
\]
For a reciprocal denominator $v\in\{p_{\tau,i},r_E,B_E\}$ and its numerator $u$, we have $u\in\mathcal V_{\rm c}$ and $g_v\in\{\mathsf{pow}_q,\mathsf{pow}_{1/(2s)},\times_2\}$. Hence
\[
\iota(u)\le d+n_{\rm c}<\iota(v).
\]
For each normalization, $g_{\omega_{\tau,i}}=\times_2$, $g_{Z_\tau}=+_2$, and $c(\omega_{\tau,i})<c(Z_\tau)$. The rank inequality therefore gives
\[
\iota(Z_\tau)-\iota(\omega_{\tau,i})\ge1.
\]
Thus every division has distinct parents, numerically ordered as numerator and denominator.
Define
\[
\widehat V=[K],\qquad \widehat E=\{(\iota(u),\iota(v)):(u,v)\in\widetilde E\},\qquad \widehat g_{\iota(v)}=g_v.
\]
For each stored operand tuple $(p_{t,1},\ldots,p_{t,a_t})$ of $v_t$, let $\pi_t\in S_{a_t}$ satisfy $\iota(p_{t,\pi_t(1)})<\cdots<\iota(p_{t,\pi_t(a_t)})$. If $g_{v_t}=\div_2$ or $a_t=1$, then $\pi_t=\operatorname{Id}$. If $a_t=2$ and $g_{v_t}\ne\div_2$, then $g_{v_t}\in\{+_2,\times_2\}$, and these gates are symmetric. Consequently, for $z\in\mathbb R^{a_t}$, $z\in\operatorname{dom}(g_{v_t})$ if and only if $(z_{\pi_t(j)})_{j=1}^{a_t}\in\operatorname{dom}(g_{v_t})$, and
\[
g_{v_t}(z)=g_{v_t}((z_{\pi_t(j)})_{j=1}^{a_t})\qquad(z\in\operatorname{dom}(g_{v_t})).
\]
With $\Omega_0=\widehat\Omega_0=\mathbb R^d$ and $E_{\mathbf x_j}=\widehat f_j=x_j$, induction in $t$ gives
\[
\begin{aligned}
\widehat\Omega_t&=\left\{x\in\widehat\Omega_{t-1}:\bigl(\widehat f_{\iota(p_{t,\pi_t(j)})}(x)\bigr)_{j=1}^{a_t}\in\operatorname{dom}(g_{v_t})\right\}\\
&=\left\{x\in\Omega_{t-1}:\bigl(E_{p_{t,j}}(x)\bigr)_{j=1}^{a_t}\in\operatorname{dom}(g_{v_t})\right\}=\Omega_t,\\
\widehat f_{\iota(v_t)}(x)&=g_{v_t}\bigl((\widehat f_{\iota(p_{t,\pi_t(j)})}(x))_{j=1}^{a_t}\bigr)\\
&=g_{v_t}\bigl((E_{p_{t,j}}(x))_{j=1}^{a_t}\bigr)=E_{v_t}(x)\qquad(x\in\Omega_t).
\end{aligned}
\]
The same induction, with $D(\mathbf x_j)=\widehat D(j)=0$, yields
\[
\begin{aligned}
\widehat D(\iota(v_t))&=1+\max_{j\in[a_t]}\widehat D(\iota(p_{t,j}))=1+\max_{j\in[a_t]}D(p_{t,j})=D(v_t),\\
\#\widehat V&=\#\widetilde V=K,\\
\#\widehat E&=\#\widetilde E,\\
\widehat V_{\rm in}&=[d],\\
\widehat V_{\rm out}&=\{\iota(v_{\rm out})\},\\
\operatorname{depth}(\widehat{\mathcal C}_{\operatorname{ext}})&=\widehat D(\iota(v_{\rm out}))=D(v_{\rm out}).
\end{aligned}
\]

\emph{3. Circuit resources and emission work.}

Let $G_F$ denote the number of vertices emitted by a call $F$, and let $W_F$ denote its direct emission work, excluding bulk workspace initialization and substitution-table maintenance. These two contributions are accounted for separately below. Put $\tau_m=h_m+\mathbf1_{\{m<2^{h_m}\}}$, so $\tau_1=0$ and $\tau_m\le h_m+1$. Let $S_{\rho_\star,m}$ denote the number of allocated reduction-heap cells. For $m=1$, $G_{\rho_\star,1}=S_{\rho_\star,1}=0$, $W_{\rho_\star,1}\le C$, and $\operatorname{dep}(\rho_\star(u_1))=\operatorname{dep}(u_1)$. For $m\ge2$, with $b=\bar m$, the fresh padding and internal vertices give
\[
\begin{aligned}
G_{\rho_\star,m}&=(b-m)+(b-1)=2b-m-1<3m,\\
W_{\rho_\star,m}&\le C\bigl(h_m+m+(b-m)+(b-1)+1\bigr)\le Cm,\\
S_{\rho_\star,m}&\le2b<4m,\\
\operatorname{dep}(\rho_\star(\boldsymbol u))&\le\max_{j\in[m]}\operatorname{dep}(u_j)+\tau_m.
\end{aligned}
\]
For $r\ge2$, $h_{\lceil r/2\rceil}=h_r-1$. Induction on $|\alpha^{(k)}|$, using at most three gates for a repeated factor, gives
\[
\begin{aligned}
\operatorname{dep}(m_k)-\max_j\operatorname{dep}(z_j)&\le3h_{|\alpha^{(k)}|},\\
G_{\operatorname{Poly}}&\le2d+3(B-d-1)+B+(B-1)+(2\bar B-B-1)<8B,\\
W_{\operatorname{Poly}}&\le C\bigl(d+(B-d-1)+B+(B-1)+B\bigr)\le C(B+d).
\end{aligned}
\]
Put $A=\max_j\operatorname{dep}(v_j)$ and $Z=\max_j\operatorname{dep}(z_j)\le A+2$. The polynomial reduction's supplied and padded leaves satisfy
\[
\begin{aligned}
\operatorname{dep}(t_0)&\le Z+1,\qquad \operatorname{dep}(t_k)\le Z+3h_s+1\quad(1\le k<B),\\
\operatorname{dep}(T[\bar B+j-1])&\le\operatorname{dep}(t_0)+1\le Z+2\quad(B<j\le\bar B),\\
\operatorname{dep}(\operatorname{Poly}(i;\boldsymbol v))&\le h_B+\max\{Z+3h_s+1,Z+2\}\le A+3h_s+h_B+3.
\end{aligned}
\]
The scalar expansions give
\[
\begin{aligned}
G_{|\cdot-X_i|}&=2d+3d+G_{\rho_{+_2},d}+1=4d+2\bar d<8d,\\
G_{I_{\cdot=X_i}}&=d+G_{\rho_{\times_2},d}=2\bar d-1<4d,\\
G_{S_a}&=5d,\\
G_\theta&\le N(G_{|\cdot-X_i|}+6)+G_{\rho_{+_2},N}\le17Nd,\\
G_R&\le G_\theta+N(G_{|\cdot-X_i|}+2)+G_{\rho_{+_2},N}+1\le31Nd,\\
G_{E_{\operatorname{exp}}}&\le G_R+N(G_{|\cdot-X_i|}+9)+7\le55Nd.
\end{aligned}
\]
The $N=1$ branch is $\operatorname{Inst}(\theta;\boldsymbol v)=\operatorname{New}(\mathsf c_1;v_1)$, so $G_\theta=1$ in that case. Parameter reads, handle reads and writes, and loop updates have bounded instruction counts per displayed operation; $N+d\le2Nd$ gives
\[
\begin{aligned}
W_{|\cdot-X_i|}+W_{I_{\cdot=X_i}}+W_{S_a}&\le Cd,\\
W_\theta&\le N W_{|\cdot-X_i|}+W_{\rho_{+_2},N}+C(N+d)\le CNd,\\
W_R&\le W_\theta+N W_{|\cdot-X_i|}+W_{\rho_{+_2},N}+C(N+d)\le CNd,\\
W_{E_{\operatorname{exp}}}&\le W_R+N W_{|\cdot-X_i|}+C(N+d)\le CNd.
\end{aligned}
\]

The algorithm contains $2N$ polynomial reductions, $N$ indicator reductions, at most $6N$ distance reductions, at most seven further $N$-input reductions, and one $(N+1)$-input reduction. Consequently, if $\mathcal R$ is the multiset of reduction calls and $m_R$ is the input count of call $R$,
\[
\begin{aligned}
\sum_{R\in\mathcal R}m_R&\le2NB+7Nd+8N+1\le2NB+16Nd,\\
\sum_{R\in\mathcal R}2\bar m_R&<4\sum_{R\in\mathcal R}m_R\le8NB+64Nd,\\
W_{\text{initializing all reduction arrays}}&\le C\sum_{R\in\mathcal R}m_R\le C(NB+Nd),
\\
K&\le d+16NB+4Nd+17Nd+31Nd+55Nd+5d+11N+13\\
&\le16NB+140Nd.
\end{aligned}
\]

Including input records and the computation of $a$,
\[
\begin{aligned}
W_{\rm emit}&\le Cd+W_a+2N W_{\operatorname{Poly}}+\sum_{i=1}^N W_{I_{\cdot=X_i}}+W_{S_a}+W_\theta+W_{E_{\operatorname{exp}}}+W_R\\
&\quad+2W_{\rho_{+_2},N}+W_{\rho_{+_2},N+1}+CN\\
&\le C\bigl(N(B+d)+Nd+N+d\bigr)\le C(NB+Nd),\qquad W_a\le C(N+d).
\end{aligned}
\]

Relative to the maximum depth of their formal inputs, the constructor depth recurrences are
\[
\begin{aligned}
D_\delta&\le2+3+\tau_d+1=6+\tau_d,&D_I&\le1+\tau_d,&D_{S_a}&\le4,\\
D_{\operatorname{Poly}}&\le2+3h_s+1+\tau_B=3h_s+\tau_B+3,\\
D_\theta&\le\begin{cases}1,&N=1,\\D_\delta+4+\tau_N,&N\ge2,\end{cases}\\
D_R&\le\max\{D_\theta,D_\delta+1\}+\tau_N+2\le D_\delta+6+2\tau_N,\\
D_{E_{\operatorname{exp}}}&\le\max\{D_R,D_\delta\}+7\le D_\delta+13+2\tau_N.
\end{aligned}
\]
For the master circuit, with $D_p=D_{\operatorname{Poly}}$,
\[
\begin{aligned}
D_\chi&\le D_I+\tau_N+2,\qquad D_y\le D_\chi+4,\\
D_t&\le1+\max\{D_I,D_p\},\qquad D_u\le D_y+D_p,\\
D_\eta&\le D_y+D_\theta,\qquad D_r\le D_y+D_R,\\
D_e&\le D_u+D_{E_{\operatorname{exp}}},\\
D_\ell&\le1+\tau_N+1+\max\{D_\eta,D_e\},\qquad D_k\le D_r+2,\\
D_h&\le1+\max\{D_\ell,D_k\},\\
D_{\rm out}&\le\tau_{N+1}+\max\{D_t,1+\max\{D_\chi,D_h\}\}\\
&\le33+3h_s+\tau_B+2\tau_d+4\tau_N+\tau_{N+1}\\
&\le42+3h_s+h_B+2h_d+5h_N\le C_s\log_2(d+N).
\end{aligned}
\]
The last bound uses $B\le(s+1)d^s$, $h_B\le1+\log_2(s+1)+s\log_2d$, $h_d,h_N\le1+\log_2(d+N)$, and $\log_2(d+N)\ge1$.

\emph{4. Uniform implementation and final bounds.}

Put $L=B+d+1$, $H_0=NL$, and $K_0=160H_0$. The finite template dependencies are
\[
E_{\operatorname{exp}}\longrightarrow R\longrightarrow\theta\longrightarrow\operatorname{Dist}\longrightarrow\rho_{+_2},
\qquad \operatorname{Poly}\longrightarrow\rho_{+_2},
\qquad I_{\cdot=X_i}\longrightarrow\rho_{\times_2}.
\]
They give at most $J=8$ simultaneously active table-bearing frames, including the compiler frame. The bounds on formal ports, reduction heaps, and emitted vertices are
\[
\begin{aligned}
m_{\rm ports}&\le\max\{N+d+1,B\}\le H_0,\qquad K\le K_0,\\
m_\rho&\le\max\{B,d,N+1\}\le H_0,\qquad 2\bar m_\rho<4H_0.
\end{aligned}
\]
For a table-bearing call $F$, let $m_F$ be its port count and $L_F$ the number of vertices emitted during the call, including nested calls. Returned handles are retained before a child table is reused; only scalar counters are reset between calls. Copying each fresh handle to the currently active tables gives
\begin{equation}
\label{eq:compiler-table-work}
\begin{aligned}
\sum_Fm_F&\le d+2Nd+Nd+(d+1)+3d+2d+(N+d)+6Nd\\
&\quad +(2NB+7Nd+8N+1)\\
&=2NB+16Nd+8d+9N+2\le2NB+35Nd,\\
\sum_FL_F&=\sum_{v\notin\widetilde V_{\rm in}}
\#\{F:F\text{ is active when }v\text{ is emitted}\}\le8(K-d),\\
W_{\rm tables}&\le C\sum_F(m_F+L_F)\le C(2NB+35Nd+8K)\le C(NB+Nd).
\end{aligned}
\end{equation}
These sums include the compiler frame and all reduction calls.

For a table-bearing call $F$, let $c_F$ be the vertex count at entry, let $\zeta_F$ be its word-table base, and let $(u_{F,j})_{j=1}^{m_F}$ be its supplied handles. Store $\sigma_F[j]$ at address $\zeta_F+j-1$ as the word $c(\sigma_F[j])$. The port copies and active-table appends are
\[
\begin{aligned}
\mathsf W[\zeta_F+j-1]&\gets c(u_{F,j}) &&(1\le j\le m_F),\\
\mathsf W[\zeta_G+m_G+c(v)-c_G-1]&\gets c(v) &&(G\text{ active when }v\text{ is emitted}).
\end{aligned}
\]
For a formal operand index $0\le\xi<m_F+c(v)-c_F$ used when emitting $v$, the word passed to the constructor is $1=c(\mathbf x_1)$ if $\xi=0$, and is read from $\mathsf W[\zeta_F+\xi-1]$ otherwise. In the latter case, $1\le\xi\le m_F+c(v)-c_F-1$, so the accessed entry is a supplied port or a previously appended handle. There are $\sum_Fm_F$ port writes, $\sum_FL_F$ append writes, and at most $2(K-d)$ operand reads. A returned tuple has at most $m_F+L_F$ handles, copied before table reuse. Since the compiler frame contributes $L_F=K-d$, all these reads, writes, address computations, and loop updates are included in $W_{\rm tables}\le C\sum_F(m_F+L_F)$ from~\eqref{eq:compiler-table-work}.

The safe point and real gate labels require only the permitted exact-real operations. With $A_0=1$, compute
\[
A_i=A_{i-1}+\begin{cases}X_{i,1},&X_{i,1}\ge0,\\-X_{i,1},&X_{i,1}<0,\end{cases}
\qquad
a_j=\begin{cases}A_N,&j=1,\\0,&2\le j\le d,\end{cases}
\qquad q_{\mathbb R}=\operatorname{embed}(q).
\]
This costs $W_a\le C(N+d)$. Every subsequently emitted real gate parameter belongs to
\begin{equation}
\label{eq:compiler-parameters}
\begin{aligned}
\mathcal L_0&=\{0,1,-1,q,2s,s+1,\tfrac12,\tfrac1{2s},\tfrac1{1-s},A_s,M,(1-s)M\},\\
\mathcal L&=\mathcal L_0\cup\{X_{i,j},-X_{i,j},\mu_i,c_{i,\alpha},a_j,-a_j:i\in[N],\ j\in[d],\ \alpha\in\Lambda\},\\
W_{\lambda(v)}&\le C.
\end{aligned}
\end{equation}
The program loads fixed rational literals, reads supplied real data, and forms these labels using $+,-,\times$, sign tests, and the single word embedding. Thus the compiler records power, exponential, and logarithm gates without evaluating those functions. Setting $D=2^{q-2s-2}\ge1$ gives $2s+\lceil\log_2D\rceil+2=q$, identifying the instance of Theorem~\ref{thrm:complexity}; the compiler does not evaluate $D$.

Lemma~\ref{lem:compiler-ram-implementation} implements these exact emissions and scans by a finite program $\mathcal P_{d,s}$ independent of $N,q,w$. Its hypotheses follow from the preceding bounds and~\eqref{eq:compiler-parameters}. Under~\eqref{eq:compiler-word-capacity}, it gives $W_{\rm init}\le CH_0$, the table-maintenance bound in~\eqref{eq:compiler-table-work}, and $W_{\rm scan}\le CK$, with every intermediate word value strictly below $2^w$. An admissible word length is
\[
w=\left\lceil\log_2\bigl(q+C_{\rm mem}(d,s)(N+1)+1\bigr)\right\rceil
=\mathcal O_{d,s}\bigl(\log(q+N)\bigr).
\]

Initialization consists of actual memory writes. Use the blocks $J_{\nu,t}$, bases $b_{\nu,t}$, and lengths $n_{\nu,t}$ in~\eqref{eq:compiler-ram-layout}. Before emission, load $\mathsf W[2]\gets0$, $\mathsf W[3]\gets1$, and $\mathsf R[2]\gets0$. For each block, set $\mathsf W[0]\gets b_{\nu,t}$ and $\mathsf W[1]\gets b_{\nu,t}+n_{\nu,t}$; while $\mathsf W[0]<\mathsf W[1]$, execute
\[
\begin{cases}\mathsf W[\mathsf W[0]]\gets\mathsf W[2],&\nu=\mathbb W,\\\mathsf R[\mathsf W[0]]\gets\mathsf R[2],&\nu=\mathbb R,\end{cases}\qquad\mathsf W[0]\gets\mathsf W[0]+\mathsf W[3].
\]
After $j$ iterations the pointer equals $b_{\nu,t}+j$, and the first $j$ block cells have been zeroed. Thus there are exactly $n_{\nu,t}$ writes, with no iteration for an empty block. The disjointness in~\eqref{eq:compiler-ram-layout} preserves the input and register cells; Lemma~\ref{lem:compiler-ram-implementation} supplies the capacity bound for the addresses and loop endpoints. Including the fixed number of block-base calculations, comparisons, increments, and scratch-register initializations gives $W_{\rm init}\le C(1+\sum_{\nu,t}n_{\nu,t})\le CH_0$.

Since $d\le B=\prod_{j=1}^s(d+j)/j\le\prod_{j=1}^sd(j+1)/j=(s+1)d^s$ and $B+d+1\le(s+3)d^s$, the vertex and edge counts are
\begin{equation}
\label{eq:compiler-vertex-edge-count}
K\le16NB+140Nd\le C_sNd^s,\qquad
|\widetilde E|=\sum_{v\notin\widetilde V_{\rm in}}\operatorname{ar}(g_v)\le2(K-d).
\end{equation}
The implementation estimates give
\begin{equation}
\label{eq:compiler-work}
\begin{aligned}
\operatorname{Work}(\texttt{Compile-Circuit})
&\le W_{\rm init}+W_{\rm emit}+W_{\rm tables}+W_{\rm scan}\\
&\le CN(B+d+1)\le C_sNd^s.
\end{aligned}
\end{equation}
The designated output consists of $4K+2$ word cells and $K$ real cells, so
\begin{equation}
\label{eq:compiler-output-storage}
\operatorname{OutputStorage}(\texttt{Compile-Circuit})=5K+2\le C_sNd^s.
\end{equation}
Finally,
\[
\begin{aligned}
\operatorname{dom}(f_{\widehat{\mathcal C}_{\operatorname{ext}},I_d})&=\Omega_{K-d}=\mathbb R^d,
&f_{\widehat{\mathcal C}_{\operatorname{ext}},I_d}&=F_{\operatorname{ext}},\\
|\widehat{\mathcal C}_{\operatorname{ext}}|&=K-d\le C_sNd^s,
&\operatorname{width}(\widehat{\mathcal C}_{\operatorname{ext}})&\le K\le C_sNd^s,\\
\operatorname{depth}(\widehat{\mathcal C}_{\operatorname{ext}})&\le C_s\log_2(d+N).
\end{aligned}
\]
With the inputs available at time $0$, evaluate each computation vertex $v\in\widehat V\setminus[d]$ at time $\widehat D(v)$. If $(u,v)\in\widehat E$, then the depth recurrence gives $\widehat D(u)\le\widehat D(v)-1$. If $v\in\widehat V\setminus[d]$, then $1\le\widehat D(v)\le D_{\rm out}$, since every vertex reaches $\iota(v_{\rm out})$. Moreover,
\[
\#\{v\in\widehat V\setminus[d]:\widehat D(v)=t\}\le K-d\le C_sNd^s\qquad(t\in[D_{\rm out}]).
\]
Induction on $t$ therefore evaluates every gate scheduled at time $t$ from operands available by time $t-1$. This schedule uses at most $K-d$ ideal processors and $D_{\rm out}\le C_s\log_2(d+N)$ unit-time steps.
\end{proof}

\clearpage
\appendix

\section{Background on Complexity Theory}
\label{s:Background_Complexity}
We briefly overview the background and conventions which we use to quantify the complexity of algorithms; both non-uniform and non-recursive and uniform and recursive.  These are the languages of the circuits and exact-real word-RAM models of computing.

\subsection{Circuits}
\label{s:Background_Complexity__ss:Circuits}

This appendix briefly reviews the terminology of $\mathbb{G}$-circuits.
Let $\mathcal D=(V,E)$ be a connected directed acyclic graph with its vertex set numerically labelled by a finite subset $V\subseteq\mathbb N$. For each $v\in V$, define its parent and child sets by
\[
\operatorname{pa}_v
\eqdef
\{u\in V:(u,v)\in E\},
\qquad
\operatorname{ch}_v
\eqdef
\{w\in V:(v,w)\in E\}.
\]
The input, output, and computation vertices are respectively
\[
V_{\operatorname{in}}
\eqdef
\{v\in V:\operatorname{pa}_v=\varnothing\},
\qquad
V_{\operatorname{out}}
\eqdef
\{v\in V:\operatorname{ch}_v=\varnothing\},
\qquad
V_{\operatorname{comp}}
\eqdef
V\setminus V_{\operatorname{in}}.
\]
We assume that $\#V_{\operatorname{in}}=\nu_{\operatorname{in}}\in\mathbb N_+$ and that $V_{\operatorname{out}}=\{v_{\operatorname{out}}\}$ consists of exactly one vertex.
A bounded-fan-in $\mathbb G$-circuit is a triple
$
\mathcal C
\eqdef
(V,E,\boldsymbol g)
$ where $\boldsymbol g\eqdef(g_v)_{v\in V}$ 
such that every input vertex is decorated by the identity map on $\mathbb R$, while every computation vertex $v\in V_{\operatorname{comp}}$ is decorated by a gate $g_v\in\mathbb G$ whose arity is
$
\operatorname{ar}(g_v)
=
\#\operatorname{pa}_v
\le2$.  
Because $V\subseteq\mathbb N$, the parent set of every computation vertex has a natural numerical ordering. Thus, if
\[
\operatorname{pa}_v
=
\{u_1<\cdots<u_{D_v}\},
\qquad
D_v\eqdef\#\operatorname{pa}_v,
\]
then the arguments of $g_v$ are supplied in the order
$(u_1,\ldots,u_{D_v})$. In particular, this ordering determines which parent is the numerator and which parent is the denominator at a $\div_2$-gate.

We now define the function realized by the circuit. Fix an external input dimension $\nu_{\operatorname{var}}\in\mathbb N_+$ and a binary row-stochastic matrix which formally encodes \textit{variables} into the inputs of the initial gates on which the circuit computes
$
\Pi\in\{0,1\}^{\nu_{\operatorname{in}}\times\nu_{\operatorname{var}}}
$ and $
\sum_{j=1}^{\nu_{\operatorname{var}}}\Pi_{ij}=1$ for every $i\in[\nu_{\operatorname{in}}]$. 
We call $\Pi$ the \emph{lifting channel}. Since $\Pi$ is binary and row-stochastic, every row of $\Pi$ contains exactly one nonzero entry; hence every circuit input node receives one coordinate of the external input, and the same coordinate may be lifted to several input nodes.
Write
\[
V_{\operatorname{in}}
=
\{v_1<\cdots<v_{\nu_{\operatorname{in}}}\}.
\]
For an external input $x\in\mathbb R^{\nu_{\operatorname{var}}}$, define the values at the input vertices by
\begin{equation}
\label{eq:circuit-input-values}
\bigl(
x^{(v_1)},\ldots,x^{(v_{\nu_{\operatorname{in}}})}
\bigr)^\top
\eqdef
\Pi x.
\end{equation}
Since $\mathcal D$ is acyclic, its computation vertices admit a topological ordering
$
V_{\operatorname{comp}}
=
\{w_1,\ldots,w_L\}
$ where $L\eqdef\#V_{\operatorname{comp}}$; 
such that every computation parent of $w_k$ belongs to
$\{w_1,\ldots,w_{k-1}\}$. This topological ordering is used only to perform the recursive evaluation and need not agree with the numerical labels that order the inputs of a gate.

Set $\Omega_0\eqdef\mathbb R^{\nu_{\operatorname{var}}}$. Suppose that $k\in[L]$ and that the values of all vertices preceding $w_k$ in the topological ordering have been defined on $\Omega_{k-1}$. Write the naturally ordered parent set as
\[
\operatorname{pa}_{w_k}
=
\{u_{k,1}<\cdots<u_{k,D_k}\},
\qquad
D_k\eqdef\#\operatorname{pa}_{w_k},
\]
and define
\begin{equation}
\label{eq:circuit-recursive-domain}
\Omega_k
\eqdef
\left\{
x\in\Omega_{k-1}:
\bigl(
x^{(u_{k,1})},\ldots,x^{(u_{k,D_k})}
\bigr)
\in\operatorname{dom}(g_{w_k})
\right\}.
\end{equation}
For every $x\in\Omega_k$, define the value at $w_k$ by
\begin{equation}
\label{eq:circuit-recursive-evaluation}
x^{(w_k)}
\eqdef
g_{w_k}
\left(
x^{(u_{k,1})},\ldots,x^{(u_{k,D_k})}
\right).
\end{equation}
The condition in~\eqref{eq:circuit-recursive-domain} is the \emph{gate-compatibility condition}. It ensures, in particular, that every denominator supplied to a $\div_2$-gate is nonzero, every argument supplied to a $\log$-gate is strictly positive, and every argument supplied to a $\mathsf{pow}_r$-gate is nonnegative; with \textit{realization domain}\footnote{The composite $x\mapsto1/x\mapsto1/(1/x)$ has realization domain $\mathbb R\setminus\{0\}$ and agrees with the identity there, whereas $x\mapsto x$ realizes the identity on all of $\mathbb R$.} $
\operatorname{dom}(f_{\mathcal C,\Pi})
\eqdef
\Omega_L$.  
For every $x\in\Omega_L$, the scalar output\footnote{The value in~\eqref{eq:circuit-output} is independent of the chosen topological ordering. } of the circuit is
\begin{equation}
\label{eq:circuit-output}
f_{\mathcal C,\Pi}(x)
\eqdef
x^{(v_{\operatorname{out}})}.
\end{equation}
For a set $\Omega\subseteq\mathbb R^{\nu_{\operatorname{var}}}$ and a function $f:\Omega\to\mathbb R$, we say that the circuit $\mathcal C$, equipped with the lifting channel $\Pi$, \textbf{computes $f$ on $\Omega$} if
\[
\Omega
\subseteq
\operatorname{dom}(f_{\mathcal C,\Pi})
\qquad\text{and}\qquad
f_{\mathcal C,\Pi}(x)=f(x)
\quad
\text{for every }x\in\Omega.
\]
When the lifting channel is fixed, we abbreviate
$f_{\mathcal C,\Pi}$ by $f_{\mathcal C}$. 

For $v\in V$, define $\operatorname{dep}(v)=0$ if $v\in V_{\operatorname{in}}$ and $\operatorname{dep}(v)=1+\max_{u\in\operatorname{pa}_v}\operatorname{dep}(u)$ otherwise. Set
\[
\begin{aligned}
|\mathcal C|&\eqdef\#V_{\operatorname{comp}},\qquad \operatorname{depth}(\mathcal C)\eqdef\operatorname{dep}(v_{\operatorname{out}}),\\
\operatorname{width}(\mathcal C)&\eqdef\max\{\#A:A\subseteq V\text{ is an antichain for reachability}\}.
\end{aligned}
\]
Here distinct vertices in an antichain are joined by no directed path in either direction. Every gate has unit cost; wires and fan-out have no cost. Thus evaluating each gate once costs $|\mathcal C|$ gate operations, while equal-depth vertices may be evaluated simultaneously.

\subsection{The Exact-Real Word-RAM Model}
\label{s:Background_Complexity__ss:ER_WR}
For $w\in\mathbb N_+$, let $
\mathbb W_w\eqdef\{0,1,\ldots,2^w-1\}
$
denote the set of $w$-bit machine words. An \emph{exact-real word-RAM} of word size $w$ has two random-access memory arrays
$
\mathsf W:\mathbb W_w\to\mathbb W_w
$ and $
\mathsf R:\mathbb W_w\to\mathbb R$ 
called respectively the \emph{word memory} and the \emph{real memory}. Both arrays are addressed by elements of $\mathbb W_w$. Thus $\mathsf W[a]$ stores one $w$-bit word, whereas $\mathsf R[a]$ stores one exact real number.

A program is a finite sequence of labelled instructions, independent of the input values, the input length, and the word size $w$. At every stage, the machine state consists of the program counter together with the contents of $\mathsf W$ and $\mathsf R$. The admissible unit-cost instructions are the following.
\begin{enumerate}
\item[(i)] \textbf{Word-memory instructions:} direct and indirect reads and writes in $\mathsf W$, including
$
\mathsf W[a]\leftarrow\mathsf W[b]
$, $\mathsf W[\mathsf W[a]]\leftarrow\mathsf W[b]$, and 
$\mathsf W[a]\leftarrow\mathsf W[\mathsf W[b]]$.  
\item[(ii)] \textbf{Real-memory instructions:} direct and indirect reads and writes in $\mathsf R$, including
$
\mathsf R[a]\leftarrow\mathsf R[b]
$, $\mathsf R[\mathsf W[a]]\leftarrow\mathsf R[b]$, and $\mathsf R[a]\leftarrow\mathsf R[\mathsf W[b]]$.
\item[(iii)] \textbf{Word operations:} addition, subtraction, multiplication, integer quotient, remainder, and Boolean operations on $w$-bit words, with arithmetic results interpreted modulo $2^w$ where appropriate; comparisons of two words; and conditional branching on the outcome of such comparisons.
\item[(iv)] \textbf{Exact-real operations:} for $\circ\in\{+,-,\times,\div\}$,
$
\mathsf R[a]\leftarrow\mathsf R[b]\circ\mathsf R[c]$, 
where division is permitted only when $\mathsf R[c]\ne0$, together with exact sign and equality tests on real registers. The machine may also copy a word into a real register through the canonical embedding $\mathbb W_w\subset\mathbb R$.
\item[(v)] \textbf{Control instructions:} unconditional jumps, conditional jumps, and a halting instruction.
\end{enumerate}
A program may load its fixed nonnegative integer literals into word registers and its fixed rational literals into real registers; word size must accommodate its integer literals. No instruction converts an arbitrary exact real number into a word. In particular, floor, ceiling, extraction of the binary representation of a real number, and indirect addressing by a real register are not permitted.
An input is a pair
\[
(a,b)\in\mathbb R^{n_{\mathbb R}}\times\mathbb W_w^{n_{\mathbb W}}
\]
loaded into designated initial segments of $\mathsf R$ and $\mathsf W$. A program is \emph{well-defined} on an input if every partial instruction encountered during its execution is applied inside its domain. Its \emph{sequential work} on that input is the number of instructions executed before halting.

A designated collection of memory cells forms the output. The \emph{output storage} is the number of word and real cells in this designated output representation. The \emph{auxiliary storage} is the maximum number of additional memory cells used during the execution, excluding the read-only input cells and the designated output cells. Each word cell and each exact-real cell contributes one unit of storage. This model therefore measures exact-real arithmetic work and random-access storage, but not bit complexity, finite-precision error, conditioning, or memory-communication cost.

For the compiler below, choose a fixed integer $C_{\rm mem}(d,s)\ge1$ large enough for the storage layout established in its proof, and require
\begin{equation}
\label{eq:compiler-word-capacity}
2^w>q+C_{\rm mem}(d,s)(N+1).
\end{equation}
This accommodates input, output, auxiliary arrays, integer literals, counters, and record offsets, each in the corresponding word or real memory. The input integer $q$ is embedded into a real register before its gate labels are computed. For bounded $q$, an admissible word size can be chosen as $w=\Theta_{d,s}(\log(1+N))$.

\subsection{RAM Implementation of the Circuit Compiler}
\label{s:Background_Complexity__ss:CompilerRAM}

\begin{lemma}[RAM implementation of the emitter program]
\label{lem:compiler-ram-implementation}
Fix $d,s\in\mathbb N_+$ with $s\ge2$. For an input satisfying Data Format~\ref{dataformat}, use the scalar emitters and three relabelling scans of Algorithm~\ref{alg:compile-circuit}, and put
\[
B=\binom{d+s}{s},\qquad L=B+d+1,\qquad H_0=NL,\qquad K_0=160H_0.
\]
Let $K$ count all emitted vertices. For each table-bearing call $F$, including the compiler frame, let $m_F$ be its port count and $L_F$ its number of emitted vertices, including nested calls. Let $J$ be the maximum number of simultaneously active calls, and let $m_\rho$ be the input count of a reduction call. If
\[
\begin{gathered}
K\le16NB+140Nd,\qquad J\le8,\\
\max_Fm_F\le\max\{N+d+1,B\},\qquad \max_\rho m_\rho\le\max\{B,d,N+1\}.
\end{gathered}
\]
then, there exists a finite exact-real word-RAM program $\mathcal P_{d,s}$, independent of $N,q,w$ and the real input values, implementing these emissions and scans whenever~\eqref{eq:compiler-word-capacity} holds, with $C_{\rm mem}(d,s)$ chosen as in the proof. For an absolute constant $C\ge 1$,
\begin{equation}
\label{eq:compiler-ram-costs}
W_{\rm init}\le CH_0,\qquad W_{\rm tables}\le C\sum_F(m_F+L_F),\qquad W_{\rm scan}\le CK.
\end{equation}
After the shared safe-point computation, each gate parameter and scalar record append costs at most $C$ instructions, excluding the table copies counted in $W_{\rm tables}$. The output uses~\eqref{eq:compiler-record} and occupies exactly $5K+2$ cells. Input, output, and auxiliary storage together occupy at most $3265H_0$ word cells and $835H_0$ real cells.
\end{lemma}

\begin{proof}[{Proof of Lemma~\ref{lem:compiler-ram-implementation}}]
Throughout the proof, $C$ denotes an absolute constant which may increase between occurrences. We set $r_*=512$ and use the polynomial emitter's coefficient order $\Lambda=\{\alpha^{(0)},\ldots,\alpha^{(B-1)}\}$. The input encoding is
\begin{equation}
\label{eq:compiler-ram-input}
\begin{aligned}
\mathsf W[r_*]&=N,&\mathsf W[r_*+1]&=q,&\mathsf R[r_*]&=M,\\
\mathfrak a(i,\ell)&=r_*+1+(i-1)L+\ell &&&(i\in[N],\ 0\le\ell<L),\\
\mathsf R[\mathfrak a(i,j-1)]&=X_{i,j} &&&(j\in[d]),\\
\mathsf R[\mathfrak a(i,d)]&=\mu_i,\\
\mathsf R[\mathfrak a(i,d+1+k)]&=c_{i,\alpha^{(k)}} &&&(0\le k<B),\\
\mathfrak a(i,0)+L&=\mathfrak a(i+1,0) &&&(i<N),\\
\mathfrak a(N,L-1)&=r_*+NL,&n_{\mathbb R}^{\rm in}&=1+H_0,&n_{\mathbb W}^{\rm in}&=2.
\end{aligned}
\end{equation}
For $\mathsf W[0]=a$, $\mathsf W[1]=b$, and $\mathsf W[2]=c$, where $a,b\ge r_*$ and $c\ge r_*+1+H_0$, real arithmetic is implemented by
\[
\begin{aligned}
\mathsf R[0]&\gets\mathsf R[\mathsf W[0]],&\mathsf R[1]&\gets\mathsf R[\mathsf W[1]],\\
\mathsf R[2]&\gets\mathsf R[0]\circ\mathsf R[1],&\mathsf R[\mathsf W[2]]&\gets\mathsf R[2],\qquad\circ\in\{+,-,\times\}.
\end{aligned}
\]
For $m\ge1$, put $\bar m=2^{\lceil\log_2m\rceil}$. The hypotheses give
\[
\begin{gathered}
m_F\le\max\{N+d+1,B\}\le H_0,\qquad L_F\le K\le K_0,\\
m_\rho\le\max\{B,d,N+1\}\le H_0,\qquad 2\bar m_\rho<4H_0.
\end{gathered}
\]
In each memory, we reserve the disjoint register banks
\[
\begin{gathered}
\mathcal B_{\rm pr}=[0,32)\cap\mathbb N_0,\qquad\mathcal B_{\rm gl}=[32,64)\cap\mathbb N_0,\\
\mathcal B_\ell=[64+32\ell,96+32\ell)\cap\mathbb N_0\qquad(0\le\ell<8),\\
\mathcal B_{\rm pr}\sqcup\mathcal B_{\rm gl}\sqcup\bigsqcup_{\ell=0}^7\mathcal B_\ell=[0,320)\cap\mathbb N_0\subset[0,r_*),\\
n_{\rm pr}\le4\ \text{fields}+6\ \text{addresses}+6\ \text{operands}=16<32,\\
n_{\rm gl}^{\mathbb W}\le14\ \text{block bases}+10\ \text{counters}+4\ \text{pointers}=28<32,\\
n_{\rm local}\le3\ \text{indices}+5\ \text{pointers}+12\ \text{handles}+4\ \text{offsets}=24<32,\\
\mathsf R[32]=q_{\mathbb R},\qquad \mathsf R[33]=M,\qquad \mathsf R[34]=(1-s)M,\qquad \mathsf R[35]=A_i.
\end{gathered}
\]
Here $A_i=1+\sum_{k=1}^i|X_{k,1}|$. A call at depth $\ell$ uses $\mathcal B_\ell$; indexed families remain in arrays. The twelve local handles of $E_{\operatorname{exp}}$ are $(r,r^{-1},p_R,p_M,p_M^{-1},\delta_i,a_i,\beta_i,b_i,c_i,d_i,e_i)$. Arithmetic uses $\mathcal B_{\rm pr}$; only the initial embedding writes $\mathsf R[32]$.

Write $c(v)$ for the creation index, including the $d$ inputs. For each table-bearing call $F$, let $c_F$ be the vertex count at entry, $(u_{F,j})_{j=1}^{m_F}$ its supplied handles, and $\zeta_F$ its word-table base. Put $\mathfrak s_F(j)=\zeta_F+j-1$. For the innermost active call $F$, execute
\begin{equation}
\label{eq:compiler-ram-tables}
\begin{aligned}
\sigma_F[j]&\gets u_{F,j} &&(1\le j\le m_F),\\
\sigma_F^*(a)&=\begin{cases}\mathbf x_1,&a=0,\\\sigma_F[a],&a\ge1,\end{cases}\\
v&\gets\operatorname{New}\bigl(g;\sigma_F^*(a_1),\ldots,\sigma_F^*(a_{\operatorname{ar}(g)})\bigr),\\
0\le a_j&<m_F+c(v)-c_F &&(j\in[\operatorname{ar}(g)]),\\
\sigma_G[m_G+c(v)-c_G]&\gets v &&(G\text{ active}),\\
1\le c(v)-c_G&\le L_G,\qquad m_G+L_G\le H_0+K_0.
\end{aligned}
\end{equation}
If $a_j=0$ for an operand index, then $g=\mathsf c_z$ for some $z\in\mathbb R$. Handles are stored as their creation indices:
\[
\begin{aligned}
\mathsf W[\mathfrak s_F(j)]&\gets c(u_{F,j}) &&(1\le j\le m_F),\\
h_j&\gets\begin{cases}1,&a_j=0,\\\mathsf W[\mathfrak s_F(a_j)],&a_j\ge1,\end{cases} &&(1\le j\le\operatorname{ar}(g)),\\
\mathsf W[\mathfrak s_G(m_G+c(v)-c_G)]&\gets c(v) &&(G\text{ active}).
\end{aligned}
\]
Each $h_j$ occupies a word register. If $a_j=0$, then the load is $1=c(\mathbf x_1)$ and no table cell is accessed. If $a_j\ge1$, then $a_j\le m_F+c(v)-c_F-1$, so the read lies among the ports or previously appended entries. Copy each returned tuple once before table reuse; it has at most $m_F+L_F$ handles. The compiler frame has $L_F=K-d$, and hence
\[
\begin{aligned}
W_{\rm tables}&\le C\left(\sum_Fm_F+\sum_FL_F+2(K-d)+\sum_F(m_F+L_F)\right)\\
&\le C\sum_F(m_F+L_F).
\end{aligned}
\]
The four terms count port writes, active-table appends, operand reads, and return copies; $C$ includes address calculations and loop updates. Only scalar counters are reset between calls. Inlining the finite acyclic call graph at its syntactic call sites gives literal jump targets and requires no return stack. Polynomial splitting ranks and coefficient offsets are fixed literals for $d,s$.

We allocate the following blocks after the input, reusing the eight template tables by call depth and the heap after each reduction:
\[
\begin{array}{c|r|r}
\text{block}&\text{word cells}&\text{real cells}\\\hline
\text{temporary records}&4K_0&K_0\\
\text{relabelling map}&K_0&0\\
\text{reserved output}&4K_0+2&K_0\\
\text{retained handles and safe point}&7N+d+8&d\\
\text{template tables}&8(K_0+H_0)&0\\
\text{monomial table}&B&0\\
\text{reduction heap}&4H_0&0
\end{array}
\]
The retained tuples are $(b_i)_i,(p_i)_i,(t_i)_i,(u_i)_i,(\eta_i)_i,(e_i)_i,(w_i)_i,(y_j)_j$; the eight further slots hold $r_0,\chi,r,\ell,k,h,\chi h,v_{\rm out}$. For the entry $n_{\nu,t}$ in row $t$ and column $\nu\in\{\mathbb W,\mathbb R\}$, set
\begin{equation}
\label{eq:compiler-ram-layout}
\begin{aligned}
\beta_{\mathbb W}&=r_*+2,\qquad\beta_{\mathbb R}=r_*+1+H_0,\\
b_{\nu,t}&=\beta_\nu+\sum_{j=1}^{t-1}n_{\nu,j},\\
J_{\nu,t}&=[b_{\nu,t},b_{\nu,t}+n_{\nu,t})\cap\mathbb N_0,\\
J_{\nu,t}\cap J_{\nu,t'}&=\varnothing\quad(t\ne t'),\qquad J_{\nu,t}\cap[0,\beta_\nu)=\varnothing,\\
n_{\mathbb W}^{\rm all}&=514+17K_0+12H_0+7N+d+B+10\le3265H_0,\\
n_{\mathbb R}^{\rm all}&=513+H_0+2K_0+d\le835H_0,\\
W_{\rm init}&\le C\sum_{\nu,t}n_{\nu,t}\le CH_0.
\end{aligned}
\end{equation}
Before emission, load $\mathsf W[2]\gets0$, $\mathsf W[3]\gets1$, and $\mathsf R[2]\gets0$. For each block $J_{\nu,t}$, set $\mathsf W[0]\gets b_{\nu,t}$ and $\mathsf W[1]\gets b_{\nu,t}+n_{\nu,t}$. While $\mathsf W[0]<\mathsf W[1]$, execute
\[
\begin{cases}\mathsf W[\mathsf W[0]]\gets\mathsf W[2],&\nu=\mathbb W,\\\mathsf R[\mathsf W[0]]\gets\mathsf R[2],&\nu=\mathbb R,\end{cases}\qquad\mathsf W[0]\gets\mathsf W[0]+\mathsf W[3].
\]
After $j$ iterations, for $0\le j\le n_{\nu,t}$, the pointer is $b_{\nu,t}+j$ and the first $j$ block cells are zero. Thus the loop makes exactly $n_{\nu,t}$ writes; if $n_{\nu,t}=0$, then it makes none. By~\eqref{eq:compiler-ram-layout}, the written blocks exclude the input and registers. Including block-base calculations, register initialization, comparisons, and increments gives $W_{\rm init}\le C(1+\sum_{\nu,t}n_{\nu,t})\le CH_0$.

Let $b_{\mathbb W}^{\rm tmp},b_{\mathbb R}^{\rm tmp}$ and $b_{\mathbb W}^{\rm out},b_{\mathbb R}^{\rm out}$ be the temporary and output record bases, and let $\zeta_\iota$ be the relabelling-map base. For $v$ with creation index $k=c(v)$ and ordered parents $(u_j)_{j=1}^{\operatorname{ar}(v)}$, put $\pi_j(v)=c(u_j)$ if $j\le\operatorname{ar}(v)$ and $\pi_j(v)=0$ otherwise, for $j\in\{1,2\}$. Append the record by
\[
\begin{aligned}
\mathsf W[b_{\mathbb W}^{\rm tmp}+4(k-1)]&\gets\operatorname{tag}(v),\\
\mathsf W[b_{\mathbb W}^{\rm tmp}+4(k-1)+1]&\gets\operatorname{ar}(v),\\
\mathsf W[b_{\mathbb W}^{\rm tmp}+4(k-1)+1+j]&\gets\pi_j(v)\qquad(j\in\{1,2\}),\\
\mathsf R[b_{\mathbb R}^{\rm tmp}+k-1]&\gets\lambda(v).
\end{aligned}
\]
For distinct fixed scratch addresses $a,b\in\mathcal B_{\rm pr}$, use the indirect reads and writes
\[
\begin{aligned}
\mathsf W[b]&\gets\mathsf W[\mathsf W[a]],&\mathsf W[\mathsf W[a]]&\gets\mathsf W[b],\\
\mathsf R[b]&\gets\mathsf R[\mathsf W[a]],&\mathsf R[\mathsf W[a]]&\gets\mathsf R[b].
\end{aligned}
\]
Memory-to-memory copies pass through a scratch register. Each record append writes four word fields and one real field. These writes, affine address calculations, and the counter update cost at most $C$ instructions per append and are charged to $W_{\rm emit}$.

The first scan initializes $n_{\rm c}\gets0$ and counts constant tags in $\mathsf W[b_{\mathbb W}^{\rm tmp}+4(k-1)]$, $k\in[K]$. The second reads the tags, computes the ranks from Step~2 of the proof of Theorem~\ref{thrm:complexity__Assembly}, and writes $\mathsf W[\zeta_\iota+c(v)-1]\gets\iota(v)$. The third reads each temporary record's four word fields and one real field and obtains
\[
\begin{aligned}
\ell&\gets\mathsf W[\zeta_\iota+c(v)-1],\\
\rho_j&\gets\begin{cases}\mathsf W[\zeta_\iota+\pi_j(v)-1],&j\le\operatorname{ar}(v),\\0,&j>\operatorname{ar}(v),\end{cases}\qquad j\in\{1,2\}.
\end{aligned}
\]
If $\operatorname{ar}(v)=2$, then one comparison gives $(p_1(v),p_2(v))=(\min\{\rho_1,\rho_2\},\max\{\rho_1,\rho_2\})$; otherwise $(p_1(v),p_2(v))=(\rho_1,\rho_2)$. Put $(f_0,f_1,f_2,f_3)=(\operatorname{tag}(v),\operatorname{ar}(v),p_1(v),p_2(v))$ and write
\[
\begin{aligned}
\mathsf W[b_{\mathbb W}^{\rm out}+2+4(\ell-1)+j]&\gets f_j\qquad(0\le j\le3),\\
\mathsf R[b_{\mathbb R}^{\rm out}+\ell-1]&\gets\mathsf R[b_{\mathbb R}^{\rm tmp}+c(v)-1].
\end{aligned}
\]
After the scan, write the two headers once:
\[
\mathsf W[b_{\mathbb W}^{\rm out}]\gets K,\qquad\mathsf W[b_{\mathbb W}^{\rm out}+1]\gets\mathsf W[\zeta_\iota+c(v_{\rm out})-1].
\]
Unused parent fields cause no map lookup. Each record requires at most two parent-label reads and one sorting comparison. All scan reads, map and output writes, comparisons, address calculations, and loop updates are charged to $W_{\rm scan}$. Bijectivity of $\iota$ gives $4K$ distinct word-field cells and $K$ distinct real-field cells, with the two headers outside the word-field range. Hence
\[
W_{\rm scan}\le C(3K+1)\le CK,\qquad\operatorname{OutputStorage}=4K+2+K=5K+2.
\]

With $A_0=1$, compute the safe point and embed $q$ by
\[
\begin{aligned}
A_i&=A_{i-1}+\begin{cases}X_{i,1},&X_{i,1}\ge0,\\-X_{i,1},&X_{i,1}<0,\end{cases}\qquad(i\in[N]),\\
a_j&=\begin{cases}A_N,&j=1,\\0,&2\le j\le d,\end{cases}\\
q_{\mathbb R}=\mathsf R[32]&\gets\operatorname{embed}(\mathsf W[r_*+1]),\qquad W_a\le C(N+d).
\end{aligned}
\]
The gate parameters in~\eqref{eq:compiler-parameters} use fixed rational literals, input reads, $+,-,\times$, sign tests, and the displayed embedding. After computing $a$, each parameter costs $W_{\lambda(v)}\le C$; neither $D$ nor a power, exponential, or logarithm is evaluated.

Let $L_*$ be the largest fixed word literal. Evaluate word expressions first in $\mathbb N$, guarding each subtraction by its lower index bound. For the set $\mathcal W$ of all word-register values, addresses, counters, and intermediate word-arithmetic results,
\begin{equation}
\label{eq:compiler-ram-capacity}
\begin{aligned}
0&\le(i-1)L\le H_0-L,\qquad 0\le4(k-1)+j<4K_0,\qquad 0\le2j+1<4H_0,\\
\max\mathcal W&\le\max\{q,L_*,4000H_0\},\qquad C_{\rm mem}(d,s)\ge4000L+L_*+1,\\
\max\mathcal W&<q+C_{\rm mem}(d,s)(N+1)<2^w.
\end{aligned}
\end{equation}
Thus word arithmetic agrees with arithmetic in $\mathbb N$, and all addresses are available. The constant $C_{\rm mem}$ is not a program literal. The construction therefore defines a finite program $\mathcal P_{d,s}$ implementing the emitters and scans, whose fixed instructions and literals depend only on $d,s$.
\end{proof}

\section{Non-Analyticity Checks Made In the Introduction}
\label{a:NonAnalticityChecks}
We now compare against two other state-of-the-art Whitney extensions theorems; namely, we check if those constructions produce extensions which are real-analytic away from their training data.  We are only interested in the case of a finite-training dataset, as in our paper.

\subsection{Non Real-Analyticity of \texorpdfstring{\cite{fefferman2009fitting}}{The Classical Whitney Formula}}
\label{a:NonAnalticityChecks__ss:FIII}
Abbreviate $\mathbb X\eqdef\mathbb X^N$ and write $U=\R^d\setminus\mathbb X$. 
For $n\in[N]$, let $e_n:\mathbb X\to\R$ be defined by $e_n(X_j)=\delta_{nj}$ for every $j\in[N]$.
For a linear map $T:\R^{\mathbb X}\to C^0(\R^d)$, put $\phi_n=Te_n$ and define its \emph{value depth} by
\begin{equation}
\label{eq:value-depth}
\operatorname{depth}_0(T)\eqdef\sup_{x\in\R^d}\#\{n\in[N]:\phi_n(x)\ne0\}.
\end{equation}
Linearity gives $(Tf)(x)=\sum_{n=1}^N f(X_n)\phi_n(x)$.

\begin{proposition}[Finite depth obstructs analyticity]
\label{prop:NO_Analytic_FIII}
Suppose $d\ge2$, $(Tf)|_{\mathbb X}=f$ for every $f\in\R^{\mathbb X}$, and $\operatorname{depth}_0(T)<N$. Then $\phi_n|_U\notin C^\omega(U)$ for some $n\in[N]$.
\end{proposition}
In particular, Proposition~\ref{prop:NO_Analytic_FIII} applies to the finite-depth linear $C^s$-extension operator recalled in~\cite[equations~(0.2)-(0.5), pp.~430-431]{fefferman2009fitting} (see also~\cite[Theorem~6(A)--(B), p.~188]{FeffermanKlartag2009II}), with $s\in\mathbb N_+$ and tolerance $\sigma\equiv0$, whenever $N>C_{\mathrm{dep}}(s,d)$, where $C_{\mathrm{dep}}(s,d)$ denotes the bound in equation~(0.3) of that paper.

\begin{proof}[Proof of Proposition~\ref{prop:NO_Analytic_FIII}]
For every $n,j\in[N]$, exact interpolation gives
\begin{equation}
\label{eq:finite-depth-cardinality}
\phi_n(X_j)=(Te_n)(X_j)=e_n(X_j)=\delta_{nj}.
\end{equation}
Fix $n\in[N]$. By~\eqref{eq:finite-depth-cardinality} and continuity, choose $\eta_n>0$ such that $\sup_{x\in B(X_n,\eta_n)}|\phi_n(x)-1|\le1/2$, and set $r_n=\frac12\min\bigl(\{1,\eta_n\}\cup\{|X_n-X_j|:j\in[N]\setminus\{n\}\}\bigr)>0$. Then
\begin{equation}
\label{eq:finite-depth-nonvanishing}
\begin{aligned}
&\varnothing\ne B(X_n,r_n)\setminus\{X_n\}\subset U,\\
&|\phi_n(x)|\ge1-|\phi_n(x)-1|\ge\tfrac12
\qquad\text{for every }x\in B(X_n,r_n).
\end{aligned}
\end{equation}

For $a,b\in U$, set $L_{2n-1}=a+\R(X_n-a)$ and $L_{2n}=b+\R(X_n-b)$ for $n\in[N]$. Each $L_\ell$ is closed and has empty interior, since $\dim L_\ell=1<d$. Starting with $B_0=B(0,1)$, choose successively nonempty open balls $B_\ell\subset B_{\ell-1}\setminus L_\ell$, $\ell\in[2N]$; each difference is nonempty and open by the preceding properties. For $z\in B_{2N}$, define
\[
\gamma(t)=
\begin{cases}
a+2t(z-a),&0\le t\le\tfrac12,\\
b+2(1-t)(z-b),&\tfrac12\le t\le1.
\end{cases}
\]
Then $\gamma\in C([0,1];\R^d)$, $\gamma(0)=a$, and $\gamma(1)=b$. Moreover, $\gamma(t)=X_n$ would imply $z\in L_{2n-1}\cup L_{2n}$, contrary to the choice of $z$. Hence $\gamma([0,1])\subset U$, proving that $U$ is path connected.

For $h\in C^\omega(U)$, set $A_h=\bigcap_{k=0}^{\infty}\{x\in U:D^kh(x)=0\}$. This set is relatively closed because every $D^kh$ is continuous. For $z\in A_h$, choose $\rho_z\in(0,\operatorname{dist}(z,\mathbb X))$ within the radius of convergence of the Taylor series of $h$ at $z$. Thus: for every $|v|<\rho_z$ we have
\[
h(z+v)=\sum_{k=0}^{\infty}\frac{D^kh(z)[v^k]}{k!}=0
.
\]
Consequently, $D^kh|_{B(z,\rho_z)}=0$ for every $k\ge0$, so $B(z,\rho_z)\subset A_h$. Thus $A_h$ is relatively open as well. If $h|_V=0$ for a nonempty open $V\subset U$, then $\varnothing\ne V\subset A_h$; connectedness therefore gives $A_h=U$ and $h|_U=0$.

Suppose, towards a contradiction, that $\phi_n|_U\in C^\omega(U)$ for every $n\in[N]$, and put $G_n=\{x\in U:\phi_n(x)\ne0\}$. Each $G_n$ is relatively open. If $\operatorname{int}_U(U\setminus G_n)\ne\varnothing$, the identity principle just proved gives $\phi_n|_U=0$, contrary to~\eqref{eq:finite-depth-nonvanishing}. Hence $\operatorname{int}_U(U\setminus G_n)=\varnothing$ and $\overline{G_n}^{\,U}=U$. Set $V_0=U$ and $V_n=V_{n-1}\cap G_n$ for $n\in[N]$. Since $G_n$ is open and dense in $U$, induction gives that every $V_n$ is nonempty and relatively open. Choosing $x_*\in V_N$ and using~\eqref{eq:value-depth}, we obtain the contradiction
\[
N=\#\{n\in[N]:\phi_n(x_*)\ne0\}\le\operatorname{depth}_0(T)<N.
\]

For the application to Part~III, equation~(0.5) with $\sigma\equiv0$ gives $(Tf)|_{\mathbb X}=f$. Let $S(x)\subseteq\mathbb X$ and $a_y(x)\in\R$ denote the sets and coefficients in equations~(0.3)-(0.4) with derivative order $\alpha=0$; these are independent of the input data. For every $x\in\R^d$: if $X_n\notin S(x)$ then
$\phi_n(x)=\sum_{y\in S(x)}a_y(x)e_n(y)=0$
\begin{align*}
\#\{n\in[N]:\phi_n(x)\ne0\}
&\le\#\{n\in[N]:X_n\in S(x)\}
=\#S(x)\le C_{\mathrm{dep}}(s,d).
\end{align*}
Taking the supremum over $x$ gives $\operatorname{depth}_0(T)\le C_{\mathrm{dep}}(s,d)<N$, as required.
\end{proof}

\subsection{Non Real-Analyticity of \texorpdfstring{\cite{ParusinskiRainer2024}}{The Definable Whitney Extension}}
\label{s:NonAnalyticityObstructions__ss:OminimalGuy}

Finally, we consider the potential for real-analyticity of the model-theoretic Whitney-extension construction in~\cite{ParusinskiRainer2024}.

\begin{proposition}[Nonanalyticity of the finite-set Parusi\'nski-Rainer extension]
\label{prop:parusinski_rainer_nonanalyticity}
Suppose $d\ge2$, and let $f$ be the extension of a nonzero Whitney jet of order $s$ on $\mathbb X$ obtained by the construction in~\cite[proof of Theorem~1.3, pp.~334-336]{ParusinskiRainer2024}. Then $f$ has compact support and $f|_U\notin C^\omega(U)$.
\end{proposition}
\begin{proof}
Set $K\eqdef\bigcup_{a\in\mathbb X}\overline{B(a,1)}$. We first prove that the finite-set recursion produces an extension supported in $K$. More precisely, for $S\subseteq\mathbb X$ and an input jet $\mathcal J=(J^\alpha)_{|\alpha|\le s}$, where $J^\alpha:\mathbb X\to\R$ and $J^\alpha|_{\mathbb X\setminus S}=0$, we prove this assertion by induction on $|S|$, with $S$ as the auxiliary support set in that recursion. If $\mathcal J=0$, the construction returns zero; this includes the case $S=\varnothing$.

Suppose $|S|\ge1$ and that the assertion holds for smaller auxiliary support sets. Let $a\in S$ be the first point selected by the recursion and put $A\eqdef\mathbb X\setminus(S\setminus\{a\})$. Then $J^\alpha|_{A\setminus\{a\}}=0$ for $|\alpha|\le s$. The one-point construction in~\cite[Lemma~4.1, p.~335]{ParusinskiRainer2024}, applied to $\mathcal J|_A$, uses $r\eqdef\min(\{1\}\cup\{|a-b|:b\in A\setminus\{a\}\})\in(0,1]$ and the function
\[
h(x)\eqdef\chi\!\left(\frac{x-a}{r}\right)\sum_{|\alpha|\le s}\frac{J^\alpha(a)}{\alpha!}(x-a)^\alpha,\qquad x\in\R^d,
\]
where $\chi\in C^{s+1}(\R^d)$ is the cutoff used in the construction, with $\chi=1$ near $0$ and $\operatorname{supp}\chi\subset B(0,1)$. The formula also gives $h=0$ if $\mathcal J|_A=0$. Since $\operatorname{supp}h\subseteq a+r\operatorname{supp}\chi\subset B(a,r)$ and $r\le|a-b|$ for $b\in A\setminus\{a\}$, the function $h$ vanishes on a neighborhood of each such $b$. At $a$, the cutoff is identically one on a neighborhood. Consequently, for every $|\alpha|\le s$,
\[
\partial^\alpha h(b)=J^\alpha(b)\quad(b\in A),\qquad \widetilde J^\alpha\eqdef J^\alpha-\partial^\alpha h|_{\mathbb X},\qquad \widetilde J^\alpha|_{\mathbb X\setminus(S\setminus\{a\})}=0.
\]
By~\cite[proof of Theorem~1.3, p.~336]{ParusinskiRainer2024}, the recursion returns $u=h+\widetilde u$, where $\widetilde u$ is the recursively constructed extension of $(\widetilde J^\alpha)_{|\alpha|\le s}$ with auxiliary support set $S\setminus\{a\}$. The induction hypothesis gives $\operatorname{supp}\widetilde u\subseteq K$, and hence $\operatorname{supp}u\subseteq\operatorname{supp}h\cup\operatorname{supp}\widetilde u\subseteq\overline{B(a,1)}\cup K=K$. This closes the induction. Applied to the prescribed jet, it yields $\operatorname{supp}f\subseteq K$; the set $K$ is compact because $\mathbb X$ is finite.

Suppose now that $f|_U\in C^\omega(U)$, and fix $x\in U$. Choose a unit vector $v$ outside the finite set $\{\pm(a-x)/|a-x|:a\in\mathbb X\}$; such a choice is possible because $d\ge2$. Then $(x+\R v)\cap\mathbb X=\varnothing$, so $g(t)\eqdef f(x+tv)$ is real-analytic on $\R$. Choose $R>0$ with $K\subset B(0,R)$. For $|t|>R+|x|$, the inequality $|x+tv|\ge|t|-|x|>R$ gives $g(t)=0$. Thus $Z\eqdef\bigcap_{j=0}^{\infty}(g^{(j)})^{-1}(\{0\})$ is nonempty and closed. For each $t_0\in Z$, real analyticity supplies $\varepsilon>0$ such that $g(t)=\sum_{j=0}^{\infty}g^{(j)}(t_0)(t-t_0)^j/j!=0$ whenever $|t-t_0|<\varepsilon$. All derivatives of $g$ therefore vanish on this interval, proving that $Z$ is open. Connectedness of $\R$ yields $Z=\R$, and hence $f(x)=g(0)=0$. Since $x$ was arbitrary, $f$ vanishes on $U$; continuity and finiteness of $\mathbb X$ then give $f\equiv0$ on $\R^d$. In particular, $\partial^\alpha f(a)=0$ for every $a\in\mathbb X$ and $|\alpha|\le s$, contradicting interpolation of the prescribed nonzero jet.
\end{proof}

\bibliographystyle{acm}
\bibliography{refs}

\end{document}